\documentclass[times,3p]{elsarticle}
\biboptions{sort,compress}
\usepackage[labelfont=bf]{caption}

\newcommand{\comment}[1]{}

\usepackage{booktabs}
\usepackage{csvsimple}
\usepackage{multirow}
\usepackage{amsmath,amsfonts,amssymb,amsthm,booktabs,color,epsfig,graphicx,hyperref,url}
\usepackage{dutchcal}
\usepackage{enumitem}
\theoremstyle{plain}
\newtheorem{theorem}{Theorem}

\newtheorem{lemma}{Lemma}
\newtheorem{corollary}{Corollary}

\theoremstyle{definition}

\newtheorem{remark}{Remark}
\newtheorem{example}{Example}

\newtheorem{assump}{Assumption}
\usepackage{tocloft}

\begin{document}

\begin{frontmatter}

\title{Identifiability, Convergence and Nonparametric Estimation of Bivariate Archimax Copulas}

\author[1]{Nicolas Dietrich}
\author[2]{Wolfgang Trutschnig}

\address[1]{IDA Lab Salzburg, Department for Artificial Intelligence \& Human Interfaces, Jakob-Haringer-Straße 6, Techno 6, 5020 Salzburg, Austria, 
\url{nicolaspascal.dietrich@plus.ac.at}}
\address[2]{IDA Lab Salzburg, Department for Artificial Intelligence \& Human Interfaces, Hellbrunnerstrasse 34, 5020 Salzburg, Austria, \url{wolfgang@trutschnig.net}}

\cortext[mycorrespondingauthor]{Corresponding author. Email address: \url{nicolaspascal.dietrich@plus.ac.at}}

\begin{abstract}
Considering that the family of bivariate Archimax copulas contains both the Archimedean and the extreme-value class, Archimax copulas constitute a flexible family allowing to model extreme and moderate levels of dependence. Despite their natural appeal, no fully nonparametric, consistent estimator which itself is an element of the Archimax family $\mathcal{C}_{am}$ has been established in the literature yet -- the main reason being that Archimax copulas are not identifiable from the involved Archimedean generator and Pickands dependence function alone. Trying to close this gap, here we resolve the identifiability issue by working with transformed generators and transformed Pickands dependence functions, and show that these functions indeed identify the Archimax copula uniquely. Building upon this result, we then prove that uniform convergence of Archimax copulas is equivalent to uniform convergence of the correspon\-ding transformed generators and Pickands dependence functions. Moreover, as it is the case for Archimedean and extreme-value copulas, in $\mathcal{C}_{am}$, uniform convergence can be shown to be equivalent to weak convergence of almost all conditional distributions. Exploiting these equivalences and working with the Kendall distribution function, we then construct two nonparametric estimators for Archimax copulas -- a Pickands as well as a CFG type estimator, both being elements of $\mathcal{C}_{am}$ -- and show that these estimators are strongly consistent under mild regularity conditions. As another consequence of the afore-mentioned weak conditional convergence, we obtain strongly consistent plug-in estimators for measures of directed dependence such as Chatterjee's $\xi$ and Trutschnig's $\zeta_1$. A large-scale simulation study demonstrates that the proposed CFG type estimator outperforms both the standard empirical co\-pula estimator and the Pickands type estimator; an application to precipitation data from Bregenz and Dornbirn (Austria) illustrates the practical usage of our estimators on real data.
\end{abstract}


\end{frontmatter}

\tableofcontents
\section{Introduction\label{sec:1}}
Given an i.i.d. sample $(X_1,Y_1), (X_2,Y_2), \ldots, (X_n,Y_n)$ of a random vector $(X,Y)$, understanding the dependence structure between the componentwise block maxima $\max_{i\leq n} X_i$ and $\max_{i\leq n} Y_i$ is of substantial interest in applications such as hydrology and finance (see \citep{longin, mcneil2005,salvadori}). The asymptotic regime, i.e., the behavior of this dependence structure as $n \rightarrow \infty$, has been studied extensively in \citep{bucher2011, caperaa1997, dur_princ, genest2009, haan1977, Mai2011, nelsen2006, pickands1981, evc-mass}. The limit dependence structure is well known to correspond 
to so-called extreme-value copulas, which arise as the copulas corresponding to the weak limit of the distribution of the componentwise maxima. 
The fundamental results of \citep{dur_princ,haan1977,nelsen2006,pickands1981} show that bivariate extreme-value copulas admit a particularly convenient analytic representation: every such copula can be characterized in terms of a Pickands dependence function $A$ (see Subsection~\ref{subsec:intro_evc} for a precise definition). Specifically, a bivariate copula $C \colon [0,1]^2 \rightarrow [0,1]$ is an extreme-value copula if and only if there exists a Pickands dependence function $A$ such that
\begin{equation}\label{eq:def.evc.long.and.short}
   C(x,y) = \exp\left((\log(x) + \log(y))A\left(\frac{\log(x)}{\log(x) + \log(y)}\right)\right) = (xy)^{A\left(\frac{\log(x)}{\log(x) + \log(y)}\right)}
\end{equation}
holds for every $(x,y) \in (0,1)^2$.\\
In practice, however, sample sizes are always finite, and the limiting extreme-value model may fail to adequately capture the dependence structure between block maxima. Motivated by this, \citep{caperaa2000a} studied copulas lying in the domain of attraction of extreme-value distributions. One prominent class of copulas with this property (under mild regularity assumptions on the Archimedean generator) is the 
so-called Archimax family $\mathcal{C}_{am}$ first introduced in \citep{caperaa2000b}, which is obtained by replacing the exponential function and the logarithm in the extreme-value representation \eqref{eq:def.evc.long.and.short}
by a distortion function $\psi$, called an Archimedean generator, and its pseudo-inverse $\varphi$. According to \citep{caperaa2000b}, a bivariate 
copula $C$ is called an Archimax copula if there exist an Archimedean generator $\psi$ with pseudo-inverse $\varphi$, and a Pickands dependence function $A$, such that
\begin{equation}\label{eq:am_cop_formula}
C(x,y) = \psi\left((\varphi(x) + \varphi(y))A\left(\frac{\varphi(x)}{\varphi(x) + \varphi(y)}\right)\right)
\end{equation}
holds for every $(x,y) \in (0,1)^2$.\\
While nonparametric estimation of bivariate extreme-value copulas is well studied, 
with consistency and asymptotic results derived, inter alia, in
\citep{bucher2011, caperaa1997, cormier2014, filsvilletard2008, genest2009, ghorbal2009, segers2004}, in the Archimax setting, to the best of the authors' knowledge, 
even in the bivariate case fully nonparametric estimators have not been 
established in the literature yet. This gap is most likely attributable to the fact that, unlike in the extreme-value case, the Archimax framework requires 
simultaneous estimation of both, the generator $\psi$ and the Pickands dependence function $A$. The existing literature offers a semiparametric estimator for the multivariate Archimax family, whose asymptotic behavior has been studied under suitable regularity conditions (see \citep{est-archimax}), as well as a neural-network-based estimator for $\psi$ and $A$ jointly (see \citep{ng2022}); the latter, however, lacks a formal proof of consistency.
Going beyond the bivariate setting, extensions have been proposed recently (see \citep{mult-archimax, mesiar2013}). Moreover, non-exchangeable Archimax copulas and their fitting to real data has been considered in \citep{bacigal2011}, clustered Archimax copulas have been studied in \citep{chatelain2025}, and hierarchical Archimax copulas in \citep{hofert2018}.\\

The main objective of this contribution is to fill this gap by deriving a fully nonparametric estimator for elements in $\mathcal{C}_{am}$, which itself is of 
Archimax type. To this end, we first address the issue of identifiability 
in $\mathcal{C}_{am}$. As shown in \citep{est-archimax}, an Archimax copula $C_{\psi,A}$ can, in general, not be identified from the generator $\psi$ and the Pickands dependence function $A$ alone. We resolve this issue by 
representing Archimax copulas in terms of so-called transformed generators $(\psi)_{1-\tau_A}$ and transformed Pickands dependence functions $(A)_{1-\tau_A}$, where $\tau_A$ is Kendall's $\tau$ of the extreme-value copula $C_A$ associated with $A$. 
More precisely, we show that every Kendall distribution function $F_{\psi,A}^K$ of an Archimax copula $C_{\psi,A}$ -- except for the case when the Pickands dependence function corresponds to the Fr\'echet--Hoeffding upper bound $M(x,y) = \min\{x,y\}$ -- induces an Archimedean copula whose generator is $(\psi)_{1-\tau_A}$.
Using these transformed functions it can be shown that $\mathcal{C}_{am}$ 
is indeed identifiable, and that, as a direct consequence, $C_{\psi,A}$ is identified by $\psi$ and $A$ alone, only if Kendall's $\tau$ of $C_A$ is held fixed.\\
Motivated by the fact that for both, the Archimedean and the extreme-value family, 
uniform convergence is equivalent to uniform convergence of the corresponding generators and Pickands dependence functions, respectively (see \citep{charpentier2008, bernoulli}), we then establish an analogous result for the Archimax family and the corresponding transformed generators and transformed Pickands dependence functions.
Moreover, we characterize convergence in $\mathcal{C}_{am}$ in several 
ways, so-called weak conditional convergence being key. 
Following \citep{bernoulli} we say that a sequence $(C_n)_{n \in \mathbb{N}}$ 
of copulas converges weakly conditional to a copula $C$ if and only
if almost all (regular) conditional distributions converge weakly. 
It is straightforward to verify that weak conditional 
convergence is a much stronger concept than standard uniform convergence 
of copulas. As shown in \citep{bernoulli}, however, the two notions turn out 
to be equivalent in the Archimedean and the extreme-value family. 
Considering that $\mathcal{C}_{am}$ contains both of the afore-mentioned classes, 
as one of our key results we show that weak conditional convergence 
and uniform convergence are equivalent in $\mathcal{C}_{am}$. 
This very result has direct implications to nonparametric estimation and to 
the behavior of dependence measures such as Chatterjee's $\xi$ or Trutschnig's $\zeta_1$ (see \citep{chatterjee2021, JGT, tru06}) -- 
both of these measures compare conditional distributions with unconditional 
ones and can therefore be expressed in terms of Markov kernels, if 
almost all conditional distributions converge, then so do the dependence measures. 
As a direct consequence, in $\mathcal{C}_{am}$ uniform convergence of 
a sequence $(C_n)_{n \in \mathbb{N}}$ implies convergence of 
$(\xi(C_n))_{n \in \mathbb{N}}$ and of $(\zeta_1(C_n))_{n \in \mathbb{N}}$, 
a property that we will directly apply by simply plugging-in our estimators 
in the expressions for $\xi$ and $\zeta_1$. \\
Since the Kendall distribution function $F_{\psi,A}^K$ of $C_{\psi,A}$ 
can be viewed as the Kendall distribution function of an Archimedean copula 
with generator $(\psi)_{1-\tau_A}$, our natural first step in the estimation 
procedure consists of estimating this pseudo-inverse via the empirical Kendall distribution function $K_n$, yielding an estimator $\varphi_n$. Using $\varphi_n$ as a plug-in, we then estimate the transformed Pickands dependence function $(A)_{1-\tau_A}$ via both a Pickands type estimator $B_{n,c}^\mathbf{P}$ and a CFG type estimator $B_{n,c}^\mathbf{CFG}$, following the approaches in \citep{caperaa1997, est-archimax, genest2009}. 
Plugging in these estimators into \eqref{eq:am_cop_formula} yields strongly consistent estimators $D_n^\mathbf{P}$ and $D_n^\mathbf{CFG}$ for the copula $C_{\psi,A}$; these estimators, however, are not even proper distribution functions as they may fail to be $2$-increasing.\\
Ensuring that our estimators are not only copulas
but even elements of $\mathcal{C}_{am}$, we regularize 
$\varphi_n$, $B_{n,c}^\mathbf{P}$, and $B_{n,c}^\mathbf{CFG}$ so that they take the transformed forms $(\varphi)_{1-\tau_A}$ and $(A)_{1-\tau_A}$, and 
proceed as follows: we first project $B_{n,c}^\mathbf{P}$ and $B_{n,c}^\mathbf{CFG}$ into the space of Pickands dependence functions $\mathcal{A}$, and then map the result back into the space of transformed Pickands dependence functions $\{(A)_{1-\tau_A} \colon A \in \mathcal{A}\}$ by computing Kendall's $\tau$ of the extreme-value copula induced by the projected Pickands dependence function. Analogously, we adjust $\varphi_n$ to the transformed form $(\varphi)_{1-\tau_A}$ working with the greatest convex minorant. After an optimization step (assuring that we are as close as possible to our original estimators $D_n^\mathbf{P}$ and $D_n^\mathbf{CFG}$), this altogether yields estimators $(A_n)_{1-\tau_{A_n}}$ and $(\varphi_n)_{1-\tau_{A_n}}$ of the desired form, which converge to the transformed functions $(A)_{1-\tau_A}$ and $(\varphi)_{1-\tau_A}$, respectively. Using these as plug-in estimators for the copula itself finally yields strongly consistent estimators which are themselves elements of $\mathcal{C}_{am}$, i.e., we remain within the $\mathcal{C}_{am}$ family.\\
In order to assess the performance of our estimators, we conduct a large-scale
simulation study comparing different Archimax copula models across various Archimedean generators, Pickands dependence functions (both parametric and nonparametric), sample sizes, and types of association. 
Comparing the Pickands type and CFG type estimators, we find that the CFG type estimator outperforms both the empirical copula estimator and the Pickands type estimator. We further assess the estimators' performance by using them as plug-in estimators for dependence measures mentioned before, comparing our plug-in estimators for Chatterjee's $\xi$ to the nearest-neighbor estimator of \citep{chatterjee2021}. Unsurprisingly, the CFG type estimator again outperforms both the Pickands type and the nearest-neighbor approach.\\
Finally, we apply the CFG type estimator to real rainfall data, consisting of
the daily maximal precipitation per month from the two cities of 
Bregenz and Dornbirn (Austria) to model their dependence structure. 
Doing so we proceed as follows: Considering that we are working with maxima, 
it might seem natural to assume that the dependence structure is of 
extreme-value type -- we therefore fit both our estimator and the nonparametric CFG estimator for extreme-value copulas presented in \citep{genest2009} to the data, 
and compare both to the empirical copula estimator. 
It turns out that our CFG type estimator is, overall, closer to the empirical estimator than the CFG estimator of \citep{genest2009}. To further quantify the 
extent of dependence in the data, we also use the extreme-value estimator and our own estimator as plug-in estimators for the dependence measures $\xi$ and 
$\zeta_1$. Again, our CFG type estimator yields values closer to those obtained from the estimators proposed by \citep{chatterjee2021} and \citep{JGT}.\\[4pt]
The remainder of this paper is organized as follows. Section \ref{section:notation} collects the notation used throughout the paper and is organized into two subsections: the first covers general notation, the second introduces the notation specific to Archimedean, extreme-value, and Archimax copulas. 
We also included a simplified version of the Markov kernel for Archimax copulas, originally derived in \citep{dietrich2026}.
Section \ref{sec:identifiability} addresses identifiability in 
$\mathcal{C}_{am}$. After introducing transformed generators and transformed Pickands dependence functions and establishing their key analytical properties, we prove that two Archimax copulas coincide if and only if their transformed generators and Pickands dependence functions coincide -- yielding a new identifiability result for the Archimax copula model. Several 
examples illustrate the underlying ideas and results obtained.\\
Section \ref{sec:conv_archimax} establishes various convergence results 
in the class $\mathcal{C}_{am} \setminus{\{M\}}$. 
In particular we show that the following assertions are equivalent for 
$C,C_1,C_2,\ldots \in \mathcal{C}_{am}$: (i) $(C_n)_{n \in \mathbb{N}}$ 
converges uniformly to $C$; (ii) the corresponding transformed generators and Pickands dependence functions converge; (iii)  $(C_n)_{n \in \mathbb{N}}$ 
converges weakly conditional $C$.\\
Section \ref{sec:estimator} introduces our novel nonparametric 
copula estimators and proves strong consistency under 
mild regularity assumptions. For the sake of readability, all technical 
proofs are deferred to \ref{sec:proof_consistency}.\\
Section \ref{sec:sim_study} contains the afore-mentioned 
simulation study, using the Integrated Squared Error, the Integrated Relative Absolute Error, and the uniform metric as performance indicators. 
The results obtained are illustrated through a series of plots and figures, with full results collected in various tables in \ref{sec:sim_study_app}.
Finally, in Section \ref{sec:real_data}, we apply the CFG type estimator to 
real rainfall data.
\section{Notation and preliminaries}\label{section:notation}
\subsection{General notation and conventions}\label{subsection:general_notation}
For an arbitrary metric space $(S,d)$, we denote the Borel $\sigma$-field on $S$ by $\mathcal{B}(S)$, and write $\mathcal{P}(S)$ for the family of all probability measures on $\mathcal{B}(S)$. For $s \in S$, the Dirac measure at $s$ is denoted by $\delta_s$. Furthermore, for $\nu \in \mathcal{P}(S)$, the support $\mathrm{supp}(\nu)$ of $\nu$ is the complement of the union of all open sets $U$ satisfying $\nu(U) = 0$. \\
For metric spaces $(S,d)$ and $(S',d')$, a Borel measurable transformation \mbox{$T:S \rightarrow S'$}, and $\nu \in \mathcal{P}(S)$, we denote by $\nu^T$ the push-forward of $\nu$ under $T$, defined by $\nu^T(F):=\nu(T^{-1}(F))$ for every $F \in \mathcal{B}(S')$.
Writing $\mathbb{I}:=[0,1]$, the Lebesgue measure on $\mathcal{B}(\mathbb{I}^2)$ and $\mathcal{B}(\mathbb{I})$ is denoted by $\lambda_2$ and $\lambda$, respectively.\\
We call a function $G\colon \mathbb{I} \rightarrow \mathbb{I}$ a distribution function/measure generating function if its extension to $\mathbb{R}$, obtained by setting $G(x) = 0$ for $x < 0$ and $G(x) = 1$ for $x > 1$, is itself a distribution function/measure generating function. \comment{Similarly, a function $H\colon [0,\infty) \rightarrow \mathbb{I}$ is called a distribution function if its extension to $\mathbb{R}$, obtained by setting $H(x) = 0$ for $x < 0$, is itself a distribution function.} Moreover, for an interval $J \subseteq \mathbb{R}$ and an arbitrary function $u \colon J \rightarrow \mathbb{R}$, we denote the right-hand and left-hand derivatives of $u$, when they exist, by $D^+u$ and $D^-u$, respectively.
In what follows, we consistently adopt the convention that $a\cdot\infty = \infty$, $\frac{a}{\infty} = 0$, and $\frac{a}{0} = \infty$ for every $a \in \mathbb{R} \setminus \{0\}$.\\[4pt]
Throughout this contribution $\mathcal{C}$ denotes the family of all bivariate copulas. For a given random vector $(X,Y)$ on a probability space $(\Omega, \mathcal{F}, \mathbb{P})$ and some copula $C \in \mathcal{C}$, we will write $(X,Y) \sim C$ if $C$ is the joint distribution function of $(X,Y)$ restricted to $\mathbb{I}^2$.
Analogously, for an arbitrary bivariate distribution 
function $H$ writing $(X,Y)\sim H$ indicates that $H$ is the distribution function 
of $(X,Y)$. The doubly stochastic measure corresponding to the copula $C\in\mathcal{C}$ will be denoted by $\mu_C$, i.e., $\mu_C([0,x] \times [0,y]) := C(x,y)$ 
for every $(x,y) \in \mathbb{I}^2$. 
$M$ will denote the minimum copula, 
$W$ the lower Fréchet--Hoeffding bound, and $\Pi$ the product 
(independence) copula.
Equipping the family $\mathcal{C}$ with the uniform metric $d_\infty \colon \mathcal{C}^2 \rightarrow \mathbb{I}$, defined by
$$
d_\infty(C_1,C_2) = \max_{(x,y) \in \mathbb{I}^2}|C_1(x,y)- C_2(x,y)|
$$
for every $C_1,C_2 \in \mathcal{C}$, yields a compact metric space
$(\mathcal{C},d_\infty)$. To keep notation simple, in 
the sequel we will at some points also write $d_\infty(f,g)$ for arbitrary 
functions $f,g: \mathbb{I}^2 \rightarrow \mathbb{R}$ which are not necessarily copulas.\\
Suppose that $(X,Y)$ has continuous distribution function 
$H$, marginals $F,G$, and underlying (unique) copula $C$. 
For every sample $(X_1,Y_1),\ldots,(X_n,Y_n)$ of $(X,Y)$, writing
$H_n$ for the bivariate, and $F_n,G_n$ for the marginal distribution functions, 
Sklar's theorem yields the existence of some copula $C_n \in \mathcal{C}$ 
fulfilling $H_n(x,y)=C_n(F_n(x),G_n(y))$ for all $(x,y) \in \mathbb{R}^2$. 
Formally, $C_n$ is only uniquely determined on 
$\textrm{Range}(F_n) \times \textrm{Range}(G_n)$ and many extensions to 
a copula are possible. In fact, considering 
$I^n_i:=[\frac{i-1}{n},\frac{i}{n}]$ for every $i \in \{1,\ldots,n\}$, setting
$R^n_{i,j}:=I^n_i \times I^n_j$ for all $i,j \in \{1,\ldots,n\}$, 
defining the similarity $f^n_{i,j}: \mathbb{I}^2 \rightarrow R^n_{i,j}$ by 
$f^n_{i,j}(x,y)=(\frac{x+i-1}{n}, \frac{y+j-1}{n})$, and
letting $(U_i,V_i):=(F_n(X_i),G_n(Y_i)) $ denote the 
$i$-th pseudo-observation, the following property holds: for every 
copula $B \in \mathcal{C}$, the measure $\mu \in \mathcal{P}(\mathbb{I}^2)$, defined
by 
\begin{equation}\label{eq:emp.general}
    \mu(E_1 \times E_2):= \frac{1}{n} \sum_{i=1}^n \mu_B^{f^n_{(nU_i,nV_i)}}(E_1 \times E_2) = \frac{1}{n} \sum_{i=1}^n 
    \mu_B \left((f^n_{(nU_i,nV_i)})^{-1} (E_1 \times E_2)\right),\qquad E_1,E_2 \in 
    \mathcal{B}(\mathbb{I})
\end{equation}
is doubly stochastic and therefore corresponds to a unique copula $C_n^B$
to which we will refer to as the empirical $B$-copula, i.e., with the 
notation from before we have $\mu=\mu_{C_n^B}$. 
Whenever the `interpolation' $B$ is relevant, we will explicitly write $C_n^B$,
if not, we will simply consider the standard bilinear interpolation 
and write $C_n:=C_n^\Pi$.
For further information on copulas and doubly stochastic measures we refer to \citep{dur_princ,nelsen2006}.

In the sequel, conditional distributions and Markov kernels will be key concepts.
A map $K: \mathbb{R}\times\mathcal{B}(\mathbb{R}) \rightarrow \mathbb{I}$ is called a Markov kernel from $\mathbb{R}$ to $\mathbb{R}$ if, and only if the function $x\mapsto K(x,E)$ is Borel-measurable for every $E\in\mathcal{B}(\mathbb{R})$ and the map $E\mapsto K(x,E)$ is a probability measure on $\mathcal{B}(\mathbb{R})$ for every $x\in\mathbb{R}$. 
Given a random vector $(X,Y)$ on a probability space $(\Omega,\mathcal{F},\mathbb{P})$, we will call a Markov kernel $K$ a regular conditional distribution of
$Y$ given $X$ if for every set $E \in \mathcal{B}(\mathbb{R})$ the equation
$$
K(X(\omega), E) = \mathbb{E}[\mathbf{1}_E \circ Y | X](\omega)
$$
holds for $\mathbb{P}$-almost every $\omega \in \Omega$.
It is a well known result that for every random vector $(X,Y)$, there exists 
a regular conditional 
distribution $K$ of $Y$ given $X$, which is unique for $\mathbb{P}^{X}$-almost every $x\in \mathbb{R}$.\\
In the case that $C \in \mathcal{C}$ and  $(X,Y) \sim C$, we will let $K_{C}:\mathbb{I} \times \mathcal{B}(\mathbb{I}) \to \mathbb{I}$ denote 
(a version of) the regular conditional 
distribution of $Y$ given $X$ and simply refer to it as 
the Markov kernel of $C$.
Considering $x \in \mathbb{I}$, we define the $x$-section $G_x$ of a given set $G \in\mathcal{B}(\mathbb{I}^2)$ by $G_{x}:=\{y
\in \mathbb{I}: (x,y) \in G\}\in\mathcal{B}(\mathbb{I})$. Using \emph{disintegration} (see \citep[Section 5]{kallenberg} and \citep[Section 8]{klenke}), the subsequent identity holds for all $G \in \mathcal{B}(\mathbb{I}^2)$:
\begin{align}\label{eq:DI}
	\mu_C(G) = \int_{\mathbb{I}} K_{C}(x,G_x)
	\, \mathrm{d}\lambda(x).
\end{align}
\comment{
Considering $d \in \{1,2\}$, in this contribution, a measure $\nu$ on $\mathcal{B}(S)$, whereby $S=\mathbb{I}^d$ or $S=[0,\infty)$, is referred to as singular (w.r.t. $\lambda_d$ or $\lambda$)
if, and only if it has no point-masses and there exists a set $G \in \mathcal{B}(S)$ such that 
$\lambda_d(G) = 0$  and $\nu(G) = \nu(S)$. For the sake of simplicity, we will use the term 'singular' for the afore-mentioned measures, rather than the longer expression 'singular without point-masses'. This aligns with the concept of singular distribution functions and is justified by the following straightforward observation: 
The doubly stochastic measure $\mu_C$ corresponding to the copula $C \in \mathcal{C}$ has no point masses and therefore always has degenerated discrete component.\\[4pt] 
}
Markov kernels were used in \cite{tru06} in order to construct metrics 
capable of strictly separating independence from complete dependence. 
In fact, defining $D_p \colon \mathcal{C}^2 \rightarrow \mathbb{I}$ by
\begin{equation}\label{eq:metric_d_p}
D_p(C_1,C_2) := \left(\int_{\mathbb{I}^2}|K_{C_1}(x,[0,y]) - K_{C_2}(x,[0,y])|^p \mathrm{d}\lambda_2(x,y)\right)^\frac{1}{p}
\end{equation}
for every $p \in [1,\infty)$ and 
\begin{equation}\label{eq:metric_infty}
D_\infty(C_1,C_2) := \sup_{y \in \mathbb{I}}\int_{\mathbb{I}}|K_{C_1}(x,[0,y]) - K_{C_2}(x,[0,y])| \mathrm{d}\lambda(x),
\end{equation}
for $p = \infty$ yields complete and separable metric spaces $(\mathcal{C},D_p)$. 
Using $D_p$, it is straightforward to construct dependence 
measures, i.e., mappings $\eta: \mathcal{C} \rightarrow \mathbb{I}$ with the following properties:
(i) $\eta(C)=0$ if and only if $C=\Pi$, (ii) $\eta(C)=1$ if and only if $C$ 
is completely dependent, i.e., if there exists some $\lambda$-preserving transformation
$g: \mathbb{I} \rightarrow \mathbb{I}$ such that for $(X,Y) \sim C$ we have
$Y=g \circ X$ almost surely. The resulting measure from \cite{tru06} is defined by
$\zeta_1(C)=3 D_1(C, \Pi)$, switching to its $L^2$-version yields 
Chatterjee's `New Coefficient of Correlation' $\xi$ studied in \cite{chatterjee2021}. 
Since then, various other measures of dependence have appeared in the literature, 
we refer to \cite{ALFT} and the references therein. All of these dependence 
measures have in common that they are continuous with respect to weak conditional convergence: 
We will say that a sequence $(C_n)_{n \in \mathbb{N}}$ of copulas
converges weakly conditional to $C \in \mathcal{C}$, if and only if there exists some 
$\Lambda \in \mathcal{B}(\mathbb{I})$ with $\lambda(\Lambda)=1$ such that 
for every $x \in \Lambda$ we have that $(K_{C_n}(x,\cdot))_{n \in \mathbb{N}}$
converges weakly to $K_C(x,\cdot)$. According to \cite{tru06} weak conditional 
converges implies convergence with respect to $D_p$, and the latter implies 
convergence with respect to $d_\infty$.
While the two dependence measures $\xi$ and 
$\zeta_1$ are conceptually very similar, their proposed estimators 
build upon different ideas - the estimator for $\xi$ is a nearest neighbor
approach directly estimating $\xi$ (see \cite{chatterjee2021}), while
the one proposed for $\zeta_1$ is just an aggregation of the underlying so-called 
checkerboard approximation/aggregation of the empirical copula $C_n$ (see \cite{JGT}).

For a copula $C \in \mathcal{C}$ with $(U,V) \sim C$, following \citep[Definition 3.9.5]{dur_princ}, the Kendall distribution function $F_C^K: \mathbb{I} \rightarrow \mathbb{I}$ of $C$ is given by
\begin{equation}\label{eq:kendall_dist}
    F_C^K(t) := \mathbb{P}[C(U,V) \leq t], \qquad t \in \mathbb{I}.
\end{equation}
Furthermore, following \citep{dur_princ, GR, joe, nelsen2006}, Kendall's $\tau$ of $C$ is defined as
\begin{equation}\label{eq:kendalls_tau_general}
    \tau(C) := 4\int_{\mathbb{I}^2} C(s,t) \, \mathrm{d}\mu_C(s,t) - 1.
\end{equation}
It is well known (see \citep{dur_princ}) that Kendall's $\tau$ can be obtained from the Kendall distribution function via
\begin{equation}\label{eq:kendalls_tau_alter}
\tau(C)=3-4\int_{\mathbb{I}} F_C^K(t)\,\mathrm{d}\lambda(t).
\end{equation}
Moreover, using \cite[Proposition 5]{GNZ}  or the extended continuous mapping theorem \cite[Theorem 18.11]{vdv}, uniform convergence of a sequence of 
copulas $(C_n)_{n \in \mathbb{N}}$ to a copula $C$ implies weak convergence
of the corresponding Kendall distribution functions, i.e., for every point of continuity $t \in \mathbb{I}$ of $F^K_C$ we have 
$\lim_{n \rightarrow \infty} F^K_{C_n}(t)=F^K_C(t)$. 
Given a sample $(X_1,Y_1),(X_2,Y_2),\ldots$ from 
$(X,Y) \sim H$ with continuous $H$ and underlying copula $C$, it is well known 
(see, e.g., \cite{janssen2012}) that the sequence 
$(C_n^B)_{n \in \mathbb{N}}$ of empirical copulas (for any choice of $B \in \mathcal{C}$) 
converges uniformly to $C$ with probability $1$. As a direct consequence, we
get that the sequence $(F_{C_n^B}^K)_{n \in \mathbb{N}}$ of 
corresponding Kendall distributions converges weakly to $F^K_C$ with probability $1$. We will use this property in Section \ref{sec:estimator} in the construction of our estimators. Moreover, related results on the weak convergence of Kendall's process can be found in \citep{barbe1996}.
\subsection{Preliminaries on Archimedean, extreme-value, and Archimax copulas}
In the following, we give an overview of extreme-value, Archimedean, and Archimax copulas. We show how to characterize these copulas via Archimedean generators and Pickands dependence functions, and define several other functions and measures that are important throughout this work.
\subsubsection{Characterizing Archimedean copulas}\label{subsec:intro_archimedean}
Let $\psi \colon [0,\infty) \rightarrow \mathbb{I}$ be non-increasing and continuous. We call $\psi$ an Archimedean generator (or simply a generator, the terminology we adopt throughout) whenever the following three conditions hold: $(i)$ $\psi(0) = 1$; $(ii)$ $\lim_{z \rightarrow \infty}\psi(z) = 0$; and $(iii)$ $\psi$ is strictly decreasing on the interval $[0,\inf\{x\colon \psi(x) = 0\})$, where $\inf(\emptyset):=+\infty$ by convention.\\
Associated to any such generator $\psi$ is its pseudo-inverse $\varphi \colon \mathbb{I} \rightarrow [0,\infty]$, given by
$$\varphi(y) := \inf\{z \in [0,\infty] \colon \psi(z) = y\}, \qquad y \in \mathbb{I}.$$
As shown in \citep{mult_arch}, $\varphi$ inherits several regularity properties from $\psi$: it is strictly decreasing on $(0,1]$, right-continuous at $0$, and satisfies $\varphi(1)=0$. Depending on whether $\varphi(0) = \infty$ or $\varphi(0) < \infty$, we call $\psi$ and $\varphi$ strict or non-strict, respectively -- equivalently, $\psi$ is strict if and only if $\psi(z) > 0$ holds for all $z \in [0,\infty)$.\\
(Archimedean) generators give rise to the family of Archimedean copulas: $C \in \mathcal{C}$ is Archimedean if it can be written as
\begin{equation}\label{eq:def_arch_cop}
C(x,y) := \psi\left(\varphi(x) + \varphi(y)\right), \qquad (x,y) \in \mathbb{I}^2,
\end{equation}
for some generator $\psi$ with pseudo-inverse $\varphi$. We write $\mathcal{C}_{ar}$ for the class of bivariate Archimedean copulas, and where the choice of generator matters, denote the copula in \eqref{eq:def_arch_cop} by $C_\psi$ rather than $C$.\\
Not every generator yields a copula via \eqref{eq:def_arch_cop}: by \citep[Proposition 2.1]{neslehova}, $C_\psi$ is a bivariate copula if and only if $\psi$ is convex. This convexity has several useful consequences. First, it guarantees that the one-sided derivatives $D^-\psi$ and $D^+\psi$ exist everywhere on $(0,\infty)$; moreover, by \citep[Theorem 3.7.4]{Kannan1996} and \citep[Appendix C]{Pollard2001}, these two derivatives can differ only on a countable subset of $(0,\infty)$, and agree, $D^-\psi(z) = D^+\psi(z)$, at every point of continuity $z$ of $D^-\psi$. Second, in \citep{mult_arch} it was proved that $\lim_{z \rightarrow \infty} D^-\psi(z) = 0$ and that in fact $D^-\psi(z)=0$ already for every $z > \varphi(0)$. Finally, convexity propagates from $\psi$ to its pseudo-inverse $\varphi$ as well.\\
To establish a one-to-one correspondence between the family of convex generators and the family of Archimedean copulas, we restrict ourselves throughout this contribution to normalized generators, i.e., those generators 
satisfying $\psi(1) = \frac{1}{2}$ (equivalently, $\varphi(\frac{1}{2}) = 1$), and write $\Psi$ for the resulting family of convex, normalized generators.\\[4pt]
An alternative route to Archimedean copulas, developed in \citep{mult_arch,neslehova,schilling2012}, proceeds via univariate probability measures rather than generators directly. We call $\gamma \in \mathcal{P}([0,\infty))$ a Williamson measure if $\gamma(\{0\}) = 0$ and $\int_\mathbb{I} (1-t) \mathrm{d}\gamma(t) = \frac{1}{2}$, and denote the family of all such measures by
$$
\mathcal{P}_\mathcal{W} := \left\{\gamma \in \mathcal{P}([0,\infty)) \colon \gamma(\{0\}) = 0, \int_\mathbb{I} (1-t) \mathrm{d}\gamma(t) = \frac{1}{2}\right\}.
$$
The contributions \citep{mult_arch,neslehova,schilling2012} establish a one-to-one correspondence between $\Psi$ and $\mathcal{P}_\mathcal{W}$, in the sense that every $\psi \in \Psi$ arises from a unique $\gamma \in \mathcal{P}_\mathcal{W}$ via
\begin{equation}\label{eq:rel_psi_gamma}
    \psi(z) = \int_{[0,\infty)}(1- tz)_+ \mathrm{d}\gamma(t), \qquad z>0,
\end{equation}
as shown in \citep[Theorem 5.1]{mult_arch}.
\comment{Moreover, if $\gamma \in \mathcal{P}_\mathcal{W}$ and $\psi \in \Psi$ is its corresponding generator, then, according to \citep{mult_arch}, the following representation of the left-hand derivative $D^-\psi$ of $\psi$ in terms of $\gamma$ holds for every $z > 0$:
\begin{equation}\label{eq:rel_deriv_psi_gamma}
    D^-\psi(z) = - \int_{[0,\frac{1}{z}]}t \mathrm{d}\gamma(t).
\end{equation}
Finally, we refer to $\gamma$ as strict if its support contains the point $0$ and otherwise as non-strict. It has been established in \citep[Lemma 5.5]{mult_arch} that $\gamma$ is strict if, and only if its corresponding generator $\psi$ is strict. If $\gamma$ is non-strict and $\varphi$ is the pseudo-inverse of the corresponding generator $\psi$, then $\gamma([0,z]) = 0$ holds for every $z < \frac{1}{\varphi(0)}$ \textcolor{red}{[Kasper, Dietrich, Trutschnig2024]}.}
At last, we define the function
$\beta \colon \mathbb{I} \rightarrow[-1,0]$ by
\begin{equation}\label{eq:beta}
\beta(t) := D^-\psi(\varphi(t))\varphi(t).
\end{equation}
for every $t \in (0,1)$ and $\beta(0) = 0 = \beta(1)$.
It is well known (see \cite{mult_arch}) that for every bivariate 
Archimedean copula $C_\psi$ the Kendall distribution 
function is given by 
\begin{equation}\label{eq:kendall.arch}
   F_{\psi}^K(t) := F_{C_\psi}^K(t)=\gamma\left(\left[0,\tfrac{1}{\varphi(t)}\right]\right) = t-\beta(t) = 
t-\tfrac{\varphi(t)}{D^+\varphi(t)}
\end{equation}
and that the Kendall distribution function characterizes the copula
and the normalized generator (see \citep{dietrich2024, GNZ, GR, bernoulli, mult_arch}).\\
Lastly, using the formula in eq. \eqref{eq:kendalls_tau_general}, for a given generator $\psi \in \Psi$, Kendall's $\tau$ of the Archimedean copula $C_\psi \in \mathcal{C}_{ar}$ is given by
$$
\tau_\psi := \tau(C_\psi) = 1 + 4\int_\mathbb{I}\beta(t) \mathrm{d}\lambda(t).
$$
\subsubsection{Characterizing extreme-value copulas}\label{subsec:intro_evc}
A copula $C \in \mathcal{C}$ is called an extreme-value copula (EVC) if it arises as the limit
$$
C(x,y) = \lim_{n \rightarrow \infty} B^n\!\left(x^\frac{1}{n},y^\frac{1}{n}\right), \qquad x,y \in \mathbb{I},
$$
for some $B \in \mathcal{C}$. We write $\mathcal{C}_{ev}$ for the family of all bivariate EVCs. As established in \citep{dur_princ,haan1977,nelsen2006,pickands1981}, the following 
characterization holds:
\begin{enumerate}
    \item $C \in \mathcal{C}_{ev}$.
    \item $C$ is max-stable, i.e., $C(x,y) = C^k\!\left(x^\frac{1}{k},y^\frac{1}{k}\right)$ for all $k \in \mathbb{N}$ and $x,y \in \mathbb{I}$.
    \item $C$ admits a representation of the form
    \begin{equation}\label{eq:map_eq_pick_copula}
    C(x,y) = (xy)^{A\left(\frac{\log(x)}{\log(xy)}\right)}, \qquad x,y \in (0,1),
    \end{equation}
    for some Pickands dependence function $A$, that is, a convex function $A \colon \mathbb{I} \rightarrow \mathbb{I}$ with 
    $\max\{1-x,x\} \leq A(x) \leq 1$ for every $x \in \mathbb{I}$.
\end{enumerate}
We denote by $\mathcal{A}$ the family of all Pickands dependence functions, and for $A \in \mathcal{A}$ write $C_A$ for the unique EVC it induces. Furthermore, setting $A_M(t):=\max\{t,1-t\}$ and $A_\Pi(t):= 1$ for every $t \in \mathbb{I}$, one readily verifies that $C_{A_M} = M$ and $C_{A_\Pi} = \Pi$. It is straightforward to verify (and, in fact, a direct consequence of the Arzel\`a-Ascoli theorem) that $(\mathcal{A},\Vert \cdot \Vert_\infty)$ is compact, where
$\Vert \cdot \Vert_\infty$ denotes the uniform norm. \\
Convexity implies that for every Pickands dependence function $A \in \mathcal{A}$ the right-hand derivative $D^+A(x)$ exists for every $x \in [0,1)$ and the left-hand derivative $D^-A(x)$ exists for every $x \in (0,1]$.\\ Moreover, $D^+A(x)=D^-A(x)$ holds for all but at most countably many $x \in (0,1)$ and therefore $A$ is differentiable on the complement of a countable set in $(0,1)$. Furthermore, $D^+A$ is non-decreasing and right-continuous on $[0,1)$ and $D^-A$ is non-decreasing and 
left-continuous on $(0,1]$ (see \citep{Kannan1996,Pollard2001}).
In the sequel, we write $\mathrm{Cont}(D^+A)$ for the set of points at which $D^+A$ is continuous. If a sequence $(A_n)_{n \in \mathbb{N}}$ in $\mathcal{A}$ converges to $A \in \mathcal{A}$, then, using convexity and monotonicity it is straightforward to verify that the following assertion holds: for every sequence 
$(t_n)_{n \in \mathbb{N}}$ in $\mathbb{I}$ converging to some $t_0 \in \textrm{Cont}(D^+A) \cap (0,1)$ we have that $D^+A_n(t_n) \overset{n \rightarrow \infty}{\longrightarrow} D^+A(t_0)$.\\ 
Defining $D^+A(1):=D^-A(1)$, we obtain that $D^+A$ is a non-decreasing and right-continuous function on the interval $\mathbb{I}$. Considering that $\max\{1-x,x\} \leq A(x) \leq 1$ for every $x \in \mathbb{I}$, it follows that $D^+A(x) \in [-1,1]$ holds for every $x \in \mathbb{I}$.\\[4pt]
An alternative characterization of EVCs, developed in \citep{beirlant2004, dietrich2024, haan1977, pickands1981}, is given in terms of probability measures on $\mathbb{I}$. We call $\vartheta \in \mathcal{P}(\mathbb{I})$ a Pickands dependence measure if its expected value satisfies $\mathbb{E}[\vartheta] := \int_\mathbb{I}x \, \mathrm{d}\vartheta(x) = \frac{1}{2}$, and define the family of all such measures by
$$
\mathcal{P}_\mathcal{A} := \left\{\vartheta \in \mathcal{P}(\mathbb{I}) \colon \int_\mathbb{I}x \, \mathrm{d}\vartheta(x) = \frac{1}{2}\right\}.
$$
As shown in \citep{dietrich2024}, this yields a one-to-one correspondence with Pickands dependence functions: given $\vartheta \in \mathcal{P}_\mathcal{A}$, \citep[Lemma Appendix A.2]{dietrich2024} shows that the corresponding $A \in \mathcal{A}$ can be expressed via
\begin{equation}\label{eq:rel_pickands_fct_measure}
    A(t) = 1-t +2\int_{[0,t]}\vartheta([0,z]) \, \mathrm{d}\lambda(z) = 
    2\int_{\mathbb{I}} \max\{t(1-s),(1-t)s\}  \, \mathrm{d}\vartheta(s), \qquad t \in \mathbb{I}.
\end{equation}

Fix now $\vartheta \in \mathcal{P}_\mathcal{A}$ with corresponding Pickands dependence function $A \in \mathcal{A}$. Proceeding as in \citep{evc-mass}, associate to $A$ the two boundary points
\begin{equation}\label{eq:definition_L_R}
    L := \max\{s \in \mathbb{I} \colon A(s) = 1-s\}, \qquad R := \min\{s \in \mathbb{I} \colon A(s) = s\}.
\end{equation}

We next introduce a function that plays a central role both in defining the $t$-level-set functions in eq. \eqref{eq:fct_f_t} and in several proofs later in the paper. For $A \in \mathcal{A}$, define $h_A \colon \mathbb{I} \rightarrow [1,\infty]$ by
\begin{equation}\label{eq:fct_h}
h_A(t) := \begin{cases}
    \infty,&\text{if } t = 0,\\
    \frac{A(t)}{t}, &\text{if } t\in (0,1]
\end{cases}
\end{equation}
for every $t \in \mathbb{I}$. By Lemma \ref{lem:regularity_fcts}, $h_A$ is strictly decreasing on $(0,\inf\{x \in \mathbb{I} \colon h_A(x) = 1\}) = (0,R)$, hence invertible there. Following \cite{dietrich2026}, we extend this inverse to a pseudo-inverse $h_A^{[-1]} \colon [1,\infty] \rightarrow \mathbb{I}$ via
$$
h_A^{[-1]}(z) :=
\begin{cases}
    R, &\text{if } z = 1,\\
    h_A^{-1}(z), &\text{if } z \in (1,\infty),\\
    0, &\text{if } z = \infty,
\end{cases}
$$
for $z \in [1,\infty]$; note that when $R=1$, this simplifies to $h_A^{[-1]}(z) = h_A^{-1}(z)$ for all $z \in [1,\infty)$.

To simplify notation, and following \citep{fuchs_tschimpke,evc-mass,dietrich2026}, we further define, for fixed $A \in \mathcal{A}$, the function $G_A \colon \mathbb{I} \rightarrow \mathbb{I}$ by
\begin{equation}\label{eq:funct_G_A}
G_A(t) :=
\begin{cases}
A(t) + D^+A(t)(1-t),& \text{if } t \in [0,1),\\
1,& \text{if } t = 1
\end{cases}
\end{equation}
for every $t \in \mathbb{I}$; note that $G_A(t)$ is simply the value at $1$ of the tangent line to $A$ at $t$.

Considering properties of the functions $h_A$, $h_A^{[-1]}$, and $G_A$, we recall the following lemma, which has already been established in \citep[Lemma 2.1]{dietrich2026}; we restate it for the sake of completeness.
\begin{lemma}\label{lem:regularity_fcts}
Let $A \in \mathcal{A}$, $h_A$ be defined as in eq. \eqref{eq:fct_h} and let $L$, $R$ be defined according to eq. \eqref{eq:definition_L_R}. Then the following assertions hold:
   \begin{enumerate}[label=(\roman*)]
        \item $h_A$ is a non-increasing, convex function on $(0,1]$, which is strictly decreasing on $(0,R]$ and fulfills $h_A(t) = 1$ for every $t \in [R,1]$;
        \item $h_A^{[-1]}$ is a strictly decreasing and convex function on $[1,\infty)$;
        \item $G_A$ is a non-negative, non-decreasing and right-continuous function on $\mathbb{I}$, which is positive on $(L,1]$. Furthermore, if $L>0$, $G_A(t) = 0$ holds for every $t \in [0,L)$ and $G_A(t) = 1$ holds for every $t \in [R,1]$.
    \end{enumerate}
\end{lemma}
Finally, as it will be key for identifiability and the construction of our  estimator in Sections \ref{sec:identifiability} and \ref{sec:estimator}, respectively, for a fixed $A \in \mathcal{A}$, we note that Kendall's $\tau$ of the associated EVC $C_A$ is given as
\begin{equation}\label{eq:kendallstau_ex}
\tau_A := \tau(C_A)= \int_\mathbb{I}\frac{t(1-t)}{A(t)} \mathrm{d}D^+A(t).
\end{equation}
\subsubsection{Bivariate Archimax copulas}
Combining a convex generator with a Pickands dependence function yields a broader family of copulas known as Archimax copulas, which includes both the Archimedean and extreme-value families as subfamilies. Formally, $C \in \mathcal{C}$ belongs to this family if there exist $\psi \in \Psi$, with pseudo-inverse $\varphi$, and $A \in \mathcal{A}$ such that
\begin{equation}\label{eq:def_archimax}
C(x,y) = \psi\left((\varphi(x) + \varphi(y))A\left(\frac{\varphi(x)}{\varphi(x) + \varphi(y)}\right)\right)
\end{equation}
holds for all $(x,y) \in (0,1)^2$; we call $C$ an Archimax copula generated by the pair $(\psi,A)$. The fact that eq. \eqref{eq:def_archimax}, suitably extended to $\mathbb{I}^2$, indeed defines a proper copula is shown in \citep{caperaa2000b,durante2018}. Stressing the connection to the generator $\psi$ and the Pickands dependence function $A$, we sometimes write $C_{\psi,A}$ for the copula in \eqref{eq:def_archimax} rather than $C$. As already mentioned in 
the introduction, $\mathcal{C}_{am}$ will denote the class of all bivariate Archimax copulas.\\
Three particular choices of $(\psi,A)$ recover familiar copula families as special cases of $C_{\psi,A}$:
\begin{enumerate}
    \item Setting $A(t) = 1$ for every $t \in \mathbb{I}$ yields the Archimedean family: for any $\psi \in \Psi$, one directly verifies that $C_{\psi,A} = C_\psi \in \mathcal{C}_{ar}$.
    \item Considering $\psi(z) = \exp(-z\log(2))$ we obtain the extreme-value family: as shown in \citep{caperaa2000b}, $C_{\psi,A} = C_A \in \mathcal{C}_{ev}$ for every $A \in \mathcal{A}$.
    \item Taking $A(x) = \max\{1-x,x\}$ collapses the dependence entirely, regardless of $\psi$: as shown in \citep{est-archimax}, $C_{\psi,A}(x,y) = M(x,y) = \min\{x,y\}$ for every $(x,y) \in \mathbb{I}^2$ and every $\psi \in \Psi$.
\end{enumerate}
Notice that using eq. \eqref{eq:def_archimax}
and the results from \cite{bernoulli}, it is straightforward to prove the following implication: If a sequence $(\psi_n)_{n \in \mathbb{N}}$ of Archimedean 
generators converges uniformly to $\psi \in \Psi$ and if a sequence 
$(A_n)_{n \in \mathbb{N}}$ of Pickands dependence functions converges to 
$A \in \mathcal{A}$, then $C_{\psi_n,A_n} \overset{n \rightarrow \infty}{\longrightarrow} C_{\psi,A}$.
\\
We now turn to a description of the level sets of Archimax copulas, which will be central to establishing Corollary \ref{cor:markov_kernel}. For a formal definition of the $t$-level sets of a copula $C \in \mathcal{C}$, we refer to \citep{dur_princ}. We follow \citep{dietrich2026} and fix $\psi \in \Psi$ with pseudo-inverse $\varphi$, and recall the function $h_A$ from eq. \eqref{eq:fct_h}, together with its pseudo-inverse $h_A^{[-1]}$. For $t \in (0,1)$, if $\psi$ is strict, or $t \in [0,1)$ if $\psi$ is non-strict, we define the $t$-level-set function $f^t \colon [t,1] \rightarrow \mathbb{I}$ by
\begin{equation}\label{eq:fct_f_t}
f^t(x) := 
\begin{cases}
\psi\left(\left(\frac{1}{h_A^{[-1]}\left(\frac{\varphi(t)}{\varphi(x)}\right)}-1\right)\varphi(x)\right), &\text{if } x \in [t,1),\\
t, &\text{if } x =1
\end{cases}
\end{equation}
for every $x \in [t,1]$. The two boundary cases are handled separately: when $\psi$ is strict, we set $f^0(x) := 0$ for every $x \in \mathbb{I}$; and in all cases, $f^1 \colon \{1\} \rightarrow \mathbb{I}$ is given by $f^1(1) := 1$. These functions were shown in \citep{dietrich2026} to parametrize the $t$-level sets of Archimax copulas.\\
Since the identifiability result and the estimators we develop in this paper both hinge on the Kendall distribution function of Archimax copulas, we recall its form here. For $\psi \in \Psi$, $A \in \mathcal{A}$, and the corresponding copula $C_{\psi,A} \in \mathcal{C}_{am}$, according to \citep{caperaa2000b, dietrich2026}, the Kendall distribution function of $C_{\psi,A}$ takes the form
\begin{equation}\label{eq:kendall_archimax}
F_{\psi,A}^K(t) = \tau_A t + (1-\tau_A)F_\psi^K(t) = t + (\tau_A-1)\beta_\psi(t), \qquad t\in \mathbb{I},
\end{equation}
where $\tau_A$ is defined in eq. \eqref{eq:kendallstau_ex}. In other words, the Kendall distribution function of $C_{\psi,A}$ is a convex combination of the Kendall distribution function of $M$ and the Kendall 
distribution function of the corresponding Archimedean copula.
\subsection{Markov kernels of Archimax copulas}
\noindent Proving the results in the next sections, Markov kernels will be an essential tool. Let $\psi \in \Psi$, $A \in \mathcal{A}$ and $C_{\psi,A} \in \mathcal{C}_{am}$. Recall that we consistently use the convention that $a\cdot\infty = \infty$, $\frac{a}{\infty} = 0$ and $\frac{a}{0} = \infty$ for $a \in \mathbb{R} \setminus \{0\}$.\\
We propose the following condensed version of the Markov kernel $K_{\psi,A}$ of $C_{\psi,A}$ proposed in \cite[Theorem 4.1]{dietrich2026}:
\begin{corollary}\label{cor:markov_kernel}
Let $\psi \in \Psi$, $A \in \mathcal{A}$, $G_A$ be defined according to eq. \eqref{eq:funct_G_A} and $C_{\psi,A} \in \mathcal{C}_{am}$. Then the function $K_{\psi,A} \colon \mathbb{I}^2 \rightarrow \mathbb{I}$ defined by
\begin{align}\label{eq:markov-kernel}
K_{\psi,A}(x,[0,y]) := \begin{cases}
1,& \text{if } (x,y) \in\{0,1\} \times \mathbb{I},\\
\frac{D^{-}\psi\left((\varphi(x) + \varphi(y)) A\left(\frac{\varphi(x)}{\varphi(x) + \varphi(y)}\right)\right)}{D^-\psi(\varphi(x))}\cdot G_A\left(\frac{\varphi(x)}{\varphi(x) + \varphi(y)}\right),& \text{if } (x,y) \in (0,1) \times \mathbb{I}
\end{cases}
\end{align}
for every $(x,y) \in \mathbb{I}^2$ is a version of the Markov-kernel of $C_{\psi,A}$.
\end{corollary}
\begin{proof}
    We fix $\psi \in \Psi$, $A \in \mathcal{A}$ arbitrarily and let $h_A$ be defined as in eq. \eqref{eq:fct_h} and $f^0$ be the function parametrizing the $0$-level set as defined in eq. \eqref{eq:fct_f_t}. Then according to \cite[Theorem 4.1]{dietrich2026} we obtain that
    \begin{align*}
    K_{\psi,A}(x,[0,y]) &= 
    \frac{D^{-}\psi\left(\varphi(x) h_A\left(\frac{\varphi(x)}{\varphi(x) + \varphi(y)}\right)\right)}{D^-\psi(\varphi(x))}\cdot G_A\left(\frac{\varphi(x)}{\varphi(x) + \varphi(y)}\right)\\&=
    \frac{D^{-}\psi\left((\varphi(x) + \varphi(y)) A\left(\frac{\varphi(x)}{\varphi(x) + \varphi(y)}\right)\right)}{D^-\psi(\varphi(x))}\cdot G_A\left(\frac{\varphi(x)}{\varphi(x) + \varphi(y)}\right)
    \end{align*}
    holds for every $(x,y) \in (0,1) \times \mathbb{I}$ fulfilling $f^0(x) \leq y$. Fix $(x,y) \in (0,1) \times \mathbb{I}$ with $y < f^0(x)$. Then we have that $\frac{\varphi(x)}{\varphi(x) + \varphi(y)} < h_A^{[-1]}\left(\frac{\varphi(0)}{\varphi(x)}\right)$ and therefore $\varphi(0) < \varphi(x) h_A\left(\frac{\varphi(x)}{\varphi(x) + \varphi(y)}\right)$. This implies that 
    $$
    0 = D^{-}\psi\left(\varphi(x) h_A\left(\frac{\varphi(x)}{\varphi(x) + \varphi(y)}\right)\right) = D^{-}\psi\left((\varphi(x) + \varphi(y)) A\left(\frac{\varphi(x)}{\varphi(x) + \varphi(y)}\right)\right),
    $$
    hence $K_{\psi,A}$ as given in eq.~\eqref{eq:markov-kernel} is a Markov kernel
    of $C_{\psi,A}$.
\end{proof}
As consequence of Corollary \ref{cor:markov_kernel}, we have the following corollary for the Archimedean case, simplifying the representation of the Markov-kernel given in \citep{fernandezsanchez2015, fuchs_tschimpke}:
\begin{corollary}
    Let $\psi \in \Psi$ and $C_\psi \in \mathcal{C}_{ar}$. Then the function $K_\psi \colon \mathbb{I}^2 \rightarrow \mathbb{I}$, defined by
    \begin{align}
    K_{\psi}(x,[0,y]) := \begin{cases}
    1,& \text{if } (x,y) \in\{0,1\} \times \mathbb{I},\\
    \frac{D^{-}\psi\left(\varphi(x) + \varphi(y)\right)}{D^-\psi(\varphi(x))},& \text{if } (x,y) \in (0,1) \times \mathbb{I}
    \end{cases}
\end{align}
for every $(x,y) \in \mathbb{I}^2$ is a version of the Markov kernel of $C_\psi$.
\end{corollary}
\section{Identifiability in $\mathcal{C}_{am}$}\label{sec:identifiability}
This section addresses identifiability in the Archimax family. 
As shown in \cite{neslehova}, the generator $\psi$ and Pickands dependence function $A$ do not identify the Archimax copula $C_{\psi,A}$ uniquely. 
Tackling and overcoming this issue, we introduce a new approach to 
identify Archimax copulas via (unique) so-called transformed generators and transformed Pickands dependence functions. We proceed in three steps: 
first we introduce and study properties of these transformed functions, 
then we derive several analytical properties, and finally we show 
that they indeed assure identifiability in the Archimax family. 
As a by-product we prove that the original generator $\psi$ and Pickands dependence 
function $A$ only identify the model, if Kendall's $\tau$ of the EVC induced by 
the Pickands dependence function is fixed.\\
\noindent Consider $\psi \in \Psi$ with pseudo inverse $\varphi$ and let $\alpha \in (0,\infty)$ be arbitrary but fixed. We define the transformed function $(\psi)_\alpha$ as
\begin{equation}\label{eq:trans_arch}
(\psi)_\alpha(z) := \psi(z^\alpha)
\end{equation}
for every $z \in [0,\infty)$ and denote its pseudo-inverse by 
$$
(\varphi)_\alpha(x) = \varphi^\frac{1}{\alpha}(x)
$$
for every $x \in \mathbb{I}$.\\
Moreover, for $A\in \mathcal{A}$ and $\alpha \in (0,\infty)$, the transformed function $(A)_\alpha$ is defined as
\begin{equation}\label{eq:trans_pickands}
(A)_\alpha(t) := (t^\alpha + (1-t)^\alpha)^\frac{1}{\alpha} A^\frac{1}{\alpha}\left(\frac{t^\alpha}{t^\alpha + (1-t)^\alpha}\right)
\end{equation}
for every $t \in \mathbb{I}$. Alternatively, we may express the function $(A)_\alpha$ via the function $h_A$ as defined in eq. \eqref{eq:fct_h} - for every $t \in \mathbb{I}$ the following identity holds:
\begin{equation}\label{eq:alter_trans_pick}
(A)_\alpha(t) = t \cdot h_A^\frac{1}{\alpha}\left(\frac{t^\alpha}{t^\alpha + (1-t)^\alpha}\right).
\end{equation}
To simplify notation, in the sequel we will work with the function $\kappa_\alpha: \mathbb{I} \rightarrow \mathbb{I}$, given by
\begin{equation}\label{eq:trans_kappa}
\kappa_\alpha(t):=\frac{t^\alpha}{t^\alpha + (1-t)^\alpha}
\end{equation}
for every $t \in \mathbb{I}$ and every $\alpha \in (0,\infty)$. Obviously $\kappa_\alpha$ is an increasing
homeomorphism of $\mathbb{I}$, which has $\frac{1}2$ as fixed point. 
For $\alpha > 1$ the mapping $\kappa_\alpha$ is strictly convex on $[0,\frac{1}{2}]$
and strictly concave on $[\frac{1}{2},1]$. For $\alpha<1$ it is strictly 
convex on $[\frac{1}{2},1]$ and strictly concave on $[0,\frac{1}{2}]$.
The next lemma was proved in \citep[Lemma 2.3]{est-archimax}. It shows that for specific values of $\alpha \in (0,\infty)$, $(\psi)_\alpha$, and $(A)_\alpha$ are indeed Archimedean generators and Pickands dependence functions, respectively.
\begin{lemma}\label{lem:trans_gen}
    Let $\psi \in \Psi$ and $A \in \mathcal{A}$. Then the following two assertions hold:
    \begin{enumerate}[label=(\roman*)]
        \item $(\psi)_\alpha \in \Psi$ for every $\alpha \in (0,1]$.
        \item $(A)_\alpha \in \mathcal{A}$ for every $\alpha \in [1,\infty)$.
    \end{enumerate}
\end{lemma}
\noindent As the next example shows, the previous lemma does not extend to arbitrary $\alpha \in (0,\infty)$.
\begin{example}
Consider the (normalized) generator $\psi(z) = \exp(\log(\frac{1}{2})z)$ for every $z \in (0,\infty)$. Then for arbitrary $\alpha \in [1,\infty)$ the corresponding transformed function $(\psi)_\alpha$ is given by
$$
(\psi)_\alpha(z) = \psi(z^\alpha) = \exp(\log(\tfrac{1}{2})z^\alpha) 
$$
for every $z \in [0,\infty)$. For the second derivative we then obtain that
$$
(\psi)_\alpha''(z) = \log(\tfrac{1}{2})\alpha\exp(\log(\tfrac{1}{2})z^\alpha)z^{\alpha-2}[(\alpha - 1) + \log(\tfrac{1}{2})\alpha z^\alpha]
$$
for every $z \in (0,\infty)$. It is easy to see that $\psi_\alpha''(z) < 0$ if and only if $z \in \left(0, \left(\frac{\alpha-1}{\log(2)\,\alpha}\right)^{\frac{1}{\alpha}}\right)$, so $\psi_\alpha$ is not convex and, consequently, is not a generator of an Archimedean copula.
\\
Considering $A_\Pi(t)=1$, for arbitrary $\beta \in (0,\infty]$ we obtain that
    $$
    (A_\Pi)_\beta(t) = (t^\beta + (1-t)^\beta)^\frac{1}{\beta}
    $$
    for every $t \in \mathbb{I}$. Calculating the second derivative of $(A)_\beta$ then yields that
    $$
    (A)_\beta''(t) = (\beta-1)(t^\beta + (1-t)^\beta)^{\frac{1}{\beta}-2}t^{\beta -2}(1-t)^{\beta -2}
    $$
    for every $t \in (0,1)$. Therefore, $(A)_\beta''(t) < 0$ holds if and only if $\beta < 1$, implying that $A_\beta$ is not convex for $\beta < 1$ and thus is not a Pickands dependence function.
\end{example}
\begin{example}\label{ex2}
Again considering $A_\Pi$, for every $\alpha \in [1,\infty)$
we obviously have $(A_\Pi)_\alpha=A_{G_\alpha}$, where $A_{G_\alpha}$ denotes the 
Pickands dependence function of the Gumbel copula $G_\alpha$. Setting $G_\infty=M$ it follows
that the full Gumbel family can be constructed from $A_\Pi$ by applying 
the transformation $\alpha \mapsto (A_\Pi)_\alpha$. Moreover, 
for every $\alpha \geq 1$ and arbitrary $A \in \mathcal{A}$ we obviously have
$(A)_\alpha(t) \leq A_{G_\alpha}(t)$ for all $t \in \mathbb{I}$. 
A straightforward calculation (or applying Lemma \ref{lem:properties_trans_pickands}) 
shows that for every $\alpha \in [1,\infty)$ 
we have that 
$$
A_\Pi=((A_\Pi)_\alpha)_{\frac{1}{\alpha}}=(A_{G_\alpha})_{\frac{1}{\alpha}}, 
$$
implying that $(A)_{\alpha}$ may be a Pickands dependence function for some $\alpha \in (0,1)$.
\end{example}
\noindent Next, we derive the right-hand derivative $D^+(A)_{\alpha}$ of 
$(A)_\alpha$ on the interval $(0,1)$ for arbitrary $\alpha \in (0,\infty)$:
\begin{lemma}
Let $A \in \mathcal{A}$, $\alpha \in (0,\infty)$ be arbitrary and let $\kappa_\alpha$ be defined as in eq. \eqref{eq:trans_kappa}. Then $D^+(A)_\alpha$ exists on the whole interval $(0,1)$. Moreover, the following identity holds for every $t \in (0,1)$:
\begin{equation}\label{eq:right-hand_der}
    D^+(A)_{\alpha}(t) =
    (A)_{\alpha}(t) \cdot \left[\frac{t^{\alpha -1} - (1-t)^{\alpha -1}}{t^{\alpha} + (1-t)^{\alpha}} + \frac{D^+A\left(\kappa_\alpha(t)\right)t^{\alpha - 1}(1-t)^{\alpha - 1}}{A\left(\kappa_\alpha(t)\right)(t^\alpha + (1-t)^\alpha)^2}\right]
\end{equation}
\end{lemma}
\begin{proof}
    First, the functions $s \mapsto (s^\alpha + (1-s)^\alpha)^\frac{1}{\alpha}$, $s \mapsto s^\frac{1}{\alpha}$ and $\kappa_\alpha$ are smooth on $(0,1)$. Since $A$ is convex, its right-hand derivative $D^+A$ exists. Observing that $\kappa_\alpha$ is strictly increasing on $(0,1)$ and applying \cite[Lemma 7.1]{dietrich2026} yields that $D^+(A\circ \kappa_\alpha)(s)$ exists for every $s \in (0,1)$. Applying the product and chain rule, existence of $D^+(A)_\alpha$, and eq. \eqref{eq:right-hand_der} follow.
\end{proof}
\noindent As it will be of use later on, we derive some additional useful properties of the function $(A)_\alpha$ (if no restriction on $\alpha$ is stated, then 
we always consider $\alpha \in (0,\infty)$). 
Assuring well-definedness, notice that the definition of 
$(A)_\alpha$ in eq. \eqref{eq:trans_pickands} can be extended to every 
function $f: \mathbb{I} \rightarrow (0,\infty)$ and every $\alpha \in (0,\infty)$ by
setting
$$
(f)_\alpha(t)=(t^\alpha + (1-t)^\alpha)^\frac{1}{\alpha} f^\frac{1}{\alpha}\left(\kappa_\alpha(t)\right)
$$
for every $t \in \mathbb{I}$. With this extension, expressions like $((A)_\alpha)_\beta$ make sense for 
all $\alpha,\beta \in (0,\infty)$.
\begin{lemma}\label{lem:properties_trans_pickands}
    Let $A \in \mathcal{A}$ be arbitrary but fixed. Then the following assertions hold:
     \begin{enumerate}[label=(\roman*)]
        \item For all $\alpha, \beta \in (0,\infty)$ we have $((A)_\alpha)_\beta=(A)_{\alpha\beta}=((A)_\beta)_\alpha$ as well as 
        $\kappa_\alpha\circ \kappa_\beta=\kappa_{\alpha\beta}=\kappa_\beta\circ \kappa_\alpha$.
       \item The function $D^+(A)_{\alpha}$ has a discontinuity at $t \in (0,1)$ if, and only if 
       $D^+A$ has a discontinuity at $\kappa_{\alpha}(t)$.
       \item The set of discontinuity points of $(A)_{\alpha}$ is at most 
       countably infinite. 
       \item $\lim_{\alpha \rightarrow \infty}(A)_\alpha(t) = A_M(t)$ for 
       every $t \in \mathbb{I}$ and the convergence is uniform in $t \in \mathbb{I}$.
       \item $A_M(t)\leq (A)_\alpha(t) \leq 2^\frac{1-\alpha}{\alpha}$ for every $t \in \mathbb{I}$ and every $\alpha \in (0,1]$.
       \item Suppose that $A$ is linear on some non-degenerated interval 
       $[a,b] \subseteq \mathbb{I}$ and there exists some $t_0 \in [a,b]$ 
       with $A_M(t_0)<A(t_0)$. Then $(A)_\alpha$ is strictly 
       convex on the interval $\kappa_{\frac{1}{\alpha}}([a,b])$ for $\alpha>1$ and 
       strictly concave for $\alpha \in (0,1)$.
       \item If $\alpha \geq 1$, then $(A)_\alpha(t) \leq A(t)$ for every 
       $t \in \mathbb{I}$. Moreover, for every $t \in (0,1)$ with 
       $A(t) > A_M(t)$ even $(A)_\alpha(t)<A(t)$ holds. 
       \item If $(\alpha_n)_{n \in \mathbb{N}}$ converges to $\alpha \in (0,\infty)$,
       then $(\kappa_{\alpha_n})_{n \in \mathbb{N}}$ converges uniformly 
       to $\kappa_\alpha$ on $\mathbb{I}$.
       \item If $(\alpha_n)_{n \in \mathbb{N}}$ converges to $\alpha \in (0,\infty)$
       and the sequence $(A_n)_{n\in \mathbb{N}}$ of Pickands dependence functions
       converges uniformly to $A \in \mathcal{A}$, then 
       $((A_n)_{\alpha_n})_{n \in \mathbb{N}}$ converges uniformly to $(A)_\alpha$.
        \end{enumerate}
\end{lemma}
\begin{proof}
     A straightforward calculation proves the first assertion, the second 
     one is a direct consequence of eq. \eqref{eq:right-hand_der}. 
     Assertion number three follows from the facts that, firstly, for 
     each $A \in \mathcal{A}$ we have that 
     $D^+A$ can have at most countably infinitely many discontinuities and, secondly, 
     that $\kappa_\alpha$ is a homeomorphism of $\mathbb{I}$. \\
     In order to prove the fourth assertion recall that for 
     every $\alpha \in [1,\infty)$ and every $t\in\mathbb{I}$ 
     we have $A_M(t)\leq (A)_\alpha(t) \leq (A_\Pi)_\alpha(t)=A_{G_\alpha}(t)$. Considering $\lim_{\alpha \rightarrow \infty}A_{G_\alpha}(t) = A_M(t)$
     therefore yields pointwise convergence, which, using the fact that 
     the family $\mathcal{A}$ is equi-continuous (all elements are Lipschitz continuous with Lipschitz constant $1$), proves the desired result. \\
     For $\alpha \in (0,1]$ we obviously have
    $$
    (A)_\alpha(t) = (t^\alpha + (1-t)^\alpha)^\frac{1}{\alpha}A^\frac{1}{\alpha}\left(\frac{t^\alpha}{t^{\alpha} + (1-t)^\alpha}\right) \leq (t^\alpha + (1-t)^\alpha)^\frac{1}{\alpha}.
    $$
    Since for $\alpha \in (0,1]$ the function $t \mapsto (t^\alpha + (1-t)^\alpha)^\frac{1}{\alpha}$ 
    is concave and symmetric on $\mathbb{I}$, we directly get 
    $(A)_\alpha(t) \leq 2^\frac{1-\alpha}{\alpha}$. \\
    In order to show assertion number six we proceed as follows: Assume that 
    $\alpha>1$ and set $[c,d]=\kappa_{\frac{1}{\alpha}}([a,b])$. Then, by assumption
    there are two constants $s_1 \in (0,1),s_2 \in (-1,1)$ fulfilling 
    $s_1+s_2 \in (0,1)$ such that for every 
    $t \in [c,d]$ we have
    \begin{align*}
       (A)_\alpha^\alpha(t) &= \left(t^\alpha + (1-t)^\alpha \right)
    \left(s_1 + s_2 \kappa_\alpha(t)\right) =
    \left(t^\alpha + (1-t)^\alpha \right)s_1 + t^\alpha s_2 \\
    &=t^\alpha(s_1+s_2) + (1-t)^\alpha s_1 = 
    \left\Vert \left(t(s_1+s_2)^{\frac{1}{\alpha}},(1-t)s_1^{\frac{1}{\alpha}}\right) \right\Vert_\alpha^\alpha,
    \end{align*}
    where $\Vert \cdot \Vert_\alpha$ denotes the $\alpha$-norm in $\mathbb{R}^2$.
    In other words, we have that $t\mapsto(A)_\alpha(t)$ can be expressed as 
    the composition of a strictly convex function (the $\alpha$-norm) and an injective affine 
    function (the mapping $t \mapsto (t(s_1+s_2)^{\frac{1}{\alpha}},(1-t)s_1^{\frac{1}{\alpha}})$). This shows that $(A)_\alpha$ is strictly convex on $[c,d]$. 
    Proving strict concaveness of $(A)_\alpha$ can be done analogously, e.g., 
    by directly 
    calculating the second derivative of the function 
    $t \mapsto (t^\alpha(s_1+s_2) + (1-t)^\alpha s_1)^{\frac{1}{\alpha}}$ 
    and showing that it is strictly smaller than $0$. \\
    For proving assertion number seven it suffices to consider $\alpha>1$. 
    Noticing that eq. \eqref{eq:rel_pickands_fct_measure} directly yields
    \begin{align}\label{eq:A.alpha.via.measure}
    ((A)_\alpha(t))^\alpha=  \left(t^\alpha + (1-t)^\alpha \right) \, 2\int_{\mathbb{I}} \max\{\kappa_\alpha(t)(1-s),(1-\kappa_\alpha(t))s\}  \mathrm{d}\vartheta(s) = 2\int_{\mathbb{I}} 
    \max\{t^\alpha(1-s),(1-t)^\alpha\,s\}  \mathrm{d}\vartheta(s)
    \end{align}
    we may proceed as follows: 
    Using the fact that $\max\{ac,bd\} \leq \max\{a,b\}\max\{c,d\}$ for $a,b,c,d \in [0,\infty)$, the last integral can be bounded from 
    above by $\max\{t^{\alpha-1},(1-t)^{\alpha -1}\} \,A(t)$. Altogether 
    it therefore follows that
    $$
    ((A)_\alpha(t))^\alpha \leq A(t) \max\{t^{\alpha-1},(1-t)^{\alpha -1}\}
    = A(t) A_M(t)^{\alpha-1}
    \leq A(t) \,(A(t))^{\alpha-1} = (A(t))^\alpha.
    $$
If $A_M(t) < A(t)$, the middle inequality above is strict, so $((A)_\alpha(t))^\alpha < (A(t))^\alpha$, and since both sides are non-negative and $\alpha>0$, taking the $\alpha$-th root yields $(A)_\alpha(t) < A(t)$.\\
    In order to prove assertion number eight, first notice that the sequence $(\kappa_{\alpha_n})_{n \in \mathbb{N}}$ obviously converges pointwise 
    to $\kappa_\alpha$. 
    Interpreting $\kappa_\alpha,\kappa_{\alpha_1},\kappa_{\alpha_2},\ldots$ as 
    continuous distribution functions (restricted to $\mathbb{I}$) and using the 
    fact that for continuous distribution functions pointwise and uniform 
    convergence are equivalent, it follows that 
    $(\kappa_{\alpha_n})_{n \in \mathbb{N}}$ converges to $\kappa_\alpha$
    uniformly on $\mathbb{I}$.   \\  
    The last assertion now follows easily by using the triangle inequality, the fact
    that all involved functions are continuous on $\mathbb{I}$, and that
    under the stated assumptions we have that 
    the functions $x \mapsto x^{\frac{1}{\alpha_n}}$ converge uniformly to 
    the function  $x \mapsto x^{\frac{1}{\alpha}}$ on $[\frac{1}{2},1]$.
\end{proof}
\begin{remark}\label{rem:illustration.A.alpha}
    Contrary to the case $\alpha\geq 1$, the function $(A)_\alpha$ may be quite irregular for $\alpha \in (0,1)$. In fact, starting with a discrete Pickands 
    dependence measure $\vartheta$ 
    with full support (such an example is straightforward to construct), 
    applying the second assertion in Lemma \ref{lem:properties_trans_pickands}
    it follows that the set of discontinuity points of $D^+(A)_\alpha$ 
    is dense in $\mathbb{I}$ for every $\alpha \in (0,1)$. 
    We conjecture that for such $\vartheta$ the function $(A_\vartheta)_\alpha$
    is not even locally concave or convex for $\alpha\in (0,1)$. 
    Proving such a property, however, 
    is outside the scope of this article, we here settle with  
    mentioning that approximating $\vartheta$ by discrete Pickands 
    measures $\vartheta_n$ with only finitely many point masses, Lemma \ref{lem:properties_trans_pickands}
    implies that for $\alpha \in (0,1)$ the corresponding function $(A_{\vartheta_n})_\alpha$ is piecewise strictly concave. 
    Figure \ref{fig:discrete.approx} depicts the (piecewise linear) 
    Pickands dependence function $A$ corresponding to a discrete Pickands dependence measure with finitely many point masses 
    as well the function $(A)_\alpha$ for $\alpha=1-\tau_A$.
\end{remark}
\begin{figure}[!ht]
	\centering
	\includegraphics[width=0.8\textwidth]{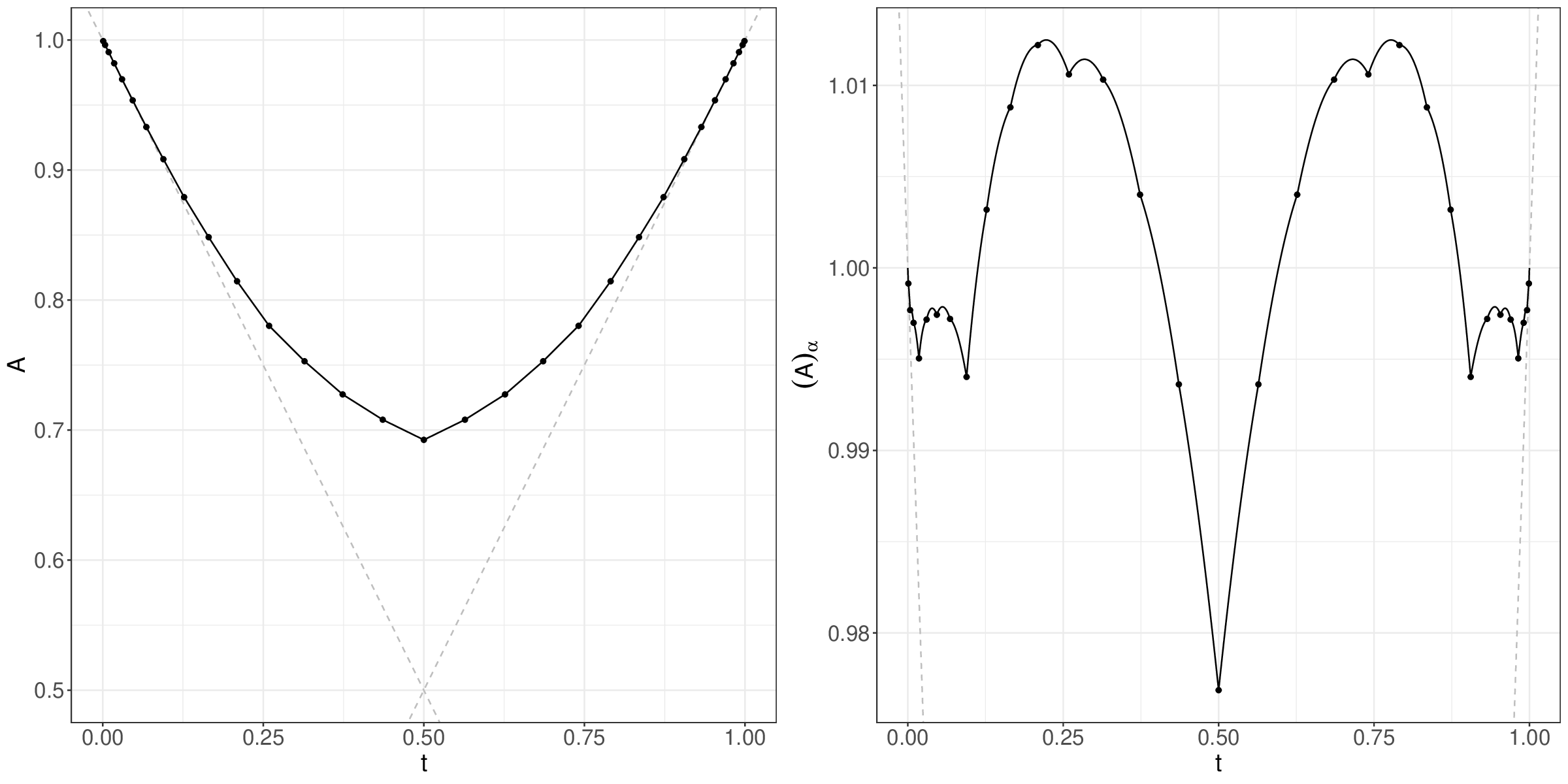}
	\caption{Symmetric, piecewise linear Pickands dependence function 
    $A \in \mathcal{A}$ corresponding to a Pickands measure with 
    31 point masses (left panel); 
    piecewise (strictly) concave function $(A_\alpha)$ for 
    $\alpha=1-\tau_A\approx 0.486$. The points in the left panel denote the edges, in the right panel the boundaries of the strictly concave segments. (right).}\label{fig:discrete.approx}
\end{figure}
\noindent Motivated by the fact that Kendall's $\tau$ plays a central role throughout this paper, we calculate Kendall's $\tau$ of a transformed Pickands dependence function $(A)_\alpha$ for $A \in \mathcal{A}$ and $\alpha \in (0,\infty)$. Since the proof of Lemma~\ref{lem:kendalls_tau_trans} is rather tedious, we defer it to~\ref{proof:lem_trans_kendall} for the sake of readability.
\begin{lemma}\label{lem:kendalls_tau_trans}
    For $A \in \mathcal{A}$ and $\alpha \in (0,\infty)$ such that $(A)_\alpha \in \mathcal{A}$ we have 
    $$
    \tau_{(A)_\alpha} = 1-\frac{1-\tau_A}{\alpha}.
    $$
Moreover, for $\psi \in \Psi$ and 
    $\alpha \in (0,1]$ the 
    identity $\tau_{(\psi)_\alpha}=(1-\alpha) + \alpha\tau_{\psi}$ holds.
\end{lemma}
\begin{proof}
    See \ref{proof:lem_trans_kendall}.
\end{proof}
\begin{remark}
According to Example \ref{ex2} it is possible that 
    for $A \in \mathcal{A}$ we even have $(A)_\alpha \in \mathcal{A}$ for 
    some $\alpha <1 $. Lemma \ref{lem:kendalls_tau_trans}, however, 
    implies that $\alpha$ can not be arbitrarily small for 
    $(A)_\alpha \in \mathcal{A}$ to hold. In fact, considering that 
    EVCs have non-negative Kendall's $\tau$, $0 \leq \tau_{(A)_\alpha}$ has to hold, 
    which, in turn, is equivalent to $\alpha \geq 1-\tau_A$. 
\end{remark}
\noindent The next lemma shows that the Kendall distribution function $F_{\psi,A}^K$ of an Archimax copula $C_{\psi,A} \in \mathcal{C}_{am}$ is the Kendall distribution function of an Archimedean copula $C_{(\psi)_{1-\tau_A}} \in \mathcal{C}_{ar}$ with corresponding generator $(\psi)_{1-\tau_A}$.
\begin{lemma}\label{lem:kendall_is_archimedean}
    Let $\psi \in \Psi$, $A\in \mathcal{A}\setminus \{A_M\}$ and let $\tau_A$ denote Kendall's $\tau$ of the corresponding EVC $C_A$. Furthermore, let $C_{\psi,A} \in \mathcal{C}_{am}$ denote the corresponding Archimax copula. 
    Then there exists a unique Archimedean copula
    whose Kendall distribution function coincides with $F^K_{\psi,A}$. 
    This Archimedean copula is $C_{(\psi)_{1-\tau_A}} \in \mathcal{C}_{ar}$.
\end{lemma}
\begin{proof}
    Considering $A \neq A_M$ we have $\tau_A < 1$. Applying the formula in eq. \eqref{eq:kendall.arch} and using straightforward calculations yields
    $$
    F_{(\psi)_{1-\tau_A}}^K(t) = t-(1-\tau_A)\beta_\psi(t) = F_{\psi,A}^K(t)
    $$
    for every $t \in \mathbb{I}$, which proves the result.
\end{proof}
\begin{remark}
\noindent As direct consequence of Lemma \ref{lem:kendall_is_archimedean} we obtain that every EVC $C \in \mathcal{C}_{ev}\setminus \{M\}$ induces an Archimedean copula via its Kendall distribution function.
\end{remark}
The following example motivates the main theorem of this section.
\begin{example}
Consider the Pickands dependence functions $A_1(t) = A_\Pi(t) = 1$ and $A_2(t) = A_{G_\alpha}(t) = (t^\alpha + (1-t)^\alpha)^\frac{1}{\alpha}$ for every $t \in \mathbb{I}$, where $\alpha \in (1, \infty)$ is arbitrary but fixed. Let $\psi_2 \in \Psi$ be arbitrary but fixed as well, and define $\psi_1(z) = \psi_2(z^\frac{1}{\alpha})$ for every $z \in [0,\infty)$. For the corresponding Archimax copulas $C_{\psi_1,A_1}, C_{\psi_2,A_2} \in \mathcal{C}_{am}$, this immediately gives
\begin{equation*}
    C_{\psi_1,A_1}(x,y) = C_{\psi_2,A_2}(x,y)
\end{equation*}
for every $x,y \in \mathbb{I}$. Since $\tau_{A_1} = 0$ and $A_2=(A_1)_\alpha$, Lemma \ref{lem:kendalls_tau_trans} yields $\tau_{A_2} = 1- \frac{1}{\alpha}$, and consequently
\begin{equation*}
    \psi_1(z) = (\psi_2)_{\frac{1-\tau_{A_2}}{1-\tau_{A_1}}}(z), \qquad A_2(t) = (A_1)_{\frac{1-\tau_{A_1}}{1-\tau_{A_2}}}(t)
\end{equation*}
for every $z \in [0,\infty)$ and $t \in \mathbb{I}$.
\end{example}
\noindent The next theorem generalizes the previous example and 
characterizes identifiability in the Archimax class.
The theorem shows that two Archimax copulas are identical if and only if their generators and Pickands dependence functions are equal up to a transformation
depending on Kendall's $\tau$ of the associated EVCs.
\begin{theorem}\label{thm:identifiability}
Let $A_1,A_2 \in \mathcal{A} \setminus \{A_M\}$ be Pickands dependence functions and let
$\tau_{A_1}, \tau_{A_2}$ be the Kendall's $\tau$s of $C_{A_1}, C_{A_2} \in \mathcal{C}_{ev}$,
respectively. Furthermore, suppose that $\psi_1,\psi_2 \in \Psi$ and let
$C_{\psi_1,A_1}, C_{\psi_2,A_2} \in \mathcal{C}_{am}$. Then the following
equivalence holds:
$$
C_{\psi_1,A_1}(x,y) = C_{\psi_2,A_2}(x,y)
$$
for every $x,y \in \mathbb{I}$ if and only if one of the following two
cases holds:
\begin{enumerate}
\item If $\tau_{A_1} \le \tau_{A_2}$, then
$$
\psi_1(z) = (\psi_2)_{\frac{1-\tau_{A_2}}{1-\tau_{A_1}}}(z)
\quad \text{and} \quad
A_2(t) = (A_1)_{\frac{1-\tau_{A_1}}{1-\tau_{A_2}}}(t),
$$
\item if $\tau_{A_1} > \tau_{A_2}$, then
$$
\psi_2(z) = (\psi_1)_{\frac{1-\tau_{A_1}}{1-\tau_{A_2}}}(z)
\quad \text{and} \quad
A_1(t) = (A_2)_{\frac{1-\tau_{A_2}}{1-\tau_{A_1}}}(t)
$$
\end{enumerate}
for every $z \in [0,\infty)$ and $t \in \mathbb{I}$.
\end{theorem}
\begin{proof}
    Consider $\psi_1,\psi_2 \in \Psi$ and $A_1,A_2 \in \mathcal{A}\setminus \{A_M\}$. We assume that $C_{\psi_1,A_1}(x,y) = C_{\psi_2,A_2}(x,y)$ for every $x,y \in \mathbb{I}$ and assume without loss of generality that $\tau_{A_1} \leq \tau_{A_2}$. The case $\tau_{A_1} > \tau_{A_2}$ follows similarly. We obviously have that the corresponding Kendall distribution functions $F_{\psi_1,A_1}$ and $F_{\psi_2,A_2}$ coincide, i.e.,
    $$
    F_{\psi_1,A_1}^K(t) = F_{\psi_2,A_2}^K(t)
    $$
    for every $t \in \mathbb{I}$.
    According to Lemma \ref{lem:kendall_is_archimedean}, $F_{\psi_1,A_1}^K$ and $F_{\psi_2,A_2}^K$ induce Archimedean generators $(\psi_1)_{1-\tau_{A_1}}$ and $(\psi_2)_{1-\tau_{A_2}}$, which obviously fulfill $(\psi_1)_{1-\tau_{A_1}}(s) = (\psi_2)_{1-\tau_{A_2}}(s)$ for every $s \in [0,\infty)$.
    Setting $s = z^\frac{1}{1-\tau_{A_1}}$, we obtain that
    $$
    \psi_1(z)= \psi_2\left(z^\frac{1-\tau_{A_2}}{1-\tau_{A_1}}\right) = (\psi_2)_{\frac{1-\tau_{A_2}}{1-\tau_{A_1}}}(z)
    $$
    for every $z \in [0,\infty)$. Having $\tau_{A_1} \leq \tau_{A_2}$, we obtain that $(\psi_2)_{\frac{1-\tau_{A_2}}{1-\tau_{A_1}}}$ is indeed an Archimedean generator. Using this fact, fixing $x,y \in \mathbb{I}$ and applying that $C_{\psi_1,A_1}(x,y) = C_{\psi_2,A_2}(x,y)$ yields
    $$
    \psi_2\left((\varphi_2(x)^\frac{1-\tau_{A_1}}{1-\tau_{A_2}} + \varphi_2(y)^\frac{1-\tau_{A_1}}{1-\tau_{A_2}})^\frac{1-\tau_{A_2}}{1-\tau_{A_1}}A_1^\frac{1-\tau_{A_2}}{1-\tau_{A_1}}\left(\frac{\varphi_2(x)^\frac{1-\tau_{A_1}}{1-\tau_{A_2}}}{\varphi_2(x)^\frac{1-\tau_{A_1}}{1-\tau_{A_2}} + \varphi_2(y)^\frac{1-\tau_{A_1}}{1-\tau_{A_2}}}\right)\right) = 
    \psi_2\left((\varphi_2(x) + \varphi_2(y))\,A_2\left(\frac{\varphi_2(x)}{\varphi_2(x) + \varphi_2(y)}\right)\right).
    $$
    Setting $t = \frac{\varphi_2(x)}{\varphi_2(x) + \varphi_2(y)}$, it is then easy to see that $A_2(t) = (A_1)_\frac{1-\tau_{A_1}}{1-\tau_{A_2}}(t)$. 
    Since the other implication is obtained by straightforward calculations,
    the proof is complete.
\end{proof}
\begin{remark}
    Notice that, using Lemma \ref{lem:properties_trans_pickands}, both cases in the previous theorem 
    boil down to  
    $$
    (\psi_1)_{1-\tau_{A_1}}=(\psi_2)_{1-\tau_{A_2}}, \quad  
    (A_1)_{1-\tau_{A_1}}=(A_2)_{1-\tau_{A_2}}.
    $$
    We distinguished the two cases in the formulation of the theorem, since
    (again see Lemma \ref{lem:properties_trans_pickands}) 
    for $A \in \mathcal{A}$ in general $(A)_{1-\tau_A}$ does not need to 
    be an element of $\mathcal{A}$.
Using this reparametrization, identifiability of the Archimax model can 
be achieved. However, the family $\{(A)_{1-\tau_A}\colon A \in \mathcal{A}\}$ does not seem to admit a straightforward analytical characterization.
\end{remark}
The following corollary shows that the Archimax model $C_{\psi,A}$ is only identified by the functions $\psi$ and $A$, if $\tau_A$ is fixed.
\begin{corollary}
    Let $\psi_1,\psi_2 \in \Psi$, $A_1,A_2 \in \mathcal{A} \setminus \{A_M\}$ and let $\tau_{A_1}, \tau_{A_2}$ be the Kendall's $\tau$s of $C_{A_1}, C_{A_2} \in \mathcal{C}_{ev}$. Moreover, let $C_{\psi_1,A_1}(x,y) = C_{\psi_2,A_2}(x,y)$ for all $x,y \in \mathbb{I}$. 
    Then $\psi_1 = \psi_2$ on $[0,\infty)$ and $A_1= A_2$ on $\mathbb{I}$ if and only if $\tau_{A_1} = \tau_{A_2}$.
\end{corollary}
\begin{proof}
Immediate consequence of Theorem \ref{thm:identifiability}.
\end{proof}
The authors in \citep{est-archimax} studied identifiability of the Archimax model via the generator $\psi$ and the Pickands dependence function $A$, under the assumption that $\psi$ satisfies a regularity condition, namely that it is regularly varying. To interpret these results in light of Theorem \ref{thm:identifiability}, we first formally introduce the notion of regularly varying functions.\\
A function $f \colon [0,\infty) \rightarrow \mathbb{I}$ is said to be regularly varying with index $\alpha \in \mathbb{R}$ if, for every $x > 0$,
$$
\lim_{t \rightarrow \infty} \frac{f(xt)}{f(t)} = x^\alpha.
$$
The family of all such functions with fixed index $\alpha \in \mathbb{R}$ is denoted by $\mathcal{R}_\alpha$.\\
\noindent Since \citep{est-archimax} works with stable tail dependence functions rather than Pickands dependence functions, we translate their identifiability result for the Archimax model to the setting of Pickands dependence functions. As this translation follows immediately from \citep{est-archimax}, we omit a formal proof.
\begin{lemma}\label{lem:identifiability_nes}
    Suppose that $A_1,A_2 \in \mathcal{A}\setminus \{A_M\}$ and let $\psi_1,\psi_2 \in \Psi$ with the property that $(1-\psi_i(\frac{1}{\cdot})) \in \mathcal{R}_{-\frac{1}{m_i}}$ for $m_i \geq 1$ and $i \in \{1,2\}$. Then $C_{\psi_1,A_1}(x,y) = C_{\psi_2,A_2}(x,y)$ holds for every $x,y \in \mathbb{I}$ if and only if one of the following two cases holds:
    \begin{enumerate}
        \item If $m_1 \leq m_2$, then
        $$
        (\psi_1)_\frac{m_1}{m_2}(z) = \psi_2(z) \text{ and } A_1(t) = (A_2)_{\frac{m_2}{m_1}}(t),
        $$
        \item if $m_1 > m_2$, then
        $$
        (\psi_2)_\frac{m_2}{m_1}(z) = \psi_1(z) \text{ and } A_2(t) = (A_1)_{\frac{m_1}{m_2}}(t)
        $$
    \end{enumerate}
    for every $z \in [0,\infty)$ and every $t \in \mathbb{I}$.
\end{lemma}
Applying Theorem \ref{thm:identifiability} and Lemma \ref{lem:identifiability_nes}, we get the following corollary as consequence:
\begin{corollary}
    Let $\psi_1,\psi_2 \in \Psi$ fulfill $(1-\psi_i(\frac{1}{\cdot})) \in \mathcal{R}_{-\frac{1}{m_i}}$ for $m_i \geq 1$ and $i \in \{1,2\}$. Furthermore let $A_1, A_2 \in \mathcal{A}\setminus\{A_M\}$, then the following two assertions hold:
    \begin{enumerate}
        \item If $m_1 \leq m_2$ and $\frac{m_2}{m_1} \neq \frac{1- \tau_{A_2}}{1- \tau_{A_1}}$, then $C_{\psi_1,A_1} \neq C_{\psi_2,A_2}$.
        \item If $m_2 < m_1$ and $\frac{m_2}{m_1} \neq \frac{1- \tau_{A_1}}{1- \tau_{A_2}}$, then $C_{\psi_1,A_1} \neq C_{\psi_2,A_2}$.
    \end{enumerate}
\end{corollary}
\begin{proof}
We prove the first assertion, the proof of the second assertion is analogous. Suppose that
$$
C_{\psi_1,A_1}(x,y) = C_{\psi_2,A_2}(x,y)
$$
holds for every $x,y \in \mathbb{I}$. We distinguish two cases.\\
\textit{Case 1: $\tau_{A_1} \leq \tau_{A_2}$.} By Theorem \ref{thm:identifiability}, we obtain
$$
\psi_1(z) = (\psi_2)_{\frac{1-\tau_{A_2}}{1-\tau_{A_1}}}(z) \quad \text{for every } z \in [0,\infty).
$$
Since $m_1 \leq m_2$, Lemma \ref{lem:identifiability_nes} yields
$$
(\psi_1)_{\frac{m_1}{m_2}}(z) = \psi_2(z) \quad \text{for every } z \in [0,\infty),
$$
and therefore $\frac{m_2}{m_1} = \frac{1-\tau_{A_2}}{1-\tau_{A_1}}$.

\textit{Case 2: $\tau_{A_1} > \tau_{A_2}$.} By Theorem \ref{thm:identifiability}, we obtain
$$
\psi_2(z) = (\psi_1)_{\frac{1-\tau_{A_1}}{1-\tau_{A_2}}}(z) \quad \text{for every } z \in [0,\infty),
$$
and Lemma \ref{lem:identifiability_nes} then yields $\frac{m_2}{m_1} = \frac{1-\tau_{A_2}}{1-\tau_{A_1}}$, a contradiction.
\end{proof}
The function $(A)_{1-\tau_A}$ will be key both for convergence and for estimation - the following 
examples show that $(A)_{1-\tau_A}$ can, but need not be a Pickands dependence 
function.
\begin{example}[Gumbel]\label{ex:gumbel_pickands_vs_transformed}
    We again consider the Gumbel family, whose Pickands dependence function is given by
    $$
    A_{G_\theta}(t) = (t^\theta + (1-t)^\theta)^\frac{1}{\theta}, \quad \theta \in [1,\infty)
    $$
    for every $t \in \mathbb{I}$. Using Lemma \ref{lem:kendalls_tau_trans} with $A =1=A_\Pi$, we obtain that $\tau_{A_{G_\theta}} = 1 - \frac{1}{\theta}$. 
    According to Example \ref{ex2} it follows immediately that 
    $$
    (A_{G_\theta})_{1-\tau_{A_\theta}}=((A_\Pi)_\theta)_{1-\tau_{A_\theta}} = A_\Pi,
    $$
    i.e., we have $(A_{G_\theta})_{1-\tau_{A_\theta}}(t) =1$ for every $t \in \mathbb{I}$.
\end{example}
\begin{remark}
    The boundary case $\alpha=1-\tau_A$ mentioned in Remark \ref{rem:illustration.A.alpha}
    is of particular interest and underlines the fact that the Gumbel family has a very special role in $\mathcal{C}_{am}$. In fact, if for 
    $A \in \mathcal{A}\setminus\{A_M\}$ we have $(A)_{1-\tau_A} \in \mathcal{A}$, then 
    Lemma \ref{lem:kendalls_tau_trans} yields $\tau_{(A)_{1-\tau_A}}=0$, implying 
    $(A)_{1-\tau_A}=A_\Pi$. Having this and using Example \ref{ex2} as well 
    as Lemma \ref{lem:properties_trans_pickands} directly yields that 
    $A$ has to be the Pickands dependence function of an element in the Gumbel family.
    In other words, for any Pickands dependence function outside the Gumbel family, 
    the function $(A)_{1-\tau_A}$ is not a Pickands dependence function.
\end{remark}
\begin{example}[Galambos]\label{ex:galambos_pickands_vs_transformed}
We consider the Pickands dependence function $A_\theta$ of a Galambos copula, given by 
$$
A_\theta(t) = 1 - [t^{-\theta} + (1-t)^{-\theta}]^{-1/\theta}, \quad 
\theta \in (0,\infty).
$$
for every $t \in \mathbb{I}$. Calculating Kendall's $\tau$ of $C_{A_\theta} \in \mathcal{C}_{ev}$ and using the fact that
$$
A_\theta''(t) = (1 + \theta)(t^{-\theta} + (1-t)^{-\theta})^{-\frac{1}{\theta}-2}(1-t)^{-\theta-2}t^{-\theta-2}
$$ for every $t \in \mathbb{I}$, we obtain that
\begin{align}
\nonumber \tau_{A_\theta} &= \int_{\mathbb{I}}\frac{t(1-t)}{A_\theta(t)} \mathrm{d}D^+A_\theta(t) = \int_\mathbb{I} \frac{t(1-t)}{A_\theta(t)} A_\theta''(t) \mathrm{d}\lambda(t) \\&=
(1+\theta)\int_\mathbb{I} \frac{t^{-\theta - 1}(1-t)^{-\theta - 1}(t^{-\theta} + (1-t)^{-\theta})^{-\frac{1}{\theta}-2}}{1-(t^{-\theta} + (1-t)^{-\theta})^{-\frac{1}{\theta}}} \mathrm{d}\lambda(t).
\end{align}
The functions $(A_\theta)_{1-\tau_{A_\theta}}$ for the parameters $\theta \in \{\frac{1}{5}, \frac{1}{2}, 1, 3, 5, 10\}$ are depicted in Figure \ref{fig:galambos_pick_vs_trans}. Obviously, none of them 
is a Pickands dependence function.
\begin{figure}[!ht]
	\centering
	\includegraphics[width=1\textwidth]{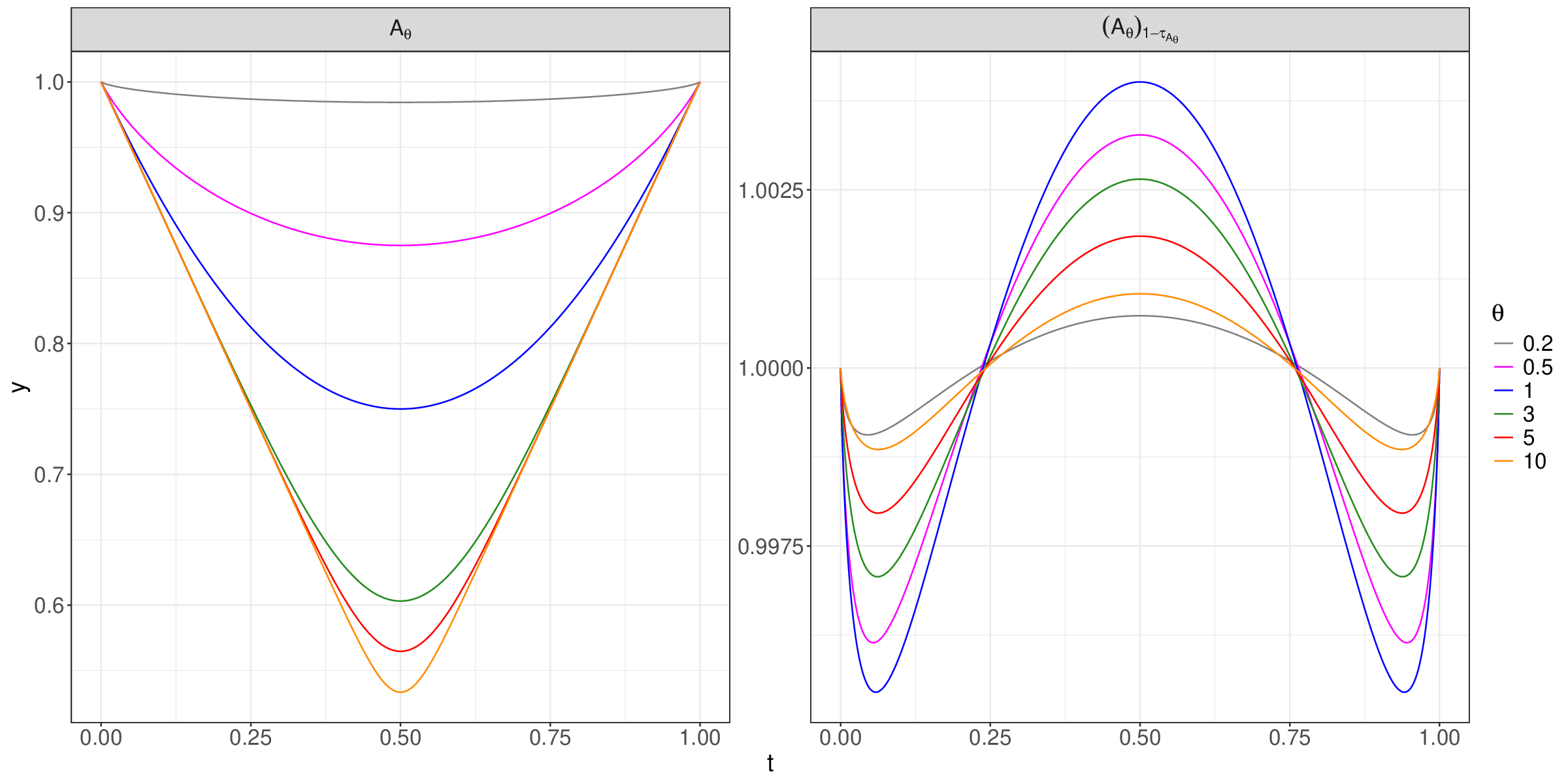}
	\caption{Plots of the Pickands dependence functions $A_\theta$ and the corresponding transformed functions $(A_\theta)_{1- \tau_{A_\theta}}$ for parameters $\theta \in \{\frac{1}{5}, \frac{1}{2}, 1, 3, 5, 10\}$ as defined in Example \ref{ex:galambos_pickands_vs_transformed}.}
	\label{fig:galambos_pick_vs_trans}
\end{figure}
\end{example}
\begin{example}[Marshall-Olkin] \label{ex:marshall_olkin}
For arbitrary parameters $\alpha, \beta \in (0,1]$ the Pickands dependence function $A_{\alpha,\beta} \in \mathcal{A}$ of a Marshall-Olkin copula $C_{A_{\alpha,\beta}} \in \mathcal{C}_{ev}$ is given by
$$
A_{\alpha,\beta}(t) = \begin{cases}
    1 - \beta t,& \text{ if } t \in [0,\frac{\alpha}{\alpha + \beta})\\
    1 - \alpha (1-t),& \text{ if } t \in [\frac{\alpha}{\alpha + \beta},1]
\end{cases}
$$
for every $t \in \mathbb{I}$. According to \citep{ghoudi1998}, Kendall's $\tau$ of $C_{A_{\alpha,\beta}}$ is given by
$$
\tau_{A_{\alpha,\beta}} = \frac{\alpha \beta}{\alpha + \beta - \alpha \beta}
$$
and therefore, simplifying, we have that
\begin{align*}
(A_{\alpha,\beta})_{1-\tau_{A_{\alpha,\beta}}}(t) &= (t^{1-\tau_{A_{\alpha,\beta}}} + (1-t)^{1-\tau_{A_{\alpha,\beta}}})^\frac{1}{1-\tau_{A_{\alpha,\beta}}}A_{\alpha,\beta}^\frac{1}{1-\tau_{A_{\alpha,\beta}}}\left(\frac{t^{1-\tau_{A_{\alpha,\beta}}}}{t^{1-\tau_{A_{\alpha,\beta}}} + (1-t)^{1-\tau_{A_{\alpha,\beta}}}}\right) \\&=
\begin{cases}
    \left((1-\beta)t^\frac{\alpha + \beta -2\alpha \beta}{\alpha + \beta -\alpha\beta} + (1-t)^\frac{\alpha + \beta -2\alpha \beta}{\alpha + \beta -\alpha\beta}\right)^\frac{\alpha + \beta -\alpha \beta}{\alpha + \beta -2\alpha \beta}, &\text{ if } t \in \left[0,\left(\frac{\alpha}{\alpha + \beta}\right)^\frac{\alpha + \beta -\alpha \beta}{\alpha + \beta -2\alpha \beta}\right)\\
       \left((1-\alpha)(1-t)^\frac{\alpha + \beta -2\alpha \beta}{\alpha + \beta -\alpha\beta} + t^\frac{\alpha + \beta -2\alpha \beta}{\alpha + \beta -\alpha\beta}\right)^\frac{\alpha + \beta -\alpha \beta}{\alpha + \beta -2\alpha \beta},&\text{ if } t \in \left[\left(\frac{\alpha}{\alpha + \beta}\right)^\frac{\alpha + \beta -\alpha \beta}{\alpha + \beta -2\alpha \beta}, 1\right]
\end{cases}
\end{align*}
for every $t \in \mathbb{I}$.
The functions $(A_{\alpha,\beta})_{1-\tau_{A_{\alpha,\beta}}}$ for some parameters are depicted in Figure \ref{fig:marshall_olkin_pick_vs_trans}. 
As in the previous example, none of them 
is a Pickands dependence function.
\begin{figure}[!ht]
	\centering
	\includegraphics[width=1\textwidth]{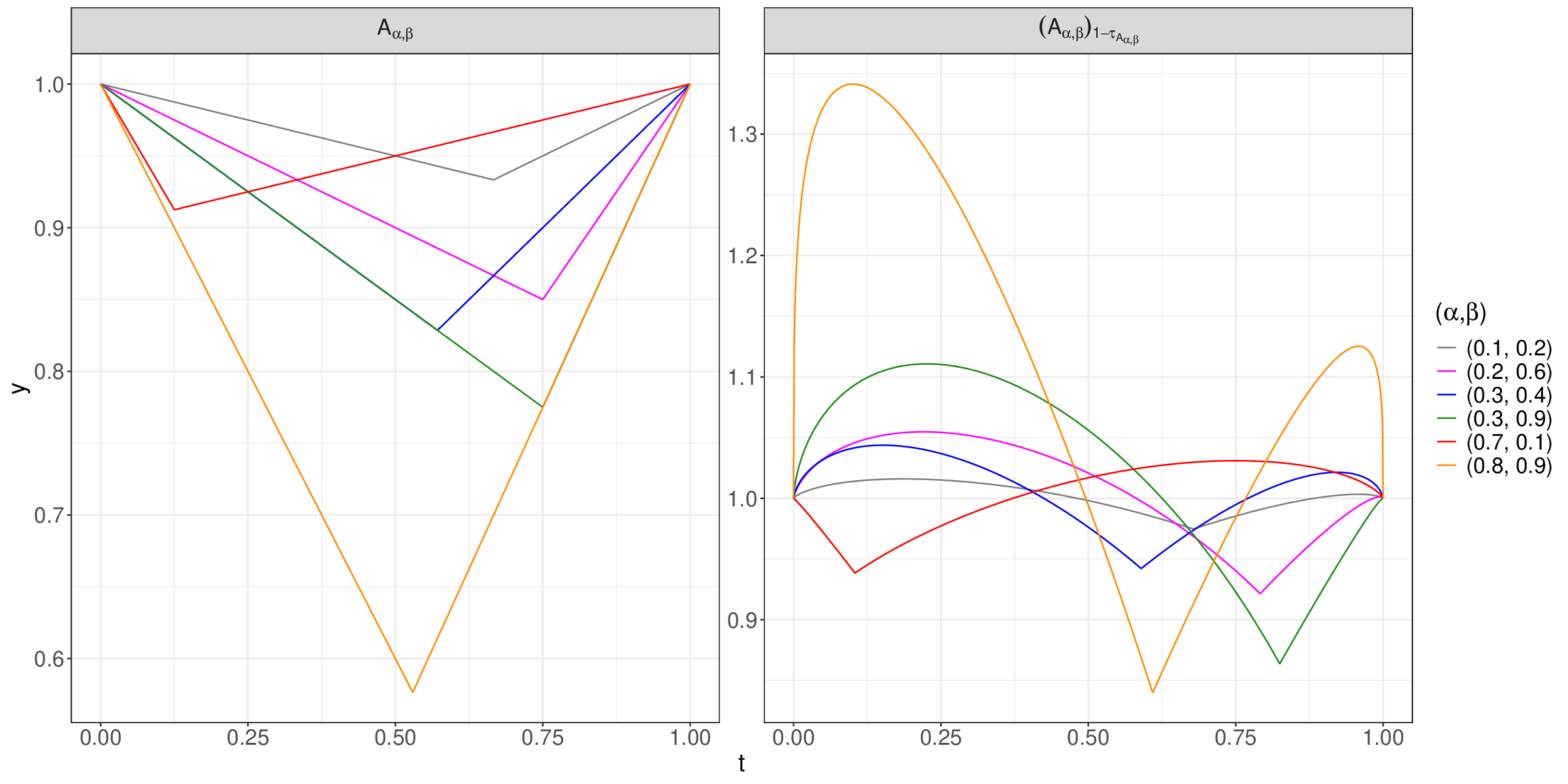}
	\caption{Plots of the Pickands dependence functions $A_{\alpha,\beta}$ and the corresponding transformed functions $(A_{\alpha,\beta})_{1- \tau_{A_{\alpha,\beta}}}$ for parameters $(\alpha,\beta) \in \{(0.1, 0.2), (0.2, 0.6), (0.3,0.4), (0.3,0.9), (0.7,0.1), (0.8,0.9)\}$ as considered 
    in Example \ref{ex:marshall_olkin}.}
	\label{fig:marshall_olkin_pick_vs_trans}
\end{figure}
\end{example}
\begin{example}[Piecewise linear Pickands dependence function] \label{ex:non_smooth_est}
    Defining the probability measure $\vartheta$ by 
    $$
    \vartheta := \tfrac{1}{5}(\delta_\frac{1}{10} + \delta_\frac{3}{10} + \delta_\frac{1}{2} + \delta_\frac{7}{10} + \delta_\frac{9}{10})
    $$
    we obviously have $\mathbb{E}[\vartheta] = \frac{1}{2}$. Calculating the associated Pickands dependence function yields that
    \begin{align*}
    A(t) = \begin{cases}
        1-t, &\text{ if } t \in [0,\frac{1}{10})\\
        \frac{24}{25} - \frac{3}{5}t, &\text{ if } t \in [\frac{1}{10},\frac{3}{10})\\
        \frac{21}{25} - \frac{1}{5}t, &\text{ if } t \in [\frac{3}{10},\frac{1}{2})\\
           \frac{16}{25} + \frac{1}{5}t, &\text{ if } t \in [\frac{1}{2},\frac{7}{10})\\
        \frac{9}{25} + \frac{3}{5}t, &\text{ if } t \in [\frac{7}{10},\frac{9}{10})\\
        t, &\text{ if } t \in [\frac{9}{10},1]\\
    \end{cases}
    \end{align*}
    for every $t \in \mathbb{I}$. Therefore, Kendall's $\tau$ is given by
    \begin{align*}
    \tau_A &= \int_\mathbb{I} \tfrac{t(1-t)}{A(t)} \mathrm{d}D^+A(t) = 2\int_\mathbb{I} \tfrac{t(1-t)}{A(t)} \mathrm{d}\vartheta(t) \\& =
    \frac{2}{5}\left[\frac{\frac{1}{10}(1-\frac{1}{10})}{A(\frac{1}{10})} + \frac{\frac{3}{10}(1-\frac{3}{10})}{A(\frac{3}{10})} + \frac{\frac{1}{2}(1-\frac{1}{2})}{A(\frac{1}{2})} + \frac{\frac{7}{10}(1-\frac{7}{10})}{A(\frac{7}{10})} + \frac{\frac{9}{10}(1-\frac{9}{10})}{A(\frac{9}{10})}\right]
    =
    \frac{5177}{12025}.
    \end{align*}
    Considering $1- \tau_A = \frac{6848}{12025}$ hence yields
    $$
    (A)_{1-\tau_A}(t) = \left(t^\frac{6848}{1205} + (1-t)^\frac{6848}{1205}\right)^\frac{1205}{6848}A^\frac{1205}{6848}\left(\frac{t^\frac{6848}{1205}}{t^\frac{6848}{1205} + (1-t)^\frac{6848}{1205}}\right).
    $$
\end{example}
\noindent Theorem \ref{thm:identifiability} has direct implications to the convergence of Archimax copulas. The next example shows that, 
even in very simple cases, convergence of $((A_n)_{1-\tau_{A_n}})_{n \in \mathbb{N}}$ to $(A)_{1-\tau_A}$ does not generally imply convergence of $(A_n)_{n \in \mathbb{N}}$ to $A$.
\begin{example}
For $A := A_{G_2}$ we have $\tau_A = \frac{1}{2}$. 
Setting $A_n(t): = A_{G_3}(t)=(t^{3} + (1-t)^{3})^{\frac{1}{3}}$ for 
every $t \in \mathbb{I}$ and every $n \in \mathbb{N}$, the sequence
$(A_n)_{n \in \mathbb{N}}$ obviously does not converge to $A$. 
Considering, however, $\tau_{A_n} = \frac{2}{3}$, and using Example \ref{ex2} 
as well as Theorem \ref{lem:properties_trans_pickands} directly yields 
$$
(A_n)_{1-\tau_{A_n}}(t) = (G_3)_{\frac{1}{3}}(t) =1 = A_\Pi(t) = (A)_{1-\tau_{A}}(t)
$$
for every $t \in \mathbb{I}$.
\end{example}
\section{Convergence of bivariate Archimax copulas}\label{sec:conv_archimax}
\noindent The main goal of this section is to prove the following theorem, which extends the corresponding results for 
$\mathcal{C}_{ar}$ and $\mathcal{C}_{ev}$ as established in \cite{bernoulli} to 
the larger class $\mathcal{C}_{am}$ and which is key for the nonparametric 
estimators developed in Section \ref{sec:estimator}.
Since the proof is rather technical, we split it into several lemmas.
\begin{theorem}\label{thrm:main_result_convergence}
    Consider $\psi,\psi_1,\psi_2,... \in \Psi$ as well as $A,A_1,A_2,... \in \mathcal{A}\setminus \{A_M\}$ and let $C_{\psi,A}, C_{\psi_1,A_1}, C_{\psi_2,A_2}, \ldots \mathcal{C}_{am}$ denote the corresponding
    Archimax copulas. Furthermore, let $\tau_{A},\tau_{A_1}, \tau_{A_2}, \ldots$ be the Kendall's $\tau$s of $C_A, C_{A_1},C_{A_2}\ldots \in \mathcal{C}_{ev}$. Then the following assertions are equivalent:
    \begin{enumerate}[label=(\roman*)]
        \item $(\psi_n)_{1-\tau_{A_n}} \overset{n \rightarrow \infty}{\longrightarrow}(\psi)_{1-\tau_{A}}$ uniformly on $[0,\infty)$ and $(A_n)_{1-\tau_{A_n}} \overset{n \rightarrow \infty}{\longrightarrow}(A)_{1-\tau_{A}}$ uniformly on $\mathbb{I}$;
          \item $(\varphi_n)_{1-\tau_{A_n}} \overset{n \rightarrow \infty}{\longrightarrow}(\varphi)_{1-\tau_{A}}$ pointwise on $(0,1]$ and $(A_n)_{1-\tau_{A_n}} \overset{n \rightarrow \infty}{\longrightarrow}(A)_{1-\tau_{A}}$ uniformly on $\mathbb{I}$;
        \item $d_\infty(C_{\psi_n,A_n}, C_{\psi,A}) \overset{n \rightarrow \infty}{\longrightarrow} 0$;
        \item $D_p(C_{\psi_n,A_n}, C_{\psi,A}) \overset{n \rightarrow \infty}{\longrightarrow} 0$, for every $p \in [1,\infty]$;
        \item $C_{\psi_n,A_n} \overset{n\rightarrow \infty}{\longrightarrow}C_{\psi,A}$ weakly conditional.
    \end{enumerate}
\end{theorem}
\begin{proof}
 Immediate consequence of Lemmas \ref{lem:convergence} and \ref{lem:convergence_wcc}.
\end{proof}
\begin{lemma}\label{lem:convergence}
    Let $\psi,\psi_1,\psi_2, ... \in \Psi$, $A,A_1,A_2,... \in \mathcal{A}\setminus\{A_M\}$ and let $\tau_A, \tau_{A_1}, \tau_{A_2}, \ldots$ be the Kendall's $\tau$s of $C_A, C_{A_1},C_{A_2}\ldots \in \mathcal{C}_{ev}$. Furthermore, let $C_{\psi,A}, C_{\psi_1,A_1}, C_{\psi_2,A_2}, ... \in \mathcal{C}_{am}$ be the corresponding Archimax copulas and let $(\psi)_{1-\tau_{A}}, (\psi_1)_{1-\tau_{A_1}}, (\psi_2)_{1-\tau_{A_2}}, \ldots$ and $(A)_{1-\tau_{A}}, (A_1)_{1-\tau_{A_1}}, (A_2)_{1-\tau_{A_2}}, \ldots$ be defined as in eqs. \eqref{eq:trans_arch} and \eqref{eq:trans_pickands}, respectively.
    Then the first of the following two assertions implies the second one:
    \begin{itemize}
        \item[(i)] $C_{\psi_n,A_n} \overset{n \rightarrow \infty}{\longrightarrow} C_{\psi,A}$ uniformly on $\mathbb{I}^2$.
        \item[(ii)] $(\psi_n)_{1-\tau_{A_n}} \overset{n \rightarrow \infty}{\longrightarrow} (\psi)_{1-\tau_{A}}$ uniformly on $[0,\infty)$ and 
        $
     (A_n)_{1-\tau_{A_n}}\overset{n \rightarrow \infty}{\longrightarrow} 
     (A)_{1-\tau_{A}}
        $ uniformly on $\mathbb{I}$.
    \end{itemize}
\end{lemma}
\begin{proof}
    Assume that $C_{\psi_n,A_n} \overset{n \rightarrow \infty}{\longrightarrow} C_{\psi,A}$ uniformly on $\mathbb{I}^2$. Then 
    $F^K_{\psi_n,A_n} \overset{n\rightarrow \infty}{\longrightarrow} F^K_{\psi,A}$ weakly on $\mathbb{I}$.
    Applying Lemma \ref{lem:kendall_is_archimedean}, each Kendall distribution function $F^K_{\psi,A}, F^K_{\psi_1,A_1}, F^K_{\psi_2,A_2}, \ldots$ induces a (unique, normalized) Archimedean generator $(\psi)_{1-\tau_A}, (\psi_1)_{1-\tau_{A_1}}, (\psi_2)_{1-\tau_{A_2}}, \ldots$, and the corresponding Archimedean copulas $C_{(\psi)_{1-\tau_A}}, C_{(\psi_1)_{1-\tau_{A_1}}}, C_{(\psi_2)_{1-\tau_{A_2}}}, \ldots$ have Kendall distribution functions $F^K_{\psi,A}, F^K_{\psi_1,A_1}, F^K_{\psi_2,A_2}, \ldots$.
    Applying \cite[Theorem 4.1]{bernoulli} then yields that the sequence $\{(\psi_n)_{1-\tau_{A_n}}\}_{n \in \mathbb{N}}$ converges uniformly to $(\psi)_{1-\tau_A}$. \\
    Using the Lipschitz property of copulas (see \citep{dur_princ}), the fact that $\{C_{\psi_n,A_n}\}_{n \in \mathbb{N}}$ converges to $C_{\psi,A}$ uniformly, and the triangle inequality, we obtain that
    \begin{align*}
    |(\psi_n)_{1-\tau_{A_n}}((A_n)_{1-\tau_{A_n}}(t)) - (\psi)_{1-\tau_{A}}((A)_{1-\tau_{A}}(t))| &= |C_{\psi_n,A_n}((\psi_n)_{1-\tau_{A_n}}(t),(\psi_n)_{1-\tau_{A_n}}(1-t)) - C_{\psi,A}((\psi)_{1-\tau_{A}}(t),(\psi)_{1-\tau_{A}}(1-t))| \\&\leq
    |C_{\psi_n,A_n}((\psi_n)_{1-\tau_{A_n}}(t),(\psi_n)_{1-\tau_{A_n}}(1-t)) - C_{\psi_n,A_n}((\psi)_{1-\tau_{A}}(t),(\psi)_{1-\tau_{A}}(1-t))| \\&\quad + |C_{\psi_n,A_n}((\psi)_{1-\tau_{A}}(t),(\psi)_{1-\tau_{A}}(1-t)) - C_{\psi,A}((\psi)_{1-\tau_{A}}(t),(\psi)_{1-\tau_{A}}(1-t))| \\&\leq
    |(\psi_n)_{1-\tau_{A_n}}(t) - (\psi)_{1-\tau_{A}}(t)| + |(\psi_n)_{1-\tau_{A_n}}(1-t) - (\psi)_{1-\tau_{A}}(1-t)|\\&\quad + |C_{\psi_n,A_n}((\psi)_{1-\tau_{A}}(t),(\psi)_{1-\tau_{A}}(1-t)) - C_{\psi,A}((\psi)_{1-\tau_{A}}(t),(\psi)_{1-\tau_{A}}(1-t))| \\& \overset{n\rightarrow \infty}{\longrightarrow} 0
    \end{align*}
uniformly in $t \in \mathbb{I}$.\\
Next, we prove that $\{(A_n)_{1-\tau_{A_n}}\}_{n \in \mathbb{N}}$ converges uniformly to $(A)_{1-\tau_{A}}$ on $\mathbb{I}$ and start with showing $\sup_{n \in \mathbb{N}} \tau_{A_n} <1$. Assume, on the contrary, that  
$\sup_{n \in \mathbb{N}} \tau_{A_n}=1$ holds. Then there exists some subsequence 
$(\tau_{A_{n_j}})_{j \in \mathbb{N}}$ with 
$\tau_{A_{n_j}} \overset{j \rightarrow \infty}{\longrightarrow} 1$, implying 
$$
t=\lim_{j \rightarrow \infty} \tau_{A_{n_j}} t + (1-\tau_{A_{n_j}})F_{\psi_{n_j}}^K(t) =
\lim_{j \rightarrow \infty} F_{\psi_{n_j},A_{n_j}}^K(t) = F^K_{\psi,A}(t)
$$
for every point of continuity $t$ of $F^K_{\psi,A}$. Since the points of continuity of $F^K_{\psi,A}$ are dense in $\mathbb{I}$, it follows that $C_{\psi,A}=M$, a contradiction. This shows that $\sup_{n \in \mathbb{N}} \tau_{A_n} <1$, and hence there exists some $\varepsilon_0 \in (0,1)$ such that
$$
\tau_{A_n} < 1- \varepsilon_0
$$
for every $n \in \mathbb{N}$. Using Lemma \ref{lem:properties_trans_pickands}, we obtain that
\begin{align}\label{eq: est_trans_A}
    (A_n)_{1-\tau_{A_n}}(t) \leq 
    2^\frac{\tau_{A_n}}{1-\tau_{A_n}} \text{ and }   (A)_{1-\tau_{A}}(t) \leq 
    2^\frac{\tau_{A}}{1-\tau_{A}}
\end{align}
for every $t \in \mathbb{I}$.
Using eq. \eqref{eq: est_trans_A} and the fact that the sequence $\{(\psi_n)_{1-\tau_{A_n}} \circ (A_n)_{1-\tau_{A_n}}\}_{n\in\mathbb{N}}$ converges uniformly to $(\psi)_{1-\tau_{A}} \circ (A)_{1-\tau_{A}}$ on $\mathbb{I}$, we obtain the existence of an $N \in \mathbb{N}$ such that for every $n \geq N$ and every $t \in \mathbb{I}$ we have that
$$
(\psi_n)_{1-\tau_{A_n}}( (A_n)_{1-\tau_{A_n}}(t)) > (\psi)_{1-\tau_{A}}( (A)_{1-\tau_{A}}(t)) - \frac{\psi(2^{\tau_A})}{2} \geq \frac{\psi(2^{\tau_A})}{2}.
$$
Furthermore, applying eq. \eqref{eq: est_trans_A}, we have that
$$
(\psi_n)_{1-\tau_{A_n}}( (A_n)_{1-\tau_{A_n}}(t)) \geq \psi_n(2^{\tau_{A_n}}) > \psi_n(2^{1-\varepsilon_0})
$$
for every $t \in \mathbb{I}$ and every $n \in \{1,\ldots, N-1\}$.
Therefore, it follows for every $n \in \mathbb{N}$ and every $t \in \mathbb{I}$ that
$$
(\psi_n)_{1-\tau_{A_n}}((A_n)_{1-\tau_{A_n}}(t)) \geq \min\left\{\psi_1(2^{1-\varepsilon_0}), \psi_2(2^{1-\varepsilon_0}), \ldots,\psi_{N-1}(2^{1-\varepsilon_0}),\tfrac{\psi(2^{\tau_A})}{2} \right\} =: a_0.
$$
Using Lemma \ref{lem:positivity_generator}, considering $1- \varepsilon_0 < 1$ and $\tau_A < 1$, it immediately follows that $a_0 > 0$.\\
Using convexity of $\varphi^\frac{1}{1-\tau_A}, \varphi_1^\frac{1}{1-\tau_{A_1}}, \varphi_2^\frac{1}{1-\tau_{A_2}}, \ldots$ and the fact that, according to \cite{bernoulli}, the sequence $\{\varphi_n^\frac{1}{1-\tau_{A_n}}\}_{n \in \mathbb{N}}$ converges pointwise to $\varphi^\frac{1}{1-\tau_A}$ on $[a_0,1]$, we obtain uniform convergence of 
$\{\varphi_n^\frac{1}{1-\tau_{A_n}}\}_{n \in \mathbb{N}}$ to $\varphi^\frac{1}{1-\tau_A}$ on $[a_0,1]$.\\
Furthermore, for every $n \in \mathbb{N}$ and $t \in \mathbb{I}$, we have $(A_n)_{1-\tau_{A_n}}(t)^{1-\tau_{A_n}} \leq 2^{\tau_{A_n}} < 2$ as well as  $(A)_{1-\tau_{A}}(t)^{1-\tau_{A}} \leq 2^{\tau_{A}}<2$. Applying Lemma \ref{lem:positivity_generator} therefore yields 
$$
\varphi_n(\psi_n((A_n)_{1-\tau_{A_n}}(t)^{1-\tau_{A_n}}))^\frac{1}{1-\tau_{A_n}} = (A_n)_{1-\tau_{A_n}}(t)
$$
as well as
$$
\varphi(\psi((A)_{1-\tau_{A}}(t)^{1-\tau_{A}}))^\frac{1}{1-\tau_{A}} = (A)_{1-\tau_{A}}(t).
$$
Using all of the above together with the fact that the sequence $\{(\psi_n)_{1-\tau_{A_n}} \circ (A_n)_{1-\tau_{A_n}}\}_{n\in\mathbb{N}}$ converges uniformly to $(\psi)_{1-\tau_{A}} \circ (A)_{1-\tau_{A}}$ on $\mathbb{I}$ and that $\varphi^\frac{1}{1-\tau_A}$ is continuous on $[a_0,1]$, again using continuous convergence, we finally obtain that 
\begin{align*}
\lim_{n \rightarrow \infty}(A_n)_{1-\tau_{A_n}}(t) &= \lim_{n \rightarrow \infty}\varphi_n(\psi_n((A_n)_{1-\tau_{A_n}}(t)^{1-\tau_{A_n}}))^\frac{1}{1-\tau_{A_n}} \\&= \lim_{n \rightarrow \infty}\varphi_n((\psi_n)_{1-\tau_{A_n}}((A_n)_{1-\tau_{A_n}}(t)))^\frac{1}{1-\tau_{A_n}} \\&=
\varphi((\psi)_{1-\tau_{A}}((A)_{1-\tau_{A}}(t)))^\frac{1}{1-\tau_{A}} \\&=
\varphi(\psi((A)_{1-\tau_{A}}(t)^{1-\tau_{A}}))^\frac{1}{1-\tau_{A}} \\& =
(A)_{1-\tau_{A}}(t)
\end{align*}
uniformly in $t \in \mathbb{I}$. 
\end{proof}
The next example shows that Lemma \ref{lem:convergence} does not extend to the case of convergence to the minimum copula $M$. 
\begin{example}
 Take any sequence $(A_n)_{n \in \mathbb{N}}$ in $\mathcal{A}\setminus\{A_M\}$ converging to $A_M$, choose an arbitrary $\psi \in \Psi$ and set $\psi_n(z) := \psi(z)$ for every $n\in \mathbb{N}$ and $z \in [0,\infty)$. 
 Then obviously
 $$
 \lim_{n \rightarrow \infty}C_{\psi_n,A_n}(x,y) = M(x,y)
 $$
 for all $x,y \in \mathbb{I}$. Considering $\tau_{A_n} \overset{n\rightarrow \infty}{\longrightarrow} 1$, we have that
 $$
 \lim_{n\rightarrow \infty} (\psi_n)_{1-\tau_{A_n}}(z) = \lim_{n \rightarrow \infty}\psi_n(z^{1-\tau_{A_n}}) = \lim_{n \rightarrow \infty}\psi(z^{1-\tau_{A_n}}) = \frac{1}{2}
 $$
 for every $z \in (0,\infty)$. Since obviously $\lim_{n\rightarrow \infty} (\psi_n)_{1-\tau_{A_n}}(0) = 1$, the limit of the sequence $((\psi_n)_{1-\tau_{A_n}})_{n \in \mathbb{N}}$ is not even a generator, implying that Lemma \ref{lem:convergence} can not be extended.
\end{example}

\noindent The following lemma shows that if the sequence $((A_n)_{1-\tau_{A_n}})_{n \in \mathbb{N}}$ converges pointwise to $(A)_{1-\tau_A}$, then the right-hand derivatives $D^+(A_n)_{1-\tau_{A_n}}$ converge to $D^+(A)_{1-\tau_A}$ at every point of continuity of $D^+(A)_{1-\tau_A}$.
\begin{lemma}\label{lem:convergence_derivatives}
    Consider $A,A_1,A_2, \ldots \in \mathcal{A}\setminus\{A_M\}$, let $\tau_A, \tau_{A_1},\tau_{A_2},\ldots$ denote the Kendall's $\tau$s of $C_A, C_{A_1}, C_{A_2}, \ldots \in \mathcal{C}_{ev}$. Define $(A)_{1-\tau_A},(A_1)_{1-\tau_{A_1}},(A_2)_{1-\tau_{A_2}}, \ldots$ according to eq. \eqref{eq:trans_pickands} and assume 
    that the sequences $\{(A_n)_{1-\tau_{A_n}}\}_{n \in \mathbb{N}}$ converges uniformly to $(A)_{1-\tau_A}$ on $\mathbb{I}$. Then for every sequence $(t_n)_{n \in \mathbb{N}}$  converging in $(0,1)$ to some $t_0 \in \mathrm{Cont}(D^+(A)_{1-\tau_A}) \cap (0,1)$, we have that
    $$
    a_n := D^+(A_n)_{1-\tau_{A_n}}(t_n) \overset{n\rightarrow\infty}{\longrightarrow}D^+(A)_{1-\tau_{A}}(t_0).
    $$
\end{lemma}
\begin{proof}
    We will prove that every subsequence $(a_{n_j})_{j \in \mathbb{N}}$ of $(a_n)_{n \in \mathbb{N}}$ has another subsequence that converges to  $D^+(A)_{1-\tau_{A}}(t_0)$. Letting $(a_{n_j})_{j \in \mathbb{N}}$ denote an arbitrary subsequence of $(a_n)_{n \in \mathbb{N}}$, using compactness of $\mathcal{A}$, there exists some subsequence $(A_{n_{j_k}})_{k \in \mathbb{N}}$ of $(A_{n_j})_{j \in \mathbb{N}}$ and some $A_0 \in \mathcal{A}$ such that $A_{n_{j_k}} \overset{k \rightarrow \infty}{\longrightarrow}A_0$ uniformly on $\mathbb{I}$. 
    This directly yields that $C_{A_{n_{j_k}}} \overset{k \rightarrow \infty}{\longrightarrow} C_{A_0}$ uniformly on $\mathbb{I}^2$ and that  $\tau_{A_{n_{j_k}}} \overset{k \rightarrow \infty}{\longrightarrow} \tau_{A_0}$. Using the aforementioned facts and uniform convergence of $(A_{n_{j_k}})_{k \in \mathbb{N}}$ to $A_0$, applying (the last assertion in) Lemma \ref{lem:properties_trans_pickands}, we obtain that $(A_{n_{j_k}})_{1-\tau_{A_{n_{j_k}}}} \overset{k \rightarrow \infty}{\longrightarrow} (A_0)_{1-\tau_{A_0}}$ uniformly on $\mathbb{I}$. Furthermore, we have by assumption that $(A_{n_{j_k}})_{1-\tau_{A_{n_{j_k}}}} \overset{k \rightarrow \infty}{\longrightarrow} (A)_{1-\tau_{A}}$ uniformly on $\mathbb{I}$. 
    Since the limit is unique, this yields that $(A)_{1-\tau_{A}}(t) = (A_0)_{1-\tau_{A_0}}(t)$ for every $t \in \mathbb{I}$ as well as $D^+(A)_{1-\tau_{A}}(t) = D^+(A_0)_{1-\tau_{A_0}}(t)$ for every $t \in (0,1)$.\\
    Considering $t_0\in \mathrm{Cont}(D^+(A)_{1-\tau_A})$, we have $t_0 \in \mathrm{Cont}(D^+(A_0)_{1-\tau_{A_0}})$, implying 
    $\kappa_{1-\tau_{A_0}}(t_0) \in \mathrm{Cont}(D^+A_0)$. Applying convexity of $A_{n_{j_k}}$ and the fact that $\kappa_{1-\tau_{A_{n_{j_k}}}} \overset{k \rightarrow \infty} {\longrightarrow} \kappa_{1-\tau_{A_{0}}}$ uniformly 
    on $\mathbb{I}$, 
    we obtain that
    $$
    \lim_{k \rightarrow \infty}D^+A_{n_{j_k}}\left(\kappa_{1-\tau_{A_{n_{j_k}}}}(t_{n_{j_k}})\right) = D^+A_0\left( \kappa_{1-\tau_{A_{0}}}(t_0)\right).
    $$
    Moreover, using uniform convergence of $(A_{n_{j_k}})_{k \in \mathbb{N}}$ 
    to $A_0$ altogether yields
    \begin{align*}
    \lim_{k\rightarrow \infty}a_{n_{j_k}} &= \lim_{k\rightarrow \infty}D^+(A_{n_{j_k}})_{1-\tau_{A_{n_{j_k}}}}(t_{n_{j_k}}) \\&= \lim_{k\rightarrow \infty}   (A_{n_{j_k}})_{1-\tau_{A_{n_{j_k}}}}(t_{n_{j_k}})\left[\frac{t_{n_{j_k}}^{-\tau_{A_{n_{j_k}}}} - (1-t_{n_{j_k}})^{-\tau_{A_{n_{j_k}}}}}{t_{n_{j_k}}^{1-\tau_{A_{n_{j_k}}}} + (1-t_{n_{j_k}})^{1-\tau_{A_{n_{j_k}}}}} + \frac{D^+A_{n_{j_k}}\left(\kappa_{1-\tau_{A_{n_{j_k}}}}(t_{n_{j_k}})\right)
    t_{n_{j_k}}^{-\tau_{A_{n_{j_k}}}}(1-t_{n_{j_k}})^{-\tau_{A{n_{j_k}}}}}{A_{n_{j_k}}\left(\kappa_{1-\tau_{A_{n_{j_k}}}}(t_{n_{j_k}})\right)\left(t_{n_{j_k}}^{1-\tau_{A_{n_{j_k}}}} + (1-t_{n_{j_k}})^{1-\tau_{A_{n_{j_k}}}}\right)^2}\right] \\&= D^+(A_0)_{1-\tau_{A_0}}(t_0) =
    D^+(A)_{1-\tau_{A}}(t_0).
    \end{align*}
    Since the subsequence $(a_{n_j})_{j \in \mathbb{N}}$ of $(a_n)_{n\in \mathbb{N}}$, was arbitrary, it already follows that $a_n \overset{n\rightarrow\infty}{\longrightarrow} D^+(A)_{1-\tau_{A}}(t_0)$.
\end{proof}
\noindent We now turn our attention to weak conditional convergence. Accordingly, for arbitrary $\psi \in \Psi$ and $A \in \mathcal{A} \setminus {A_M}$, define the functions
\begin{equation}\label{eq:fct_H}
H_{\psi,A}(x,y) := \frac{D^{-}(\psi)_{1-\tau_A}\left(((\varphi)_{1-\tau_A}(x) + (\varphi)_{1-\tau_A}(y)) \cdot (A)_{1-\tau_A}\left(\frac{(\varphi)_{1-\tau_A}(x)}{(\varphi)_{1-\tau_A}(x) + (\varphi)_{1-\tau_A}(y)}\right)\right)}{D^-(\psi)_{1-\tau_A}((\varphi)_{1-\tau_A}(x))}
\end{equation}
for every $(x,y) \in (0,1) \times \mathbb{I}$ and
\begin{equation}\label{fct_Q}
Q_{\psi,A}(x,y) := G_{(A)_{1-\tau_A}}\left(\frac{(\varphi)_{1-\tau_A}(x)}{(\varphi)_{1-\tau_A}(x) + (\varphi)_{1-\tau_A}(y)}\right)
\end{equation}
for every $(x,y) \in (0,1) \times \mathbb{I}$, where $G_{(A)_{1-\tau_A}}$ is defined as in eq.~\eqref{eq:funct_G_A}, with $(A)_{1-\tau_A}$ substituted for $A$.
Letting $C_{\psi,A} \in \mathcal{C}_{am}$ denote the corresponding Archimax copula and let $K_{\psi,A}$ be a version of its Markov kernel according to 
eq. \eqref{eq:markov-kernel}, tedious (but straightforward) calculations show that 
for every $(x,y) \in \mathbb{I}^2$ the following identity holds:
\begin{equation}\label{eq:kernel_prd}
    K_{\psi,A}(x,[0,y]) = H_{\psi,A}(x,y) \cdot Q_{\psi,A}(x,y).
\end{equation}
In the following two lemmas, we first consider convergence towards $H_{\psi,A}$ and subsequently convergence towards $Q_{\psi,A}$.
\begin{lemma}\label{lem:conv_H}
Let $\psi,\psi_1,\psi_2, ... \in \Psi$, $A,A_1,A_2,... \in \mathcal{A}\setminus\{A_M\}$ and let $\tau_A, \tau_{A_1}, \tau_{A_2}, \ldots$ be the Kendall's $\tau$s of $C_A, C_{A_1}, C_{A_2}, \ldots \in \mathcal{C}_{ev}$. Furthermore, let $(\psi)_{1-\tau_{A}}, (\psi_1)_{1-\tau_{A_1}}, (\psi_2)_{1-\tau_{A_2}}, \ldots$ and $(A)_{1-\tau_{A}}, (A_1)_{1-\tau_{A_1}}, (A_2)_{1-\tau_{A_2}}, \ldots$ be defined according to 
eqs. \eqref{eq:trans_arch} and \eqref{eq:trans_pickands}, respectively, and assume 
that the following two assertions hold:
\begin{enumerate}[label=(\roman*)]
    \item $(\psi_n)_{1-\tau_{A_n}} \overset{n \rightarrow \infty}{\longrightarrow} (\psi)_{1-\tau_{A}}$ uniformly on $[0,\infty)$ and 
    \item $(A_n)_{1-\tau_{A_n}}\overset{n \rightarrow \infty}{\longrightarrow}(A)_{1-\tau_{A}}$ uniformly on $\mathbb{I}$.
\end{enumerate}
Then there exists a set $\Lambda \in \mathcal{B}(\mathbb{I})$ with $\lambda(\Lambda) = 1$ such that the following property holds: for every $x \in \Lambda$ there exists some set $\mathcal{D}_x \in \mathcal{B}(\mathbb{I})$ which is dense in $\mathbb{I}$, such that  
$$
H_{\psi_n,A_n}(x,y) \overset{n \rightarrow \infty}{\longrightarrow} H_{\psi,A}(x,y)
$$
for every $y \in \mathcal{D}_x$. 
\end{lemma}
\begin{proof}
Define the set $\Lambda \in \mathcal{B}(\mathbb{I})$ as
$$
\Lambda := \{s \in (0,1) \colon (\varphi)_{1-\tau_A}(s) \in \mathrm{Cont}(D^-(\psi)_{1-\tau_A})\}.
$$
As $(\psi)_{1-\tau_A}$ is a convex function, $D^-(\psi)_{1-\tau_A}$ has at most countably many points of discontinuity. Hence, using the fact that $(\varphi)_{1-\tau_A}$ is  strictly decreasing on $[0,1]$, we obtain that the complement $\Lambda^c$ is at most countably infinite, so $\lambda(\Lambda) = 1$ holds.\\
Let $x \in \Lambda$ be arbitrary but fixed. 
Since the sequence $\{(\psi_n)_{1-\tau_{A_n}}\}_{n\in \mathbb{N}}$ converges uniformly to $(\psi)_{1-\tau_{A}}$, it follows (see \cite{bernoulli, mult_arch}) that $(\varphi_n)_{1-\tau_{A_n}}$ converges pointwise to $(\varphi)_{1-\tau_{A}}$ on $(0,1]$. 
Using convexity of all involved functions, continuous 
convergence holds in every point of continuity, i.e., for every 
$0<z \in \mathrm{Cont}(D^-(\psi)_{1-\tau_A})$ and every sequence 
$(z_n)_{n \in \mathbb{N}}$ converging in $(0,\infty)$ to $z$ we have that 
\begin{equation}\label{eq:cont.conv.interchange}
   \lim_{n\rightarrow \infty}D^-(\psi_n)_{1-\tau_{A_n}}(z_n) = D^-(\psi)_{1-\tau_{A}}(z). 
\end{equation}
In particular, eq. \eqref{eq:cont.conv.interchange} holds for 
$z=(\varphi)_{1-\tau_{A}}(x)$ by construction. 
Defining the set $\mathcal{D}_x$ by 
$$
\mathcal{D}_x := \left\{v \in \mathbb{I} \colon ((\varphi)_{1-\tau_A}(x) + (\varphi)_{1-\tau_A}(v)) \cdot (A)_{1-\tau_A}\left(\frac{(\varphi)_{1-\tau_A}(x)}{(\varphi)_{1-\tau_A}(x) + (\varphi)_{1-\tau_A}(v)}\right) \in \mathrm{Cont}(D^-(\psi)_{1-\tau_A})\right\},
$$
it suffices to show that $\mathcal{D}_x$ is dense in $\mathbb{I}$. 
Using the alternative expression for $(A)_{1-\tau_A}$ according to 
eq. \eqref{eq:alter_trans_pick}, the set $\mathcal{D}_x$ boils down to
$$
\mathcal{D}_x = \left\{v \in \mathbb{I} \colon (\varphi)_{1-\tau_A}(x)h_A^\frac{1}{1-\tau_A}\left(\frac{\varphi(x)}{\varphi(x) + \varphi(v)}\right) \in  \mathrm{Cont}(D^-(\psi)_{1-\tau_A})\right\}.
$$
Set $g^R(x) = \psi((\frac{1}{R}-1)\varphi(x))$, where $R \in (\frac{1}{2},1]$ is defined as in eq. \eqref{eq:definition_L_R}.
Then, according to Lemma \ref{lem:regularity_fcts}, the function $v \mapsto h_A\left(\frac{\varphi(x)}{\varphi(x) + \varphi(v)}\right)$ is strictly decreasing on $\left[0,g^R(x)\right]$ and the same holds for $v \mapsto\varphi(x)^\frac{1}{1-\tau_A}h_A^\frac{1}{1-\tau_A}\left(\frac{\varphi(x)}{\varphi(x) + \varphi(v)}\right)$. Considering that $D^-(\psi)_{1-\tau_A}$ has at most countably many discontinuities, the set of all $v \in \left[0,g^R(x)\right]$ fulfilling $\varphi(x)^\frac{1}{1-\tau_A}h_A^\frac{1}{1-\tau_A}\left(\frac{\varphi(x)}{\varphi(x) + \varphi(v)}\right) \in  \mathrm{Cont}(D^-(\psi)_{1-\tau_A})$ is dense in $\left[0,g^R(x)\right]$. 
Additionally, since for every 
$v \in (g^R(x),1]$ we have
$$
(\varphi)_{1-\tau_A}(x)h_A^\frac{1}{1-\tau_A}\left(\frac{\varphi(x)}{\varphi(x) + \varphi(v)}\right) = (\varphi)_{1-\tau_A}(x)
$$
and $(\varphi)_{1-\tau_A}(x)$ is a point of continuity of $D^-(\psi)_{1-\tau_A}$, 
it follows that $(g^R(x),1] \subseteq \mathcal{D}_x$. Altogether, $\mathcal{D}_x$ is dense in $\mathbb{I}$ and the lemma is proved.
\end{proof}
\begin{lemma}\label{lem:conv_Q}
Let $\psi,\psi_1,\psi_2, ... \in \Psi$, $A,A_1,A_2,... \in \mathcal{A}\setminus\{A_M\}$ and let $\tau_A, \tau_{A_1}, \tau_{A_2}, \ldots$ be the Kendall's $\tau$s of $C_A, C_{A_1}, C_{A_2}, \ldots \in \mathcal{C}_{ev}$. Furthermore, let $(\psi)_{1-\tau_{A}}, (\psi_1)_{1-\tau_{A_1}}, (\psi_2)_{1-\tau_{A_2}}, \ldots$ and $(A)_{1-\tau_{A}}, (A_1)_{1-\tau_{A_1}}, (A_2)_{1-\tau_{A_2}}, \ldots$ be defined according to eqs. \eqref{eq:trans_arch} and \eqref{eq:trans_pickands}, respectively, and assume that 
the following two assertions hold:
\begin{enumerate}[label=(\roman*)]
    \item $(\psi_n)_{1-\tau_{A_n}} \overset{n \rightarrow \infty}{\longrightarrow} (\psi)_{1-\tau_{A}}$ uniformly on $[0,\infty)$ and 
    \item $(A_n)_{1-\tau_{A_n}}\overset{n \rightarrow \infty}{\longrightarrow}(A)_{1-\tau_{A}}$ uniformly on $\mathbb{I}$.
\end{enumerate}
Then for every $x \in (0,1)$ there exists some 
set $\mathcal{D}_x \in \mathcal{B}(\mathbb{I})$ which is dense in $\mathbb{I}$, such that  
$$
Q_{\psi_n,A_n}(x,y) \overset{n \rightarrow \infty}{\longrightarrow} Q_{\psi,A}(x,y)
$$
for every $y \in \mathcal{D}_x$. 
\end{lemma}
\begin{proof}
Let $x \in (0,1)$ be arbitrary but fixed and define $\mathcal{D}_x \in \mathcal{B}(\mathbb{I})$ by
$$
\mathcal{D}_x := \left\{v \in \mathbb{I} \colon \frac{\varphi(x)}{\varphi(x) + \varphi(v)} \in \mathrm{Cont}(D^+(A)_{1-\tau_A})\right\}.
$$
According to Lemma \ref{lem:properties_trans_pickands}, $D^+(A)_{1-\tau_A}$ has at most countably many points of discontinuity. Hence, using the fact that the function
$v \mapsto\frac{\varphi(x)}{\varphi(x) + \varphi(v)}$ is strictly increasing on $\mathbb{I}$, it is immediate that $\mathcal{D}_x$ is dense in $\mathbb{I}$. Since $(\psi_n)_{1-\tau_{A_n}} \overset{n \rightarrow \infty}{\longrightarrow} (\psi)_{1-\tau_{A}}$ uniformly on $[0,\infty)$, proceeding as in \cite{bernoulli,mult_arch} we obtain that
$$
\frac{(\varphi_n)_{1-\tau_{A_n}}(x)}{(\varphi_n)_{1-\tau_{A_n}}(x) + (\varphi_n)_{1-\tau_{A_n}}(y)} \overset{n \rightarrow \infty}{\longrightarrow}\frac{(\varphi)_{1-\tau_{A}}(x)}{(\varphi)_{1-\tau_{A}}(x) + (\varphi)_{1-\tau_{A}}(y)} 
$$
for every $x,y \in (0,1)^2$. 
Moreover, for arbitrary $y \in \mathcal{D}_x$, by construction we have 
that $\frac{\varphi(x)}{\varphi(x) + \varphi(y)}$ is a point of continuity of $D^+(A)_{1-\tau_A}$, so applying Lemma \ref{lem:convergence_derivatives} directly yields
$$
D^+(A_n)_{1-\tau_{A_n}}\left(\frac{(\varphi_n)_{1-\tau_{A_n}}(x)}{(\varphi_n)_{1-\tau_{A_n}}(x) + (\varphi_n)_{1-\tau_{A_n}}(y)}\right) \overset{n \rightarrow \infty}{\longrightarrow}D^+(A)_{1-\tau_{A}}\left(\frac{(\varphi)_{1-\tau_{A}}(x)}{(\varphi)_{1-\tau_{A}}(x) + (\varphi)_{1-\tau_{A}}(y)}\right).
$$
Finally, using uniform convergence of $((A_n)_{1-\tau_{A_n}})_{n \in \mathbb{N}}$ to 
$(A)_{1-\tau_{A}}$ on $\mathbb{I}$, we finally obtain that
$$
Q_{\psi_n,A_n}(x,y) \overset{n \rightarrow \infty}{\longrightarrow} Q_{\psi,A}(x,y).
$$
Since $y \in \mathcal{D}_x$ was arbitrary, the proof is complete.
\end{proof}
\noindent Finally considering that weak convergence of univariate distribution functions is equivalent to pointwise convergence on a dense set, proceeding 
analogously to \cite{bernoulli} yields weak conditional convergence - the following result holds:
\begin{lemma}\label{lem:convergence_wcc}
    Let $\psi,\psi_1,\psi_2,\ldots \in \Psi$, $A,A_1,A_2,\ldots \in \mathcal{A}\setminus \{A_M\}$ and let $\tau_A, \tau_{A_1}, \tau_{A_2}, \ldots$ be the Kendall's $\tau$s of $C_A, C_{A_1}, C_{A_2}, \ldots \in \mathcal{C}_{ev}$. Furthermore, let $(\psi)_{1-\tau_{A}}, (\psi_1)_{1-\tau_{A_1}}, (\psi_2)_{1-\tau_{A_2}}, \ldots$ and $(A)_{1-\tau_{A}}, (A_1)_{1-\tau_{A_1}}, (A_2)_{1-\tau_{A_2}}, \ldots$ be defined according to eqs. \eqref{eq:trans_arch} and \eqref{eq:trans_pickands}, respectively, and let $C_{\psi,A},C_{\psi_1,A_1},C_{\psi_2,A_2},\ldots \in \mathcal{C}_{am}$ be the corresponding Archimax copulas with Markov kernels $K_{\psi,A},K_{\psi_1,A_1},K_{\psi_2,A_2},\ldots$. Assume that the following two assertions hold:
    \begin{enumerate}[label=(\roman*)]
        \item $(\psi_n)_{1-\tau_{A_n}} \overset{n \rightarrow \infty}{\longrightarrow} (\psi)_{1-\tau_{A}}$ uniformly on $[0,\infty)$ and 
        \item $(A_n)_{1-\tau_{A_n}}\overset{n \rightarrow \infty}{\longrightarrow}(A)_{1-\tau_{A}}$ uniformly on $\mathbb{I}$.
    \end{enumerate}
    Then there exists a set $\Lambda \in \mathcal{B}(\mathbb{I})$ with $\lambda(\Lambda) = 1$ such that for every $x \in \Lambda$ the conditional distributions $K_{\psi_n,A_n}(x,\cdot)$ converge weakly to $K_{\psi,A}(x,\cdot)$, as $n \rightarrow \infty$. 
    In other words: The sequence 
    $(C_{\psi_n,A_n})_{n \in \mathbb{N}}$ converges weakly conditional to 
    $C_{\psi,A}$.
    \end{lemma}
\begin{proof}
Immediate consequence of Lemmas \ref{lem:conv_H}, \ref{lem:conv_Q} and eq. \eqref{eq:kernel_prd}.
\end{proof}
The following corollary will be useful in the sequel since 
it implies that our estimators developed in Section \ref{sec:estimator} can be directly applied to estimate the extent of (directed) dependence.
\begin{corollary}\label{cor:continuity_dependence_measures}
    Let $\eta$ be a dependence measure which is continuous with respect 
    to weak conditional convergence. Furthermore, consider $C_{\psi,A},C_{\psi_1,A_1},C_{\psi_2,A_2},... \in \mathcal{C}_{am}$ with $\psi,\psi_1,\psi_2,... \in \Psi$ and $A,A_1,A_2,... \in \mathcal{A}\setminus\{A_M\}$.
    Then the following two implications hold: 
    \begin{enumerate}[label=(\roman*)]
        \item If $(\psi_n)_{1-\tau_{A_n}} \overset{n\rightarrow\infty}{\longrightarrow} (\psi)_{1-\tau_A}$ and $(A_n)_{1-\tau_{A_n}} \overset{n\rightarrow\infty}{\longrightarrow} (A)_{1-\tau_{A}}$ pointwise, then $\eta(C_{\psi_n,A_n}) \overset{n\rightarrow\infty}{\longrightarrow}\eta(C_{\psi,A})$. 
        \item If $C_{\psi_n,A_n} \overset{n\rightarrow\infty}{\longrightarrow} 
        C_{\psi,A}$, then $\eta(C_{\psi_n,A_n}) \overset{n\rightarrow\infty}{\longrightarrow}\eta(C_{\psi,A})$.
    \end{enumerate}
\end{corollary}
\begin{proof}
    Immediate consequence of Theorem \ref{thrm:main_result_convergence} and dominated convergence. 
\end{proof}
We close this section with 
the following direct consequence of Theorem \ref{thrm:main_result_convergence}.
\begin{corollary}\label{cor:dense.subclass}
The family $\mathcal{C}_{am}^{pl}$ of all Archimax copulas $C$ for which there 
exists some piecewise linear $\psi \in \Psi$ and some 
piecewise linear $A \in \mathcal{A}$ with $C=C_{\psi,A}$ is dense in 
$(\mathcal{C}_{am},d_\infty)$ as well as in $(\mathcal{C}_{am},D_p)$, for every $p \in [1,\infty)$.
\end{corollary}
\begin{proof}
    Direct consequence of Theorem \ref{thrm:main_result_convergence} and the 
    fact that every $\psi \in \Psi$ and every 
    $A \in \mathcal{A}$ are the limits of piecewise linear elements in 
    $\Psi$ and $\mathcal{A}$, respectively.
\end{proof}
\section{A novel nonparametric estimator for bivariate Archimax copulas}\label{sec:estimator}
\noindent Throughout this section we consider a stationary and ergodic sequence (see \citep[Section 6.1]{durrett2019} for formal definitions) $(X_1,Y_1), (X_2,Y_2), \ldots$ of random vectors with common bivariate, continuous distribution function $H$, and univariate marginal distribution functions $F$ and $G$; this stationarity and ergodicity assumption will be used frequently throughout the remainder of the paper. We further assume that the corresponding copula $C$ of $H$ is an Archimax copula $C_{\psi,A} \in \mathcal{C}_{am}$ with generator $\psi \in \Psi$ and Pickands dependence function $A \in \mathcal{A}\setminus\{A_M\}$.\\
Under stationarity and ergodicity, Birkhoff's Ergodic Theorem (see \citep{durrett2019}) implies that the empirical distribution function $H_n$ converges uniformly to the population distribution $H$ almost surely, i.e.,
\begin{equation*}
	\sup_{(x,y) \in \mathbb{R}^2} \left| H_n(x,y) - H(x,y) \right| \xrightarrow{\text{a.s.}} 0.
\end{equation*}
As a consequence, the empirical copula $\hat{C}_n$ converges uniformly to the population copula $C_{\psi,A}$ almost surely as well. This, in turn, implies that the copulas $C_n^B$ -- that is, the empirical copula with copula $B$ used as interpolation -- also converge uniformly to $C_{\psi,A}$ almost surely.\\
Building upon Theorem \ref{thrm:main_result_convergence}, 
the goal of this section is to construct a strongly consistent, fully 
 non-parametric estimator $C_{\psi_n,A_n}$ of the Archimax copula $C_{\psi,A}$ without leaving the Archimax class, i.e.,  
 each $C_{\psi_n,A_n}$ should be an Archimax copula, too.
 Considering Theorems \ref{thm:identifiability} and \ref{thrm:main_result_convergence}, the model $C_{\psi_n,A_n}$ is then identifiable in terms of 
 $(\psi_n)_{1-\tau_{A_n}}$ and $(A_n)_{1-\tau_{A_n}}$.
\subsection{Estimating the Kendall distribution function}\label{subs:estimator_kendall}
\noindent According to Lemma \ref{lem:kendall_is_archimedean}, 
the Kendall distribution function $F_{\psi,A}^K$ of the Archimax copula $C_{\psi,A}$
is the Kendall distribution function of the Archimedean copula 
$C_{(\psi)_{1-\tau_A}}$.
We therefore follow \cite{GNZ} and estimate $F^K_{\psi,A}=F^K_{(\psi)_{1-\tau_A}}$ 
by the function $K_n$ constructed as follows: 
Consider the pseudo sample $W_1, W_2, \ldots W_n$, given by
\begin{equation}\label{eq:pseudo_sample}
W_{i}= \frac{1}{n+1}\sum_{k = 1}^n \mathbf{1}_{(-\infty,X_{i}) \times (-\infty,Y_{i})}(X_{k},Y_k) 
\in \left\{0,\frac{1}{n+1},\frac{2}{n+1},\ldots,\frac{n-1}{n+1}\right\} 
\end{equation}
for every $i \in \{1,2,\ldots,n\}$ and define $K_n: \mathbb{I} \rightarrow \mathbb{I}$ as
\begin{equation}\label{eq:est_kend}
K_{n}(t) := \frac{1}{n} \sum_{k = 1}^n \mathbf{1}_{[0,t]}(W_{k})
\end{equation}
for every $t \in \mathbb{I}$. 
Obviously $K_n$ fulfills $K_n(0) \geq \frac{1}{n}$ 
as well as $K_n(t)=1$ for every $t \in [\frac{n-1}{n+1},1]$ and 
$K_n(t) \geq \frac{n+1}{n}t = t + \frac{t}{n}$ for every $t \in [0, \frac{n-1}{n+1}]$. 
Based on these properties, it is straightforward to 
verify that $K_n$ is bounded below by the distribution function 
$\underline{K}_n: \mathbb{I}
\rightarrow \mathbb{I}$ of the uniform distribution on the set $\left\{0,\frac{1}{n+1},\frac{2}{n+1},\ldots,\frac{n-1}{n+1}\right\} $, i.e., by the function
\begin{equation}\label{eq:lower.bound.K_n}
  \underline{K}_n(t)=\frac{1}{n} \sum_{i=1}^n \mathbf{1}_{[0,t]}(\tfrac{i-1}{n+1}), 
\quad t \in \mathbb{I}.
\end{equation}
According to \cite{GNZ}, the estimator $K_n$ (and also the function 
$\underline{K}_n$) is the Kendall distribution function of an Archimedean copula $C_{\psi_n}$. 
Notice that $K_n$ is not the Kendall distribution function 
of any empirical copula $C_n^B$, it is, however, close to $F^K_{C_n^W}$ and has the same limiting behavior. In fact, setting $\hat{W}_i:=\frac{n+1}{n}W_i \in \mathbb{I}$ and defining 
$\hat{K}_{n}(t) := \frac{1}{n} \sum_{k = 1}^n \mathbf{1}_{[0,t]}(\hat{W}_{k})$ it is 
straightforward to verify that $\hat{K}_n$ coincides with the Kendall distribution function 
of $C_n^W$ (i.e., of the empirical copula with the Fréchet--Hoeffding bound $W$ as 
`interpolation'). As mentioned in \cite{GNZ}, using the fact that convergence of 
copulas implies convergence of the corresponding Kendall distribution functions
together with continuous mapping theorem and $\lim_{n \rightarrow \infty} \frac{n}{n+1}=1$, 
it is therefore straightforward to prove that with probability $1$ we have that $(K_n)_{n \in \mathbb{N}}$ converges weakly to $F^K_{\psi,A}=F^K_{(\psi)_{1-\tau_A}}$.
\subsection{Unadjusted estimator for the transformed generator $(\psi)_{1-\tau_A}$}\label{subsec:6.2}
\noindent Since for bivariate Archimedean copulas, weak convergence 
of Kendall distribution functions is equivalent to uniform convergence of the 
(normalized) generators $\psi_n$ as well as to pointwise convergence of the 
pseudo-inverses $\varphi_n$ on $(0,1]$ (see \cite{bernoulli, mult_arch}) 
we already know that $(\psi_n)_{n \in \mathbb{N}}$ converges uniformly to 
$(\psi)_{1-\tau_A}$ on $[0,\infty)$ and $(\varphi_n)_{n \in \mathbb{N}}$ 
pointwise to $(\varphi)_{1-\tau_A}$ on $(0,1]$ with probability $1$.
Moreover, using eq. \eqref{eq:kendall.arch} the pseudo-inverse $\varphi_n$ 
of $\psi_n$ corresponding to $C_{\psi_n}$ can be constructed from $K_n$ 
in terms of (also see \cite{GR})
\begin{equation}\label{eq:phi.via.kendall}
\log{\varphi_n(t)} = \mathrm{sgn}(t-\tfrac{1}{2})\int_{\min\{t,\frac{1}{2}\}}^{\max\{t,\frac{1}{2}\}} \frac{1}{s-K_n(s)} \mathrm{d}\lambda(s).
\end{equation}
Notice that the resulting $\varphi_n$ is normalized and piecewise linear (with $W_1,\ldots,W_n$ being edges) and the same holds for the corresponding generator 
$\psi_n$. Moreover, again considering $K_n \geq \underline{K}_n $ we obviously have $K_n(s)-s \geq \underline{K}_n(s)-s>0$ for every $s \in [0,1)$, so for every $t \in [0,\frac{1}{2}]$
$$
\log{\varphi_n(t)} = \int_{[t,\frac{1}{2}]} \frac{1}{K_n(s)-s} \mathrm{d}\lambda(s)
\leq \int_{[t,\frac{1}{2}]} \frac{1}{\underline{K}_n(s)-s} \mathrm{d}\lambda(s)
$$
holds. Considering $t=0$, distinguishing the case of even and odd $n$, 
and calculating the latter integral yields 
\begin{align*}
   \log{\varphi_n(0)} &\leq  \int_{[0,\frac{1}{2}]} \frac{1}{\underline{K}_n(s)-s} \mathrm{d}\lambda(s) = 
   \begin{cases}
 \log\left(\binom{\frac{3n+1}{2}}{\frac{n+1}{2}}\right), & n \text{ odd}, \\[4mm]
\log\left(\binom{\frac{3n}{2}}{\frac{n}{2}}\right) + \,\log\left(\tfrac{3n+2}{2n+2}\right), & n \text{ even}.
\end{cases}
\end{align*}
so every generator $\psi_n$ is non-strict.
Summing up, we have that the sequence $(\psi_n)_{n \in \mathbb{N}}$ of 
normalized, non-strict, piecewise linear Archimedean generators 
converges uniformly to $(\psi)_{1-\tau_A}$ with probability $1$.  \\
In the sequel we will work with $\mathbb{E}[Z_n]$ for $Z_n$
having survival function $\psi_n$. Doing so, considering that $\varphi_n$ and 
$\psi_n$ are piecewise linear, it suffices to determine the values $\varphi_n(W_i)$
for every $i \in \{1,\ldots,n\}$, which can easily be done as follows: 
Letting $0=w_1,w_2,\ldots,w_m \leq \frac{n-1}{n+1}$ denote the unique values of 
$W_1, W_2, \ldots W_n$ in increasing order and setting 
$\ell:= \max\{i\colon: \, w_i < \frac{1}{2} \}$, using eq. \eqref{eq:phi.via.kendall}, as 
first step we obtain 
$$
\log{\varphi_n(w_{\ell+1})} = \log{\varphi_n(w_{\ell+1})} - \log{\varphi_n(\tfrac{1}{2})} =
\log{\frac{K_n(\tfrac{1}{2})-w_{\ell+1}}{K_n(\tfrac{1}{2})-\tfrac{1}{2}}},
$$
which is equivalent to 
$$
\frac{\varphi_n(w_{\ell+1})}{\varphi_n(\tfrac{1}{2})} = 
\frac{K_n(\tfrac{1}{2})-w_{\ell+1}}{K_n(\tfrac{1}{2})-\tfrac{1}{2}}.
$$
Proceeding in the same manner yields  
$$
\frac{\varphi_n(w_{\ell+2})}{\varphi_n(w_{\ell+1})} = 
\frac{K_n(w_{\ell+1})-w_{\ell+2}}{K_n(w_{\ell+1})-w_{\ell+1}}, \quad 
\frac{\varphi_n(w_{\ell+3})}{\varphi_n(w_{\ell+2})} = 
\frac{K_n(w_{\ell+2})-w_{\ell+3}}{K_n(w_{\ell+2})-w_{\ell+2}},\ldots,
$$
i.e., all values $\varphi_n(w_i)$ with $i > \ell$ can easily be determined recursively. 
For $i< \ell$ obviously an analogous recursion holds and we obtain
$$
\frac{\varphi_n(w_{\ell})}{\varphi_n(\tfrac{1}{2})} = 
\frac{K_n(w_\ell)-w_{\ell}}{K_n(w_\ell)-\tfrac{1}{2}}, \quad 
\frac{\varphi_n(w_{\ell-1})}{\varphi_n(w_{\ell})} = 
\frac{K_n(w_{\ell-1})-w_{\ell-1}}{K_n(w_{\ell-1})-w_{\ell}}, \quad 
\frac{\varphi_n(w_{\ell-2})}{\varphi_n(w_{\ell-1})} = 
\frac{K_n(w_{\ell-2})-w_{\ell-2}}{K_n(w_{\ell-2})-w_{\ell-1}}, \ldots
$$
\subsection{Unadjusted estimator for the transformed Pickands dependence function $(A)_{1-\tau_A}$}
\noindent Now we estimate the function $(A)_{1-\tau_A}$ using both -  
a version of the Pickands and the CFG type estimator as proposed in 
\citep{caperaa1997, est-archimax, genest2009, pickands1981}, and start by motivating
the estimators theoretically. Using the assumptions for the distribution 
function $H$, the underlying copula $C$ and the Pickands function $A \neq A_M$ 
as stated at the beginning of this section.
Defining
\begin{equation}\label{eq:population_xi}
\xi(t) := \min\left\{\frac{(\varphi)_{1-\tau_A}(F(X))}{t},\frac{(\varphi)_{1-\tau_A}(G(Y))}{1-t}\right\}
\end{equation}
for every $t \in (0,1)$ and setting $\xi(0) = (\varphi)_{1-\tau_A}(G(Y))$ as well as $\xi(1) = (\varphi)_{1-\tau_A}(F(X))$, it is straightforward to verify that
$$
\mathbb{P}[\xi(t) > x] = (\psi)_{1-\tau_A}(x\,(A)_{1-\tau_A}(t))
$$
holds for every $x>0$.\\
Suppose now that $Z$ is a random variable with survival function $(\psi)_{1-\tau_A}$, i.e.,
$$
\mathbb{P}[Z>z] = (\psi)_{1-\tau_A}(z)
$$
for every $z \in [0,\infty)$. Then obviously for every $t \in \mathbb{I}$ 
the random variable $\frac{Z}{(A)_{1-\tau_A}(t)}$ has the same distribution as $\xi(t)$. Building on this fact, we obtain, for every $t \in \mathbb{I}$, the following representation of $(A)_{1-\tau_A}(t)$ in terms of the expected values of $Z$ and $\xi(t)$:
\begin{equation}\label{eq:theoretical_pickands}
(A)_{1-\tau_A}(t)= \frac{\mathbb{E}[Z]}{\mathbb{E}[\xi(t)]}.
\end{equation}
Moreover, if $\mathbb{E}[\log Z]$ exists, we 
similarly obtain that
\begin{equation}\label{eq:theoretical_cfg}
(A)_{1-\tau_A}(t) = \exp\left(\mathbb{E}[\log Z] - \mathbb{E}[\log \xi(t)]\right)
\end{equation}
for every $t \in \mathbb{I}$.\\
As discussed in \ref{subsec:strong_consistency_pick}, the expected values $\mathbb{E}[Z]$ and $\mathbb{E}[\log Z]$ admit the representations
$$
\mathbb{E}[Z] = \int_\mathbb{I}(\varphi)_{1-\tau_A}(s) \, \mathrm{d}\lambda(s)
\qquad \text{and} \qquad
\mathbb{E}[\log Z] = \int_\mathbb{I}\log(\varphi)_{1-\tau_A}(s) \, \mathrm{d}\lambda(s).
$$
This representation enables direct estimation of $\mathbb{E}[Z]$ and $\mathbb{E}[\log Z]$ via $\varphi_n$, as defined in \eqref{eq:phi.via.kendall}.
\begin{remark}
    The fact that $\frac{Z}{(A)_{1-\tau_A}(t)}$ has the same distribution as $\xi(t)$ could alternatively be used to determine 
    $(A)_{1-\tau_A}(t)$ in the sense of eq. \eqref{eq:theoretical_pickands} in terms 
    of other `aggregations' of $Z$ and $\xi(t)$.
    Avoiding the potential problem of non-existence of $\mathbb{E}[Z]$ one could, 
    e.g., work with quantiles of $Z$ and $\xi(t)$. In analogy with \cite{est-archimax},
    however, we here stick to the approach via expectation. 
    Figure~\ref{fig:quantile} depicts an example of the estimator $B_{n,c}^\mathbf{P}$ (proposed in eq.~\eqref{eq:cfg_type_estimator_c}) with the expected values replaced by $\alpha$-quantiles. 
    Whether such a quantile-based estimator is strongly consistent in general, however, remains an open question which we plan to tackle in the near future.
\end{remark}
\begin{figure}[!ht]
	\centering
	\includegraphics[width=1\textwidth]{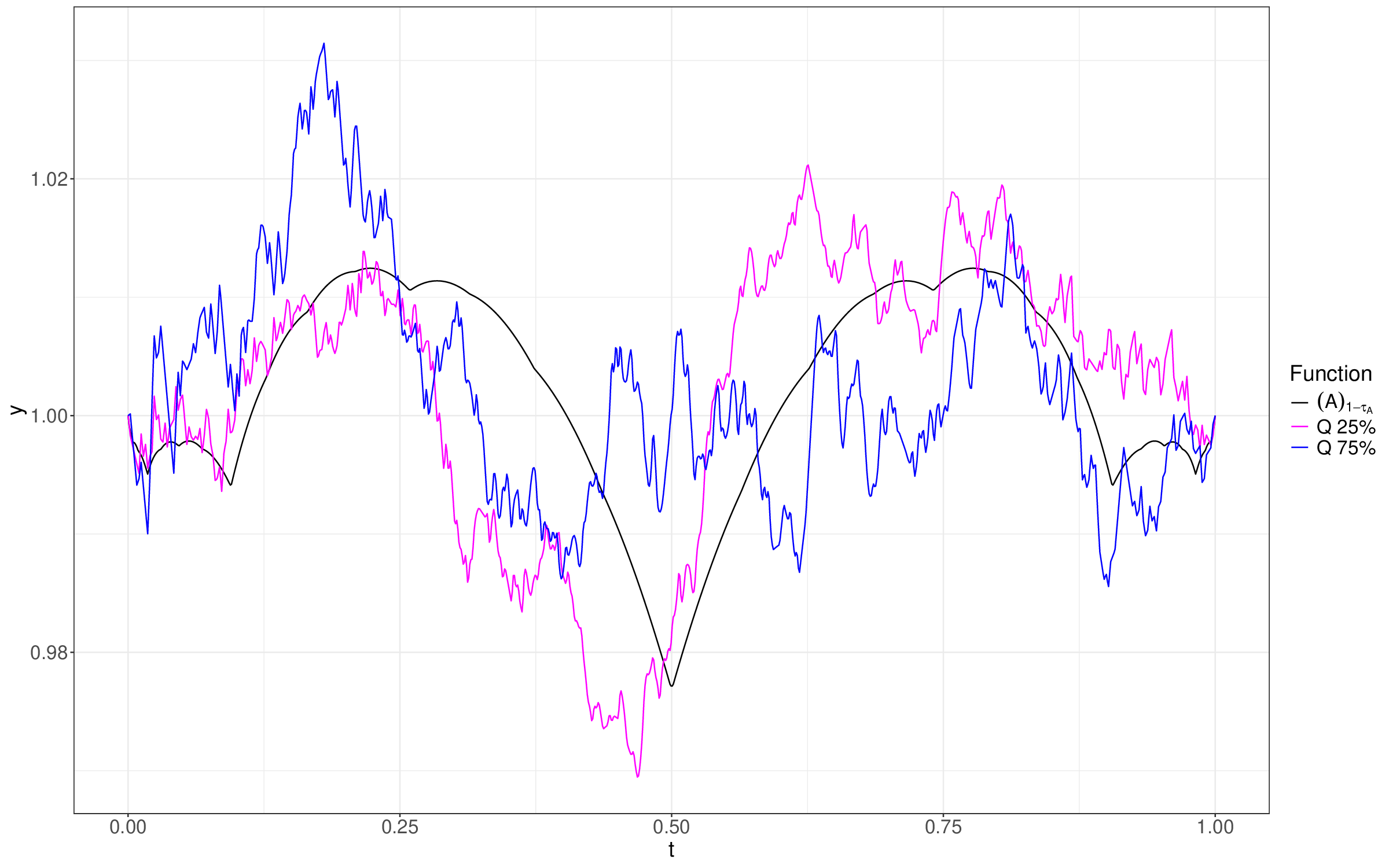}
	\caption{Plots of the estimator $B_{n,c}^{\mathbf{P}}$ proposed in eq. \eqref{eq:cfg_type_estimator_c} based on $\alpha$-quantiles rather than on expected values, with $B_{n,c}^{\mathbf{P}}$ shown in magenta for $\alpha = 0.25$ and in blue for $\alpha = 0.75$. The black line depicts the function $(A)_{1-\tau_A}$.}
	\label{fig:quantile}
\end{figure}
Considering the aforementioned properties, it seems natural 
to plug-in the empirical versions and proceed as follows: 
Letting $F_n$ and $G_n$ denote the empirical (marginal) distribution functions of the first
$n$ elements of the sample, we set
$$
\hat{U}_i := \frac{n}{n+1}F_n(X_i) \,\,\, \text{ and } \,\,\, \hat{V}_i := \frac{n}{n+1}G_n(Y_i)
$$
and define
\begin{equation}\label{eq:xis}
\xi_{n,i}(t) := \min\left\{\frac{\varphi_n(\hat{U}_i)}{t},\frac{\varphi_n(\hat{V}_i)}{1-t}\right\}
\end{equation}
for every $i \in \{1,...,n\}$ and $t \in (0,1)$; for $t \in \{0,1\}$ as before we set
$\xi_{n,i}(0) := \varphi_n(\hat{V}_i)$ as well as $\xi_{n,i}(1) := \varphi_n(\hat{U}_i)$.\\
Using eq. \eqref{eq:theoretical_pickands}, our uncorrected Pickands type 
estimator is then defined as \\
\begin{equation*}
B_{n,u}^\mathbf{P}(t)  := \frac{\frac{1}{n}\sum_{i = 1}^n \varphi_n\left(\frac{i}{n+1}\right)}{\frac{1}{n}\sum_{i=1}^n \xi_{n,i}(t)}.
\end{equation*}
Moreover, considering eq. \eqref{eq:theoretical_cfg} our uncorrected CFG type estimator $B_{n,u}^\mathbf{CFG}$ is given by
\begin{equation*}
B_{n,u}^\mathbf{CFG}(t)  := \exp\left(\frac{1}{n}\sum_{i = 1}^n \log\left(\varphi_n\left(\frac{i}{n+1}\right)\right) - \frac{1}{n}\sum_{i=1}^n\log \xi_{n,i}(t)\right) = \left[\frac{\prod_{i=1}^n\varphi_n(\frac{i}{n+1})}{\prod_{i=1}^n\xi_{n,i}(t)}\right]^\frac{1}{n}
\end{equation*}
for every $t \in \mathbb{I}$. 
For $\Xi \in \{\mathbf{P},\mathbf{CFG}\}$ we obviously have that $B_{n,u}^\Xi(0) = 1 = B_{n,u}^\Xi(1)$. Assuring that the estimator is strictly greater than $\frac{1}{2}$
at $t=\frac{1}{2}$, we correct the estimators $B_{n,u}^\Xi$ and define the corrected estimators $B_{n,c}^\Xi$ as
\begin{equation}\label{eq:cfg_type_estimator_c}
B_{n,c}^\Xi(t) := B_{n,u}^\Xi(t) + \varepsilon_n4t(1-t)
\end{equation}
for every $t \in \mathbb{I}$, whereby $(\varepsilon_n)_{n \in\mathbb{N}}$ can be any sequence fulfilling $\varepsilon_n > 0$ and $\varepsilon_n \overset{n \rightarrow \infty}{\longrightarrow} 0$. Since $\varepsilon_n$ serves solely as a correction term 
ensuring $B_{n,c}^\Xi(1/2) > 1/2$, we choose $\varepsilon_n = \exp(-n)$ throughout this paper for simplicity, which ensures that $(\varepsilon_n)_{n \in \mathbb{N}}$ decays exponentially fast. One might ask why the quantities $\mathbb{E}[Z]$ and $\mathbb{E}[\log Z]$ are estimated via a Riemann type sum of $\varphi_n$ rather than by a direct approach. We also considered computing the integrals $\int_{\mathbb{I}}\varphi_n(t)\,\mathrm{d}\lambda(t)$ and $\int_{\mathbb{I}}\log\varphi_n(t)\,\mathrm{d}\lambda(t)$ directly, but found that the Riemann sum approach is easy to handle and provided superior results. 
\subsection{Adjusted estimators for $(\psi)_{1-\tau_A}$, $(A)_{1-\tau_A}$, and the copula estimator}\label{subsec:adjusted_est}
\noindent 
Throughout this section, we always consider $\Xi \in \{\mathbf{P}, \mathbf{CFG}\}$. 
As mentioned before, our objective is to construct an estimator 
for the copula $C_{\psi,A}$, which is itself of Archimax type. 
One natural approach is to define such an estimator as the composition of the estimators $\psi_n$, $\varphi_n$, and $B_{n,c}^\Xi$ established before. 
We therefore propose the estimator 
$D_n^\Xi \colon \mathbb{I}^2 \to \mathbb{I}$, given by
\begin{equation}\label{eq:D_n}
D_n^\Xi(x,y) := \psi_n\!\left((\varphi_n(x) + \varphi_n(y)) \, B_{n,c}^\Xi\left(\frac{\varphi_n(x)}{\varphi_n(x) + \varphi_n(y)}\right)\right),
\end{equation}
for every $(x,y) \in \mathbb{I}^2 \setminus \{(1,1)\}$, and set $D_n^\Xi(1,1) = 1$. Lemma \ref{lem:copula_est_part1} shows that $D_n^\Xi$ converges uniformly to $C_{\psi,A}$ with probability $1$. However, $D_n^\Xi$ is not necessarily a copula. In fact, it need not even be a distribution function, since it does not satisfy the $2$-increasingness property (see \citep{dur_princ, klenke}) in general. An example illustrating this fact is presented in Subsection \ref{subsec:discussion_2_increasing}.\\
\noindent Since we want to stay within the Archimax class, we modify $D_n^\Xi$ accordingly and proceed as follows. 
For the remainder of this section we fix $\alpha \in [0, 1-\frac{1}{n}]$. Projecting the estimators $B_{n,c}^\mathbf{P}$ and $B_{n,c}^\mathbf{CFG}$ into the space of Pickands dependence functions $\mathcal{A}$, we define the function
\begin{equation*}
T_{\alpha,n}^\Xi(t) := \left[\left(t^\frac{1}{1-\alpha} + (1-t)^\frac{1}{1-\alpha}\right)\,B_{n,c}^\Xi\left(\kappa_{\frac{1}{1-\alpha}}(t) \right)
\right]^{1-\alpha} =(B_{n,c}^\Xi)_{\frac{1}{1-\alpha}}(t)
\end{equation*}
for every $t \in \mathbb{I}$ and set
\begin{equation}\label{eq:A_alpha_n}
A_{\alpha,n}^\Xi := \mathrm{gcm}\min\left\{T_{\alpha,n}^\Xi,1\right\}.
\end{equation}
Lemma \ref{lem:pick_is_pick} states that $A_{\alpha,n}^\Xi$ is indeed a Pickands dependence function.
We calculate Kendall's $\tau$ of the corresponding EVC $C_{A_{\alpha,n}^\Xi} \in \mathcal{C}_{ev}$ according to eq. \eqref{eq:kendallstau_ex} and denote it by $\tau_{A_{\alpha,n}^\Xi}$. Figure~\ref{fig:explanation_regularization} depicts the functions $T_{\alpha,n}^\mathbf{CFG}$, $A_{\alpha,n}^\mathbf{CFG}$, and their transformed versions $(A_{\alpha,n}^\mathbf{CFG})_{1-\tau_{A_{\alpha,n}^\mathbf{CFG}}}$ for different values of $\alpha \in [0,1-\frac{1}{n}]$.
Moreover, defining the function $\varphi_{\alpha,n}^\Xi$ as
\begin{equation}\label{eq:est_phi_alpha}
\varphi_{\alpha,n}^\Xi(t) := \frac{\mathrm{gcm}(\varphi_n^{1-\tau_{A_{\alpha,n}^\Xi}})(t)}{\mathrm{gcm}(\varphi_n^{1-\tau_{A_{\alpha,n}^\Xi}})(\frac{1}{2})}
\end{equation}
for every $t \in \mathbb{I}$, Lemma \ref{lem:gen_is_gen} states that $\varphi_{\alpha,n}^\Xi$ is indeed the pseudo-inverse of a normalized generator 
$\psi_{\alpha,n}^\Xi$ of an Archimedean copula $C_{\psi_{\alpha,n}^\Xi}$. \\
Given $A_{\alpha,n}^\Xi \in \mathcal{A}$ and $\varphi_{\alpha,n}^\Xi \in \Psi$, the corresponding Archimax copula $C_{\alpha,n}^\Xi$ is defined as
\begin{equation}\label{eq:C_alpha_n}
C_{\alpha,n}^\Xi(x,y) :=\psi_{\alpha,n}^\Xi\left((\varphi_{\alpha,n}^\Xi(x) + \varphi_{\alpha,n}^\Xi(y))A_{\alpha,n}^\Xi\left(\frac{\varphi_{\alpha,n}^\Xi(x)}{\varphi_{\alpha,n}^\Xi(x) + \varphi_{\alpha,n}^\Xi(y)}\right)\right),
\end{equation}
for every $(x,y) \in \mathbb{I}^2\setminus \{(1,1)\}$ and 
by $C_{\alpha,n}(1,1) = 1$.\\
As discussed above, the estimator $D_n^\Xi$ is strongly consistent and therefore converges uniformly to the true copula $C_{\psi,A}$ with probability $1$. 
It therefore seems natural to choose $\alpha \in [0,1-\frac{1}{n}]$ in such a way that $C_{\alpha,n}^\Xi$ is as close as possible to $D_n^\Xi$. Defining
\begin{equation}\label{eq:alphas}
\alpha_n := \min\mathrm{argmin}_{\alpha \in [0,1-\frac{1}{n}]} \Vert C_{\alpha,n}^\Xi -D_n^\Xi\Vert_\infty,
\end{equation}
according to Lemma \ref{lem:existence_uniqueness_alpha} $\alpha_n$ exists and is unique.
We optimize over the interval $[0,1-\frac{1}{n}]$ instead of the whole interval $\mathbb{I}$ to prevent the term $\frac{1}{1-\alpha}$ form becoming unbounded.
Building upon $\alpha_n$, our proposed estimator for 
$(A)_{1-\tau_A}$ takes the form
\begin{equation}\label{eq:final_estimator}
(A_{\alpha_n,n}^\Xi)_{1-\tau_{A_{\alpha_n,n}^\Xi}}(t) = \left[(t^{1-\tau_{A_{\alpha_n,n}^\Xi}} + (1-t)^{1-\tau_{A_{\alpha_n,n}^\Xi}})A_{\alpha_n,n}^\Xi\left(\kappa_{1-\tau_{A_{\alpha_n,n}^\Xi}}(t)\right)
\right]^\frac{1}{1-\tau_{A_{\alpha_n,n}^\Xi}}
\end{equation}
and the estimator for $(\varphi)_{1-\tau_A}$ is given by
\begin{equation}\label{eq:est_trans_phi}
(\varphi_{\alpha_n,n}^\Xi)_{1-\tau_{A_{\alpha_n,n}^\Xi}}(t) = \left[\varphi_{\alpha_n,n}^\Xi(t)\right]^\frac{1}{1-\tau_{A_{\alpha_n,n}^\Xi}}
\end{equation}
for every $t \in \mathbb{I}$. Accordingly, our estimator for $(\psi)_{1-\tau_A}$ 
is given by  
\begin{equation}\label{eq:est_trans_psi}
(\psi_{\alpha_n,n}^\Xi)_{1-\tau_{A_{\alpha_n,n}^\Xi}}(z) = \psi_{\alpha_n,n}^\Xi\left(z^{1-\tau_{A_{\alpha_n,n}^\Xi}}\right)
\end{equation}
for every $z \in [0,\infty)$.
Notice that, since both the estimator for $(A)_{1-\tau_A}$ and the estimator for $(\varphi)_{1-\tau_A}$ rely on the greatest convex minorant, both are slightly biased.\\
Altogether, our proposed estimator of the Archimax copula $C_{\psi,A}$ is given by
\begin{equation}\label{eq:copula_estimator}
C_{\alpha_n,n}^\Xi(x,y) := C_{\psi_{\alpha_n,n}^\Xi, A_{\alpha_n,n}^\Xi}(x,y) =  C_{(\psi_{\alpha_n,n}^\Xi)_{1-\tau_{A_{\alpha_n,n}^\Xi}}, (A_{\alpha_n,n}^\Xi)_{1-\tau_{A_{\alpha_n,n}^\Xi}}}(x,y)
\end{equation}
for every $x,y \in \mathbb{I}$. \\
Since directly verifying consistency of the estimators $(\psi_{\alpha_n,n}^\Xi)_{1-\tau_{A_{\alpha_n,n}^\Xi}}$ and $(\varphi_{\alpha_n,n}^\Xi)_{1-\tau_{A_{\alpha_n,n}^\Xi}}$ in a simulation study is difficult (see Section \ref{subsec:sim_study_pick_cfg_beta} for further explanations), prior to stating and proving consistency, we introduce an auxiliary function that will serve as 
reference function. For every $t \in (0,1)$, we define $(\beta)_{1-\tau_A}$ in accordance with eq. \eqref{eq:beta} by
\begin{align*}
(\beta)_{1-\tau_A}(t) &:= D^-(\psi)_{1-\tau_A}((\varphi)_{1-\tau_A}(t))(\varphi)_{1-\tau_A}(t) = (1-\tau_A)\beta(t)
\end{align*}
and set $(\beta)_{1-\tau_A}(0) = 0 = (\beta)_{1-\tau_A}(1)$.\\
Correspondingly, using eq. \eqref{eq:phi.via.kendall} we define the empirical counterpart as
\begin{equation}\label{eq:beta_n}
\beta_n(t) := D^-\psi_n(\varphi_n(t))\varphi_n(t)
\end{equation}
together with its transformed version, obtained via eq. \eqref{eq:est_trans_phi},
\begin{align}\label{eq:beta_n_trans}
(\beta_{\alpha_n,n}^\Xi)_{1-\tau_{A_{\alpha_n,n}^\Xi}}(t) &:= D^-(\psi_{\alpha_n,n}^\Xi)_{1-\tau_{A_{\alpha_n,n}^\Xi}}((\varphi_{\alpha_n,n}^\Xi)_{1-\tau_{A_{\alpha_n,n}^\Xi}}(t))(\varphi_{\alpha_n,n}^\Xi)_{1-\tau_{A_{\alpha_n,n}^\Xi}}(t) \nonumber\\
&= (1-\tau_{A_{\alpha_n,n}^\Xi})D^-\psi_{\alpha_n,n}^\Xi(\varphi_{\alpha_n,n}^\Xi(t))\varphi_{\alpha_n,n}^\Xi(t)
\end{align}
for every $t \in \mathbb{I}$.
\begin{figure}[!ht]
	\centering
	\includegraphics[width=1\textwidth]{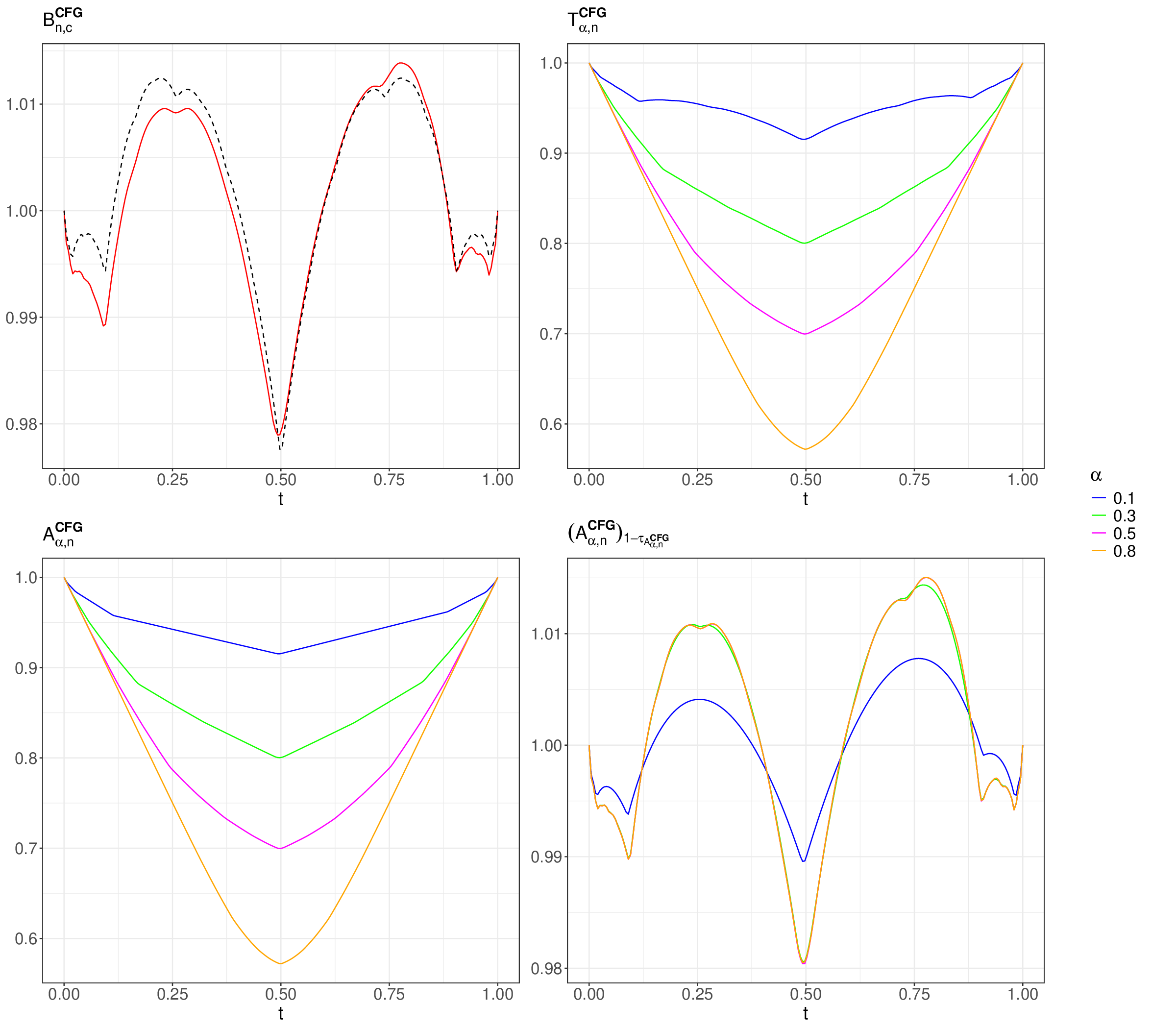}
	\caption{Plots of the estimator $B_{n,c}^{\mathbf{CFG}}$ (upper left panel, red line) and the function $(A)_{1-\tau_A}$ (upper left panel, dashed black line), the functions $T_{\alpha,n}^{\mathbf{CFG}}$ (upper right panel), the Pickands dependence functions $A_{\alpha,n}^{\mathbf{CFG}}$ (lower left panel), and their transformed versions $(A_{\alpha,n}^{\mathbf{CFG}})_{1-\tau_{A_{\alpha,n}^{\mathbf{CFG}}}}$ (lower right panel), for $\alpha \in \{0.1, 0.3, 0.5, 0.8\}$.}
	\label{fig:explanation_regularization}
\end{figure}
\subsection{Consistency of the Archimax estimator}
Building upon the previous observations, we can now state and prove 
strong consistency of our Archimax estimator. Fixing $\psi \in \Psi$ and $A \in \mathcal{A}$, denoting Kendall's $\tau$ of $C_A \in \mathcal{C}_{ev}$ 
by $\tau_A$, and letting $Z$ be a random variable with survival function 
$(\psi)_{1-\tau_A}$, throughout this section we assume that $\mathbb{E}[Z] < \infty$ if $\Xi = \mathbf{P}$, and $\mathbb{E}[|\log Z|] < \infty$ if $\Xi = \mathbf{CFG}$.
\begin{theorem}\label{thrm:strong_consistency_copula}
Let $(X_1,Y_1), (X_2,Y_2), \ldots$ be a stationary and ergodic sequence of random vectors with common bivariate, continuous distribution function $H$, univariate marginal distribution functions $F$ and $G$, and underlying Archimax copula $C_{\psi,A} \in \mathcal{C}_{am}$ with generator $\psi \in \Psi$ and Pickands dependence function $A \in \mathcal{A}\setminus\{A_M\}$. Furthermore, let $\Xi \in \{\mathbf{P},\mathbf{CFG}\}$ and let $C_{\psi_{\alpha_n,n}^\Xi, A_{\alpha_n,n}^\Xi}$ denote the estimator proposed in eq.~\eqref{eq:copula_estimator}. If $\Xi = \mathbf{P}$, we assume that Assumption~\ref{assumption1_pick} is satisfied; if $\Xi = \mathbf{CFG}$, we assume that Assumptions~\ref{assumption1_cfg} and~\ref{assumption2_cfg} are satisfied. Then, in either case,
\begin{equation}
	d_\infty \left(C_{\psi_{\alpha_n,n}^\Xi, A_{\alpha_n,n}^\Xi}, C_{\psi,A}\right) \xrightarrow{n\to\infty} 0
\end{equation}
with probability $1$.
\end{theorem}
\begin{proof}
See \ref{proof_copula_estimator}.
\end{proof}
\noindent Applying Theorem \ref{thrm:main_result_convergence} 
yields strong consistency of the estimators $(A_{\alpha_n,n}^\Xi)_{1-\tau_{A_{\alpha_n,n}^\Xi}}$, $(\psi_{\alpha_n,n}^\Xi)_{1-\tau_{A_{\alpha_n,n}^\Xi}}$, $(\varphi_{\alpha_n,n}^\Xi)_{1-\tau_{A_{\alpha_n,n}^\Xi}}$ and $(\beta_{\alpha_n,n}^\Xi)_{1-\tau_{A_{\alpha_n,n}^\Xi}}$. The following theorem holds:
\begin{theorem}\label{thrm:strong_consistency_fcts} 
Let $(X_1,Y_1), (X_2,Y_2), \ldots$ be a stationary and ergodic sequence of random vectors with common bivariate, continuous distribution function $H$, univariate marginal distribution functions $F$ and $G$, and underlying Archimax copula $C_{\psi,A} \in \mathcal{C}_{am}$ with generator $\psi \in \Psi$ and Pickands dependence function $A \in \mathcal{A}\setminus\{A_M\}$. Furthermore, let $\Xi \in \{\mathbf{P},\mathbf{CFG}\}$ and $C_{\psi_{\alpha_n,n}^\Xi, A_{\alpha_n,n}^\Xi}$ be the estimator as proposed in eq. \eqref{eq:copula_estimator}. 
    If $\Xi = \mathbf{P}$, we assume that Assumption \ref{assumption1_pick} is 
    satisfied, if $\,\Xi = \mathbf{CFG}$, we assume that Assumptions  \ref{assumption1_cfg} and \ref{assumption2_cfg} are satisfied. Then the estimators in eqs. \eqref{eq:final_estimator}, \eqref{eq:est_trans_phi}, \eqref{eq:est_trans_psi} and \eqref{eq:beta_n_trans} are strongly consistent, i.e., with probability $1$ the following four assertions hold:
\begin{enumerate}[label=(\roman*)]
    \item 
$
(A_{\alpha_n,n}^\Xi)_{1-\tau_{A_{\alpha_n,n}^\Xi}} \overset{n \rightarrow \infty}{\longrightarrow} (A)_{1-\tau_{A}}
$
uniformly on $\mathbb{I}$,
\item
$
(\psi_{\alpha_n,n}^\Xi)_{1-\tau_{A_{\alpha_n,n}^\Xi}} \overset{n\rightarrow\infty}{\longrightarrow} (\psi)_{1-\tau_A}
$
uniformly on $[0,\infty)$,
\item
$
(\varphi_{\alpha_n,n}^\Xi)_{1-\tau_{A_{\alpha_n,n}^\Xi}} \overset{n\rightarrow\infty}{\longrightarrow} (\varphi)_{1-\tau_A}
$
pointwise on $(0,1]$, and
\item
$
(\beta_{\alpha_n,n}^\Xi)_{1-\tau_{A_{\alpha_n,n}^\Xi}} (t)\overset{n\rightarrow\infty}{\longrightarrow} (\beta)_{1-\tau_A}(t)
$
in every point of continuity $t \in \mathbb{I}$ of $(\beta)_{1-\tau_A}$.
\end{enumerate}
\end{theorem}
\begin{proof}
The first three assertions are immediate consequences of Theorems \ref{thrm:main_result_convergence} and \ref{thrm:strong_consistency_copula}. 
The remaining fourth assertion follows from the first three assertions and the equivalences in \cite{mult_arch,bernoulli}.
\end{proof}
Since according to Theorem \ref{thrm:main_result_convergence}, in the Archimax family $\mathcal{C}_{am}$ uniform convergence and weak conditional convergence are equivalent, under the assumptions of Theorem \ref{thrm:strong_consistency_copula}
our estimator $C_{\psi_{\alpha_n,n}^\Xi, A_{\alpha_n,n}^\Xi}$
converges weakly conditional to $C_{\psi,A}$. The following result 
concerning the estimation of dependence measures is therefore immediate:
\begin{theorem}\label{thrm:strong_consistency_dep_meas}
Let $(X_1,Y_1), (X_2,Y_2), \ldots$ be a stationary and ergodic sequence of random vectors with common bivariate, continuous distribution function $H$, univariate marginal distribution functions $F$ and $G$, and underlying Archimax copula $C_{\psi,A} \in \mathcal{C}_{am}$ with generator $\psi \in \Psi$ and Pickands dependence function $A \in \mathcal{A}\setminus\{A_M\}$. Furthermore, let $\Xi \in \{\mathbf{P},\mathbf{CFG}\}$ and $C_{\psi_{\alpha_n,n}^\Xi, A_{\alpha_n,n}^\Xi}$ be the estimator as proposed in eq. \eqref{eq:copula_estimator}. 
    If $\Xi = \mathbf{P}$, we assume that Assumption \ref{assumption1_pick} is satisfied, if $\Xi = \mathbf{CFG}$, we assume that Assumptions \ref{assumption1_cfg} and \ref{assumption2_cfg} are satisfied.
    Then for every dependence measure $\eta: \mathcal{C} \rightarrow \mathbb{I}$ 
    which is continuous with respect to weak conditional convergence, the following 
    identity holds with probability $1$:
    \begin{equation}\label{eq:consistency_dependence_measure}
    \eta(C_{\psi_{\alpha_n,n}^\Xi, A_{\alpha_n,n}^\Xi}) \overset{n \rightarrow \infty }{\longrightarrow} \eta(C_{\psi,A}).
    \end{equation}
\end{theorem}
\begin{proof}
Immediate consequence of Theorem \ref{thrm:strong_consistency_copula} and Corollary \ref{cor:continuity_dependence_measures}.
\end{proof}

\section{Simulation study}\label{sec:sim_study}
In this section, we study the finite-sample behavior of the Pickands and CFG type estimators $(A_{\alpha_n,n}^\mathbf{P})_{1-\tau_{A_{\alpha_n,n}^\mathbf{P}}}$ and $(A_{\alpha_n,n}^\mathbf{CFG})_{1-\tau_{A_{\alpha_n,n}^\mathbf{CFG}}}$, $(\beta_{\alpha_n,n}^\mathbf{P})_{1-\tau_{A_{\alpha_n,n}^\mathbf{P}}}$ and $(\beta_{\alpha_n,n}^\mathbf{CFG})_{1-\tau_{A_{\alpha_n,n}^\mathbf{CFG}}}$, together with the copula estimators $C_{\alpha_n,n}^\mathbf{P}$ and $C_{\alpha_n,n}^\mathbf{CFG}$ introduced in Section \ref{sec:estimator}. In view of Theorem \ref{thrm:strong_consistency_dep_meas}, we also briefly examine how these estimators perform when used to estimate dependence measures such as Chatterjee's $\xi$ (see \citep{chatterjee2021}). Samples from parametric families of Archimax copulas were generated using the \textsf{R} package \textsf{acopula} (see \citep{bacigal2023}). Samples for Archimax copulas whose Pickands dependence function is
not implemented in the package \textsf{acopula} were generated via conditional inverse sampling. \\
\begin{table}[h]
\centering
\small
\begin{tabular}{|l|c|c|c|}
\hline
\textbf{Family} & \textbf{Generator } $\psi(z)$ & \textbf{Pseudo-inverse } $\varphi(t)$ & \textbf{Parameter space} \\
\hline
Clayton & $\left(1+(0.5^{-\theta}-1)z\right)^{-1/\theta}$ & $\dfrac{t^{-\theta}-1}{0.5^{-\theta}-1}$ & $\theta \in (0,\infty)$ \\[10pt]
Frank & $-\dfrac{1}{\theta}\log\!\left(1+\left(1+\exp(-\theta/2)\right)^{-z}(\exp(-\theta)-1)\right)$ & $-\dfrac{1}{\log\!\left(1+\exp(-\theta/2)\right)}\log\!\left(\dfrac{\exp(-\theta t)-1}{\exp(-\theta)-1}\right)$ & $\theta \in \mathbb{R}\setminus\{0\}$ \\[12pt]
Gumbel & $\exp\!\left(-(\log 2)\,z^{1/\theta}\right)$ & $\left(\dfrac{-\log t}{\log 2}\right)^{\theta}$ & $\theta \in [1,\infty)$ \\[10pt]
Joe & $1-\left(1-(1-2^{-\theta})^{z}\right)^{1/\theta}$ & $\dfrac{\log\!\left(1-(1-t)^{\theta}\right)}{\log\!\left(1-2^{-\theta}\right)}$ & $\theta \in [1,\infty)$ \\[10pt]
\hline
\end{tabular}
\caption{Normalized generators, pseudo-inverses and parameter ranges of 
the Archimedean copula families considered in Section \ref{sec:sim_study}; 
each generator satisfies $\psi(1)=0.5$ and $\varphi(0.5)=1$.}
\label{tab:archimedean_generators_normalized}
\end{table}
\begin{table}[h]
\centering
\small
\begin{tabular}{|l|c|c|}
\hline
\textbf{Family} & \textbf{Pickands dependence function } $A(t)$ & \textbf{Parameter space} \\
\hline
Gumbel & $\left(t^{1/\theta} + (1-t)^{1/\theta}\right)^{\theta}$ & $\theta \in (0,1]$ \\[10pt]
Galambos & $1 - \left(t^{-\theta} + (1-t)^{-\theta}\right)^{-1/\theta}$ & $\theta \in (0,\infty)$ \\[10pt]
Marshall--Olkin & $1 - \min\!\left(\alpha t,\, \beta(1-t)\right)$ & $\alpha,\beta \in [0,1]$ \\[6pt]
\hline
\end{tabular}
\caption{Pickands dependence functions and parameter ranges of 
the EVC families considered in Section \ref{sec:sim_study}.}
\label{tab:pickands_functions}
\end{table}
The setup of our simulation study is as follows: We 
consider the sample sizes $n \in \{100,500,1000\}$ and work with the generators of the Clayton, Frank, Gumbel and Joe copulas, along with the Pickands dependence functions of the Gumbel, Galambos and Marshall--Olkin family; the corresponding definitions are gathered in Table \ref{tab:archimedean_generators_normalized} and Table \ref{tab:pickands_functions}. Generator parameters were selected to span a range of association strengths, specifically $\tau_\psi \in \{\frac{1}{5},\frac{2}{5},\frac{3}{5},\frac{4}{5}\}$ and $\tau_A \in \{\frac{1}{5},\frac{2}{5},\frac{3}{5},\frac{4}{5}\}$ across the different generators and Pickands dependence functions. Some parameter choices deliberately push against the condition $\mathbb{E}[Z] < \infty$ required in Theorems \ref{thrm:strong_consistency_copula} and \ref{thrm:strong_consistency_fcts}, respectively, allowing us to 
observe the estimators' behavior near these boundaries. 
Below we summarize the key findings and highlight a selection of representative results; the full set of results is deferred to \ref{sec:sim_study_app}. Following the approach of \citep{est-archimax}, estimator accuracy for $(A_{\alpha_n,n}^\Xi)_{1-\tau_{A_{\alpha_n,n}^\Xi}}$ and $(\beta_{\alpha_n,n}^\Xi)_{1-\tau_{A_{\alpha_n,n}^\Xi}}$, $\Xi \in \{\mathbf{P}, \mathbf{CFG}\}$, is assessed via the Integrated Squared Error (ISE) and the Integrated Relative Absolute error (IRAE). For continuous, bounded functions $f,g \colon \mathbb{I}\rightarrow \mathbb{R}$, these are defined as
\begin{equation*}
\mathrm{ISE}(f,g) := \int_\mathbb{I} |f(t) - g(t)|^2 \mathrm{d}\lambda(t)
\end{equation*}
and, provided $f,g \geq 0$,
\begin{equation*}
\mathrm{IARE}(f,g) := \int_\mathbb{I} \frac{|f(t) -g(t)|}{g(t)} \mathrm{d}\lambda(t),
\end{equation*}
assuming the integral is well-defined. Both quantities are approximated via Monte Carlo integration, drawing 10,000 uniform points on $\mathbb{I}$ per replicate, with 1000 Monte Carlo replicates used for each simulation setting. The performance of the copula estimators $C_{\alpha_n,n}^\mathbf{P}$ and $C_{\alpha_n,n}^\mathbf{CFG}$ is assessed via the uniform distance to the true copula, the latter distance is approximated by evaluating the estimated and true copula on the grid 
$\{0,\frac{1}{49},\frac{2}{49},\dots,\frac{48}{49},1\}^2 \subset \mathbb{I}^2$.
As before, this procedure is repeated over 1000 Monte Carlo replicates.
\subsection{Comparison of the adjusted Pickands and CFG type estimators and the estimators for the generator $(\psi)_{1-\tau_A}$}\label{subsec:sim_study_pick_cfg_beta}
\noindent We begin by comparing the Pickands and CFG type estimators, as defined in eq.\ \eqref{eq:final_estimator}, across a range of scenarios; the corresponding results are reported in Tables \ref{fig:cfg_vs_pick_gumbel_0.4}--\ref{fig:cfg_vs_pick_mar_ol_all_vs_all} in \ref{sec:sim_study_app}. Figure \ref{fig:simulation_trans_pick} is representative of the pattern observed throughout: the CFG type estimator outperforms the Pickands type estimator on average, both in terms of ISE and IRAE. This is consistent with the results reported for the CFG estimator in \citep{genest2009} and for the CFG type estimator in \citep{est-archimax}. 
Letting $Z$ denote the random variable with survival function $(\psi)_{1-\tau_A}$, note that the regularity condition $\mathbb{E}[Z] < \infty$ is more restrictive than $\mathbb{E}[|\log Z|] < \infty$. As shown in Lemma \ref{lem:regularity_fcts}, $\mathbb{E}[\log Z]$ exists for all families considered in this simulation study, whereas $\mathbb{E}[Z]$ fails to exist for most parameter choices in the Clayton case.\\
Tables \ref{fig:beta_cfg_vs_pick_gumbel_0.4}--\ref{fig:beta_cfg_vs_pick_mar_ol_0.4} compare the two estimators for various scenarios, with the parameter of the Pickands dependence function fixed so that Kendall's $\tau$ of the corresponding EVC equals $\tau_A = \frac{2}{5}$. The CFG type estimator consistently outperforms the Pickands type estimator. Both estimators perform better when the Pickands dependence function is of Gumbel or of 
Galambos type than when the Marshall-Olkin family is considered, likely because the Pickands dependence function of the latter one fails to be differentiable at certain points. The Pickands type estimator performs particularly poorly when the generator of the Archimax copula is from the Clayton family, 
which is consistent with the fact that $\mathbb{E}[Z]$ need not exist for all of the chosen parameter values in this case.\\
Since the CFG type estimator generally outperforms the Pickands type estimator across the values of $\tau_\psi$ and $\tau_A$ considered, we now focus exclusively on the CFG type estimator and compute the IRAE and ISE for each generator and Pickands dependence function over the full grid $(\tau_\psi,\tau_A) \in \{\frac{1}{5},\frac{2}{5},\frac{3}{5},\frac{4}{5}\}^2$. In the Galambos and Marshall-Olkin cases (Tables \ref{fig:cfg_vs_pick_galambos_all_vs_all} and \ref{fig:cfg_vs_pick_mar_ol_all_vs_all}) the performance of the CFG type estimator deteriorates as $\tau_\psi$ and $\tau_A$ increase.\\
For the parameter values underlying Figure \ref{fig:simulation_trans_pick}, both population quantities $\mathbb{E}[Z]$ and $\mathbb{E}[\log Z]$ exist. Nonetheless, the Pickands type estimator continues to exhibit a substantial number of outliers, some of which are sufficiently large that the $y$-axis is displayed on a square-root scale to preserve visibility. The comparative absence of outliers in the CFG type estimator is most plausibly attributable to the logarithmic transformation, which maps large values of $\varphi_n$ (in particular, values near 0) onto a smaller scale; this results in the estimator of $\mathbb{E}[\log Z]$ substantially outperforming the estimator of $\mathbb{E}[Z]$.\\
Comparing the performance of the Pickands type estimator $(\varphi_{\alpha_n,n}^\mathbf{P})_{1-\tau_{A_{\alpha_n,n}^\mathbf{P}}}$ and the CFG type estimator $(\varphi_{\alpha_n,n}^\mathbf{CFG})_{1-\tau_{A_{\alpha_n,n}^\mathbf{CFG}}}$, we proceed as follows: 
Since Archimedean generators $\psi$ are defined on $[0,\infty)$ and their pseudo-inverses $\varphi$ may be unbounded at $0$, a direct comparison of estimators of these functions is not straightforward. We therefore follow the approach adopted in \citep{GNZ,GR} and instead compare the functions $(\beta_{\alpha_n,n}^\mathbf{P})_{1-\tau_{A_{\alpha_n,n}^\mathbf{P}}}$ and $(\beta_{\alpha_n,n}^\mathbf{CFG})_{1-\tau_{A_{\alpha_n,n}^\mathbf{CFG}}}$, which depend on the corresponding Kendall's $\tau$, the generator, and its left-hand derivative. As these functions are defined on the compact interval $\mathbb{I}$ 
and only attain finite values, they constitute a more suitable basis for assessing estimator performance. The results in Figure \ref{fig:simulation_trans_beta} and Tables \ref{fig:beta_cfg_vs_pick_gumbel_0.4}--\ref{fig:beta_cfg_vs_pick_mar_ol_all_vs_all} indicate that the CFG-based estimator continues to outperform its Pickands-based counterpart, although the discrepancy between the two is significantly smaller than that observed for the estimators $(A_{\alpha_n,n}^\mathbf{P})_{1-\tau_{A_{\alpha_n,n}^\mathbf{P}}}$ and $(A_{\alpha_n,n}^\mathbf{CFG})_{1-\tau_{A_{\alpha_n,n}^\mathbf{CFG}}}$, respectively. This is most likely attributable to the fact that, with the exception of Kendall's $\tau$ of the extreme-value component, the estimators $(\beta_{\alpha_n,n}^\mathbf{P})_{1-\tau_{A_{\alpha_n,n}^\mathbf{P}}}$ and $(\beta_{\alpha_n,n}^\mathbf{CFG})_{1-\tau_{A_{\alpha_n,n}^\mathbf{CFG}}}$ depend exclusively on the estimator of the generator itself. Moreover, both estimators tend to perform more favorably when the underlying Pickands dependence function is smooth. To assess the estimators under less favorable conditions, we therefore additionally analyze a setting in which the true Pickands dependence function $A$ is highly irregular and consider $A$ as depicted in Figure \ref{fig:discrete.approx}.\\
Taking into account the simulations in Tables~\ref{fig:beta_cfg_vs_pick_gumbel_0.4}--\ref{fig:beta_cfg_vs_pick_mar_ol_all_vs_all}, the performance of the aforementioned estimators gets worse for larger values of $\tau_\psi$ and $\tau_A$. This is in contrast to the performance of $(A_{\alpha_n,n}^\mathbf{P})_{1-\tau_{A_{\alpha_n,n}}^\mathbf{P}}$ and $(A_{\alpha_n,n}^\mathbf{CFG})_{1-\tau_{A_{\alpha_n,n}}^\mathbf{CFG}}$, where performance decreases for smaller values of $\tau_\psi$ and $\tau_A$. We conjecture that this behavior is most likely due to the fact that the population version $(\beta)_{1-\tau_A}$ tends to exhibit a steeper slope for smaller values of $\tau_\psi$ and $\tau_A$, leading to a larger ISE, even though the estimators might seemingly deviate only slightly from the population version $(\beta)_{1-\tau_A}$.
\begin{figure}[!ht]
	\centering
	\includegraphics[width=1\textwidth]{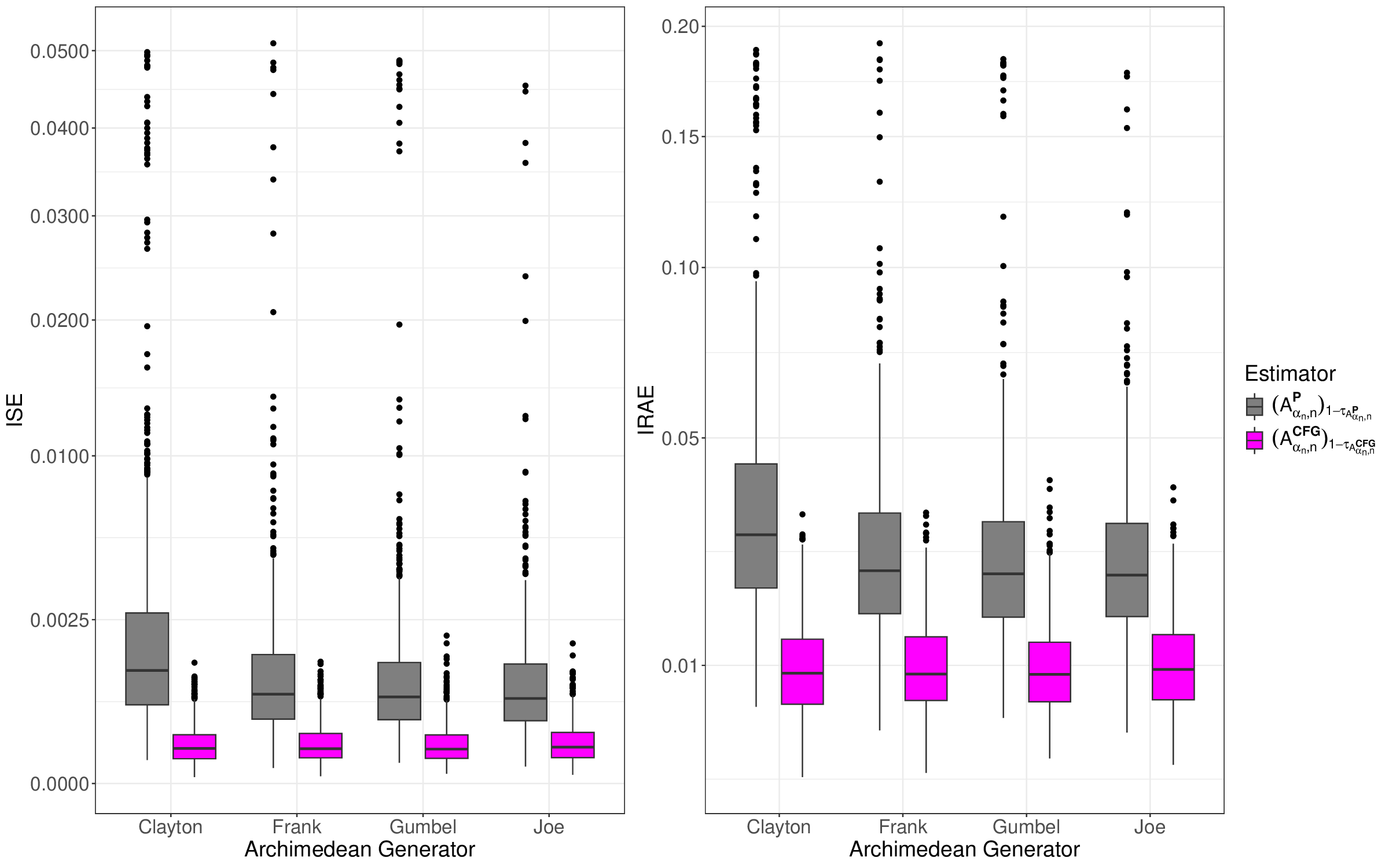}
	\caption{Boxplots of the ISE (left) and IRAE (right) for the Pickands type estimator $(A_{\alpha_n,n}^\mathbf{P})_{1-\tau_{A_{\alpha_n,n}^\mathbf{P}}}$ (gray) and the CFG type estimator $(A_{\alpha_n,n}^\mathbf{CFG})_{1-\tau_{A_{\alpha_n,n}^\mathbf{CFG}}}$ (magenta), based on a sample size of $n = 200$ and $1000$ Monte Carlo replications. The true Pickands dependence function is that of a Galambos copula with $\tau_A = \frac{1}{2}$. The true generators $\psi$ are those of a Clayton, Frank, Gumbel, and Joe copula with $\tau_\psi = \frac{1}{10}$, respectively. The $y$-axis is displayed on a square root scale.}
	\label{fig:simulation_trans_pick}
\end{figure}
\begin{figure}[!ht]
	\centering
	\includegraphics[width=1\textwidth]{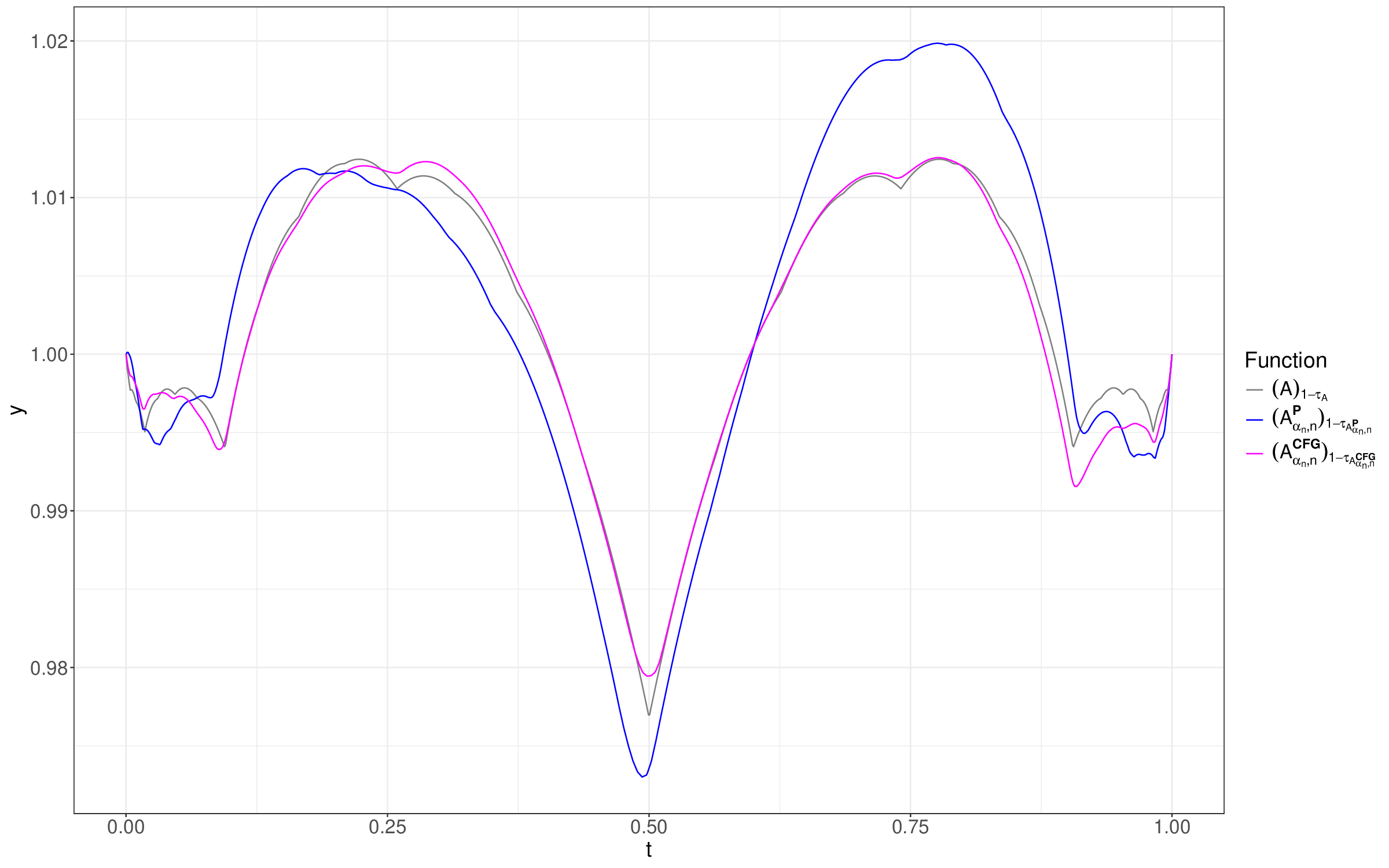}
	\caption{Plots of transformed Pickands dependence function $(A)_{1-\tau_A}$ (gray) as considered in Example \ref{ex:non_smooth_est} and the corresponding estimators $(A_{\alpha_n,n}^\mathbf{P})_{1-\tau_{A_{\alpha_n,n}^\mathbf{P}}}$ (blue) and $(A_{\alpha_n,n}^\mathbf{CFG})_{1-\tau_{A_{\alpha_n,n}^\mathbf{CFG}}}$ (magenta), respectively.}
	\label{fig:est_path_pick}
\end{figure}
\begin{figure}[!ht]
	\centering
	\includegraphics[width=1\textwidth]{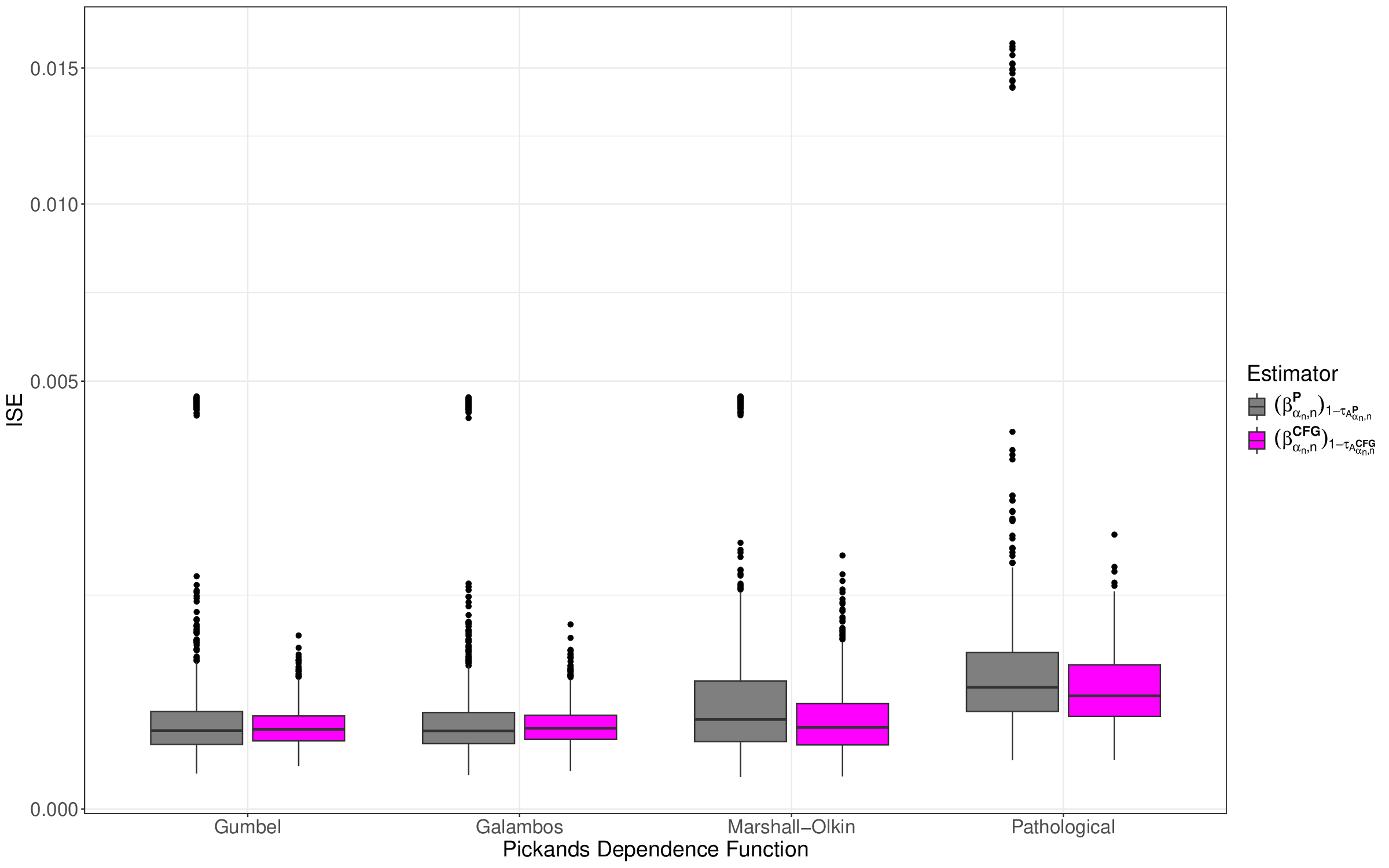}
	\caption{Boxplots of the ISE for the Pickands type estimator $(\beta_{\alpha_n,n}^\mathbf{P})_{1-\tau_{A_{\alpha_n,n}^\mathbf{P}}}$ (gray) and the CFG type estimator $(\beta_{\alpha_n,n}^\mathbf{CFG})_{1-\tau_{A_{\alpha_n,n}^\mathbf{CFG}}}$ (magenta), based on a sample size of $n = 200$ and $1000$ Monte Carlo replications. The true generator is that of a Frank copula with $\tau_\psi = \frac{3}{20}$. The true Pickands dependence functions $A$ (from left to right) are those of a Gumbel, Galambos, and Marshall--Olkin copula with $\tau_A = \frac{7}{10}$, as well as the pathological Pickands dependence function given in Figure~\ref{fig:discrete.approx}. The $y$-axis is displayed on a square root scale.}
	\label{fig:simulation_trans_beta}
\end{figure}
\subsection{Comparison of the copula estimators and their 
plug-ins for estimating directed dependence}
\noindent As one of the goals of this paper is to provide a fully non-parametric estimator for Archimax copulas that is itself an Archimax copula, it is natural to also compare the estimators $C_{\alpha_n,n}^\mathbf{P}$ and $C_{\alpha_n,n}^\mathbf{CFG}$ with the real Archimax copula $C_{\psi,A}$
for different sample sizes. Furthermore, in order to assess, how these estimators perform relative to the empirical copula estimator 
(as described in Section \ref{subsection:general_notation}) -- that is, the standard estimator ignoring the Archimax information -- we compared both $C_{\alpha_n,n}^\mathbf{P}$ and $C_{\alpha_n,n}^\mathbf{CFG}$ to the  empirical copula estimator $\hat{C}_n$. For the simulations summarized in Tables \ref{fig:cop_cfg_vs_pick_gumbel_0.4}--\ref{fig:cop_cfg_vs_pick_mar_ol_all_vs_all}, we again considered the sample sizes $n \in \{100,500,1000\}$. 
To gain additional insight into the small-sample behavior, we 
furthermore included Figure \ref{fig:simulation_trans_cop}, comparing the three estimators for sample sizes $n \in \{25,50,100,250,500,1000\}$. The results show that, even for small sample sizes, the Pickands-based estimator outperforms the empirical copula estimator--except in the Clayton case, and occasionally in the Gumbel regime--while the CFG-based estimator in turn outperforms the Pickands-based estimator. Analogous to Tables \ref{fig:cop_cfg_vs_pick_gumbel_0.4}--\ref{fig:cop_cfg_vs_pick_mar_ol_all_vs_all}, the performance of the estimators $C_{\alpha_n,n}^\mathbf{P}$ and $C_{\alpha_n,n}^\mathbf{CFG}$ again improves for larger values of $\tau_\psi$ and $\tau_A$.\\
Turning to the (natural plug-in) estimation of dependence measures, in light of Theorems \ref{thrm:main_result_convergence} and \ref{thrm:strong_consistency_dep_meas}, we compare the performance of the
estimators $\xi(C_{\alpha_n,n}^\mathbf{P})$ and $\xi(C_{\alpha_n,n}^\mathbf{CFG})$ for Chatterjee's $\xi$ against the originally proposed nearest neighbor based estimator $\hat{\xi}_n$. Figure \ref{fig:simulation_dep_meas} summarizes the 
obtained results for sample sizes $n \in \{25, 50, 100, 250, 500, 1000\}$. Unsurprisingly, the CFG-based estimator outperforms the other two estimators, particularly for small sample sizes. Notably, the Pickands-based estimator appears to exhibit a bias toward complete dependence.
\begin{figure}[!ht]
	\centering
	\includegraphics[width=1\textwidth]{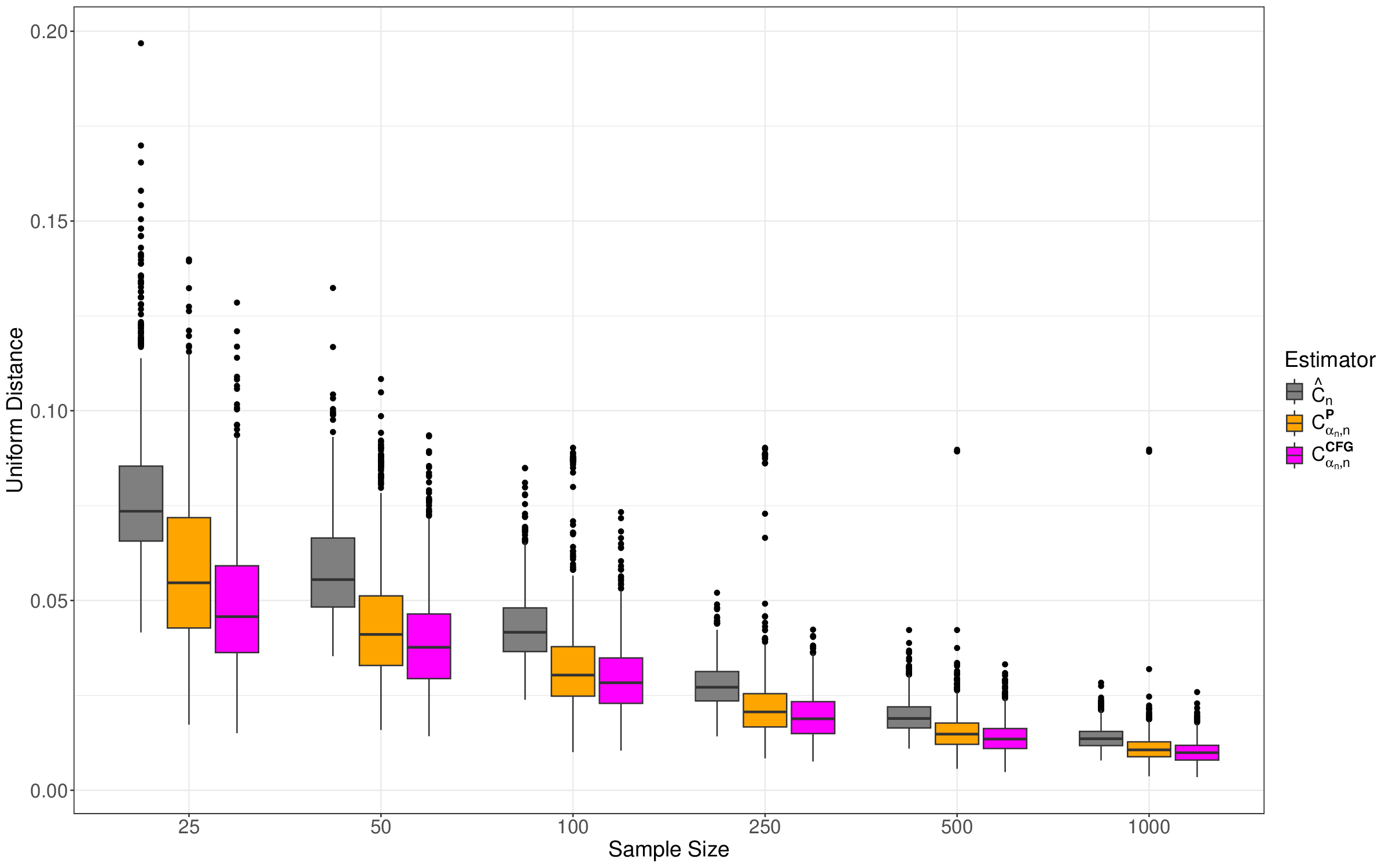}
	\caption{Boxplots of the uniform distance for the (bilinear extension of the) empirical copula estimator $\hat{C}_n$ (gray), the Pickands type estimator $C_{\alpha_n,n}^\mathbf{P}$ (orange) and the CFG type estimator $C_{\alpha_n,n}^\mathbf{CFG}$ (magenta), for sample sizes $n \in \{25, 50, 100, 250, 500, 1000\}$ and $1000$ Monte Carlo replications. The true copula is an Archimax copula $C_{\psi,A}$ with generator $\psi$ of a Joe copula with $\tau_\psi = \frac{1}{4}$. The true Pickands dependence function $A$ is the pathological Pickands dependence function given in Figure~\ref{fig:discrete.approx}.}
	\label{fig:simulation_trans_cop}
\end{figure}
\begin{figure}[!ht]
	\centering
	\includegraphics[width=1\textwidth]{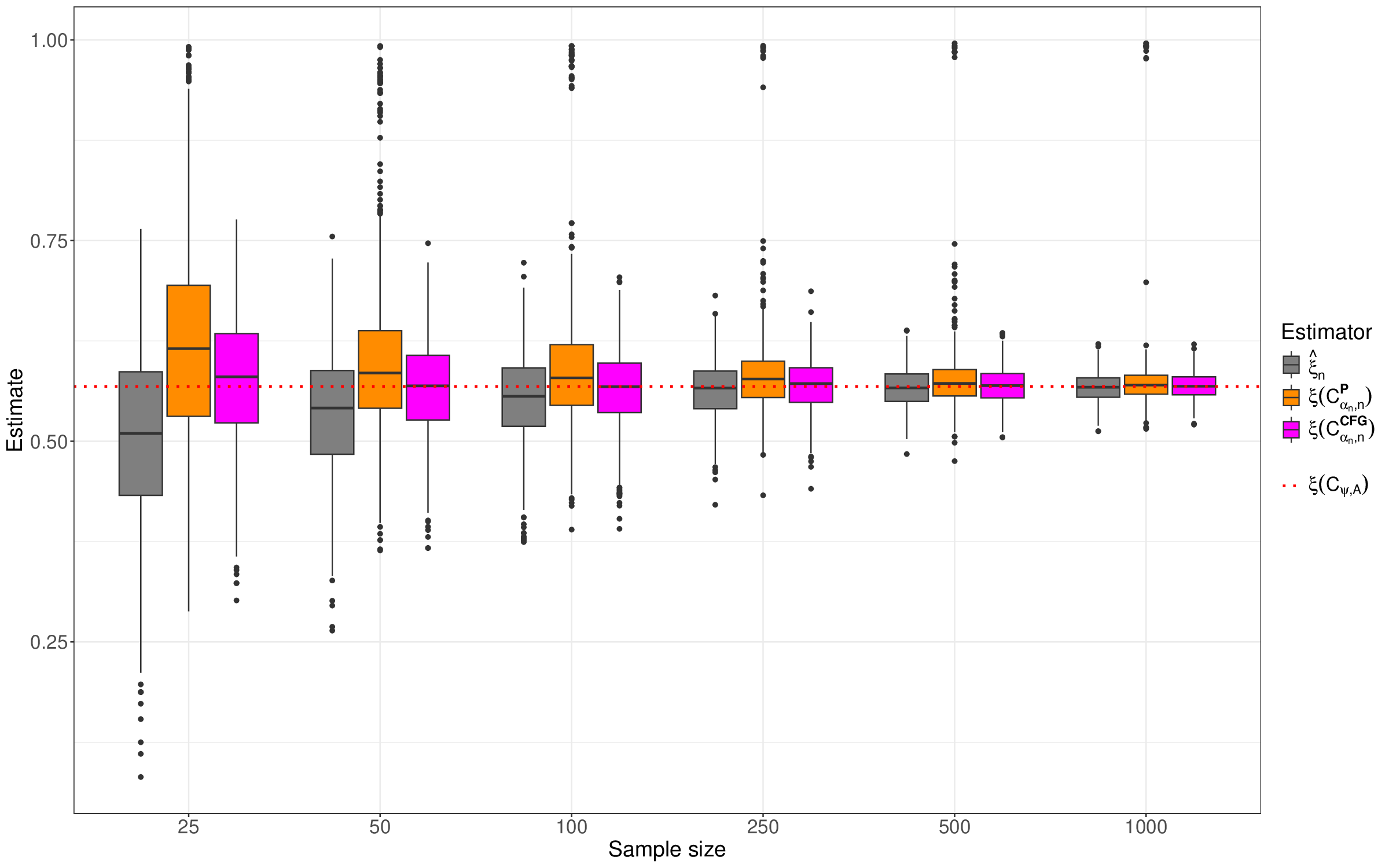}
	\caption{Boxplots of the Chatterjee estimator $\hat{\xi}_n$ (gray) and the plug-in estimators $\xi(C_{\alpha_n,n}^{\mathbf{P}})$ (orange) and $\xi(C_{\alpha_n,n}^{\mathbf{CFG}})$ (magenta), based on 1000 Monte Carlo replicates for each sample size $n \in \{25, 50, 100, 250, 500, 1000\}$. The dotted red line indicates the true value $\xi(C_{\psi,A}) \approx 0.5683$.}
	\label{fig:simulation_dep_meas}
\end{figure}
\section{Real data example}\label{sec:real_data}
In this section we illustrate the usefulness of our estimator in practice in terms of a real data example. We consider precipitation data from two locations in Austria, namely Bregenz and Dornbirn in Vorarlberg. 
The data, consisting of a bivariate sample of precipitation measurements in mm, recorded at 10-minute intervals between June 1, 1994, and May 31, 2026, is publicly available at Geosphere data hub \url{https://data.hub.geosphere.at/en/}.
Preparing the data for our purposes, we retain only the first 28 days per month and remove NA values. The data also contains a very low number of  (unrealistic/impossible) values 
of more than 100 mm of precipitation in a 10-minute interval, which we 
removed. Moreover, avoiding seasonal effects and snowfall we 
only consider the months March through August. 
Reducing short-term variability and reducing ties, we then calculated the 
average precipitation per day and finally the maximum of these 28 
daily averages per month. The resulting data contains 5 ties for the Bregenz series and 1 tie for the Dornbirn series, which were removed. 
Altogether we ended up with $181$ pairs of observations. All theoretical results established in the previous sections rely on the assumption that the data are generated by a stationary and ergodic sequence of random vectors. As ergodicity is not directly verifiable in practice, we instead assess the stronger, more tractable condition of stationarity and strong ($\alpha$-)mixing (see \citep[Definition 4.1]{est-archimax}), which is sufficient (though not necessary) for ergodicity.\\
As depicted in Figure~\ref{real_data_rolling mean}, the data exhibits no signs of non-stationarity, the rolling mean does not deviate substantially from the global mean. We also examined the behavior of the lagged variance and the autocorrelations and found no evidence against stationarity. Moreover, the Augmented Dickey--Fuller test (see \citep{fuller1996}) provides no evidence of a unit root. Applying both, the Ljung--Box and Box--Pierce tests, 
we found that both do not reject the null hypothesis of temporal independence at lags $\geq 5$ (see \citep{box1970, harvey1993, ljung1978}).

Figure~\ref{real_data_plots} 
depicts the data and the normalized ranks.
Considering that our data come from maxima, it seems natural to 
conjecture that the underlying dependence structure is of EVC type.
Assessing this assumption, we tested the null hypothesis that the data arise from an EVC. Applying the test proposed in \citep{genest2009, kojadinovic2010} (with parameters `$N = 1000$' and `derivatives $= C_n$') yields 
a p-value of approximately $p \approx 0.0099$. 
Hence, at conventional significance levels, we reject the EVC hypothesis.
Similarly to \citep{est-archimax}, a possible explanation is the pronounced clustering of observations in the lower-left corner of the scatterplot. Such a pattern suggests the presence of lower-tail dependence, a feature generally incompatible with EVCs.

Since Archimax copulas -- under suitable regularity conditions on the generator -- lie in the domain of attraction of an EVC and encompass both Archimedean copulas 
and EVCs as special cases, it seems reasonable to expect that the estimator proposed in eq.~\eqref{eq:copula_estimator} performs well for these data and is capable of capturing the lower-tail dependence seemingly present in the sample. In the following subsection, we therefore fit our nonparametric Archimax model to the data and compare its performance to the nonparametric CFG estimator for EVCs proposed in \citep{genest2009}.
\begin{figure}[!ht]
	\centering
	\includegraphics[width=1\textwidth]{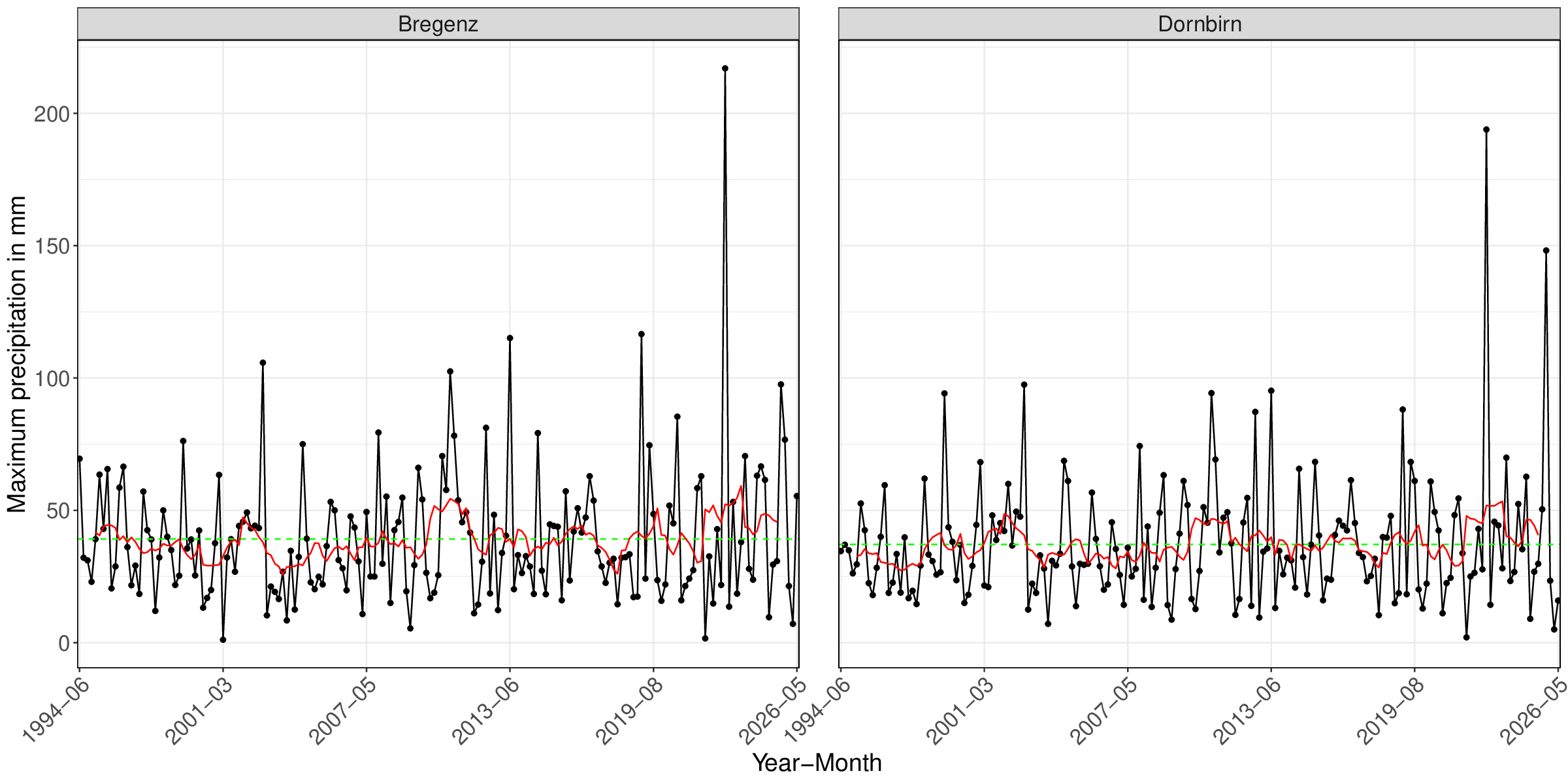}
	\caption{Plots of the monthly maximal (average daily)
    precipitation from 1994-06-01 to 2026-05-31 of Bregenz (left panel) and Dornbirn (right panel), Austria. 
   The red line denotes the rolling mean with lag $k = 10$,  the dashed green one the global mean.}
	\label{real_data_rolling mean}
\end{figure}
\begin{figure}[!ht]
	\centering
	\includegraphics[width=1\textwidth]{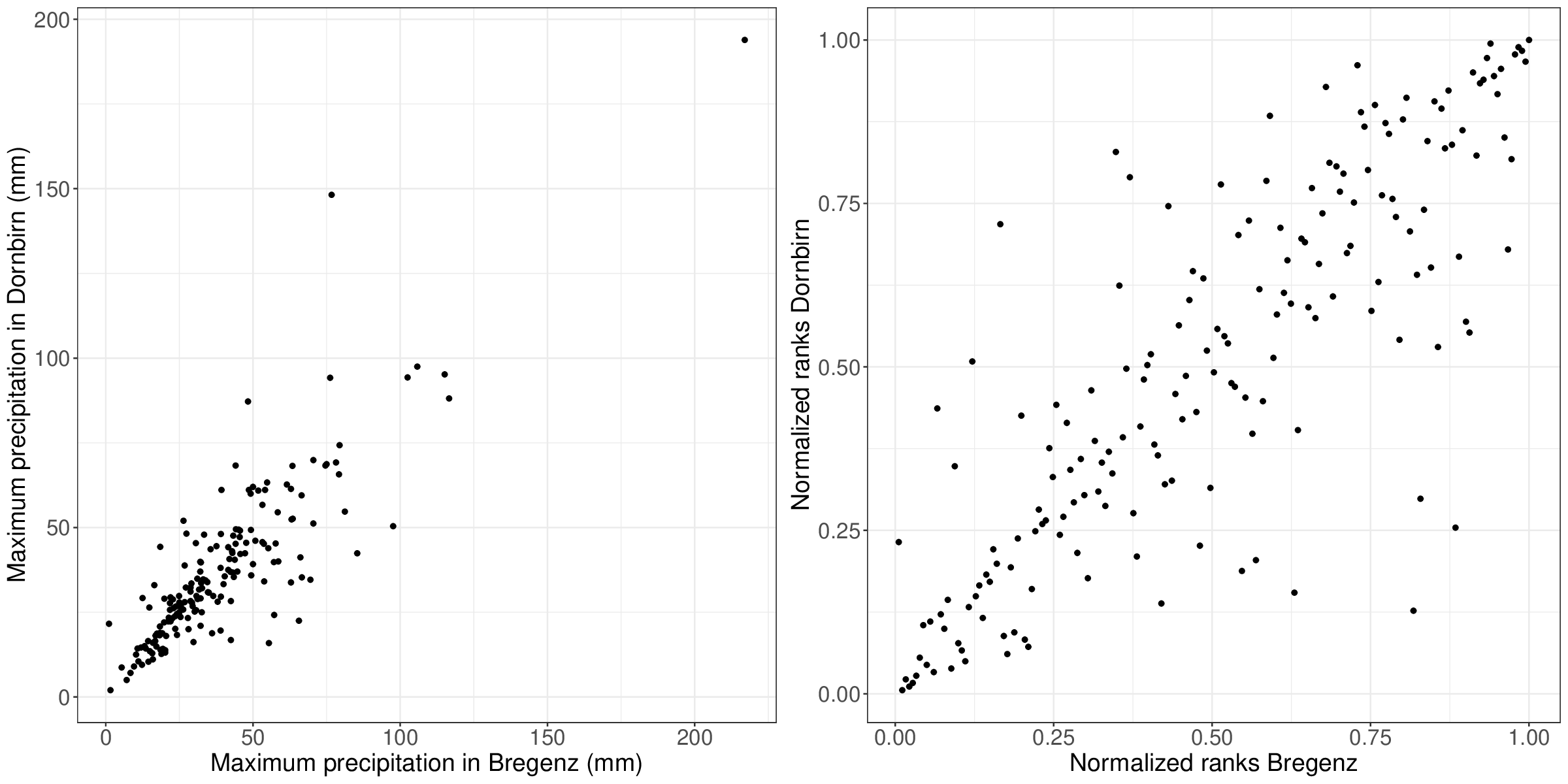}
	\caption{Plots of the monthly maximal precipitation from 1994-06-01 to 2026-05-31 of Bregenz and Dornbirn (left panel) and of the associated normalized ranks (right panel).}
	\label{real_data_plots}
\end{figure}
\subsection{Fitting the non-parametric Archimax model}
\noindent  We calculate the estimator $C_{\alpha_n,n}^{\mathbf{CFG}}$ 
defined according to eq.~\eqref{eq:C_alpha_n} for our rainfall data and compare it to the CFG-estimator (proposed in \citep{genest2009}) and to the empirical copula $C_n=C_n^\Pi$.
Since the estimator in eq.~\eqref{eq:C_alpha_n} relies only on Assumptions~\ref{assumption1_cfg}, \ref{assumption2_cfg}, and is 
fully nonparametric, no additional assumptions need to be verified; 
we can therefore expect it to yield reliable estimates for the 
underlying dependence structure. As demonstrated in the simulation study of Section~\ref{sec:sim_study}, the CFG type estimator $(A_{\alpha_n,n}^{\mathbf{CFG}})_{1-\tau_{A_{\alpha_n,n}^{\mathbf{CFG}}}}$ proposed in eq.~\eqref{eq:final_estimator} generally outperforms its Pickands type counterpart $(A_{\alpha_n,n}^{\mathbf{P}})_{1-\tau_{A_{\alpha_n,n}^{\mathbf{P}}}}$. We therefore focus on $(A_{\alpha_n,n}^{\mathbf{CFG}})_{1-\tau_{A_{\alpha_n,n}^{\mathbf{CFG}}}}$ as the estimator for the transformed Pickands dependence function $(A)_{1-\tau_A}$, and on $(\varphi_{\alpha_n,n}^{\mathbf{CFG}})_{1-\tau_{A_{\alpha_n,n}^{\mathbf{CFG}}}}$ as the estimator for the transformed pseudo-inverse $(\varphi)_{1-\tau_A}$.

Following the estimation procedure studied in Section~\ref{sec:estimator}, we begin by estimating $(\varphi)_{1-\tau_A}$ via the estimator $\varphi_n$ defined in eq.~\eqref{eq:est_trans_phi}, and use it to construct $B_{n,c}^{\mathbf{CFG}}$. Proceeding as described in Section~\ref{subsec:adjusted_est}, we obtain the parameter estimate $\alpha_n \approx 0.4368388$ (see eq.~\eqref{eq:alphas}), which yields the estimator $A_{\alpha_n,n}^{\mathbf{CFG}}$ given in eq.~\eqref{eq:A_alpha_n}. The corresponding Kendall's $\tau$ is
$$
\tau_{A_{\alpha_n,n}^{\mathbf{CFG}}} \approx 0.4755104.
$$
Having this yields the estimators $(A_{\alpha_n,n}^{\mathbf{CFG}})_{1-\tau_{A_{\alpha_n,n}^{\mathbf{CFG}}}}$ and $(\varphi_{\alpha_n,n}^{\mathbf{CFG}})_{1-\tau_{A_{\alpha_n,n}^{\mathbf{CFG}}}}$ as defined in eqs.~\eqref{eq:final_estimator} and~\eqref{eq:est_trans_phi}, respectively, which are displayed in Figures~\ref{real_data_plots_parameters} and~\ref{real_data_plots_parameters2}. Combining these two estimators finally 
yields our Archimax copula estimator $C_{\alpha_n,n}^{\mathbf{CFG}}$ for the underlying hypothesized true copula $C_{\psi,A}$ (see eq.~\eqref{eq:C_alpha_n}).

To benchmark our estimator against the one obtained under the assumption that the true underlying copula is an EVC, we proceed as follows. We estimate the Pickands dependence function using the CFG estimator of \citep{genest2009} and, since this estimator is not guaranteed to be convex, replace it by its greatest convex minorant. The resulting estimated EVC and its associated Pickands dependence function are denoted by $C_{EV,n}^{\mathbf{CFG}}$ and $A_{EV,n}^{\mathbf{CFG}}$, respectively.

To assess the relative goodness of fit, we compare the pointwise differences 
$C_n - C_{EV,n}^{\mathbf{CFG}}$ and $C_n - C_{\alpha_n,n}^{\mathbf{CFG}}$, where $C_n$ is again the (bilinear extension of the) empirical copula. 
Both differences are evaluated on a fine grid over $\mathbb{I}^2$ and visualized by means of heatmaps and boxplots in Figure~\ref{real_data_empirical_vs_cfg}, 
which shows that $C_{\alpha_n,n}^{\mathbf{CFG}}$ is uniformly closer to $C_n$ 
than $C_{EV,n}^{\mathbf{CFG}}$, suggesting that relaxing the EVC 
assumption leads to a more accurate representation of the dependence structure present in the data.
\begin{figure}[!ht]
	\centering
	\includegraphics[width=1\textwidth]{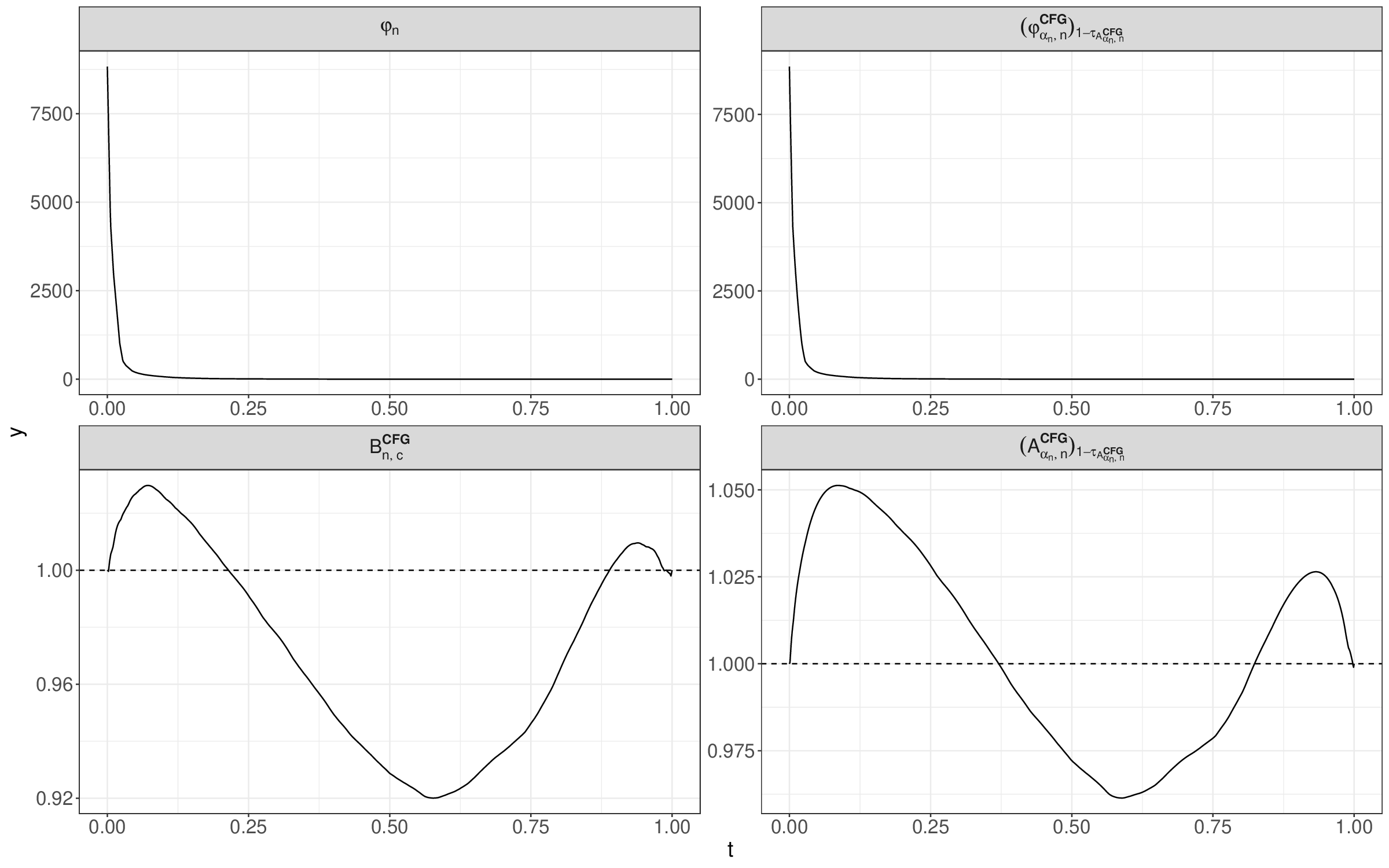}
	\caption{Plots of the estimators $\varphi_n$ (upper left panel), $(\varphi_{\alpha_n,n}^\mathbf{CFG})_{1-\tau_{A_{\alpha_n,n}^\mathbf{CFG}}}$ (upper right panel), $B_{n,c}^\mathbf{CFG}$ (lower left panel) and $(A_{\alpha_n,n}^\mathbf{CFG})_{1-\tau_{A_{\alpha_n,n}^\mathbf{CFG}}}$ (lower right panel).}
	\label{real_data_plots_parameters}
\end{figure}
\begin{figure}[!ht]
	\centering
	\includegraphics[width=1\textwidth]{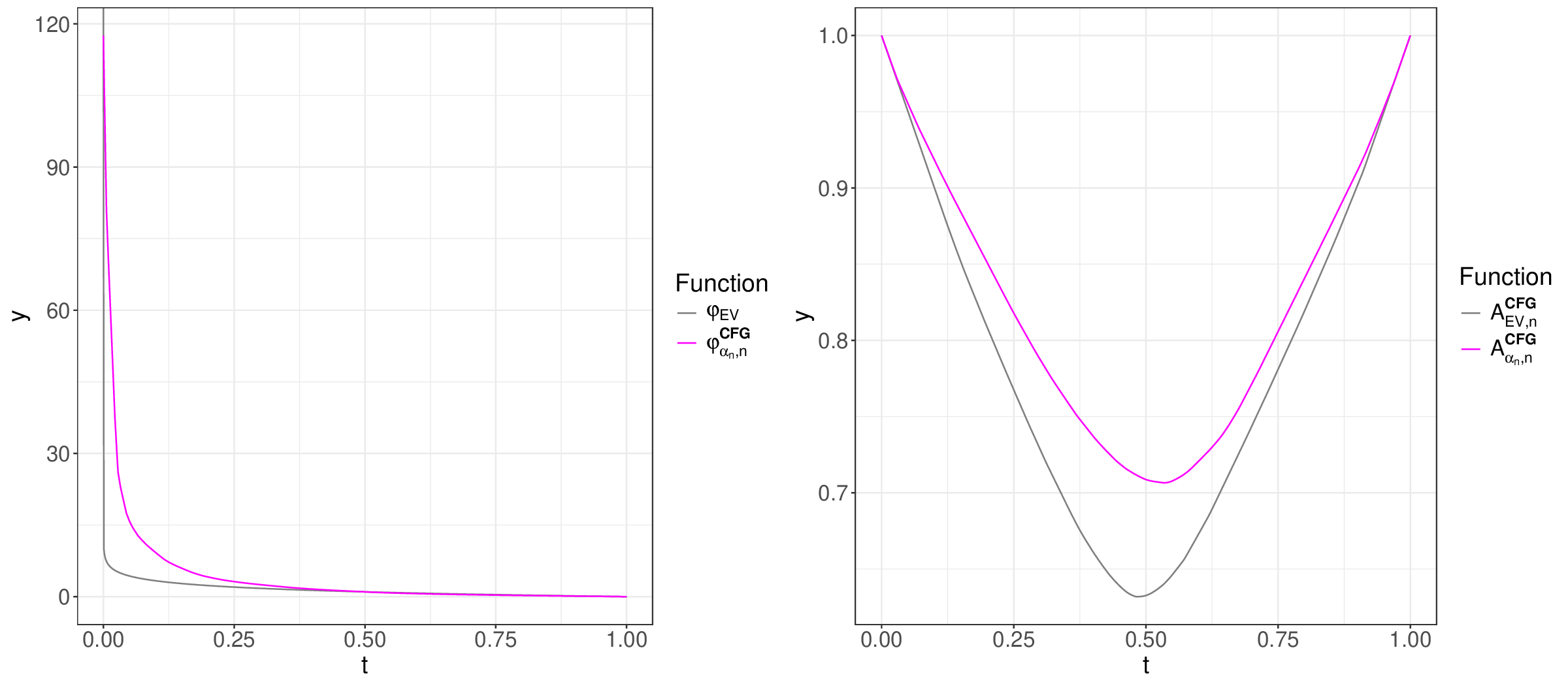}
	\caption{Plots of the estimator $\varphi_{\alpha_n,n}$ (magenta), the Archimedean generator $\varphi_{EV}(t) = \frac{\log(t)}{\log(\frac{1}{2})}$ (gray) in the left panel, and the estimators $A_{\alpha_n,n}^\mathbf{CFG}$ (magenta), $A_{EV,n}^\mathbf{CFG}$ (gray) in the right panel.}
	\label{real_data_plots_parameters2}
\end{figure}
\begin{figure}[!ht]
	\centering
	\includegraphics[width=1\textwidth]{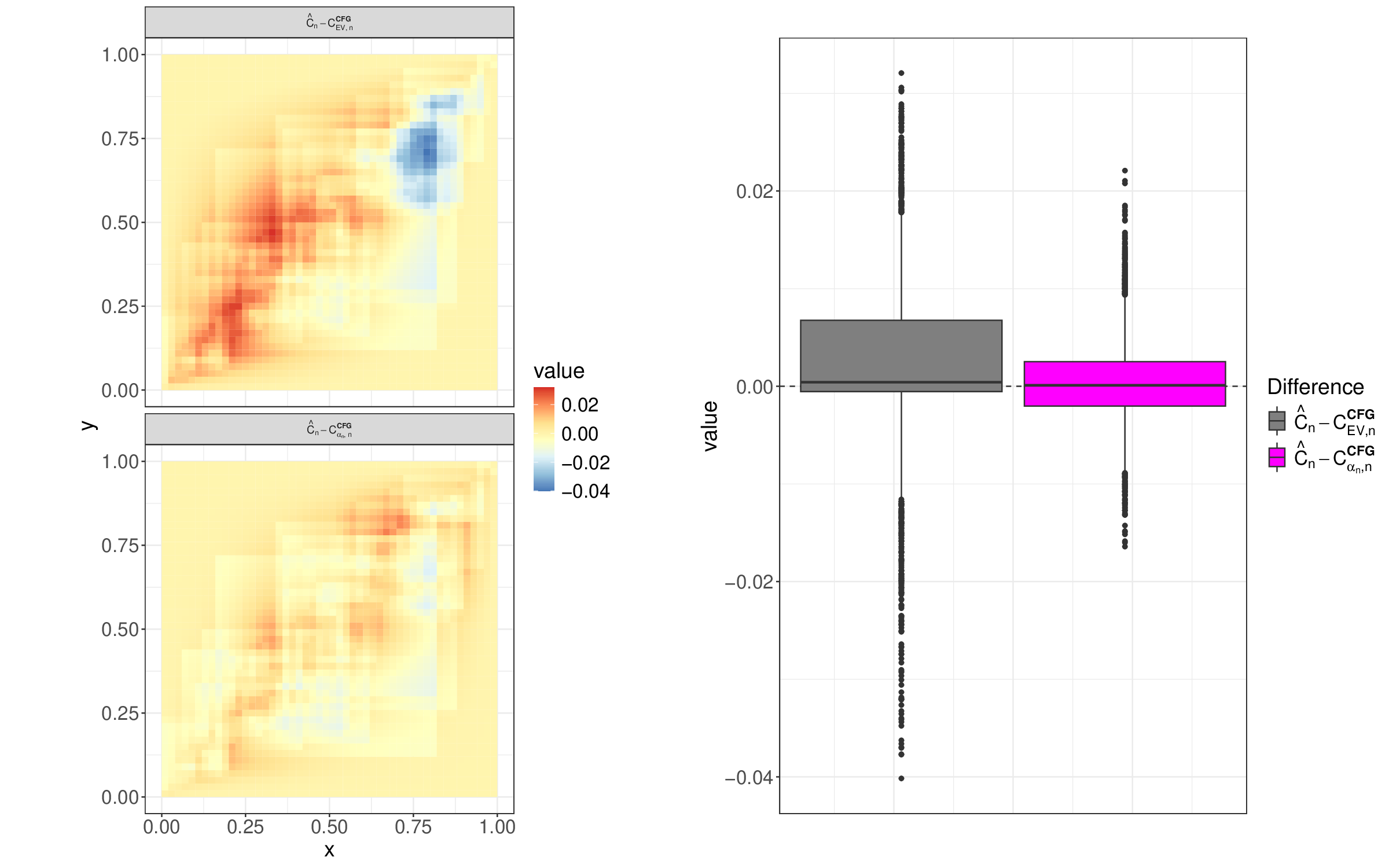}
	\caption{Heatmaps (left panel) and boxplots (right panel) of the differences $\hat{C}_n - C_{EV,n}^\mathbf{CFG}$ and $\hat{C}_n - C_{\alpha_n,n}^\mathbf{CFG}$, respectively, where $\hat{C}_n$ denotes the bilinear extension of the empirical copula estimator.}
	\label{real_data_empirical_vs_cfg}
\end{figure}
\subsection{Conditional distributions and estimating directed dependence}
In view of Corollary \ref{cor:continuity_dependence_measures} and Theorem~\ref{thrm:strong_consistency_dep_meas}, it may also be of interest to examine how the estimator $C_{\alpha_n,n}^{\mathbf{CFG}}$ performs as a plug-in estimator for dependence measures, in our case 
Chatterjee's $\xi$ and Trutschnig's $\zeta_1$ (as mentioned in Section 
\ref{subsection:general_notation}). Since Markov kernels play a central role in the definition of these measures, we first study the plug-in estimators of the kernels $K_{C_{\alpha_n,n}^{\mathbf{CFG}}}$, $K_{C_{EV,n}^{\mathbf{CFG}}}$ and $K_{\mathcal{CB}_N(E_n)}$, where $K_{C_{EV,n}^{\mathbf{CFG}}}$ denotes the Markov kernel associated with the copula $C_{EV,n}^{\mathbf{CFG}}$ and $K_{\mathcal{CB}_N(E_n)}$ denotes the Markov kernel corresponding to the empirical checkerboard copula $\mathcal{CB}_N(E_n)$ with resolution $N = \lfloor \sqrt{n} \rfloor$. We refer to \citep{dur_princ, JGT} for further details on checkerboard copulas and empirical checkerboard copulas. Figure~\ref{real_data_kernels_plots} compares the probability measures $K_{C_{\alpha_n,n}^{\mathbf{CFG}}}(x,\cdot)$, $K_{C_{EV,n}^{\mathbf{CFG}}}(x,\cdot)$, and $K_{\mathcal{CB}_N(E_n)}(x,\cdot)$ for fixed $x \in \{0.3, 0.5, 0.9\}$. It can be seen that $K_{C_{\alpha_n,n}^{\mathbf{CFG}}}(x,\cdot)$ is closer to the empirical checkerboard estimator $K_{\mathcal{CB}_N(E_n)}(x,\cdot)$ for $x = 0.9$, suggesting that $C_{\alpha_n,n}^{\mathbf{CFG}}$ may provide a better fit to the data at hand. Considering the estimators for the dependence measures $\zeta_1$ and $\xi$, we obtain the values $\zeta_1(C_{EV,n}^\mathbf{CFG}) \approx 0.676$, $\zeta_1(C_{\alpha_n,n}^\mathbf{CFG})\approx 0.694$ and $\zeta_1(\mathcal{CB}_N(E_n)) \approx 0.698$ as well as $\xi(C_{EV,n}^\mathbf{CFG}) \approx 0.485$, $\xi(C_{\alpha_n,n}^\mathbf{CFG}) \approx 0.504$ and $\hat{\xi}_n \approx 0.531$. Not surprisingly, the values of $\zeta_1(C_{\alpha_n,n}^\mathbf{CFG})$ and $\xi(C_{\alpha_n,n}^\mathbf{CFG})$ are closer to $\zeta_1(\mathcal{CB}_N(E_n))$ and $\hat{\xi}_n$, respectively. 
\begin{figure}[!ht]
	\centering
	\includegraphics[width=1\textwidth]{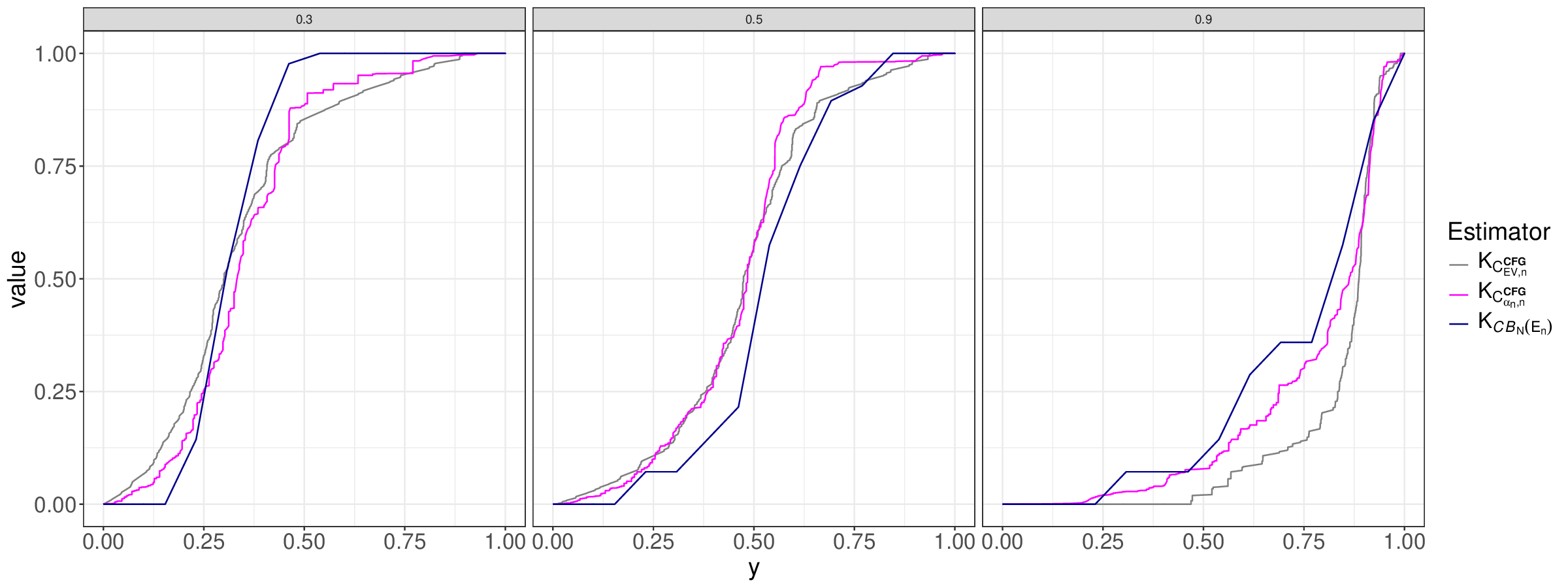}
	\caption{Plots of the distribution functions $y \mapsto  K_{C_{EV,n}^\mathbf{CFG}}(x, [0,y])$ (gray), $y \mapsto K_{C_{\alpha_n,n}^\mathbf{CFG}}(x, [0,y])$ (magenta) and $y \mapsto K_{\mathcal{CB}_N(E_n)}(x,[0,y])$ (darkblue) for $x = 0.3$ (left panel), $x = 0.5$ (middle panel) and $x = 0.9$ (right panel).}
	\label{real_data_kernels_plots}
\end{figure}
\vspace{1cm} \\
\textbf{Acknowledgement:}
Both authors gratefully acknowledge the support of the WISS 2025 project
‘IDA-lab Salzburg’ (20204-WISS/225/197-2019 and 20102-F1901166-KZP).

\clearpage
  \appendix
  \section{Auxiliary lemmas and proofs}\label{sec:auxil_lemmas}
In this short section, we present two auxiliary lemmas on properties of the generator $\psi \in \Psi$ and the Pickands dependence function $A \in \mathcal{A}$, which we use in Section~\ref{sec:conv_archimax}.
  \begin{lemma}\label{lem:positivity_generator}
      Every normalized, convex Archimedean generator $\psi \in \Psi$ fulfills $\psi(z) > 0$ for every $z \in [0,2)$. Moreover, $\psi$ is strictly decreasing on $[0,2)$ and the identity  $\varphi(\psi(z)) = z$ holds for every $z \in [0,2)$.
  \end{lemma}
  \begin{proof}
  For $\psi \in \Psi$ and $z \in [1,2)$, convexity yields 
  $$
  \frac{\psi(1) - \psi(0)}{1-0} \leq \frac{\psi(z) - \psi(0)}{z-0}.
  $$
  Hence, using $\psi(1) = \frac{1}{2}$ yields
  $$
  0<1- \frac{z}{2} \leq \psi(z)
  $$
  for every $z \in [1,2)$. Since $\psi$ is strictly decreasing on the interval $[0,\inf\{z \in [0,\infty) \colon \psi(z) = 0\})$, we obviously have that $\psi$ is strictly decreasing and hence invertible on $[0,2)$, with inverse $\varphi$. 
  This immediately implies that $\varphi(\psi(z)) = z$ for every $z \in [0,2)$.
  \end{proof}
As the following lemma is often used implicitly throughout this paper, we provide a formal proof for the sake of completeness.
  \begin{lemma}\label{lem:pickands_1/2}
      Let $A \in \mathcal{A}$. Then $A=A_M$ on $\mathbb{I}$, if and only if $A(\frac{1}{2}) = \frac{1}{2}$.
  \end{lemma}
  \begin{proof}
      $A=A_M$ obviously implies $A(\frac{1}{2}) = \frac{1}{2}$. 
      Suppose now that $A(\frac{1}{2}) = \frac{1}{2}$ holds. Then using 
      convexity of $A \in \mathcal{A}$ directly yields $A(t) \leq 1-t$ for every 
      $t \in [0,\frac{1}{2}]$ as well as $A(t) \geq t$ for every 
      $t \in [\frac{1}{2},1]$. Considering $A(t) \geq \max\{t,1-t\}$ directly yields
      $A=A_M$.
  \end{proof}
  Next, a proof for Lemma \ref{lem:kendalls_tau_trans} is provided.
  \begin{proof}[Proof of Lemma \ref{lem:kendalls_tau_trans}] \label{proof:lem_trans_kendall}
According to \cite[Proof of Theorem 5.1]{dietrich2026} we have that:
\begin{align*}
\tau_B &= 1-\int_\mathbb{I}\frac{B(s) -s D^+B(s)}{B(s)^2}G_B(s) \mathrm{d}\lambda(s) \\&=
-\int_\mathbb{I}\left[\frac{D^+B(s)(1-2s)}{B(s)} + \frac{s(s-1)D^+B(s)^2}{B(s)^2}\right]\mathrm{d}\lambda(s)
\end{align*}
for every $B \in \mathcal{A}$, where $G_B$ is defined as in eq. \eqref{eq:funct_G_A}. Plugging in $(A)_\alpha$ for $B$ and using the transformation $\kappa_\alpha$ as defined in eq. \eqref{eq:trans_kappa}, we have that
\begin{align*}
    \frac{(1-2s)D^+(A)_\alpha(s)}{(A)_\alpha(s)} = \underbrace{\frac{(1-2s)(s^{\alpha-1} - (1-s)^{\alpha-1})}{s^\alpha + (1-s)^\alpha}}_{=:I_\alpha(s)} + \underbrace{\frac{D^+A(\kappa_\alpha(s))s^{\alpha-1}(1-s)^{\alpha-1}(1-2s)}{A(\kappa_\alpha(s))(s^\alpha + (1-s)^\alpha)^2}}_{=:II_\alpha(s)}
\end{align*}
and
\begin{align*}
    \frac{s(s-1)D^+(A)_\alpha(s)^2}{(A)_\alpha(s)^2} &= \underbrace{\frac{(s^{\alpha-1} - (1-s)^{\alpha-1})^2s(s-1)}{(s^\alpha + (1-s)^\alpha)^2}}_{=:III_\alpha(s)} + \underbrace{2 \frac{D^+A(\kappa_\alpha(s))s^{\alpha-1}(1-s)^{\alpha-1}}{A(\kappa_\alpha(s))(s^\alpha + (1-s)^\alpha)^2}\cdot\frac{(s^{\alpha-1} - (1-s)^{\alpha-1})s(s-1)}{s^\alpha + (1-s)^\alpha}}_{=:IV_\alpha(s)}\\&\quad +\underbrace{\frac{D^+A\left(\kappa_\alpha(s)\right)^2s^{2\alpha-2}(1-s)^{2\alpha-2}s(s-1)}{A^2(\kappa_\alpha(s))(s^\alpha + (1-s)^\alpha)^4}}_{=:V_\alpha(s)}.
\end{align*}
We consider the transformation $\kappa_\frac{1}{\alpha}(t) = \frac{t^\frac{1}{\alpha}}{t^\frac{1}{\alpha} + (1-t)^\frac{1}{\alpha}}$ with derivative $\kappa_\frac{1}{\alpha}'(t) = \frac{t^{\frac{1}{\alpha}-1}(1-t)^{\frac{1}{\alpha}-1}}{\alpha((1-t)^\frac{1}{\alpha} + t^\frac{1}{\alpha})^2}$ and obtain that
\begin{align*}
    \int_\mathbb{I} V_\alpha(s) \mathrm{d}\lambda(s) = \int_\mathbb{I}  \frac{D^+A\left(\kappa_\alpha(s)\right)^2s^{2\alpha-2}(1-s)^{2\alpha-2}s(s-1)}{A^2(\kappa_\alpha(s))(s^\alpha + (1-s)^\alpha)^4} \mathrm{d}\lambda(s) &= \int_\mathbb{I}-\frac{D^+A(t)^2t^\frac{2\alpha-1}{\alpha}(1-t)^\frac{2\alpha-1}{\alpha}}{A^2(t)((1-t)^\frac{1}{\alpha} + t^\frac{1}{\alpha})^{-2}}\kappa_\frac{1}{\alpha}'(t) \mathrm{d}\lambda(t)\\&=
    \frac{1}{\alpha}\int_\mathbb{I}\frac{D^+A(t)^2t(t-1)}{A^2(t)} \mathrm{d}\lambda(t).
\end{align*}
Now we consider the term
\begin{align*}
II_\alpha(s) + IV_\alpha(s) = \frac{D^+A(\kappa_\alpha(s))}{A(\kappa_\alpha(s))}\left[\frac{s^{\alpha-1}(1-s)^{\alpha-1}(1-2s)}{(s^\alpha + (1-s)^\alpha)^2} + 2 \frac{s^{\alpha-1}(1-s)^{\alpha-1}}{(s^\alpha + (1-s)^\alpha)^2}\cdot\frac{(s^{\alpha-1} - (1-s)^{\alpha-1})s(s-1)}{s^\alpha + (1-s)^\alpha}\right].
\end{align*}
Then using similar arguments as before yields
\begin{align*}
    \int_\mathbb{I}[II_\alpha(s) + IV_\alpha(s)]  \mathrm{d}\lambda(t) &= 
    \int_\mathbb{I}\frac{D^+A(\kappa_\alpha(s))}{A(\kappa_\alpha(s))}\frac{s^{\alpha-1}(1-s)^{\alpha-1}[(1-s)^\alpha-s^\alpha]}{(s^\alpha + (1-s)^\alpha)^3}\mathrm{d}\lambda(s) \\&=
    \int_\mathbb{I}\left[\frac{D^+A(t)}{A(t)}(1-2t)t^{1-\frac{1}{\alpha}}(1-t)^{1-\frac{1}{\alpha}}(t^\frac{1}{\alpha} + (1-t)^\frac{1}{\alpha})^2\kappa_\frac{1}{\alpha}'(t)\right]\mathrm{d}\lambda(t) \\&=
    \frac{1}{\alpha}\int_\mathbb{I}\frac{D^+A(t)(1-2t)}{A(t)}\mathrm{d}\lambda(t)
\end{align*}
Considering the remaining part we have that
$$
I_\alpha(s) + III_\alpha(s) = \frac{(1-2s)(s^{\alpha-1} - (1-s)^{\alpha-1})}{s^\alpha + (1-s)^\alpha} + \frac{(s^{\alpha-1} - (1-s)^{\alpha-1})^2s(s-1)}{(s^\alpha + (1-s)^\alpha)^2}
$$
Substituting $s = \frac{t^\frac{1}{\alpha}}{t^\frac{1}{\alpha} + (1-t)^\frac{1}{\alpha}} = \kappa_{\frac{1}{\alpha}}(t)$ again we obtain that
$$
I_\alpha(s) = ((1-t)^\frac{1}{\alpha} -t^\frac{1}{\alpha})(t^{1-\frac{1}{\alpha}} - (1-t)^{1-\frac{1}{\alpha}})
$$
and
$$
III_\alpha(s) = -(t^{1-\frac{1}{\alpha}} - (1-t)^{1-\frac{1}{\alpha}})^2t^\frac{1}{\alpha}(1-t)^\frac{1}{\alpha}.
$$
Simplifying, the following identity holds:
$$
I_\alpha(s) + III_\alpha(s) = (t^{1-\frac{1}{\alpha}} - (1-t)^{1-\frac{1}{\alpha}})((1-t)^{1+\frac{1}{\alpha}}-t^{1+\frac{1}{\alpha}}).
$$
Applying change of coordinates and integration by parts we have
\begin{align*}
    \alpha\int_\mathbb{I}[I_\alpha(s) + III_\alpha(s)] \mathrm{d}\lambda(s) &= \int_\mathbb{I} \frac{(t^{1-\frac{1}{\alpha}} - (1-t)^{1-\frac{1}{\alpha}})((1-t)^{1+\frac{1}{\alpha}}-t^{1+\frac{1}{\alpha}})t^{\frac{1}{\alpha}-1}(1-t)^{\frac{1}{\alpha}-1}}{((1-t)^\frac{1}{\alpha} + t^\frac{1}{\alpha})^2}\mathrm{d}\lambda(t) \\&= \int_\mathbb{I} \frac{((1-t)^{\frac{1}{\alpha}-1} - t^{\frac{1}{\alpha}-1})((1-t)^{1+\frac{1}{\alpha}}-t^{1+\frac{1}{\alpha}})}{((1-t)^\frac{1}{\alpha} + t^\frac{1}{\alpha})^2}\mathrm{d}\lambda(t) \\&= \frac{\alpha}{(1-t)^\frac{1}{\alpha} + t^\frac{1}{\alpha}}((1-t)^{\frac{1}{\alpha}+1} - t^{\frac{1}{\alpha}+1})\bigg|_{t = 0}^1 + (1+\alpha)\int_\mathbb{I}\frac{(1-t)^\frac{1}{\alpha} + t^\frac{1}{\alpha}}{(1-t)^\frac{1}{\alpha} + t^\frac{1}{\alpha}} \mathrm{d}\lambda(t)\\&=
    -2\alpha +1+\alpha = 1-\alpha.
\end{align*}
Putting all of the above together we finally obtain that
$$
    \tau_{(A)_\alpha} = 1-\frac{1-\tau_A}{\alpha}.
$$
The identity concerning $\tau_{(\psi)_\alpha}$ is straightforward to verify.
\end{proof}
\section{Strong consistency of the copula estimator}\label{sec:proof_consistency}
\noindent Throughout this section we assume that the properties stated at the beginning of Section~\ref{sec:estimator} hold. Let $F_n$ and $G_n$ denote the empirical marginal distribution functions based on the first $n$ elements of the sample $(X_1,Y_1),(X_2,Y_2),\ldots$ from $H$. Furthermore, we consider the (slightly modified) pseudo-observations $(\hat{U}_i,\hat{V}_i)$, given by
$$
\hat{U}_i = \frac{n}{n+1}F_n(X_i), \qquad \hat{V}_i = \frac{n}{n+1}G_n(Y_i),
$$
and define $\hat{C}_n: \mathbb{I}^2 \rightarrow [0,1]$ by
\begin{equation}\label{eq:empirical_copula_est}
	\hat{C}_n(x,y) = \frac{1}{n}\sum_{i=1}^n \mathbf{1}_{[0,x] \times [0,y]}(\hat{U}_i,\hat{V}_i)
\end{equation}
for every $x,y \in \mathbb{I}$. Notice that $\hat{C}_n$ is not a
copula, the function is not even continuous. It is straightforward to verify, 
however, that $\hat{C}_n$ exhibits the same limiting behaviour as the empirical copula $C_n$ considered in Section \ref{subsection:general_notation}, i.e., with probability $1$ we have that $(\hat{C}_n)_{n \in \mathbb{N}}$ converges uniformly to $C$.
The following elementary property of $\hat{C}_n$ will be useful in the sequel:
\begin{lemma}\label{lem:upper.lower.cn.hat}
    The function $\hat{C}_n$, defined according to eq. \eqref{eq:empirical_copula_est},
    has the following properties: 
    For all $x,y\geq \frac{n}{n+1}$ we have $\hat{C}_n(x,y)=1$. 
    Moreover, for all $x,y \in [0,1)$  
    \begin{align}\label{eq:estimates_copulas}
    W\left(\tfrac{\lfloor (n+1) x\rfloor}{n},\tfrac{\lfloor (n+1) y\rfloor}{n}\right) \leq \hat{C}_n(x,y) &\leq M\left(\tfrac{\lfloor (n+1) x\rfloor}{n},\tfrac{\lfloor (n+1) y\rfloor}{n}\right)
    \end{align}
    holds. 
\end{lemma}
\begin{proof}
    Since the first assertion obviously holds, it suffices to prove the second one, 
    which can be done as follows: 
    Interpreting $\hat{C}_n$ as a bivariate distribution function, applying Sklar's 
    theorem, and using the fact that $M$ and $W$ is the upper and lower bound in $\mathcal{C}$, respectively, directly yields 
    \begin{align*}
    M \left(\frac{1}{n} \sum_{i=1}^n \mathbf{1}_{[0,x]}(\hat{U}_i), \frac{1}{n} \sum_{i=1}^n \mathbf{1}_{[0,x]}(\hat{V}_i)\right)
        \geq \hat{C}_n(x,y) \geq W \left(\frac{1}{n} \sum_{i=1}^n \mathbf{1}_{[0,x]}(\hat{U}_i), \frac{1}{n} \sum_{i=1}^n \mathbf{1}_{[0,x]}(\hat{V}_i)\right).
    \end{align*}
    Considering that for $x,y \in [0,1)$ we have 
    $$
    \lfloor (n+1) x\rfloor = \sum_{i=1}^n \mathbf{1}_{[0,x]}(\hat{U}_i), \quad 
    \lfloor (n+1) y\rfloor = \sum_{i=1}^n \mathbf{1}_{[0,y]}(\hat{V}_i),
    $$
    it therefore follows that 
    $$
    W\left(\tfrac{\lfloor (n+1) x\rfloor}{n},\tfrac{\lfloor (n+1) y\rfloor}{n}\right) \leq \hat{C}_n(x,y) \leq  M\left(\tfrac{\lfloor (n+1) x\rfloor}{n},\tfrac{\lfloor (n+1) y\rfloor}{n}\right),
    $$
    which completes the proof.
\end{proof}
\begin{remark}
Viewing copulas as distribution functions on $\mathbb{I}^2$ and 
extending them in the standard way to real distribution functions on $\mathbb{R}^2$, 
eq. \eqref{eq:estimates_copulas} holds for all $x,y \in \mathbb{I}$. 
In the sequel we will use this extension whenever the arguments are outside
$\mathbb{I}$.
\end{remark}
\subsection{Strong consistency of the estimators $\psi_n$ and $\varphi_n$}
We note that throughout this section, most of the statements in the theorems,  lemmas and their proofs, hold with probability $1$. For the sake of readability, 
we do not explicitly state this repetitively.
\noindent We start with the following theorem on strong consistency of our estimators:
\begin{theorem}\label{thm:consistency_psi_n}
 Let $\varphi_n$ be defined according to eq. \eqref{eq:phi.via.kendall}, let 
 $\psi_n \in \Psi$ denote the generator with pseudo-inverse $\varphi_n$ and let 
 $\beta_n$ be defined as in eq. \eqref{eq:beta_n}. 
 Then the following three assertions hold with probability $1$:
 \begin{enumerate}[label=(\roman*)]
          \item $\psi_n \overset{n\rightarrow \infty}{\longrightarrow} (\psi)_{1-\tau_A}$ uniformly on $[0,\infty)$.
          \item $\varphi_n(x) \overset{n\rightarrow \infty}{\longrightarrow} (\varphi)_{1-\tau_A}(x)$ pointwise for every $x \in (0,1]$.
        \item $\beta_n(t) \overset{n\rightarrow \infty}{\longrightarrow} (\beta)_{1-\tau_A}(t)=(1-\tau_A)\beta(t)$ for every $t \in \mathbb{I}$ at which $\beta$ is continuous. 
      \end{enumerate}
\end{theorem}
\begin{proof}
The first assertion has already been shown in Section \ref{subsec:6.2}, the second one is a direct consequence of the equivalences in \cite[Theorem 4.1]{bernoulli}.
To prove the third statement recall that we have 
$K_n(t) = t - \beta_n(t)$ as well as 
$$
F^K_{\psi,A}(t)=F^K_{C_{(\psi)_{1-\tau_A}}}(t)= t-(\beta)_{1-\tau_A}(t).
$$ 
for every $t \in \mathbb{I}$. Having this, again applying \cite[Theorem 4.1]{bernoulli},
concludes the proof.
\end{proof}
\subsection{Strong consistency of the Pickands type and the CFG type estimators $B_{n,c}^\mathbf{P}$ and $B_{n,c}^\mathbf{CFG}$}\label{subsec:strong_consistency_pick}
 \noindent Let $Z$ be the random variable whose survival function is 
 $(\psi)_{1-\tau_A}$. Using convexity of $(\psi)_{1-\tau_A}$, we obviously have that 
 $Z$ is absolutely continuous and that 
 $-D^-(\psi)_{1-\tau_A}$ is (a version of) the density of $\mathbb{P}^Z$. 
 Similarly to \cite{GNZ}, throughout this subsection we assume that 
 $\mathbb{E}[Z] < \infty$ and $\mathbb{E}[\vert \log Z \vert] < \infty$.
The following identity expressing $\mathbb{E}[Z]$ and $\mathbb{E}[\log Z]$ 
in terms of $(\psi)_{1-\tau_A}$ and $(\varphi)_{1-\tau_A}$ will be useful in the sequel.
Using Fubini's theorem (Cavalieri's principle) directly yields
\begin{align}\label{eq:int.psi.phi.gleich}
 \mathbb{E}[Z] &= \int_{[0,\infty)} \mathbb{P}[Z>z] \, \mathrm{d}\lambda(z) =  
 \int_{[0,\infty)} (\psi)_{1-\tau_A}(z) \, \mathrm{d}\lambda(z)=
 \int_{\mathbb{I}} \lambda\left(\{t \in [0,\infty): (\psi)_{1-\tau_A}(t) \geq z \} \right) 
 \, \mathrm{d}\lambda(z) \nonumber \\
 &= \int_{\mathbb{I}} \lambda\left([0,(\varphi)_{1-\tau_A}(z)] \right) 
 \, \mathrm{d}\lambda(z) = \int_{\mathbb{I}}  (\varphi)_{1-\tau_A}(z) \, \mathrm{d}\lambda(z)
\end{align}
and
\begin{align}\label{eq:int.log.psi.log.phi.gleich}
 \mathbb{E}[\log Z] &= \int_{[0,\infty)} \mathbb{P}[\log Z>z] \, \mathrm{d}\lambda(z) - \int_{(-\infty,0]} \mathbb{P}[\log Z\leq z] \, \mathrm{d}\lambda(z) \nonumber \\&=  
 \int_{[0,\infty)} (\psi)_{1-\tau_A}(\exp(z)) \, \mathrm{d}\lambda(z) - \int_{(-\infty,0]} [1-(\psi)_{1-\tau_A}(\exp(z))] \, \mathrm{d}\lambda(z) \nonumber \\&=
 \int_{[0,\frac{1}{2}]} \lambda\left(\{t \in [0,\infty): (\psi)_{1-\tau_A}(\exp(t)) \geq z \} \right) 
 \, \mathrm{d}\lambda(z) \nonumber \\&\quad-  \int_{[\frac{1}{2},1]} \lambda\left(\{t \in (-\infty,0]: 1-(\psi)_{1-\tau_A}(\exp(t)) \geq z \} \right) 
 \, \mathrm{d}\lambda(z)
 \nonumber \\
 &= \int_{[0,\frac{1}{2}]} \lambda\left([0,\log((\varphi)_{1-\tau_A}(z))] \right) 
 \, \mathrm{d}\lambda(z) - \int_{[0,\frac{1}{2}]} \lambda\left([\log((\varphi)_{1-\tau_A}(1-z)),0] \right) 
 \, \mathrm{d}\lambda(z) \nonumber\\&= \int_{\mathbb{I}}  \log((\varphi)_{1-\tau_A}(z)) \, \mathrm{d}\lambda(z),
\end{align}
where we used change of coordinates in the last equality. 
Alternatively $\mathbb{E}[\log{Z}]$ allows the following presentation in terms of
$(\psi)_{1-\tau_A}$:
\begin{align*}
    \mathbb{E}[\log{Z}] &= \int_{(-\infty,0)} (1-(\psi)_{1-\tau_A}(\exp(x))) \mathrm{d}\lambda(x) - 
    \int_{(0,\infty)} (\psi)_{1-\tau_A}(\exp(x)) \mathrm{d}\lambda(x) \\
    &= \int_{\mathbb{I}} \frac{1-(\psi)_{1-\tau_A}(u)}{u} \mathrm{d}\lambda(u) - 
       \int_{[1,\infty)} \frac{(\psi)_{1-\tau_A}(u)}{u} \mathrm{d}\lambda(u).
\end{align*}
Furthermore, since every $\varphi_n$ is non-strict and piece-wise linear we obviously have that
\begin{align*}
\int_{\mathbb{I}}\varphi_n(t) \mathrm{d}\lambda(t)  < \infty \,\,\, \text{and} \,\,\, \int_{\mathbb{I}}|\log(\varphi_n(t))| \mathrm{d}\lambda(t) < \infty
\end{align*}
and therefore $\Vert \varphi_n \Vert_1 < \infty$ and $\Vert \log \circ \varphi_n \Vert_1 < \infty$ with $\Vert \cdot \Vert_1$ denoting the $L_1$-norm with respect to $\lambda$ on 
$\mathbb{I}$. Applying Assumption \ref{assumption1_pick}, the next lemma shows that $(\varphi)_{n\in \mathbb{N}}$ converges to $(\varphi)_{1-\tau_A}$ almost surely in the $L^1$-norm.
\noindent Assuming \ref{assumption1_pick}, \ref{assumption1_cfg} and \ref{assumption2_cfg}, the estimators for $\mathbb{E}[Z]$ and $\mathbb{E}[\log Z]$ are strongly consistent - the following lemma holds:
\begin{lemma}\label{lem:est_eval}
Let $Z$ be the random variable with survival function $(\psi)_{1-\tau_A}$, 
and consider $f \in \{\mathrm{id}, \log\}$. If 
$f = \mathrm{id}$, we assume that $\mathbb{E}[Z] < \infty$ and 
Assumption~\ref{assumption1_pick} holds; if $f = \log$, we assume 
that $\mathbb{E}[|\log Z|] < \infty$ and Assumptions~\ref{assumption1_cfg} and~\ref{assumption2_cfg} hold.
Furthermore, let $\varphi_n$ be defined as in eq. \eqref{eq:phi.via.kendall}. 
Then, in either case, the following convergence holds with probability $1$:
\begin{equation}
\label{eq:your_label}
    \frac{1}{n}\sum_{i=1}^n f\!\left(\varphi_n\!\left(\frac{i}{n+1}\right)\right) 
    \xrightarrow{\;n \to \infty\;} \mathbb{E}[f(Z)].
\end{equation}
\end{lemma}
\begin{proof}
Let $f \in \{\mathrm{id},\log\}$ and let $\theta_n$ denote the uniform distribution on the set $\{\frac{1}{n+1},\ldots,\frac{n}{n+1}\}$. Then obviously $(\theta_n)_{n \in \mathbb{N}}$ converges weakly to $\lambda$ on 
$\mathbb{I}$. Using the triangle inequality we have 
\begin{align*}
    \left\vert \int_\mathbb{I} f(\varphi_n(t)) \mathrm{d}\theta_n(t) - 
    \int_\mathbb{I} f((\varphi)_{1-\tau_A}(t)) \mathrm{d}\lambda(t)\right\vert 
    &\leq \underbrace{\left\vert \int_\mathbb{I} f(\varphi_n(t)) \mathrm{d}\theta_n(t) - 
    \int_\mathbb{I} f((\varphi)_{1-\tau_A}(t)) \mathrm{d}\theta_n(t) \right\vert}_{=I_n}  \\&\quad + 
    \underbrace{\left\vert \int_\mathbb{I} f((\varphi)_{1-\tau_A}(t)) \mathrm{d}\theta_n(t) - 
    \int_\mathbb{I} f((\varphi)_{1-\tau_A}(t)) \mathrm{d}\lambda(t) \right\vert}_{II_n}.
\end{align*}
Let $\varepsilon>0$ be arbitrary but fixed. 
We first focus on the first summand $I_n$ and distinguish the cases $f = \mathrm{id}$ and $f = \log$. By Assumption \ref{assumption1_pick} or Assumption \ref{assumption1_cfg} we can find some $\delta \in (0,\frac{1}{2})$ such that 
$$
\sup_{n \in \mathbb{N}}\int_{[0,\delta]}f(\varphi_n(t))\mathrm{d}\lambda(t)  + 
\int_{[0,\delta]}f((\varphi)_{1-\tau_A}(t))\mathrm{d}\lambda(t)< \frac{\varepsilon}{4}
$$
in the case that $f = \mathrm{id}$, and such that
\begin{align*}
&\sup_{n \in \mathbb{N}}\int_{[0,\delta]}f(\varphi_n(t))\mathrm{d}\lambda(t)  + 
\int_{[0,\delta]}f((\varphi)_{1-\tau_A}(t))\mathrm{d}\lambda(t) + \quad\sup_{n \in \mathbb{N}}\int_{[1-\delta,1]}|f(\varphi_n(t))|\mathrm{d}\lambda(t)  + 
\int_{[1-\delta,1]}|f((\varphi)_{1-\tau_A}(t))|\mathrm{d}\lambda(t)< \frac{\varepsilon}{4}
\end{align*}
for $f = \log$. Notice that for $\frac{i}{n+1} \leq \delta$ we have 
$\log(\varphi_n(t)) \geq \log(\varphi_n(\frac{i}{n+1}))\geq \log(\varphi_n(\delta)) > 0$ for every $t \in [\frac{i-1}{n+1},\frac{i}{n+1}]$; moreover, for 
$\frac{i}{n+1} \geq 1-\delta$ it follows that $\vert \log(\varphi_n(t)) \vert 
\geq \vert \log(\varphi_n(\frac{i-1}{n})) \vert $ holds.  
Hence, using monotonicity of pseudo-inverses and the logarithm, 
and applying the triangle inequality several times, for $f = \log$ the summand $I_n$ can be bounded from above by
\begin{align*}
    I_n &\leq \int_\mathbb{I} \vert f(\varphi_n(t)) - f((\varphi)_{1-\tau_A}(t))\vert \mathrm{d}\theta_n(t)
    = \frac{1}{n} \sum_{i=1}^n \left\vert f\left(\varphi_n\left(\frac{i}{n+1}\right)\right) 
    - f\left((\varphi)_{1-\tau_A}\left(\frac{i}{n+1}\right)\right)  \right\vert \\
    &= \frac{1}{n} \sum_{i: \frac{i}{n+1} \leq \delta} \left\vert f\left(\varphi_n\left(\frac{i}{n+1}\right)\right) - f\left((\varphi)_{1-\tau_A}\left(\frac{i}{n+1}\right)\right)  \right\vert + 
    \frac{1}{n} \sum_{i:\delta < \frac{i}{n+1} < 1-\delta} \left\vert f\left(\varphi_n\left(\frac{i}{n+1}\right)\right) 
    - f\left((\varphi)_{1-\tau_A}\left(\frac{i}{n+1}\right)\right)  \right\vert \\&\quad+
    \frac{1}{n} \sum_{i: \frac{i}{n+1} \geq 1-\delta}\left\vert f\left(\varphi_n\left(\frac{i}{n+1}\right)\right) 
    - f\left((\varphi)_{1-\tau_A}\left(\frac{i}{n+1}\right)\right)  \right\vert \\
    &\leq \frac{n+1}{n} \frac{1}{n+1}  \sum_{i: \frac{i}{n+1} \leq \delta}  f\left(\varphi_n\left(\frac{i}{n+1}\right)\right) + \frac{n+1}{n} \frac{1}{n+1} \sum_{i: \frac{i}{n+1} \leq \delta}  f\left((\varphi)_{1-\tau_A}\left(\frac{i}{n+1}\right)\right)\, \\& \quad+
    \frac{n+1}{n} \frac{1}{n+1}  \sum_{i: \frac{i}{n+1} \geq 1-\delta}  \left|f\left(\varphi_n\left(\frac{i}{n+1}\right)\right)\right| + \frac{n+1}{n} \frac{1}{n+1} \sum_{i: \frac{i}{n+1} \geq 1-\delta}  \left|f\left((\varphi)_{1-\tau_A}\left(\frac{i}{n+1}\right)\right)\right|\,\\&\quad
    + \, \sup_{t \in [\delta,1-\delta]}|f(\varphi_n(t)) - f((\varphi)_{1-\tau_A}(t))| \\
    & \leq \frac{n+1}{n} \frac{\varepsilon}{4} + \, \sup_{t \in [\delta,1-\delta]}|f(\varphi_n(t)) - f((\varphi)_{1-\tau_A}(t))|.
\end{align*}
Using similar arguments, $I_n$ can be bounded from above in the same way for the 
case $f = \mathrm{id}$.
Considering that $(\varphi_n)_{n \in \mathbb{N}}$ converges uniformly to 
$(\varphi)_{1-\tau_A}$ on every compact subinterval $[a,b]$ of $(0,1]$, this shows the existence of some index 
$n_1 \in \mathbb{N}$ such that 
$I_n\leq \frac{\varepsilon}{2}$ for all $n \geq n_1$.\\
\comment{
For the second summand notice that by convexity of $(\varphi)_{1-\tau_A}$ 
we have 
$$
\int_{[0,\delta]} (\varphi)_{1-\tau_A} \mathrm{d}\theta_n \leq \tfrac{n+1}{n} \int_{[0,\delta]} (\varphi)_{1-\tau_A} \mathrm{d}\theta_n 
\leq \int_{[0,\delta]} (\varphi)_{1-\tau_A} \mathrm{d}\lambda < \tfrac{\varepsilon}{4},
$$
}
Using a similar argument as for $I_n$ or weak convergence of 
$(\theta_n)_{n \in \mathbb{N}}$ to $\lambda$ we can find some index $n_2 \geq n_1$ 
with $II_n \leq \frac{\varepsilon}{2}$ for all $n \geq n_2$. 
Altogether this yields $I_n + II_n < \varepsilon$ for all $n \geq n_2$, which completes the 
proof since $\varepsilon>0$ was arbitrary.
\end{proof}
\noindent The proof of the subsequent lemma is analogous to \citep{est-archimax, genest2009}:
\begin{lemma}\label{lem:repres_empirical_copula} Let $\xi$ and $\xi_{i,n}$ be defined according to eqs. \eqref{eq:population_xi} and \eqref{eq:xis}, respectively, $\hat{C}_n$ according to eq. \eqref{eq:empirical_copula_est} and let $f \in \{\mathrm{id}, \log\}$.
Furthermore write $C := C_{\psi,A}$. 
Then the following identity holds for every $t \in \mathbb{I}$:
    \begin{equation*}
     \frac{1}{n}\sum_{i=1}^n(f \circ \xi_{i,n})(t) -\mathbb{E}[(f\circ\xi)(t)]= \int_{[0,\infty)}[\hat{C}_n(\psi_n(tx),\psi_n((1-t)x)) - C((\psi)_{1-\tau_A}(tx),(\psi)_{1-\tau_A}((1-t)x))]f'(x)\mathrm{d}\lambda(x).
    \end{equation*}
\end{lemma}
\begin{proof}
For fixed $t\in \mathbb{I}$ and every $x > 0$ we have
$$
\mathbb{P}[\xi(t) > x] = (\psi)_{1-\tau_A}\left(x\,(A)_{1-\tau_A}(t)\right) 
= C\left((\psi)_{1-\tau_A}(tx),(\psi)_{1-\tau_A}((1-t)x)\right)
$$
as well as
$$
\frac{1}{n}\sum_{i=1}^n\mathbf{1}_{[0,\xi_{i,n}(t)]}(x) = \frac{1}{n}\sum_{i=1}^n \mathbf{1}_{[0,\psi_n(tx)] \times [0,\psi_n((1-t)x)]}(\hat{U}_i,\hat{V}_i) = \hat{C}_n\left(\psi_n(tx),\psi_n((1-t)x)\right).
$$
Having this and considering the case $f = \mathrm{id}$, integration over $x$ directly yields
    \begin{align*}
           \frac{1}{n}\sum_{i=1}^n\xi_{i,n}(t) -\mathbb{E}[\xi(t)] &= \int_{[0,\infty)}\left[\frac{1}{n}\sum_{i=1}^n\mathbf{1}_{[0,\xi_{i,n}(t)]}(x) - \mathbb{P}[\xi(t) > x]\right]\mathrm{d}\lambda(x) \\& = 
            \int_{[0,\infty)}[\hat{C}_n(\psi_n(tx),\psi_n((1-t)x)) - C((\psi)_{1-\tau_A}(tx),(\psi)_{1-\tau_A}((1-t)x))]f'(x)\mathrm{d}\lambda(x).
    \end{align*}
    Considering the case $f = \log$ and proceeding as in \citep{est-archimax}, it is easy to see that $\log(t) = \int_{[0,\infty)}\frac{\mathbf{1}_{[0,t]}(x) - \mathbf{1}_{\mathbb{I}}(x)}{x}\mathrm{d}\lambda(x)$ for every $t > 0$.
    Proceeding as before and using Fubini's theorem, we obtain
\begin{align*}
        \frac{1}{n}\sum_{i=1}^n\log\xi_{i,n}(t) -\mathbb{E}[\log\xi(t)] &= \frac{1}{n}\sum_{i = 1}^n\int_{[0,\infty)} \frac{\mathbf{1}_{[0,\xi_{i,n}(t)]}(x) - \mathbf{1}_\mathbb{I}(x)}{x}\mathrm{d}\lambda(x) - \mathbb{E}\left[\int_{[0,\infty)}\frac{\mathbf{1}_{[0,\xi(t)]}(x) - \mathbf{1}_\mathbb{I}(x)}{x}\mathrm{d}\lambda(x)\right] \\&=
        \frac{1}{n}\sum_{i = 1}^n\int_{[0,\infty)} \frac{\mathbf{1}_{[0,\xi_{i,n}(t)]}(x) - \mathbf{1}_\mathbb{I}(x)}{x}\mathrm{d}\lambda(x) - \int_{[0,\infty)}\left[\int_\Omega\frac{\mathbf{1}_{[0,\xi(t)]}(x) - \mathbf{1}_\mathbb{I}(x)}{x}\mathrm{d}\mathbb{P}\right]\mathrm{d}\lambda(x) \\& =
          \frac{1}{n}\sum_{i = 1}^n\int_{[0,\infty)} \frac{\mathbf{1}_{[0,\xi_{i,n}(t)]}(x) - \mathbf{1}_\mathbb{I}(x)}{x}\mathrm{d}\lambda(x) - \int_{[0,\infty)}\frac{\mathbb{P}[x \leq \xi(t)] - \mathbf{1}_\mathbb{I}(x)}{x}\mathrm{d}\lambda(x) \\& =
        \int_{[0,\infty)} \frac{\frac{1}{n}\sum_{i = 1}^n\mathbf{1}_{[0,\xi_{i,n}(t)]}(x) - \mathbb{P}[x \leq \xi(t)]}{x}\mathrm{d}\lambda(x) \\& =
                \int_{[0,\infty)} [\hat{C}_n(\psi_n(tx),\psi_n((1-t)x)) -C(\psi_{1-\tau_A}(tx),\psi_{1-\tau_A}((1-t)x))] f'(x)\mathrm{d}\lambda(x).
\end{align*}
This proves the result.
\end{proof}
\noindent As next step, we use the expression in Lemma \ref{lem:repres_empirical_copula} to prove uniform convergence over $\mathbb{I}$.
\begin{lemma}\label{lem:convergence_xi}
Let $\xi$ and $\xi_{i,n}$ be defined according to eqs. \eqref{eq:population_xi} and \eqref{eq:xis}, respectively. Furthermore, consider $f \in \{\mathrm{id},\log\}$.  
If $f = \mathrm{id}$, we assume that $\mathbb{E}[Z] < \infty$ and that 
Assumption~\ref{assumption1_pick} holds; if $f = \log$, we assume 
that $\mathbb{E}[|\log Z|] < \infty$ and that Assumptions~\ref{assumption1_cfg} and~\ref{assumption2_cfg} hold.
Then in each of the two cases the following property holds with probability $1$:
$$
\sup_{t \in \mathbb{I}} \left\vert\frac{1}{n}\sum_{i=1}^n (f \circ \xi_{i,n})(t) - \mathbb{E}[(f\circ\xi)(t)]\right\vert \overset{n \rightarrow \infty}{\longrightarrow} 0.
$$
\end{lemma}
\begin{proof}
    Applying Lemma \ref{lem:repres_empirical_copula} for $f \in \{\mathrm{id},\log\}$, 
    using the triangle inequality we have that
    \begin{align*}
    \sup_{t \in \mathbb{I}}\bigg|\frac{1}{n}\sum_{i=1}^n(f \circ \xi_{i,n})(t) &-\mathbb{E}[(f \circ\xi)(t)]\bigg|\\
    &= \sup_{t \in \mathbb{I}}\left|\int_{[0,\infty)}[\hat{C}_n(\psi_n(tx),\psi_n((1-t)x)) - C(\psi_{1-\tau_A}(tx),\psi_{1-\tau_A}((1-t)x))]f'(x)\mathrm{d}\lambda(x)\right| \\
    &\leq
    \underbrace{\sup_{t \in \mathbb{I}}\bigg|\int_{[0,\infty)}[\hat{C}_n(\psi_n(tx),\psi_n((1-t)x)) - C(\psi_n(tx),\psi_n((1-t)x))]f'(x)\mathrm{d}\lambda(x)\bigg|}_{=:I_n} \\
    &\quad +
    \underbrace{\sup_{t \in \mathbb{I}}\left|\int_{[0,\infty)}[C(\psi_n(tx),\psi_n((1-t)x)) - C(\psi_{1-\tau_A}(tx),\psi_{1-\tau_A}((1-t)x))]f'(x)\mathrm{d}\lambda(x)\right|}_{=:II_n}.
    \end{align*}
    \textbf{(i)} We first consider the case $f = \mathrm{id}$. Let $m > 0$ be arbitrary but fixed. Then for $I_n$ we obviously have  
    \begin{align*}
        I_n &\leq \sup_{t \in \mathbb{I}}\int_{[0,m]} |\hat{C}_n(\psi_n(tx),\psi_n((1-t)x)) - C(\psi_n(tx),\psi_n((1-t)x))| \mathrm{d}\lambda(x) \\&\quad+ \sup_{t \in \mathbb{I}}\int_{[m,\infty)} |\hat{C}_n(\psi_n(tx),\psi_n((1-t)x)) - C(\psi_n(tx),\psi_n((1-t)x))| \mathrm{d}\lambda(x).
    \end{align*}
    Thereby the first integral converges to $0$ since 
    \begin{align*}
    \sup_{t \in \mathbb{I}}\int_{[0,m]} |\hat{C}_n(\psi_n(tx),\psi_n((1-t)x)) - C(\psi_n(tx),\psi_n((1-t)x))| \mathrm{d}\lambda(x) \leq
    m \,\Vert\hat{C}_n - C\Vert_\infty \overset{n\rightarrow \infty}{\longrightarrow}0,
    \end{align*}
    and for the second summand we proceed as follows: Considering
    \begin{align}\label{eq:estimates_copulas2}
    \nonumber W\left(\tfrac{\lfloor (n+1) \psi_n(tx)\rfloor}{n},\tfrac{\lfloor (n+1) \psi_n((1-t)x)\rfloor}{n}\right) \leq \hat{C}_n(\psi_n(tx),\psi_n((1-t)x)) &\leq M\left(\tfrac{\lfloor (n+1) \psi_n(tx)\rfloor}{n},\tfrac{\lfloor (n+1) \psi_n((1-t)x)\rfloor}{n}\right)\\
    W(\psi_n(tx),\psi_n((1-t)x)) \leq C(\psi_n(tx),\psi_n((1-t)x)) &\leq M(\psi_n(tx),\psi_n((1-t)x))
    \end{align}
   and using Lemma \ref{lem:implication_regularity}, setting 
    $$
    a_{n,m}:=\sup_{t \in \mathbb{I}}\int_{[m,\infty)} \vert \hat{C}_n(\psi_n(tx),\psi_n((1-t)x)) - C(\psi_n(tx),\psi_n((1-t)x))\vert \mathrm{d}\lambda(x)
    $$
    yields
    \begin{align}\label{eq:estimate_xis}
    \nonumber \inf_{m > 0}\sup_{n \in \mathbb{N}}a_{n,m}&\leq  \inf_{m > 0}\sup_{n \in \mathbb{N}} \sup_{t \in \mathbb{I}}\bigg[\int_{[m,\infty)} \hat{C}_n(\psi_n(tx),\psi_n((1-t)x)) \mathrm{d}\lambda(x) + \int_{[m,\infty)} C(\psi_n(tx),\psi_n((1-t)x)) \mathrm{d}\lambda(x)\bigg] \\& \nonumber\leq
    \inf_{m > 0} \sup_{n\in \mathbb{N}}\left[\left(\tfrac{n+1}{n} +1\right)\sup_{t \in \mathbb{I}}\int_{[m,\infty)} M(\psi_n(tx),\psi_n((1-t)x)) \mathrm{d}\lambda(x)\right] \\& \nonumber
    \leq 3\inf_{m > 0}\sup_{n \in \mathbb{N}}\sup_{t \in \mathbb{I}}\left[\min\left\{\tfrac{1}{t},\tfrac{1}{1-t}\right\}\int_{[\max\{t,1-t\}m,\infty)} \psi_n(x) \mathrm{d}\lambda(x)\right] \\& \leq
    6\inf_{m > 0}\sup_{n \in \mathbb{N}}\left[\int_{[\frac{m}{2},\infty)} \psi_n(x) \mathrm{d}\lambda(x)\right] = 0.
    \end{align}
    Thereby the third inequality follows by considering the two cases 
    $t \geq 1-t$ and $t\leq 1-t$ separately and using change of coordinates both 
    times. Summing up, we have shown that $I_n \overset{n\rightarrow \infty}{\longrightarrow} 0$ almost surely.\\
    For $II_n$ we proceed similarly and first use 
    \begin{align*}
    II_n &\leq \sup_{t \in \mathbb{I}}\int_{[0,m]}|C(\psi_n(tx),\psi_n((1-t)x)) - C((\psi)_{1-\tau_A}(tx),(\psi)_{1-\tau_A}((1-t)x))|\mathrm{d}\lambda(x) \\& \quad +
    \sup_{t \in \mathbb{I}}\int_{[m,\infty)}|C(\psi_n(tx),\psi_n((1-t)x)) - C((\psi)_{1-\tau_A}(tx),(\psi)_{1-\tau_A}((1-t)x))|\mathrm{d}\lambda(x).
    \end{align*}
    Using Lipschitz continuity of copulas and applying 
    Theorem \ref{thm:consistency_psi_n} for the first summand we have that
    \begin{align*}
    \sup_{t \in \mathbb{I}}\int_{[0,m]}|C(\psi_n(tx),\psi_n((1-t)x)) - C((\psi)_{1-\tau_A}(tx),(\psi)_{1-\tau_A}((1-t)x))|\mathrm{d}\lambda(x) \leq
    2m\Vert \psi_n - (\psi)_{1-\tau_A}\Vert_\infty \overset{n\rightarrow \infty}{\longrightarrow}0.
    \end{align*}
    Furthermore, letting $b_{n,m}$ denote the second summand, proceeding as in in eq. \eqref{eq:estimate_xis} yields
    \begin{align*}
    \inf_{m > 0}\sup_{n \in \mathbb{N}}b_{n,m} &\leq
    2 \inf_{m > 0} \bigg[\sup_{n \in \mathbb{N}}\bigg[\int_{[\frac{m}{2},\infty)}\psi_n(x)\mathrm{d}\lambda(x)\bigg] +\int_{[\frac{m}{2},\infty)}(\psi)_{1-\tau_A}(x)\mathrm{d}\lambda(x)\bigg] = 0.
    \end{align*}
    This shows $II_n \overset{n\rightarrow \infty}{\longrightarrow} 0$ almost surely, so the proof is complete for the case $f = \mathrm{id}$. \\
    \textbf{(ii)} For $f = \log$ we obviously have 
      \begin{align*}
        I_n &\leq \sup_{t \in \mathbb{I}}\int_{[0,\frac{1}{m}]} \frac{|\hat{C}_n(\psi_n(tx),\psi_n((1-t)x)) - C(\psi_n(tx),\psi_n((1-t)x))|}{x} \mathrm{d}\lambda(x) \\&\quad +
        \sup_{t \in \mathbb{I}}\int_{[\frac{1}{m},m]} \frac{|\hat{C}_n(\psi_n(tx),\psi_n((1-t)x)) - C(\psi_n(tx),\psi_n((1-t)x))|}{x} \mathrm{d}\lambda(x) \\&\quad+ \sup_{t \in \mathbb{I}}\int_{[m,\infty)} \frac{|\hat{C}_n(\psi_n(tx),\psi_n((1-t)x)) - C(\psi_n(tx),\psi_n((1-t)x))|}{x} \mathrm{d}\lambda(x).
    \end{align*}
    Considering the first summand and setting 
    $$
    c_{n,m}:=\sup_{t \in \mathbb{I}}\int_{[0,\frac{1}{m}]} \frac{|\hat{C}_n(\psi_n(tx),\psi_n((1-t)x)) - C(\psi_n(tx),\psi_n((1-t)x))|}{x} \mathrm{d}\lambda(x),
    $$
    applying eq. \eqref{eq:estimates_copulas2}, the obvious fact that $\lfloor z \rfloor \geq z -1$ for every $z \in \mathbb{R}$ and Lemma \ref{lem:implication_regularity} yields
    \begin{align*}
    \nonumber \inf_{m > 0}\sup_{n \in \mathbb{N}}c_{n,m}&\leq \inf_{m > 0}\sup_{n \in \mathbb{N}} \sup_{t \in \mathbb{I}} \int_{[0,\frac{1}{m}]}\frac{1-W\left(\frac{(n+1)\psi_n(tx) - 1}{n},\frac{(n+1)\psi_n((1-t)x) - 1}{n}\right)}{x} \mathrm{d}\lambda(x) \\& \leq 
    4\inf_{m > 0}\sup_{n \in \mathbb{N}} \sup_{t \in \mathbb{I}} \min\{\tfrac{1}{t}, \tfrac{1}{1-t}\}\int_{[0,\max\{t,1-t\}\frac{1}{m}]}\frac{1-\psi_n(s)}{s} \mathrm{d}\lambda(x) \\& \leq
    8\inf_{m > 0}\sup_{n \in \mathbb{N}} \int_{[0,\frac{1}{m}]}\frac{1-\psi_n(s)}{s} \mathrm{d}\lambda(x) = 0.
    \end{align*}
    The second, the third term and $II_n$ all converge to zero, using similar arguments as in the case $f = \mathrm{id}$. This proves the result.
\end{proof}
\noindent Combining Lemma \ref{lem:est_eval} and Lemma \ref{lem:convergence_xi} directly 
yields the following main result for this subsection concerning 
strong consistency of $B_{n,c}^\Xi$.
\begin{theorem}\label{thm:pick_est_empirical}
Consider $\Xi \in \{\mathbf{P},\mathbf{CFG}\}$, and let $B_{n,c}^\Xi$ be defined 
according to eq. \eqref{eq:cfg_type_estimator_c}.
If $\Xi = \mathbf{P}$, assume that $\mathbb{E}[Z] < \infty$ and that Assumption \eqref{assumption1_pick} holds. If $\Xi = \mathbf{CFG}$, assume that $\mathbb{E}[|\log Z|] < \infty$ and that Assumptions \ref{assumption1_cfg} and \ref{assumption2_cfg} hold.
Then, in either case, the following identity holds with probability $1$:
$$
\|B_{n,c}^\Xi - (A)_{1-\tau_A}\|_\infty \overset{n\rightarrow \infty}{\longrightarrow} 0.
$$
\end{theorem}
\subsection{Proof of Theorem \ref{thrm:strong_consistency_copula}}
\noindent We start with the following lemma, which will be essential for proving the existence of the parameter $\alpha_n$ in eq. \eqref{eq:alphas}. 
\begin{lemma}\label{lem:estimate_B_n_c}
Consider $\Xi \in \{\mathbf{P},\mathbf{CFG}\}$ and let $B_{n,c}^\Xi$ be defined 
according to eq. \eqref{eq:cfg_type_estimator_c}. Then the following three assertions hold:
\begin{enumerate}[label=(\roman*)]
    \item There exists some constant $K_n \in (0,\infty)$ such that $A_M(t) \leq B_{n,c}^\Xi(t)\leq K_n$ holds for every $t \in \mathbb{I}$.
    \item  $B_{n,c}^\Xi$ is uniformly continuous on $\mathbb{I}$.
    \item $B_{n,c}^\Xi(t) > A_M(t)$ for every $t \in (0,1)$ and $B_{n,c}^\Xi(0) = B_{n,c}^\Xi(1) = 1$.
\end{enumerate}
\end{lemma}
\begin{proof}
We start with the first assertion. For every $i \in \{1,\ldots,n\}$ and $t \in \mathbb{I}$ we obviously have that 
$$
\varphi_n\left(\tfrac{1}{n+1}\right)\min\left\{\tfrac{1}{t},\tfrac{1}{1-t}\right\}\geq \hat{\xi}_{n,i}(t) \geq \varphi_n\left(\tfrac{n}{n+1}\right)\min\left\{\tfrac{1}{t},\tfrac{1}{1-t}\right\}
$$
with $\min\left\{\frac{1}{t},\frac{1}{1-t}\right\} = 1$ for $t = 0$ and $\min\left\{\frac{1}{t},\frac{1}{1-t}\right\} = 1$ for $t = 1$.
Both, for $B_{n,u}^\mathbf{P}$ and for $B_{n,u}^\mathbf{CFG}$ it therefore follows that
\begin{align*}
B_{n,u}^\Xi(t) &\leq
\frac{\varphi_n(\frac{1}{n+1})}{\varphi_n(\frac{n}{n+1})}A_M(t).
\end{align*}
and thus
$$
B_{n,c}^\Xi(t) \leq \frac{\varphi_n(\frac{1}{n+1})}{\varphi_n(\frac{n}{n+1})} + \varepsilon_n
$$
for every $t \in \mathbb{I}$.
Moreover, considering $\frac{1}{n}\sum_{i=1}^n\varphi_n(\frac{i}{n+1}) = \frac{1}{n}\sum_{i=1}^n\varphi_n(\hat{U}_i) = \frac{1}{n}\sum_{i=1}^n\varphi_n(\hat{V}_i)$, it 
is straightforward to verify that 
$$
B_{n,u}^\Xi(t) \geq A_M(t)
$$
for every $t \in \mathbb{I}$, which yields $B_{n,c}^\Xi(t) > A_M(t)$ for every $t \in (0,1)$. 
Since $B_{n,c}^\Xi(0) = B_{n,c}^\Xi(1) = 1$ is obvious, 
it remains to show the second assertion.
Considering that $\xi_{n,i}$ is continuous on $\mathbb{I}$, using 
compactness of $\mathbb{I}$, directly 
yields uniform continuity of $\xi_{n,i}$. Since finite sums, compositions and fractions (provided that the denominator is unequal to $0$) of uniformly continuous functions are uniformly continuous too, uniform continuity of $B_{n,c}^\Xi$ follows immediately.
\end{proof}
\noindent Our next step is to show that the function $A_{\alpha,n}^\Xi$ as proposed in eq. \eqref{eq:A_alpha_n} is a Pickands dependence function for every $\alpha \in [0,1-\frac{1}{n}]$.
\begin{lemma}\label{lem:pick_is_pick}
    Consider $\Xi \in \{\mathbf{P},\mathbf{CFG}\}$ and let $\alpha \in [0,1-\frac{1}{n}]$ be arbitrary. Then $A_{\alpha,n}^\Xi$ as defined in eq. \eqref{eq:A_alpha_n} has the following properties:
    \begin{enumerate}[label=(\roman*)]
        \item $A_M(t) \leq A_{\alpha,n}^\Xi(t) \leq 1$ holds for every $t \in \mathbb{I}$;
        \item $A_{\alpha,n}^\Xi$ is convex on $\mathbb{I}$.
    \end{enumerate}
    In other words, $A_{\alpha,n}^\Xi$ is a Pickands dependence function.
\end{lemma}
\begin{proof}
According to Lemma \ref{lem:estimate_B_n_c} we have 
$B_{n,c}^\Xi(t) \geq A_M(t)$ for every $t \in \mathbb{I}$. It therefore 
follows immediately that
$$
T_{\alpha,n}^\Xi(t) = B_{n,c}^\Xi\left(\frac{t^\frac{1}{1-\alpha}}{t^\frac{1}{1-\alpha} + (1-t)^\frac{1}{1-\alpha}}\right)^{1-\alpha}\left(t^\frac{1}{1-\alpha} + (1-t)^\frac{1}{1-\alpha}\right)^{1-\alpha} = (B_{n,c}^\Xi)_{\frac{1}{1-\alpha}}(t) \geq A_M(t)
$$
holds for every $t \in \mathbb{I}$ and arbitrary $\alpha \in [0,1-\frac{1}{n}]$. 
Hence, applying Lemma \ref{lem:properties_gcm} directly yields
$$
A_M(t) \leq A_{\alpha,n}^\Xi(t) \leq 1
$$
for every $t \in \mathbb{I}$, which proves the second assertion and concludes the 
proof, since $A_{\alpha,n}^\Xi$ is convex by definition.
\end{proof}
\noindent The next technical lemma is key for proving Theorem \ref{lem:existence_uniqueness_alpha}.
\begin{lemma}\label{lem:help_lem}
Let $z \in (0,\infty)$, $c_1,c_2\in (0,\infty)$ with $c_1 < c_2 $ and $\beta_1,\beta_2 \in [c_1,c_2]$. Then
    \begin{equation}\label{eq:est_kappa_t}
    |z^{\beta_1} - z^{\beta_2}| \leq |\beta_1 - \beta_2| \,\,|\log(z)|z^r,
    \end{equation}
    where $ r:= \min\{\beta_1, \beta_2\}$, if $z \in (0,1]$ and $r := \max\{\beta_1, \beta_2\}$, if $z \in (1,\infty)$.\\
    Furthermore, if $z \in (0,1]$ we have that
    \begin{equation}\label{eq:est_0_1}
    |z^{\beta_1} - z^{\beta_2}| \leq \frac{|\beta_1 - \beta_2|}{\min\{\beta_1,\beta_2\}e} \leq \frac{|\beta_1 - \beta_2|}{c_1e}.
    \end{equation}
\end{lemma}
\begin{proof}
Fixing $z \in (0,\infty)$, $c \in (0,\infty)$ and $\beta_1,\beta_2 \in [0,c]$ with $\beta_1 \neq \beta_2$, using the mean value theorem there exists some $s \in (\beta_1,\beta_2)$ such that 
$$
z^s \log(z) = \frac{z^{\beta_1} - z^{\beta_2}}{\beta_1 - \beta_2}
$$
and therefore
\begin{align*}
|z^{\beta_1} - z^{\beta_2}| = |\beta_1 - \beta_2|\,z^s\,|\log(z)| \leq |\beta_1 - \beta_2\,|z^r\,|\log(z)| 
\end{align*}
where $ r:= \min\{\beta_1, \beta_2\}$, if $z \in (0,1]$ and $r := \max\{\beta_1, \beta_2\}$, if $z \in (1,\infty)$.\\
Having $z \in (0,1]$ it is easily seen that the function $z \mapsto z^r|\log z|$ attains its maximum at $z = \exp(-\frac{1}{r})$, which implies 
$$
    |z^{\beta_1} - z^{\beta_2}| \leq \frac{|\beta_1 - \beta_2|}{\min\{\beta_1,\beta_2\}e} \leq \frac{|\beta_1 - \beta_2|}{c_1e}.
$$
\end{proof}
\begin{lemma}\label{lem:cont_A_alpha}
For $\Xi \in \{\mathbf{P},\mathbf{CFG}\}$ let $A_{\alpha,n}^\Xi$ be defined 
according to eq. \eqref{eq:A_alpha_n} and let $\alpha_0,\alpha_1,\alpha_2, \ldots \in [0,1-\frac{1}{n}]$ be such that $\alpha_j \overset{j \rightarrow \infty}{\longrightarrow} \alpha_0$. Then the following convergence holds:
$$
\Vert A_{\alpha_j,n}^\Xi - A_{\alpha_0,n}^\Xi \Vert_\infty \overset{j \rightarrow \infty}{\longrightarrow} 0.
$$
\end{lemma}
\begin{proof}
    Using the fact that the function $x \mapsto \min\{x,1\}$ is $1$-Lipschitz continuous and Lemma \ref{lem:properties_gcm}, it suffices to 
    prove that 
     $$
    (B_{n,c}^\Xi)_{\frac{1}{1-\alpha_j}} = T_{\alpha_j,n}^\Xi \overset{j\rightarrow \infty}{\longrightarrow} T_{\alpha_0,n}^\Xi = 
    (B_{n,c}^\Xi)_{\frac{1}{1-\alpha_0}}
    $$
    uniformly on $\mathbb{I}$, which can be done as follows.
    Under the stated assumption we obviously have 
    $\lim_{j \rightarrow \infty} \frac{1}{1-\alpha_j}=\frac{1}{1-\alpha_0}$ as well as
    $\frac{1}{1-\alpha_j} \in [1,n]$ for every $j \in \mathbb{N}_0$.
    Applying assertion 8 from Lemma \ref{lem:properties_trans_pickands}
    yields uniform convergence of $(\kappa_{\frac{1}{1-\alpha_j}})_{j \in \mathbb{N}}$ to $\kappa_{\frac{1}{1-\alpha_0}}$ on $\mathbb{I}$. 
    Having this, using the second assertion of Lemma \ref{lem:estimate_B_n_c}, 
    uniform convergence of $B_{n,c}^\Xi \circ \kappa_{\frac{1}{1-\alpha_j}}$ 
    to $B_{n,c}^\Xi \circ \kappa_{\frac{1}{1-\alpha_0}}$ follows. 
    Moreover, from Lemma \ref{lem:estimate_B_n_c} we also know that 
    $B_{n,c}^\Xi(t) \in [\frac{1}{2},K_n]$ for every $t \in \mathbb{I}$, 
    hence, applying the fact that the functions 
    $x \mapsto x^{1-\alpha_j}$ converge uniformly to 
    the function  $x \mapsto x^{1-\alpha_0}$ on $[\frac{1}{2},K_n]$ 
    yields that 
    $$
    \left(B_{n,c}^\Xi \circ \kappa_{\frac{1}{1-\alpha_j}}\right)^{1-\alpha_j}
    \overset{j \rightarrow \infty}{\longrightarrow} 
     \left(B_{n,c}^\Xi \circ \kappa_{\frac{1}{1-\alpha_0}}\right)^{1-\alpha_0}
    $$
    uniformly on $\mathbb{I}$. 
    Having this, the result now follows easily.
\end{proof}
\noindent Assuming that $\alpha \in [\frac{1}{n},1-\frac{1}{n}]$, we always have $A_{\alpha,n}^\Xi(\frac{1}{2}) > \frac{1}{2}$ and therefore we obtain the following bound on Kendall's $\tau$ of $A_{\alpha,n}^\Xi$:
\begin{lemma}\label{lem:bounds_tau}
    Let $\Xi \in \{\mathbf{P},\mathbf{CFG}\}$. 
    Then for every $\alpha \in [0,1-\frac{1}{n}]$ we have  $A_{\alpha,n}^\Xi\left(\frac{1}{2}\right) > \frac{1}{2}$ as well as
    $0 \leq \tau_{A_{\alpha,n}^\Xi} < 1.$
\end{lemma}
\begin{proof}
    According to Lemma \ref{lem:estimate_B_n_c} 
    $B_{n,c}^\Xi(\frac{1}{2}) > \frac{1}{2}$ holds, which implies
    $T_{\alpha,n}^\Xi(\frac{1}{2}) > \frac{1}{2}$. Using continuity of $T_{\alpha,n}^\Xi$, we obviously have that $c_\alpha := \min_{s \in \mathbb{I}}T_{\alpha,n}^\Xi(s) > \frac{1}{2}$. 
    Defining the function $g_\alpha: \mathbb{I} \rightarrow [c_\alpha,1]$ by
    $$
    g_{\alpha}(t) := \begin{cases}
        1-t, &\text{ if } t \in [0,1-c_\alpha],\\
        c_\alpha, &\text{ if } t \in (1-c_\alpha,c_\alpha),\\
        t, &\text{ if } t \in [c_\alpha,1],\\
    \end{cases}
    $$ 
    obviously $g_{\alpha}$ is convex, so we obtain that
    $$
    A_{\alpha,n}^\Xi\left(\tfrac{1}{2}\right) = \mathrm{gcm}(\min\{T_{\alpha,n}^\Xi,1\})\left(\tfrac{1}{2}\right) \geq  g_{\alpha}\left(\tfrac{1}{2}\right) > \tfrac{1}{2}.
    $$
    Considering $A_{\alpha,n}^\Xi \in \mathcal{A}$, both 
    $C_{A_{\alpha,n}^\Xi} \neq M$ and 
    $0 \leq \tau_{A_{\alpha,n}^\Xi} < 1$ follow.  
\end{proof}
\noindent Our next step is to show that $\varphi_{\alpha,n}^\Xi$ as defined in eq. \eqref{eq:est_phi_alpha} is indeed the pseudo-inverse of a normalized Archimedean 
generator $\psi_{\alpha,n}^\Xi$:
\begin{lemma}\label{lem:gen_is_gen}
    Let $\Xi \in \{\mathbf{P},\mathbf{CFG}\}$ and $\alpha \in [0,1-\frac{1}{n}]$ be arbitrary. Then $\varphi_{\alpha,n}^\Xi$, defined according to 
    eq. \eqref{eq:est_phi_alpha}, has the following properties: 
    \begin{enumerate}[label=(\roman*)]
        \item $\varphi_{\alpha,n}^\Xi(0)< \infty$,  $\varphi_{\alpha,n}^\Xi(\frac{1}{2}) = 1$, 
        $\varphi_{\alpha,n}^\Xi(1) = 0$ and $\varphi_{\alpha,n}^\Xi(t) > 0$ for every $t \in [0,1)$;
        \item $\varphi_{\alpha,n}^\Xi$ is strictly decreasing on $\mathbb{I}$;
        \item $\varphi_{\alpha,n}^\Xi$ is convex on $\mathbb{I}$.
    \end{enumerate}
    In other words: $\varphi_{\alpha,n}^\Xi$ is the
    pseudo-inverse of a non-strict Archimedean generator $\psi_{\alpha,n}^\Xi \in \Psi$.
    \end{lemma}
\begin{proof}
    Let $\Xi \in \{\mathbf{P},\mathbf{CFG}\}$ and $\alpha \in [0,1-\frac{1}{n}]$ 
    be arbitrary but fixed and let $\varphi_n$ and $\tau_{A_{\alpha,n}^\Xi}$ be the estimators as defined in eqs. \eqref{eq:phi.via.kendall} and \eqref{eq:kendallstau_ex}, respectively. According to Lemma \ref{lem:bounds_tau} we have that $\tau_{A_{\alpha,n}^\Xi} \in [0,1)$. Since $\varphi_n$ is the pseudo-inverse of  an Archimedean generator, setting $f_n := \varphi_n^{1-\tau_{A_{\alpha,n}^\Xi}}$, we obtain that
    \begin{enumerate}
        \item $f_n > 0$ on $[0,1)$, $f_n(1) = 0$ and $f_n(0)< \infty$;
        \item $f_n$ is strictly decreasing on $\mathbb{I}$;
        \item $f_n$ is continuous on $\mathbb{I}$.
    \end{enumerate}
    Furthermore, convexity of $\varphi_n$ and the fact 
    that $\varphi_n$ is piecewise linear yields 
    $\varphi_n(t) \geq (t-1)D^-\varphi_n(1)=(1-t)(-D^-\varphi_n(1))$ for every $t \in \mathbb{I}$, 
    implying
    $$
    f_n(t) \geq (1-t)^{1-\tau_{A_{\alpha,n}^\Xi}}(-D^-\varphi_n(1))^{1-\tau_{A_{\alpha,n}^\Xi}} \geq (1-t)(-D^-\varphi_n(1))^{1-\tau_{A_{\alpha,n}^\Xi}}=: \iota_n(t)
    $$
    for every $t \in \mathbb{I}$. Since $\iota_n$ is convex and  $\iota_n(t) > 0$ for every $t \in [0,1)$, we obtain that
    \begin{equation}\label{eq:est_gcm}
    \mathrm{gcm}(f_n)(t) \geq \iota_n(t) > 0
    \end{equation}
    holds for every $t \in [0,1)$. Applying the fact that $f_n(1) = 0=\iota_n(1) $, it follows that 
    $$
    0 = f_n(1) \geq \mathrm{gcm}(f_n)(1) \geq 0
    $$
    and therefore $\mathrm{gcm}(f_n)(1) = 0$. Considering the left-hand derivative of $\mathrm{gcm}(f_n)$ in $1$, using eq. \eqref{eq:est_gcm} yields 
    $$
    D^-[\mathrm{gcm}(f_n)](1) = \lim_{h \uparrow 0}\frac{\mathrm{gcm}(f_n)(1+h) - \mathrm{gcm}(f_n)(1)}{h} = \lim_{h \uparrow 0}\frac{\mathrm{gcm}(f_n)(1+h)}{h} \leq 0.
    $$
    Since the left-hand derivative of a convex function is non-decreasing, this 
    shows 
    $$
    D^-[\mathrm{gcm}(f_n)](x) \leq D^-[\mathrm{gcm}(f_n)](1) \leq 0
    $$
    for every $x \in (0,1)$, so $\mathrm{gcm}(f_n)$ is non-increasing on $(0,1)$.\\
    Suppose that there exist $x_1,x_2 \in (0,1)$ with $x_1 < x_2$ and $\mathrm{gcm}(f_n)(x_1) = \mathrm{gcm}(f_n)(x_2)$. Then monotonicity yields that $\mathrm{gcm}(f_n)(t) = \mathrm{gcm}(f_n)(x_1)$ for every $t \in [x_1,x_2]$ and therefore $D^-\mathrm{gcm}(f_n)(t) = 0$ for every $t \in (x_1,x_2]$, implying that $D^-\mathrm{gcm}(f_n)(t) = 0$ for every $t \in (x_1,1]$. This shows 
    that $\mathrm{gcm}(f_n)(t) = 0$ for every $t \in (x_1,1]$, a contradiction to eq. \eqref{eq:est_gcm}. Therefore we have that $\mathrm{gcm}(f_n)$ is strictly decreasing on $\mathbb{I}$. Since obviously $\mathrm{gcm}(f_n)(\frac{1}{2}) > 0$ and $\mathrm{gcm}(f_n)$ is convex by definition, the desired properties follow immediately.
\end{proof}
\begin{lemma}\label{lem:cont_phi_alpha}
Let $\Xi \in \{\mathbf{P},\mathbf{CFG}\}$, $\varphi_{\alpha,n}^\Xi$ be defined as in eq. \eqref{eq:est_phi_alpha} and assume that $\alpha_0,\alpha_1,\alpha_2, \ldots \in [0,1-\frac{1}{n}]$ fulfills $\alpha_j \overset{j \rightarrow \infty}{\longrightarrow} \alpha_0$. Then it holds that
$$
\Vert\varphi_{\alpha_j,n}^\Xi -  \varphi_{\alpha_0,n}^\Xi\Vert_\infty \overset{j \rightarrow \infty}{\longrightarrow} 0.
$$
\end{lemma}
\begin{proof}
    According to Lemma \ref{lem:cont_A_alpha} we have $A_{\alpha_j,n}^\Xi \overset{j \rightarrow \infty}{\longrightarrow} A_{\alpha_0,n}^\Xi$ uniformly on $\mathbb{I}$, so (see \cite{bernoulli}) the sequence $(C_{A_{\alpha_j,n}^\Xi})_{j \in \mathbb{N}}$ in $\mathcal{C}_{ev}$ converges uniformly to $C_{A_{\alpha_0,n}^\Xi} \in \mathcal{C}_{ev}$, which in turn yields $\tau_{A_{\alpha_j,n}^\Xi} \overset{j \rightarrow \infty}{\longrightarrow} \tau_{A_{\alpha_0,n}^\Xi}$. Applying Lemma \ref{lem:bounds_tau} yields that $0 \leq \tau_{A_{\alpha_0,n}^\Xi} < 1$, 
    so there exists some constant $c_0 > 0$ with $1- \tau_{A_{\alpha_0,n}^\Xi} = c_0$. Furthermore, there exists some $J_1 \in \mathbb{N}$ such that for every $j \geq J_1$ we obtain that
    $$
    |\tau_{A_{\alpha_j,n}^\Xi} - \tau_{A_{\alpha_0,n}^\Xi}| < \frac{\varepsilon}{\max\{\frac{2}{c_0e},\log(\varphi_n(0))\varphi_n(0)\}}.
    $$
    Moreover, there exists an index $J_2 \geq J_1 $ such that for every $j \geq J_2$ we have
    $$
    \tau_{A_{\alpha_j,n}^\Xi}  < \tau_{A_{\alpha_0,n}^\Xi} + \frac{c_0}{2}.
    $$
    Applying eqs. \eqref{eq:est_kappa_t} and \eqref{eq:est_0_1} in Lemma \ref{lem:help_lem} we therefore obtain that for every $j \geq J_2$ it holds that
    \begin{align*}
    |\varphi_n(t)^{1-\tau_{A_{\alpha_j,n}^\Xi}}-\varphi_n(t)^{1-\tau_{A_{\alpha_0,n}^\Xi}}| 
    &\leq |\tau_{A_{\alpha_j,n}^\Xi} - \tau_{A_{\alpha_0,n}^\Xi}| \, |\log(\varphi_n(t))|\,\varphi_n(t)^r\\&\leq
    |\tau_{A_{\alpha_j,n}^\Xi} - \tau_{A_{\alpha_0,n}^\Xi}|\max\left\{\tfrac{2}{c_0e},|\log(\varphi_n(0))|\varphi_n(0)\right\} < \varepsilon
    \end{align*}
    for every $t \in \mathbb{I}$, whereby $r = \max\{1-\tau_{A_{\alpha_j,n}^\Xi},1-\tau_{A_{\alpha_0,n}^\Xi}\}$ if
    $t \in [0,\frac{1}{2}]$ and
    $r = \min\{1-\tau_{A_{\alpha_j,n}^\Xi},1-\tau_{A_{\alpha_0,n}^\Xi}\}$ if $t \in (\frac{1}{2},1]$. 
    In other words, we have 
    $\varphi_n^{1-\tau_{A_{\alpha_j,n}^\Xi}}
    \overset{j \rightarrow \infty}{\longrightarrow}
    \varphi_n^{1-\tau_{A_{\alpha_0,n}^\Xi}}$
    uniformly on $\mathbb{I}$.
    Applying Lemma \ref{lem:properties_gcm} on properties of the gcm yields that  $\varphi_{\alpha_j,n}^\Xi \overset{j \rightarrow \infty}{\longrightarrow}\varphi_{\alpha_0,n}^\Xi$ uniformly on $\mathbb{I}$.
\end{proof}
\begin{theorem}\label{lem:existence_uniqueness_alpha}
Let $\Xi \in \{\mathbf{P},\mathbf{CFG}\}$ and $D_n^\Xi$ be defined according to 
eq. \eqref{eq:D_n} and $C_{\alpha,n}^\Xi$ according to 
eq. \eqref{eq:C_alpha_n}. Then
\begin{equation}\label{eq:def.alpha.n}
\alpha_n = \min\mathrm{argmin}_{\alpha \in [0,1-\frac{1}{n}]} \Vert C_{\alpha,n}^\Xi - D_{n}^\Xi \Vert_\infty
\end{equation}
exists and is unique.
\end{theorem}
\begin{proof}
Fix $\Xi \in \{\mathbf{P},\mathbf{CFG}\}$ and consider an arbitrary sequence $\alpha_0,\alpha_1,\alpha_2,\ldots \in [0,1-\frac{1}{n}]$ with $\alpha_j \overset{j \rightarrow \infty}{\longrightarrow} \alpha_0$. Applying Lemma \ref{lem:cont_A_alpha} 
and Lemma \ref{lem:cont_phi_alpha} we have, firstly, that 
$A_{\alpha_j,n}^\Xi \overset{j \rightarrow \infty}{\longrightarrow}A_{\alpha_0,n}^\Xi$ and $\varphi_{\alpha_j,n}^\Xi \overset{j \rightarrow \infty}{\longrightarrow} \varphi_{\alpha_0,n}^\Xi$ uniformly on $\mathbb{I}$, and, secondly, $\tau_{A_{\alpha_j,n}^\Xi} \overset{j \rightarrow \infty}{\longrightarrow} \tau_{A_{\alpha_0,n}^\Xi}$. 
Having this, using the last assertion in Lemma \ref{lem:properties_trans_pickands} and applying similar arguments as in the proofs of Lemma \ref{lem:cont_A_alpha} and Lemma \ref{lem:cont_phi_alpha} we obtain that
$$
(A_{\alpha_j,n}^\Xi)_{1-\tau_{A_{\alpha_j,n}^\Xi}} \overset{j \rightarrow \infty}{\longrightarrow} (A_{\alpha_0,n}^\Xi)_{1-\tau_{A_{\alpha_0,n}^\Xi}},\quad 
(\varphi_{\alpha_j,n}^\Xi)_{1-\tau_{A_{\alpha_j,n}^\Xi}} \overset{j \rightarrow \infty}{\longrightarrow} (\varphi_{\alpha_0,n}^\Xi)_{1-\tau_{A_{\alpha_0,n}^\Xi}}
$$
uniformly on $\mathbb{I}$. Using Theorem \ref{thrm:main_result_convergence} it then immediately follows that $C_{\alpha_j,n}^\Xi \overset{j\rightarrow\infty}{\longrightarrow}C_{\alpha_0,n}^\Xi$ uniformly on $\mathbb{I}^2$.
This shows that the mapping 
$\alpha \mapsto \Vert C_{\alpha,n}^\Xi - D_{n}^\Xi \Vert_\infty$ is continuous on the compact interval $[0,1-\frac{1}{n}]$, implying that the set 
$$
\mathrm{argmin}_{\alpha \in [0,1-\frac{1}{n}]} \Vert C_{\alpha,n}^\Xi - D_{n}^\Xi \Vert_\infty
$$
is non-empty and compact. As a direct consequence $\alpha_n$, defined according 
to eq. \eqref{eq:def.alpha.n} exists and is unique.
\end{proof}
\noindent As next step we show that the estimators proposed in eq. \eqref{eq:D_n} converge uniformly on $\mathbb{I}^2$ to the copula $C_{\psi,A}$.
\begin{lemma}\label{lem:copula_est_part1}
Consider $\psi \in \Psi$, $A \in \mathcal{A}\setminus\{A_M\}$, $\Xi \in \{\mathbf{P},\mathbf{CFG}\}$ and let $D_n^\Xi$ be defined 
according to eq. \eqref{eq:D_n}. Then with probability $1$ we have
    $$
    \Vert D_n^\Xi - C_{\psi,A}\Vert_\infty \overset{n \rightarrow \infty}{\longrightarrow} 0.
    $$
\end{lemma}
\begin{proof}
Let $\varepsilon > 0$ be arbitrary but fixed and consider some 
$\delta \in (0,\frac{\varepsilon}{5})$. We already know that the sequence 
$(\varphi_n)_{n\in\mathbb{N}}$ converges uniformly 
to $(\varphi)_{1-\tau_A}$ on every compact subinterval $[a,b]$ of $(0,1]$
and that (according to Theorems \ref{thm:consistency_psi_n} and  \ref{thm:pick_est_empirical}) $(B_{n,c}^\Xi)_{n \in \mathbb{N}}$ 
converges uniformly on $\mathbb{I}$ to $(A)_{1-\tau_A}$.
We therefore can find some $n_0 \in \mathbb{N}$ such that for every 
$n \geq n_0$
$$
\sup_{(x,y)\in[\delta,1]^2}|D_n^\Xi(x,y) - C_{\psi,A}(x,y)| < \frac{\varepsilon}{5}
$$
holds. According to Lemma \ref{lem:estimate_B_n_c} we have that $B_{n,c}^\Xi(t) \geq A_M(t)$ for every $t \in \mathbb{I}$, hence 
\begin{align*}
D_n^\Xi(x,y) \leq \psi_n\left((\varphi_n(x) + \varphi_n(y))\max\left\{\frac{\varphi_n(x)}{\varphi_n(x) + \varphi_n(y)},\frac{\varphi_n(y)}{\varphi_n(x) + \varphi_n(y)}\right\}\right) = M(x,y)
\end{align*}
for all $x,y \in \mathbb{I}$. 
For $n \geq n_0 $ this altogether yields
\begin{align*}
\Vert D_n^\Xi - C_{\psi,A}\Vert_\infty &\leq \sup_{(x,y)\in[\delta,1]^2}|D_n^\Xi(x,y) - C_{\psi,A}(x,y)| + \sup_{(x,y)\in[0,\delta] \times \mathbb{I}}|D_n^\Xi(x,y) - C_{\psi,A}(x,y)| + \sup_{(x,y)\in[\delta,1] \times [0,\delta]}|D_n^\Xi(x,y) - C_{\psi,A}(x,y)| \\&\leq
\frac{\varepsilon}{5} + \sup_{(x,y)\in[0,\delta] \times \mathbb{I}}D_n^\Xi(x,y) +C_{\psi,A}(\delta,1) +  \sup_{(x,y)\in[\delta,1] \times [0,\delta]}D_n^\Xi(x,y) + C_{\psi,A}(1,\delta) \\&\leq
\frac{\varepsilon}{5} + 4\delta < \varepsilon.
\end{align*}
This completes the proof since $\varepsilon>0$ was arbitrary.
\end{proof}
Before finally proving Theorem \ref{thrm:strong_consistency_copula}
we need the following last lemma:
\begin{lemma}\label{lem:copula_est_part2}
  Consider $\psi \in \Psi$, $A \in \mathcal{A}\setminus\{A_M\}$, $\Xi \in \{\mathbf{P},\mathbf{CFG}\}$, and let 
  $A_{\tau_A,n}^\Xi$ be defined according to eq. \eqref{eq:A_alpha_n} and $\psi_{\tau_A,n}^\Xi$ be the pseudo-inverse of $\varphi_{\tau_A,n}^\Xi$ 
  according to eq. \eqref{eq:est_phi_alpha}. Then we have that
    $$
    \Vert C_{\psi_{\tau_A,n}^\Xi,A_{\tau_A,n}^\Xi}-C_{\psi,A}\Vert_\infty \overset{n \rightarrow \infty}{\longrightarrow}0
    $$
with probability $1$.
\end{lemma}
\begin{proof}
    According to Theorem \ref{thm:pick_est_empirical} we have that 
    $B_{n,c}^\Xi \overset{n \rightarrow \infty} {\longrightarrow} (A)_{1-\tau_A}$
    uniformly on $\mathbb{I}$ with probability $1$. Using Lemma \ref{lem:properties_trans_pickands} therefore yields that 
    $$
    T_{\tau_A,n}^\Xi =(B_{n,c}^\Xi)_{\frac{1}{1-\tau_A}} 
    \overset{n \rightarrow \infty}{\longrightarrow} 
    ((A)_{1-\tau_A})_{\frac{1}{1-\tau_A}} = A
    $$
    uniformly on $\mathbb{I}$, so, applying Lemma \ref{lem:properties_gcm} we obtain that $A_{\tau_A,n}^\Xi \overset{n\rightarrow \infty}{\longrightarrow}A$ uniformly on $\mathbb{I}$ and therefore $\tau_{A_{\tau_A,n}^\Xi} \overset{n \rightarrow \infty}{\longrightarrow} \tau_A \in [0,1)$. Using Theorem \ref{thm:consistency_psi_n}, the fact that $\varphi, \varphi_1, \varphi_2, \ldots$ are convex functions, and applying arguments analogous to the proof of Lemma \ref{lem:cont_phi_alpha} we have that for every $\delta > 0$
    $$
    \sup_{t \in [\delta,1]}|\varphi_n^{1-\tau_{A_{\tau_A,n}^\Xi}}(t) - \varphi(t)| \overset{n\rightarrow \infty}{\longrightarrow} 0
    $$
    holds almost surely. 
    Again applying Lemma \ref{lem:properties_gcm} therefore yields that  $\varphi_{\tau_A,n}^\Xi \overset{n \rightarrow \infty}{\longrightarrow} \varphi$ uniformly on $[\delta,1]$ for every $\delta > 0$, implying that $\varphi_{\tau_A,n}^\Xi \overset{n \rightarrow \infty}{\longrightarrow} \varphi$ pointwise on $(0,1]$. Using Theorem \ref{thm:consistency_psi_n}, the fact that conjunctions of uniformly and pointwise convergent sequences of functions converge pointwise and Lipschitz continuity of copulas 
    finally yields $\Vert C_{\psi_{\tau_A,n}^\Xi,A_{\tau_A,n}^\Xi}-C_{\psi,A}\Vert_\infty \overset{n \rightarrow \infty}{\longrightarrow}0$.
\end{proof}
\noindent After all these preparations we can now 
prove strong consistency of the copula-estimator proposed in eq. \eqref{eq:copula_estimator}.
\begin{proof}[Proof of Theorem \ref{thrm:strong_consistency_copula}]\label{proof_copula_estimator}
    Choose $n \in \mathbb{N}$ sufficiently large so that $0 \leq \tau_A < 1-\frac{1}{n}$
    holds. Using the definition of $\alpha_n$ as given in eq. \eqref{eq:alphas} and applying Lemmas \ref{lem:copula_est_part1} and \ref{lem:copula_est_part2}, we obtain that
    \begin{align*}
    \Vert C_{\psi_{\alpha_n,n}^\Xi, A_{\alpha_n,n}^\Xi} - C_{\psi,A}\Vert_\infty \leq
    \Vert C_{\psi_{\alpha_n,n}^\Xi, A_{\alpha_n,n}^\Xi} - D_n^\Xi\Vert_\infty + \Vert D_n^\Xi - C_{\psi,A}\Vert_\infty \leq \Vert C_{\psi_{\tau_A,n}^\Xi, A_{\tau_A,n}^\Xi} - D_n^\Xi\Vert_\infty + \Vert D_n^\Xi - C_{\psi,A}\Vert_\infty \overset{n\rightarrow\infty}{\longrightarrow}0
    \end{align*}
    with probability $1$.
\end{proof}
\subsection{Auxiliary lemma}
Let $f \colon \mathbb{I} \rightarrow \mathbb{R}$ be a function, that is bounded from below. Then the greatest convex minorant $\mathrm{gcm}(f): \mathbb{I} \rightarrow \mathbb{R}$ 
of $f$ is the function defined as
$$
\mathrm{gcm}(f)(t) := \sup\{h(t)\mid h \colon \mathbb{I} \rightarrow \mathbb{R},\,\, h \text{ convex and } h(s)\leq f(s) \text{ for all } s \in \mathbb{I}\}.
$$
For the sake of completeness, we prove the following properties for the greatest convex minorant (we write $f_1 \leq f_2$ if $f_1(s) \leq f_2(s)$ for all $s \in \mathbb{I}$):
\begin{lemma}[Properties of the gcm]\label{lem:properties_gcm}
Let $f,f_1,f_2,\ldots \colon \mathbb{I} \rightarrow \mathbb{R}$ be functions, that  are bounded from below. Then the following properties hold:
\begin{enumerate}[label=(\roman*)]
    \item $\mathrm{gcm}(f)$ exists, is convex and unique.
    \item If $f$ is convex, then $\mathrm{gcm}(f)=f$.
    \item If $f_1 \leq f_2$, then $\mathrm{gcm}(f_1) \leq \mathrm{gcm}(f_2)$.
    \item $\mathrm{gcm}(f+c) = \mathrm{gcm}(f) + c$ for every constant $c \in \mathbb{R}$.
    \item If $f_n \overset{n\rightarrow\infty}{\longrightarrow} f$ uniformly on $\mathbb{I}$, then $\mathrm{gcm}(f_n) \overset{n\rightarrow\infty}{\longrightarrow}\mathrm{gcm}(f)$ uniformly on $\mathbb{I}$.
\end{enumerate}
\end{lemma}
\begin{proof}
As we always have $\mathrm{gcm}(f) \leq f$ and the supremum of families of 
convex functions is a convex function, $\mathrm{gcm}(f)$ is itself convex. 
Uniqueness is an immediate consequence of the maximality of $\mathrm{gcm}(f)$.\\
The third assertion is an immediate consequence of 
the fact that for $f_1 \leq f_2$ obviously $\mathrm{gcm}(f_1)$ is a convex minorant
of $f_2$. Since assertion number four is straightforward to verify, 
we focus on the last assertion: 
Using uniform convergence, for every $\varepsilon > 0$ there exists some 
$N_0 \in \mathbb{N}$ such that for every $n \geq N_0$ and every $t \in \mathbb{I}$ we obtain that
   $$
   f(t) - \varepsilon < f_n(t) < f(t) + \varepsilon.
   $$
Using the first four properties, we obtain that for every $n \geq N_0$ and 
every $t \in \mathbb{I}$
   $$
   \mathrm{gcm}(f)(t) - \varepsilon = \mathrm{gcm}(f-\varepsilon)(t) \leq \mathrm{gcm}(f_n)(t) \leq \mathrm{gcm}(f+ \varepsilon)(t) = \mathrm{gcm}(f)(t) + \varepsilon
   $$
   holds, which completes the proof.
\end{proof}
\subsection{Regularity assumptions for the Pickands and CFG type estimator}\label{subsec:assumptions_pickands_cfg}
In this subsection we discuss the regularity assumptions stated in  Theorems~\ref{thrm:strong_consistency_copula} and~\ref{thrm:strong_consistency_fcts}. Both theorems require $\mathbb{E}[Z] < \infty$ for the Pickands type estimator
and $\mathbb{E}[|\log Z|] < \infty$ for the CFG type estimator. 
Motivated by the simulation study in Section~\ref{sec:sim_study}, we check these conditions for each of the generators $\psi$ listed in Table~\ref{tab:archimedean_generators_normalized}.
\begin{lemma}
    Let $\psi_\theta \in \Psi$ be one of the generators listed 
    in Table \ref{tab:archimedean_generators_normalized} with admissible parameter $\theta$  and $A \in \mathcal{A} \setminus\{A_M\}$ with corresponding $\tau_A \in [0,1)$. Furthermore let $Z_\theta$ be the random variable with survival function $(\psi)_{1-\tau_A}$. If $\psi_\theta$ is the generator of a Frank, Gumbel or Joe copula, then
    \begin{equation} \label{eq:param.ev.pick}
    \mathbb{E}[Z_\theta] =  \int_\mathbb{I} (\varphi_\theta)_{1-\tau_A}(t) \mathrm{d}\lambda(t) = \int_\mathbb{I} \varphi_\theta^\frac{1}{1-\tau_A}(t) \mathrm{d}\lambda(t) < \infty.
    \end{equation}
    In the case that $\psi_\theta$ is the generator of a Clayton copula, eq. \eqref{eq:param.ev.pick} holds if and only if $\theta \in [0,1-\tau_A)$.\\
    Furthermore, in each of the four cases it holds that
     $$
     \mathbb{E}[|\log(Z_\theta)|] = \int_\mathbb{I} \left|\log\left(\varphi_\theta^\frac{1}{1-\tau_A}(t)\right)\right| \mathrm{d}\lambda(t) = \frac{1}{1-\tau_A}\int_\mathbb{I} |\log(\varphi_\theta(t))|\mathrm{d}\lambda(t) < \infty.
    $$ 
\end{lemma}
\begin{proof}
(i) We first consider the stated equivalence for the Clayton copula
and assume that $\psi_\theta$ is a Clayton generator. If $\theta \in [0,1-\tau_A)$, we obviously have
$$
\mathbb{E}[Z_\theta] = \int_{\mathbb{I}}\left[\dfrac{t^{-\theta}-1}{0.5^{-\theta}-1}\right]^\frac{1}{1-\tau_A} \mathrm{d}\lambda(t) \leq \frac{1}{(0.5^{-\theta}-1)^\frac{1}{1-\tau_A}} \int_\mathbb{I} t^{-\frac{\theta}{1-\tau_A}} \mathrm{d}\lambda(t) < \infty.
$$
We assume now that $\theta \in [1-\tau_A,\infty)$ and show $\mathbb{E}[Z_\theta] = \infty$; the following inequality holds:
$$
\mathbb{E}[Z_\theta] \geq \int_{[0,2^{-\frac{1}{\theta}}]}\left[\dfrac{t^{-\theta}-1}{0.5^{-\theta}-1}\right]^\frac{1}{1-\tau_A} \mathrm{d}\lambda(t) \geq \frac{1}{2^\frac{1}{1-\tau_A}(0.5^{-\theta}-1)^\frac{1}{1-\tau_A}} \int_{[0,2^{-\frac{1}{\theta}}]} t^{-\frac{\theta}{1-\tau_A}} \mathrm{d}\lambda(t) = \infty.
$$
(ii) As second family we consider the case where $\psi_\theta$ is the generator of the Frank copula. For $\theta \in \mathbb{R} \setminus \{0\}$, applying change of coordinates yields:
\begin{align*}
\mathbb{E}[Z_\theta] &= \frac{1}{\log(1+\exp(\frac{\theta}{2}))^\frac{1}{1-\tau_A}}\int_{\mathbb{I}}\left[-\log\left(\frac{\exp(-\theta t)-1}{\exp(-\theta)-1}\right)\right]^\frac{1}{1-\tau_A} \mathrm{d}\lambda(t) \\&=
\frac{1- \exp(-\theta)}{\theta\log(1+\exp(\frac{\theta}{2}))^\frac{1}{1-\tau_A}}\int_{\mathbb{I}}\left[-\log\left(z\right)\right]^\frac{1}{1-\tau_A}\frac{1}{z(\exp(-\theta)-1) + 1} \mathrm{d}\lambda(z)
\end{align*}
For $\theta > 0$, we obtain that
$$
\frac{1}{z(\exp(-\theta)-1) + 1} \leq \frac{1}{\exp(-\theta)},
$$
while for $\theta < 0$ we have
$$
\frac{1}{z(\exp(-\theta)-1) + 1} \leq 1.
$$
Combining both cases, we obtain for every $\theta \in \mathbb{R} \setminus \{0\}$
\begin{align*}
    \mathbb{E}[Z_\theta] \;\leq\; \frac{\big(1-\exp(-\theta)\big)\max\{\exp(\theta),1\}}{\theta \,\log\!\big(1+\exp(\tfrac{\theta}{2})\big)^\frac{1}{1-\tau_A}}\int_\mathbb{I}[-\log z]^\frac{1}{1-\tau_A} \mathrm{d}\lambda(z) \;=\; c_1(\theta,\tau_A)\,\Gamma\!\left(\tfrac{2-\tau_A}{1-\tau_A}\right) < \infty,
\end{align*}
where $\Gamma$ denotes the $\Gamma$-function and
$$
c_1(\theta,\tau_A) := \frac{\big(1-\exp(-\theta)\big)\max\{\exp(\theta),1\}}{\theta \,\log\!\big(1+\exp(\tfrac{\theta}{2})\big)^\frac{1}{1-\tau_A}}.
$$\\
(iii) If $\psi_\theta$ is the generator of a Gumbel copula and $\theta \in [1,\infty)$, applying change of coordinates we have that
\begin{align*}
    \mathbb{E}[Z_\theta] &= \int_\mathbb{I}\left[\frac{-\log t}{\log 2}\right]^\frac{\theta}{1-\tau_A} \mathrm{d}\lambda(t) = \frac{\Gamma(\tfrac{\theta}{1-\tau_A}+1)}{\left[\log 2\right]^\frac{\theta}{1-\tau_A}} < \infty.
\end{align*}
(iv) If $\psi_\theta$ is the generator of a Joe copula, where $\theta \in [1,\infty)$, applying change of coordinates, we obtain that
\begin{align*}
\mathbb{E}[Z_\theta] &= \int_\mathbb{I} \left[\dfrac{\log\!\left(1-(1-t)^{\theta}\right)}{\log\!\left(1-2^{-\theta}\right)}\right]^\frac{1}{1-\tau_A} \mathrm{d}\lambda(t) \\&\leq
    \frac{1}{[-\log\!\left(1-2^{-\theta}\right)]^\frac{1}{1-\tau_A}}\int_\mathbb{I} [-\log\!\left(1-(1-t)^{\theta}\right)]^\frac{1}{1-\tau_A} \mathrm{d}\lambda(t) \\&=
    c_2(\theta,\tau_A)\int_\mathbb{I} [-\log(z)]^\frac{1}{1-\tau_A} (1-z)^{\frac{1}{\theta}-1} \mathrm{d}\lambda(z),
\end{align*}
where $c_2(\theta,\tau_A) := \frac{1}{\theta[-\log\!\left(1-2^{-\theta}\right)]^\frac{1}{1-\tau_A}}$. Considering the remaining integral and proceeding as in the previous case, we further obtain:
\begin{align*}
    \int_\mathbb{I} [-\log(z)]^\frac{1}{1-\tau_A} (1-z)^{\frac{1}{\theta}-1} \mathrm{d}\lambda(z) &= \int_{[0,\frac{1}{2}]} [-\log(z)]^\frac{1}{1-\tau_A} (1-z)^{\frac{1}{\theta}-1} \mathrm{d}\lambda(z) + \int_{[\frac{1}{2},1]} [-\log(z)]^\frac{1}{1-\tau_A} (1-z)^{\frac{1}{\theta}-1} \mathrm{d}\lambda(z) \\&\leq
    2^{1-\frac{1}{\theta}}\int_{[0,\frac{1}{2}]} [-\log(z)]^\frac{1}{1-\tau_A} \mathrm{d}\lambda(z) + \int_{[\frac{1}{2},1]}(1-z)^\frac{1+\tau_A(\theta -1)}{\theta(1-\tau_A)}z^{-\frac{1}{1-\tau_A}} \mathrm{d}\lambda(z) < \infty.
\end{align*}
(v) We show that $\mathbb{E}[|\log Z_\theta|] < \infty$, when $\psi_\theta$ is the Clayton generator. We have that
$$
\int_\mathbb{I} |\log(\varphi_\theta(t))|\mathrm{d}\lambda(t) \leq \int_{\mathbb{I}} |\log(t^{-\theta}-1)|\mathrm{d}\lambda(t) + c_3
$$
for $c_3 := -\log(2^\theta - 1)$. Furthermore it holds that
$$
\int_{\mathbb{I}} |\log(t^{-\theta}-1)|\mathrm{d}\lambda(t) \leq \theta \int_\mathbb{I}|\log(t)| \mathrm{d}\lambda(t) + \int_\mathbb{I}|\log(1-t^\theta)| \mathrm{d}\lambda(t).
$$
Obviously, the first integral on the right-hand side is finite. Moreover, using change of coordinates and distinguishing the cases $\theta \in (0,1)$ and $\theta \in [1,\infty)$ we have
$$
 \int_\mathbb{I}|\log(1-t^\theta)| \mathrm{d}\lambda(t) = \frac{1}{\theta}\int_\mathbb{I} \frac{|\log(z)|}{(1-z)^{1-\frac{1}{\theta}}} \mathrm{d}\lambda(z) < \infty.
$$
The remaining cases follow analogously. For each fixed parameter value
$\theta$, the function $|\log \varphi_\theta(t)|$ exhibits at most
logarithmic singularities at the endpoints of $\mathbb{I}$, and such
singularities are integrable with respect to $\lambda$. Hence,
$\mathbb{E}[|\log Z_\theta|]<\infty$ for all considered families.
\end{proof}
Proving strong consistency in Theorem \ref{thm:pick_est_empirical}, we used 
the following regularity assumptions. 
Assumption \ref{assumption1_pick} ensures consistency for the Pickands type estimator,  Assumptions \ref{assumption1_cfg} and \ref{assumption2_cfg} for the CFG type estimator.
\begin{assump}\label{assumption1_pick}
 $$
 \inf_{\delta > 0}\sup_{n \in \mathbb{N}}\int_{[0,\delta]}\varphi_n(t)\mathrm{d}\lambda(t) = 0;
 $$
\end{assump}
\begin{assump}\label{assumption1_cfg}
 $$
 \inf_{\delta > 0}\sup_{n \in \mathbb{N}}\int_{[0,\delta]}\log(\varphi_n(t)) \mathrm{d}\lambda(t) = 0;
 $$
\end{assump}
\begin{assump}\label{assumption2_cfg}
         \begin{align*}
        \inf_{\delta > 0} \sup_{n \in \mathbb{N}}\int_{[1-\delta,1]}|\log(\varphi_n(t))| \mathrm{d}\lambda(t) = 0.
     \end{align*}
\end{assump}
Proving Theorem \ref{thm:pick_est_empirical}, the following lemma will be key:
\begin{lemma}\label{lem:implication_regularity}
    Let  $(\psi_n)_{n \in \mathbb{N}}$ converge uniformly to $(\psi)_{1-\tau_A}$. If Assumption \ref{assumption1_pick} holds, then
    $$
    \inf_{m > 0}\sup_{n \in \mathbb{N}}\int_{[m,\infty)} \psi_n(x) \mathrm{d}\lambda(x) = 0.
    $$
    If Assumption  \ref{assumption1_cfg} holds, then
     $$
 \inf_{m > 0}\sup_{n \in \mathbb{N}}\int_{[m,\infty)}\frac{\psi_n(z)}{z} \mathrm{d}\lambda(t) = 0.
 $$
     If Assumption  \ref{assumption2_cfg} holds, then
      $$
 \inf_{\delta > 0}\sup_{n \in \mathbb{N}}\int_{[0,\delta]}\frac{1-\psi_n(z)}{z} \mathrm{d}\lambda(t) = 0.
 $$
\end{lemma}
\begin{proof}
    We assume that $(\psi_n)_{n \in \mathbb{N}}$ converges uniformly to $(\psi)_{1-\tau_A}$ and prove the first implication. Let 
$\delta_0 \in (0,1]$ be arbitrary but fixed. Then Cavalieri implies  
\begin{align*}
   \int_{[\varphi_n(\delta_0),\infty)} \psi_n(z) d\lambda(z) &= 
   \int_{(0,\delta_0]} \lambda\left(\{t \geq \varphi_n(\delta_0): \psi_n(t) \geq x\} \right) 
   d\lambda(x) = \int_{(0,\delta_0]}  \varphi_n(x)-\varphi_n(\delta_0) d\lambda(x) \\
   &= \int_{(0,\delta_0]}  \varphi_n(x) d\lambda(x) - \delta_0 \varphi_n(\delta_0).
\end{align*}
Considering $\lim_{n \rightarrow \infty} \varphi_n(\delta_0)=(\varphi)_{1-\tau_A}(\delta_0) >0$, 
we can find some $M> (\varphi)_{1-\tau_A}(\delta_0)$ such that for every $n \in \mathbb{N}$ we have
$\varphi_n(\delta_0) < M$, implying 
\begin{align*}
  \int_{[M,\infty)} \psi_n(z) d\lambda(z) & \leq  \int_{[\varphi_n(\delta_0),\infty)} \psi_n(z) d\lambda(z) + \delta_0 \varphi_n(\delta_0)  = \int_{(0,\delta_0]}  \varphi_n(x) d\lambda(x),
\end{align*}
so in particular
\begin{align*}
\sup_{n \in \mathbb{N}}  \int_{[M,\infty)} \psi_n(z) d\lambda(z) & \leq  
\sup_{n \in \mathbb{N}} \int_{(0,\delta_0]}  \varphi_n(x) d\lambda(x).
\end{align*}
We prove the second implication and assume that Assumption \ref{assumption1_cfg} holds. For $\delta_0 \in (0,\frac{1}{2})$, applying Cavalieri's principle and change of coordinates 
we have that
\begin{align*}
\int_{[\varphi_n(\delta_0),\infty)}\frac{\psi_n(z)}{z} \mathrm{d}\lambda(z) &= 
\int_{[\log(\varphi_n(\delta_0)),\infty)} \psi_n(\exp(s)) \mathrm{d}\lambda(s) = 
\int_{[0,\delta_0]} \lambda(\{t \in [\log(\varphi_n(\delta_0),\infty) \colon \psi_n(\exp(t)) >z\}) \mathrm{d}\lambda(z) \\&=
\int_{[0,\delta_0]}[\log(\varphi_n(z)) - \log(\varphi_n(\delta_0))]\mathrm{d}\lambda(z) = 
\int_{[0,\delta_0]}\log(\varphi_n(z))\mathrm{d}\lambda(z) - \delta_0\log(\varphi_n(\delta_0))
\end{align*}
Applying the same arguments as before we have that
$$
\int_{[M,\infty)}\frac{\psi_n(z)}{z} \mathrm{d}\lambda(z) \leq  \int_{[\varphi_n(\delta_0),\infty)}\frac{\psi_n(z)}{z} \mathrm{d}\lambda(z) + \delta_0\log(\varphi_n(\delta_0))= \int_{[0,\delta_0]}\log(\varphi_n(z))\mathrm{d}\lambda(z),
$$
implying that
$$
\sup_{n \in \mathbb{N}} \int_{[M,\infty)}\frac{\psi_n(z)}{z} \mathrm{d}\lambda(z) \leq \sup_{n \in \mathbb{N}}\int_{[0,\delta_0]}\log(\varphi_n(z))\mathrm{d}\lambda(z).
$$
Assuming \ref{assumption2_cfg}, we proceed similarly and observe
\begin{align*}
    \int_{[0,\varphi_n(1-\delta_0)]}\frac{1-\psi_n(z)}{z} \mathrm{d}\lambda(z) &= 
\int_{(-\infty,\log(\varphi_n(1-\delta_0))]} [1- \psi_n(\exp(s))] \mathrm{d}\lambda(s) \\&= \int_{[0,\delta_0]} \lambda(\{t \in (-\infty,\log(\varphi_n(1-\delta_0))] \colon \log(\varphi_n(1-z)) \leq t\}) \mathrm{d}\lambda(z) \\&=
-\delta_0|\log(\varphi_n(1-\delta_0))| + \int_{[1-\delta_0,1]}|\log(\varphi_n(s))| \mathrm{d}\lambda(s).
\end{align*}
Again, we can find some $\varepsilon > 0$ such that for every $n \in \mathbb{N}$ we have $\varepsilon < \varphi_n(1-\delta_0)$ and therefore
$$
\int_{[0,\varepsilon]}\frac{1-\psi_n(z)}{z} \mathrm{d}\lambda(z)\leq \int_{[0,\varphi_n(1-\delta_0)]}\frac{1-\psi_n(z)}{z} \mathrm{d}\lambda(z) + \delta_0|\log(\varphi_n(1-\delta_0))| = \int_{[1-\delta_0,1]}|\log(\varphi_n(s))| \mathrm{d}\lambda(s).
$$
This proves the result.
\end{proof}
The next remarks discuss the regularity conditions in eqs. \eqref{assumption1_pick}, \eqref{assumption1_cfg} and \eqref{assumption2_cfg}:
\begin{remark}
Assumption~\ref{assumption2_cfg} is not overly restrictive. We justify this from two angles.\\
\noindent\textbf{(i) A single generator.} Let $\varphi$ be the pseudo-inverse of a generator $\psi \in \Psi$ satisfying $D^-\varphi(1)<0$. By convexity of $\varphi$, we have that
$$
\varphi(t) \geq (t-1)\,D^-\varphi(1), \qquad t \in \mathbb I.
$$
For $\delta \leq \tfrac12$ we have $\varphi(t)\leq \varphi(1/2)=1$ for every $t \in [1-\delta,1]$ by normalization, so $|\log\varphi(t)| = -\log\varphi(t)$ there, and the bound above gives
\begin{align*}
\int_{[1-\delta,1]} |\log \varphi(t)|\, \mathrm d\lambda(t)
&\leq - \int_{[1-\delta,1]} \log\big((t-1)\,D^-\varphi(1)\big)\, \mathrm d\lambda(t) \\
&= - \int_{[1-\delta,1]} \log(1-t)\, \mathrm d\lambda(t) \;-\; \delta \log\big({-}D^-\varphi(1)\big) \\
&= -\delta\log\delta + \delta - \delta \log\big({-}D^-\varphi(1)\big)
\;\overset{\delta \to 0^+}{\longrightarrow}\; 0.
\end{align*}
In particular, this applies whenever $\varphi$ is the pseudo-inverse of a piecewise linear generator $\psi$: in that case $\varphi$ itself is piecewise linear and fulfills $D^-\varphi(1)<0$. The argument above then yields
$$
\lim_{\delta \to 0^+} \int_{[1-\delta,1]} |\log\varphi(t)|\, \mathrm d\lambda(t) = 0
$$
for every piecewise linear generator, with no further assumption needed.\\
\noindent\textbf{(ii) A converging sequence of piecewise linear generators.}
Now let $(\psi_n)_{n\in\mathbb N}$ be a sequence of piecewise linear generators converging uniformly to a generator $\psi\in\Psi$, with corresponding pseudo-inverses $\varphi_n$ and $\varphi$. Suppose further that $\limsup_{n\to\infty}D^-\varphi_n(1)<0$. Since each $\varphi_n$ is piecewise linear, this condition already forces the stronger statement
$$
a:=\sup_{n \in \mathbb{N}}  D^-\varphi_n(1) <0.
$$
We use this to verify Assumption~\ref{assumption2_cfg}. First, taking suprema directly gives
$$
\sup_{n\in\mathbb N} \big[-\log(-D^-\varphi_n(1))\big]
= -\log\Big({-}\sup_{n\in\mathbb N} D^-\varphi_n(1)\Big) = -\log(-a).
$$

Next, convexity of $\varphi_n$ yields the linear lower bound
$$
\varphi_n(t) \geq (t-1)\,D^-\varphi_n(1), \qquad t \in \mathbb I,\ n \in \mathbb N.
$$

Integrating this bound over $[1-\delta,1]$ as in (i), and applying the uniform estimate above, gives
\begin{align*}
\sup_{n\in\mathbb N} \int_{[1-\delta,1]} |\log \varphi_n(t)|\, \mathrm d\lambda(t)
&\leq \int_{[1-\delta,1]} -\log(1-t)\, \mathrm d\lambda(t)
\;+\; \delta \sup_{n\in\mathbb N} \big[-\log(-D^-\varphi_n(1))\big] \\
&= -\delta\log\delta + \delta
+ \delta \sup_{n\in\mathbb N} \big[-\log(-D^-\varphi_n(1))\big]
\;\overset{\delta \to 0^+}{\longrightarrow}\; 0,
\end{align*}
which establishes Assumption~\ref{assumption2_cfg}.
\end{remark}
\begin{remark}
    For Assumption \ref{assumption1_cfg} the situation is more ambivalent: 
    On the one hand, considering that each $\varphi_n$ is piecewise linear,
    we obviously have that 
    $$
    \lim_{\delta \rightarrow 0+} \int_{[0,\delta]} \log \varphi_n(t) \mathrm{d} \lambda(t)=0
    $$
    and the same holds for every non-strict $\varphi$. On the other hand, for every
    $n$ with positive probability we may be in the situation that $W_i$, 
    defined according to eq. \eqref{eq:pseudo_sample} fulfills    
    $W_i=\frac{i-1}{n+1}$ for every $i \in \{1,\ldots,n\}$. 
    In this case we have $K_n=\underline{K}_n$ with 
    $\underline{K}_n$ defined according to eq. \eqref{eq:lower.bound.K_n}, so 
    for $t \in [0,\frac{1}{2}]$ using Fubini's theorem we get   
    \begin{align*}
        \int_{[0,\frac{1}{2}]} \log \varphi_n(t) \mathrm{d} \lambda(t) 
        &=  \int_{[0,\frac{1}{2}]} \int_{[t,\frac{1}{2}]} 
        \frac{1}{\underline{K}_n(s)-s} \, \mathrm{d} \lambda(s) \,\,
        \mathrm{d} \lambda(t) \\&
        = \int_{[0,\frac{1}{2}]} \int_{[0,\frac{1}{2}]} \mathbf{1}_{[t,\frac{1}{2}]}(s) 
        \frac{1}{\underline{K}_n(s)-s} \, \mathrm{d} \lambda(s) \,\,
        \mathrm{d} \lambda(t) \\
        &= \int_{[0,\frac{1}{2}]} \frac{1}{\underline{K}_n(s)-s} \,  \int_{[0,\frac{1}{2}]} \mathbf{1}_{[t,\frac{1}{2}]}(s) 
        \,\mathrm{d} \lambda(t) \, \mathrm{d} \lambda(s) \,\,
         \\& = \int_{[0,\frac{1}{2}]} \frac{s}{\underline{K}_n(s)-s} \,\, \mathrm{d} \lambda(s) 
    \end{align*}
     Considering that for every $s \in [0,1]$ we have 
    $\underline{K}_n(s)-s \leq \frac{2}{n+1}$, it follows that
    \begin{align*}
        \int_{[0,\frac{1}{2}]} \frac{s}{\underline{K}_n(s)-s} \,\, \mathrm{d} \lambda(s) \geq \frac{n+1}{2} \int_{[0,\frac{1}{2}]} s \, \mathrm{d} \lambda(s)
        = \frac{n+1}{16},
    \end{align*}
    implying 
    \begin{align*}
       \sup_{n \in \mathbb{N}}  \int_{[0,\frac{1}{2}]} \log \varphi_n(t) \mathrm{d} \lambda(t) 
        &\geq  \sup_{n \in \mathbb{N}} \frac{n+1}{16}=\infty.
    \end{align*}
    Since replacing $[0,\frac{1}{2}]$ with $[0,\delta]$ for 
    some fixed $\delta \in (0,\frac{1}{2})$ 
    yields a similar result, this shows that in this situation Assumption \ref{assumption1_cfg} does not hold.
\end{remark}
\subsection{Violation of $2$-Increasingness of $D_n^\mathbf{P}$ and $D_n^\mathbf{CFG}$ } \label{subsec:discussion_2_increasing}
In Section \ref{sec:estimator} we noted that $D_n^\mathbf{P}$ and $D_n^\mathbf{CFG}$ as defined in eq. \eqref{eq:D_n}, are not copulas in general. 
In fact, contrary to the regularized copula estimators $C_{\alpha_n,n}^\mathbf{P}$ and $C_{\alpha_n,n}^\mathbf{CFG}$ defined in eq. \eqref{eq:C_alpha_n}, which are real copulas, 
these functions even fail to be distribution functions since 
they violate the $2$-increasingness property:
In fact, for $\Xi \in \{\mathbf{P}, \mathbf{CFG}\}$ and fixed $h > 0$, 
defining the quantities $\Delta_{D_n^\Xi}^h(x,y)$ and 
$\Delta_{C_{\alpha_n,n}^\Xi}^h(x,y)$ for every $(x,y) \in \mathbb{I}$ by
$$
 \Delta_{D_n^\Xi}^h(x,y) := D_n^\Xi(x + h,y + h) - D_n^\Xi(x + h,y) - D_n^\Xi(x,y + h) + D_n^\Xi(x,y)
$$
and
$$
 \Delta_{C_{\alpha_n,n}^\Xi}^h(x,y) := C_{\alpha_n,n}^\Xi(x + h,y + h) - C_{\alpha_n,n}^\Xi(x + h,y) - C_{\alpha_n,n}^\Xi(x,y + h) + C_{\alpha_n,n}^\Xi(x,y).
$$
It is straightforward to verify that $\Delta_{D_n^\Xi}^h$ and 
$\Delta_{C_{\alpha_n,n}^\Xi}^h$ can indeed attain negative values. 
For illustrative purposes, we drew a sample of size $n = 5$ from an Archimax copula 
(with generator and Pickands dependence function mentioned 
in Figure \ref{fig:rect_mon_viol}) and computed $\Delta_{D_n^\Xi}^h$ and $\Delta_{C_{\alpha_n,n}^\Xi}^h$ for $h = 0.05$ over a fine grid. As depicted 
in Figure \ref{fig:rect_mon_viol}, $\Delta_{D_n^\Xi}^h$ indeed becomes negative at several grid points, while $\Delta_{C_{\alpha_n,n}^\Xi}^h(x,y) \geq 0$ holds everywhere.
\begin{figure}[!ht]
	\centering
	\includegraphics[width=0.8\textwidth]{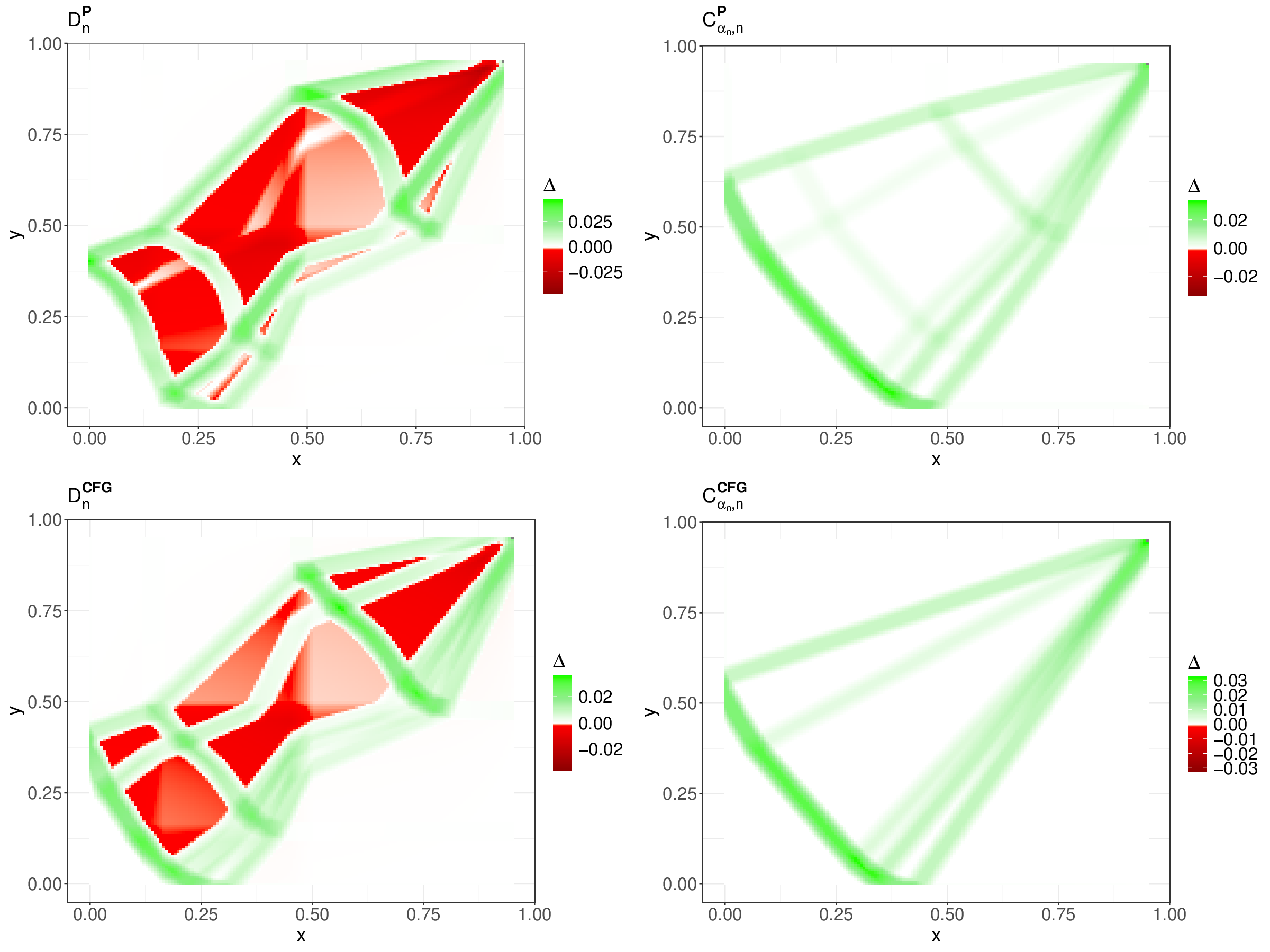}
	\caption{Heatmaps of $\Delta_{D_n^\Xi}^h$ and $\Delta_{C_{\alpha_n,n}^\Xi}^h$ for $\Xi \in \{\mathbf{P}, \mathbf{CFG}\}$, with sample size $n = 5$ and $h = 0.05$. The underlying Archimax copula has a Clayton generator with parameter $\theta = \frac{1}{5}$ and the Pickands dependence function from Figure \ref{fig:discrete.approx}. The upper panel depicts the Pickands-based estimators, the lower panel shows the CFG-based estimators.}
	\label{fig:rect_mon_viol}
\end{figure}
\section{Results of the simulation study}\label{sec:sim_study_app}
This section gathers the results of the simulation study evaluating the estimators proposed in Section~\ref{sec:estimator}; the results are discussed in detail in Section~\ref{sec:sim_study}. 
Aiming for well structured summaries, we proceed in three steps, first examining the estimators of $(A)_{1-\tau_A}$, then the estimators of $(\beta)_{1-\tau_A}$, and finally the estimators of the copula $C_{\psi,A}$.
\begin{figure}[!ht]
	\centering
	\includegraphics[width=1\textwidth]{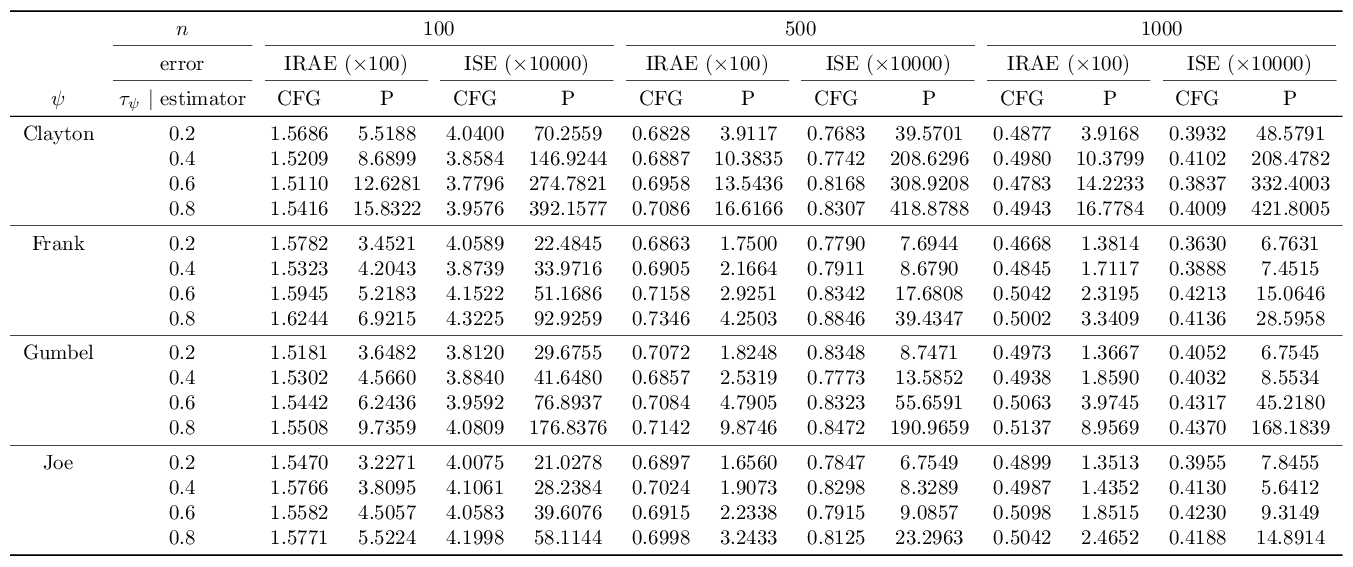}
	\caption{Average performance of the CFG-based estimator $(A_{\alpha_n,n}^\mathbf{CFG})_{1-\tau_{A_{\alpha_n,n}^\mathbf{CFG}}}$ and the Pickands-based estimator $(A_{\alpha_n,n}^\mathbf{P})_{1-\tau_{A_{\alpha_n,n}^\mathbf{P}}}$,  measured via the integrated relative absolute error IRAE ($\times 100$) and the integrated squared error ISE ($\times 10000$), based on samples drawn from a $2$-dimensional Archimax copula $C_{\psi,A}$ with sample sizes $n \in \{100, 500, 1000\}$. The Pickands dependence function $A$ follows the Gumbel family with parameter $\alpha = \frac{5}{3}$, chosen such that $\tau_A = \frac{2}{5}$. Four Archimedean generators $\psi$ are considered -- Clayton, Frank, Gumbel and Joe -- each for four values of Kendall's $\tau$, $\tau_\psi \in \{\tfrac{1}{5}, \tfrac{2}{5}, \tfrac{3}{5}, \tfrac{4}{5}\}$. Reported values are averages over $1000$ Monte Carlo repetitions.}\label{fig:cfg_vs_pick_gumbel_0.4}
\end{figure}
\begin{figure}[!ht]
	\centering
	\includegraphics[width=1\textwidth]{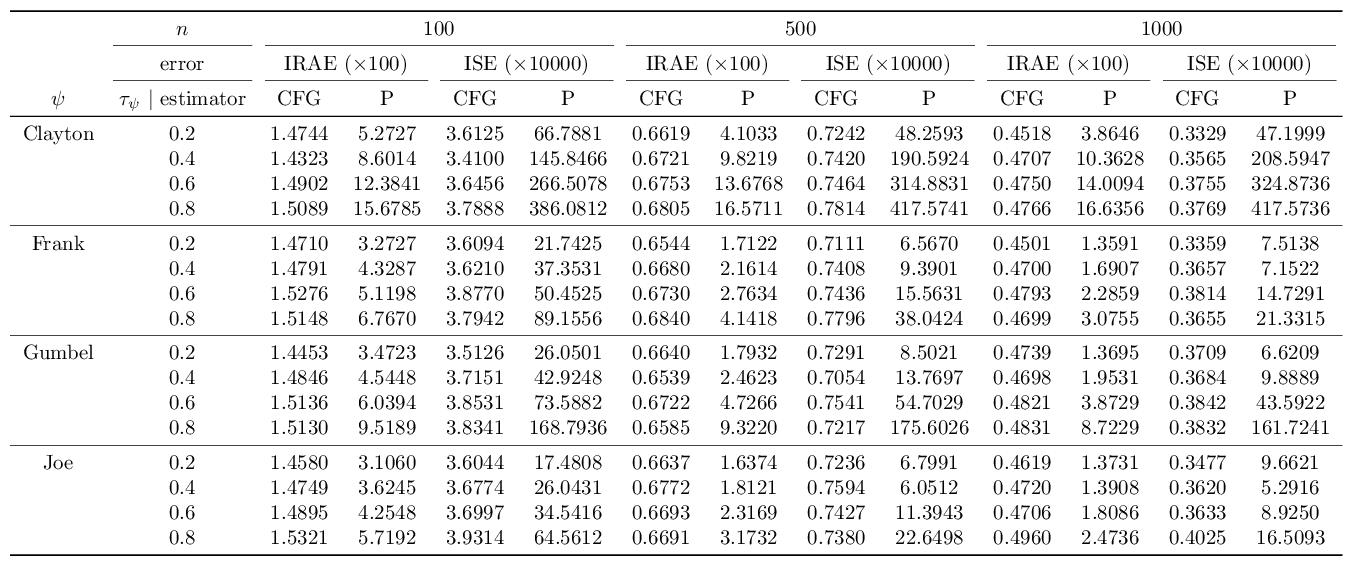}
	\caption{Average performance of the CFG-based estimator $(A_{\alpha_n,n}^\mathbf{CFG})_{1-\tau_{A_{\alpha_n,n}^\mathbf{CFG}}}$ and the Pickands-based estimator $(A_{\alpha_n,n}^\mathbf{P})_{1-\tau_{A_{\alpha_n,n}^\mathbf{P}}}$,  measured via the integrated relative absolute error IRAE ($\times 100$) and the integrated squared error ISE ($\times 10000$), based on samples drawn from a $2$-dimensional Archimax copula $C_{\psi,A}$ with sample sizes $n \in \{100, 500, 1000\}$. The Pickands dependence function $A$ follows the Galambos family with parameter $\alpha = 0.9460773$, chosen such that $\tau_A = \frac{2}{5}$. Four Archimedean generators $\psi$ are considered -- Clayton, Frank, Gumbel and Joe -- each for four values of Kendall's $\tau$, $\tau_\psi \in \{\tfrac{1}{5}, \tfrac{2}{5}, \tfrac{3}{5}, \tfrac{4}{5}\}$. Reported values are averages over $1000$ Monte Carlo repetitions.}\label{fig:cfg_vs_pick_galambos_0.4}
\end{figure}
\begin{figure}[!ht]
	\centering
	\includegraphics[width=1\textwidth]{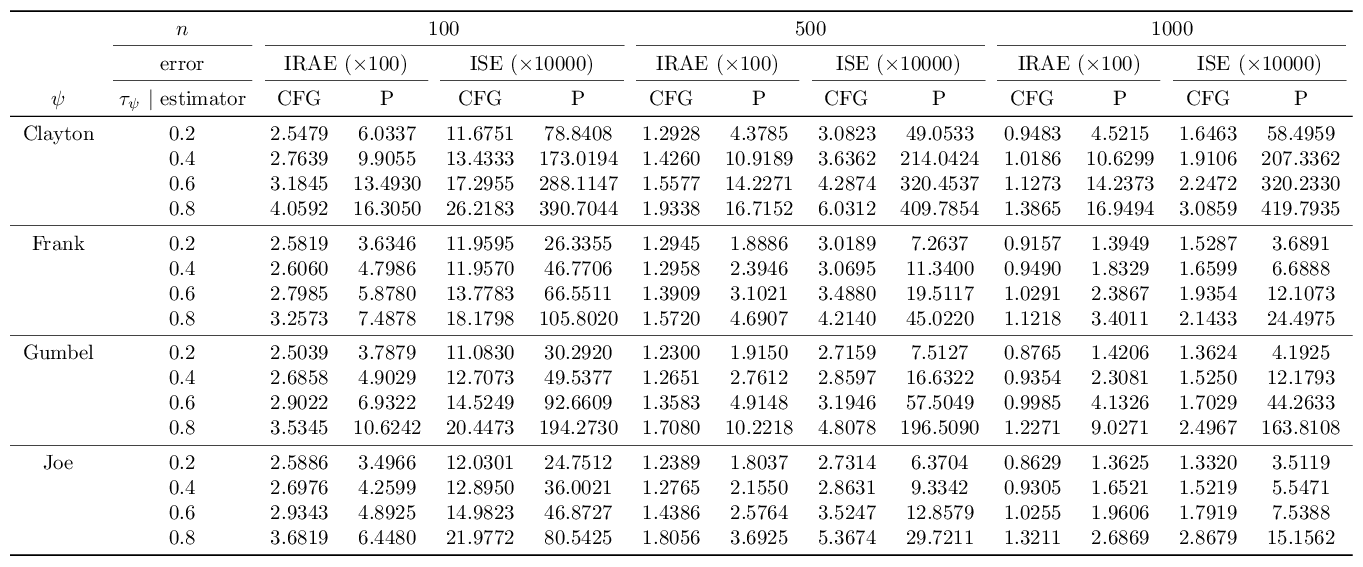}
	\caption{Average performance of the CFG-based estimator $(A_{\alpha_n,n}^\mathbf{CFG})_{1-\tau_{A_{\alpha_n,n}^\mathbf{CFG}}}$ and the Pickands-based estimator $(A_{\alpha_n,n}^\mathbf{P})_{1-\tau_{A_{\alpha_n,n}^\mathbf{P}}}$,  measured via the integrated relative absolute error IRAE ($\times 100$) and the integrated squared error ISE ($\times 10000$), based on samples drawn from a $2$-dimensional Archimax copula $C_{\psi,A}$ with sample sizes $n \in \{100, 500, 1000\}$. The Pickands dependence function $A$ follows the Marshall-Olkin family with parameters $\alpha_1 = \frac{3}{7}$ and $\alpha_2 = \frac{6}{7}$, chosen such that $\tau_A = \frac{2}{5}$. Four Archimedean generators $\psi$ are considered -- Clayton, Frank, Gumbel and Joe -- each for four values of Kendall's $\tau$, $\tau_\psi \in \{\tfrac{1}{5}, \tfrac{2}{5}, \tfrac{3}{5}, \tfrac{4}{5}\}$. Reported values are averages over $1000$ Monte Carlo repetitions.}\label{fig:cfg_vs_pick_mar_ol_0.4}
\end{figure}
\begin{figure}[!ht]
	\centering
	\includegraphics[width=0.8\textwidth]{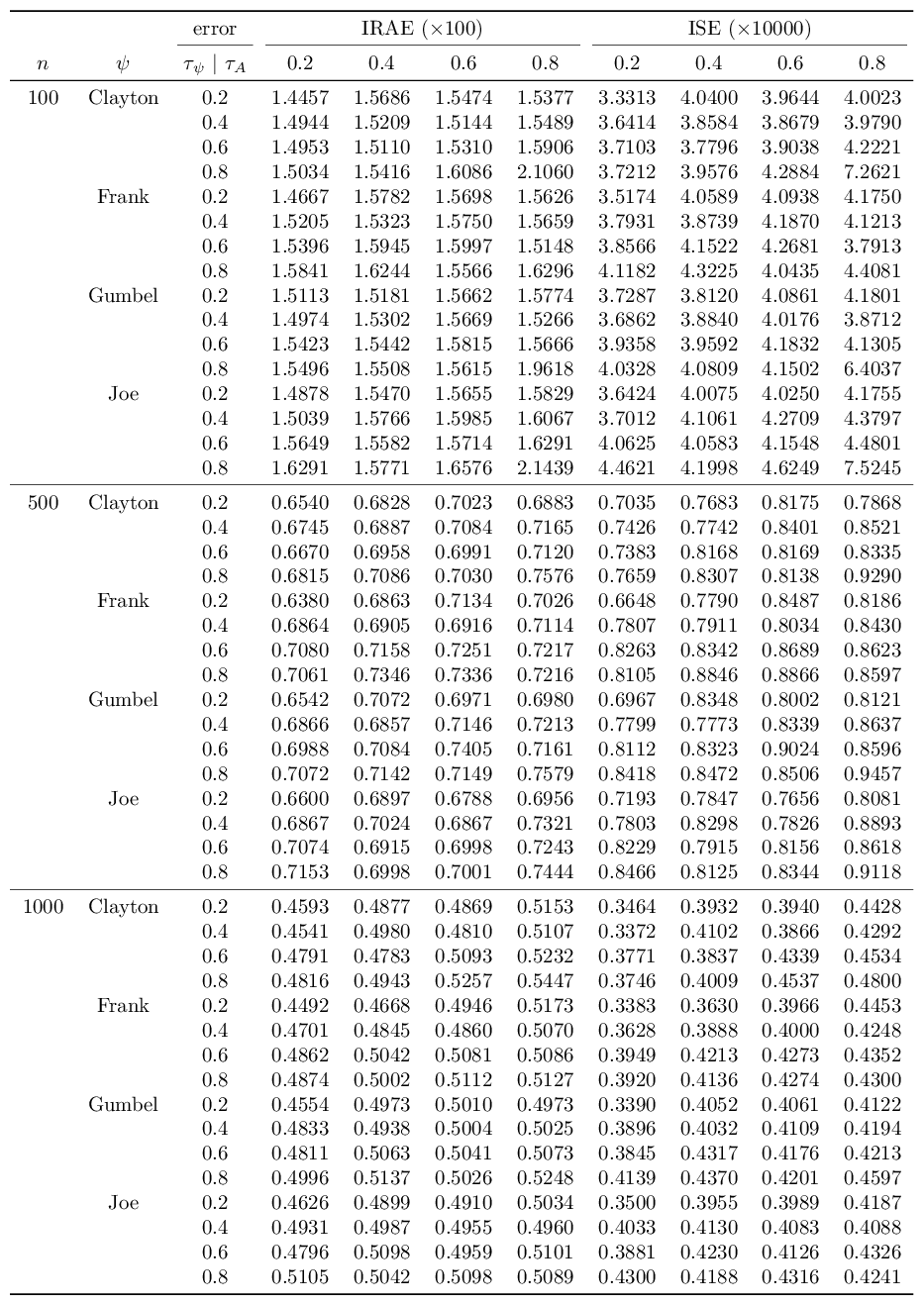}
	\caption{Average performance of the CFG-based estimator $(A_{\alpha_n,n}^\mathbf{CFG})_{1-\tau_{A_{\alpha_n,n}^\mathbf{CFG}}}$,  measured via the integrated relative absolute error IRAE ($\times 100$) and the integrated squared error ISE ($\times 10000$), based on samples drawn from a $2$-dimensional Archimax copula $C_{\psi,A}$ with sample sizes $n \in \{100, 500, 1000\}$. The Pickands dependence function $A$ follows the Gumbel family with parameters $\alpha \in [1,\infty)$ chosen such that $\tau_A \in \{\frac{1}{5}, \frac{2}{5}, \frac{3}{5}, \frac{4}{5}\}$. Four Archimedean generators $\psi$ are considered -- Clayton, Frank, Gumbel and Joe -- each for four values of Kendall's $\tau$, $\tau_\psi \in \{\tfrac{1}{5}, \tfrac{2}{5}, \tfrac{3}{5}, \tfrac{4}{5}\}$. Reported values are averages over $1000$ Monte Carlo repetitions.}\label{fig:cfg_vs_pick_gumbel_all_vs_all}
\end{figure}
\begin{figure}[!ht]
	\centering
	\includegraphics[width=0.8\textwidth]{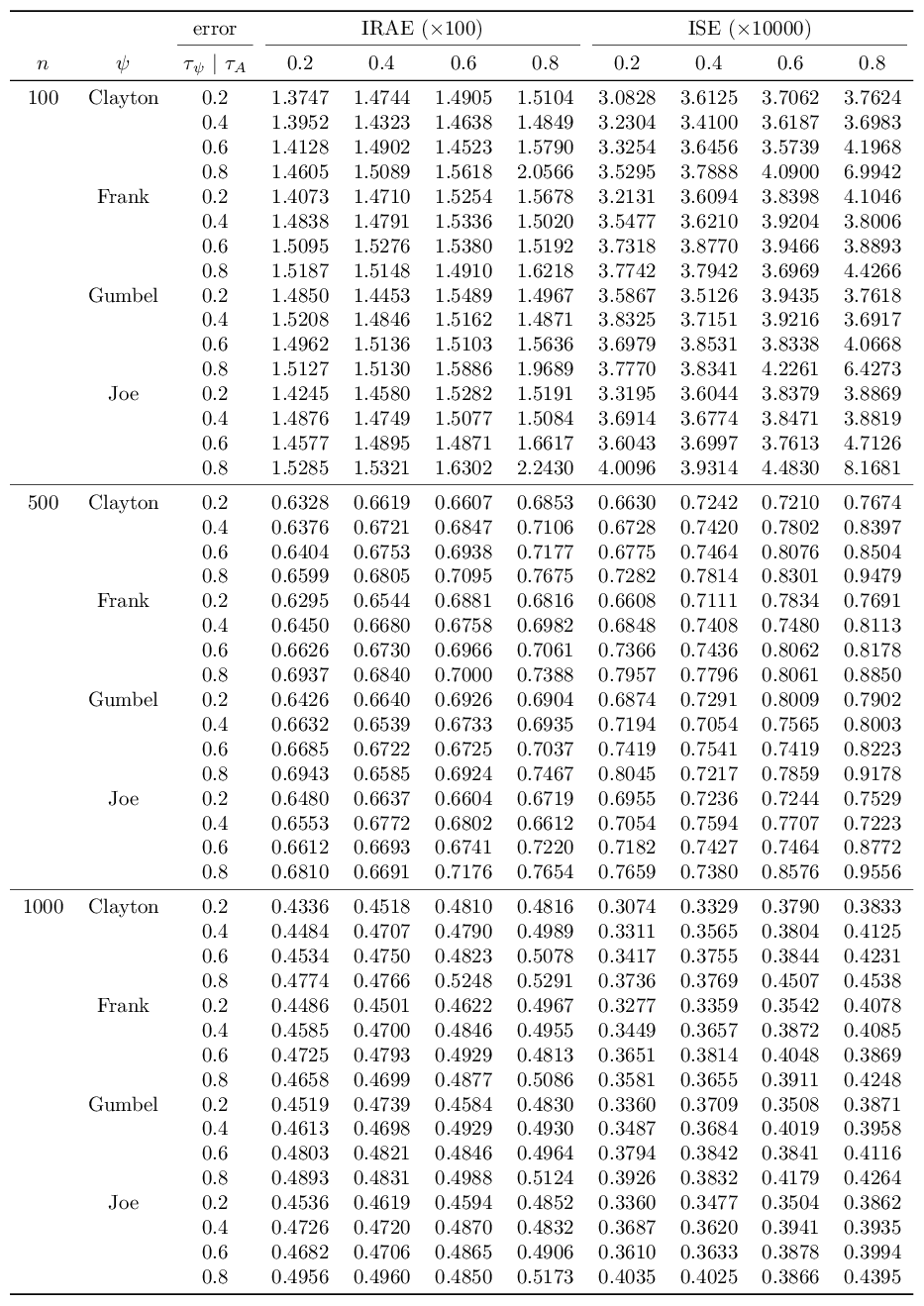}
	\caption{Average performance of the CFG-based estimator $(A_{\alpha_n,n}^\mathbf{CFG})_{1-\tau_{A_{\alpha_n,n}^\mathbf{CFG}}}$,  measured via the integrated relative absolute error IRAE ($\times 100$) and the integrated squared error ISE ($\times 10000$), based on samples drawn from a $2$-dimensional Archimax copula $C_{\psi,A}$ with sample sizes $n \in \{100, 500, 1000\}$. The Pickands dependence function $A$ follows the Galambos family with parameters $\alpha \in (0,\infty)$ chosen such that $\tau_A \in \{\frac{1}{5}, \frac{2}{5}, \frac{3}{5}, \frac{4}{5}\}$. Four Archimedean generators $\psi$ are considered -- Clayton, Frank, Gumbel and Joe -- each for four values of Kendall's $\tau$, $\tau_\psi \in \{\tfrac{1}{5}, \tfrac{2}{5}, \tfrac{3}{5}, \tfrac{4}{5}\}$. Reported values are averages over $1000$ Monte Carlo repetitions.}\label{fig:cfg_vs_pick_galambos_all_vs_all}
\end{figure}
\begin{figure}[!ht]
	\centering
	\includegraphics[width=0.8\textwidth]{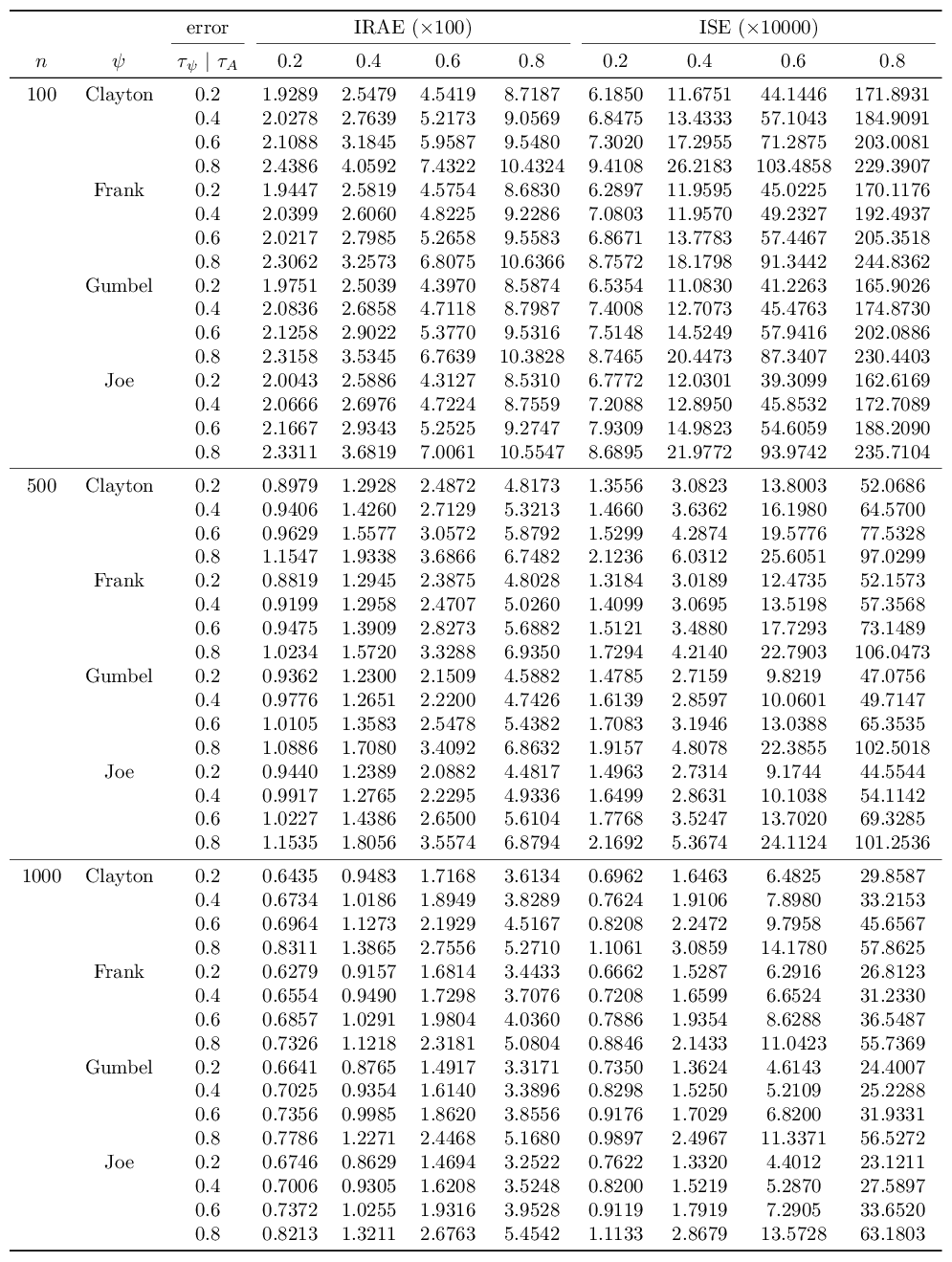}
	\caption{Average performance of the CFG-based estimator $(A_{\alpha_n,n}^\mathbf{CFG})_{1-\tau_{A_{\alpha_n,n}^\mathbf{CFG}}}$,  measured via the integrated relative absolute error IRAE ($\times 100$) and the integrated squared error ISE ($\times 10000$), based on samples drawn from a $2$-dimensional Archimax copula $C_{\psi,A}$ with sample sizes $n \in \{100, 500, 1000\}$. The Pickands dependence function $A$ follows the Marshall-Olkin family with parameters $\alpha_1, \alpha_2 \in (0,1)$ chosen such that $\tau_A \in \{\frac{1}{5}, \frac{2}{5}, \frac{3}{5}, \frac{4}{5}\}$. Four Archimedean generators $\psi$ are considered -- Clayton, Frank, Gumbel and Joe -- each for four values of Kendall's $\tau$, $\tau_\psi \in \{\tfrac{1}{5}, \tfrac{2}{5}, \tfrac{3}{5}, \tfrac{4}{5}\}$. Reported values are averages over $1000$ Monte Carlo repetitions.}\label{fig:cfg_vs_pick_mar_ol_all_vs_all}
\end{figure}
\begin{figure}[!ht]
	\centering
	\includegraphics[width=0.55\textwidth]{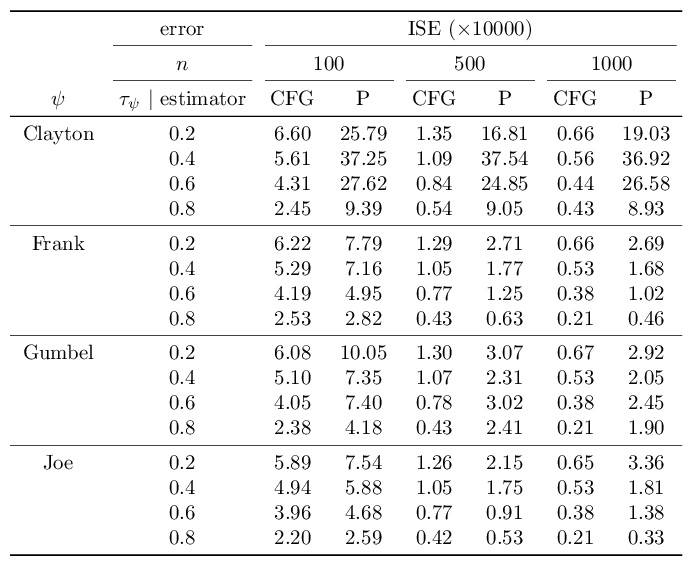}
	\caption{Average performance of the CFG-based estimator $(\beta_{\alpha_n,n}^\mathbf{CFG})_{1-\tau_{A_{\alpha_n,n}^\mathbf{CFG}}}$ and the Pickands-based estimator $(\beta_{\alpha_n,n}^\mathbf{P})_{1-\tau_{A_{\alpha_n,n}^\mathbf{P}}}$, measured via the integrated squared error ISE ($\times 10000$), based on samples drawn from a $2$-dimensional Archimax copula $C_{\psi,A}$ with sample sizes $n \in \{100, 500, 1000\}$. The Pickands dependence function $A$ follows the Gumbel family with parameter $\alpha = \frac{5}{3}$, chosen such that $\tau_A = \tfrac{2}{5}$. Four Archimedean generators $\psi$ are considered -- Clayton, Frank, Gumbel and Joe -- each for four values of Kendall's $\tau$, $\tau_\psi \in \{\tfrac{1}{5}, \tfrac{2}{5}, \tfrac{3}{5}, \tfrac{4}{5}\}$. Reported values are averages over $1000$ Monte Carlo repetitions.}\label{fig:beta_cfg_vs_pick_gumbel_0.4}
\end{figure}
\begin{figure}[!ht]
	\centering
	\includegraphics[width=0.55\textwidth]{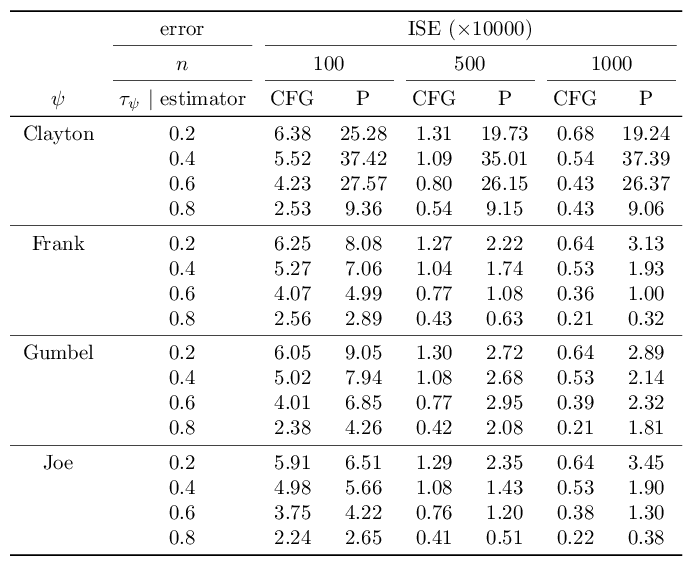}
	\caption{Average performance of the CFG-based estimator $(\beta_{\alpha_n,n}^\mathbf{CFG})_{1-\tau_{A_{\alpha_n,n}^\mathbf{CFG}}}$ and the Pickands-based estimator $(\beta_{\alpha_n,n}^\mathbf{P})_{1-\tau_{A_{\alpha_n,n}^\mathbf{P}}}$, measured via the integrated squared error ISE ($\times 10000$), based on samples drawn from a $2$-dimensional Archimax copula $C_{\psi,A}$ with sample sizes $n \in \{100, 500, 1000\}$. The Pickands dependence function $A$ follows the Galambos family with parameter $\alpha = 0.9460773$, chosen so that $\tau_A = \tfrac{2}{5}$. Four Archimedean generators $\psi$ are considered -- Clayton, Frank, Gumbel and Joe -- each for four values of Kendall's $\tau$, $\tau_\psi \in \{\tfrac{1}{5}, \tfrac{2}{5}, \tfrac{3}{5}, \tfrac{4}{5}\}$. Reported values are averages over $1000$ Monte Carlo repetitions.}\label{fig:beta_cfg_vs_pick_galambos_0.4}
\end{figure}
\begin{figure}[!ht]
	\centering
	\includegraphics[width=0.55\textwidth]{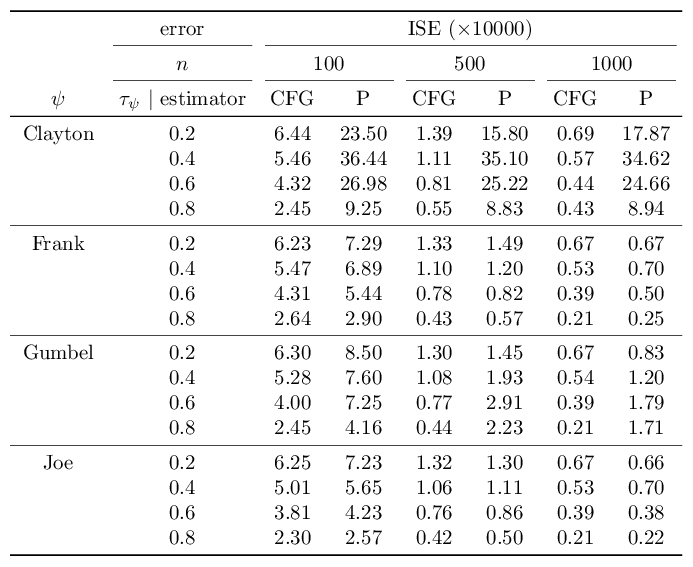}
	\caption{Average performance of the CFG-based estimator $(\beta_{\alpha_n,n}^\mathbf{CFG})_{1-\tau_{A_{\alpha_n,n}^\mathbf{CFG}}}$ and the Pickands-based estimator $(\beta_{\alpha_n,n}^\mathbf{P})_{1-\tau_{A_{\alpha_n,n}^\mathbf{P}}}$, measured via the integrated squared error ISE ($\times 10000$), based on samples drawn from a $2$-dimensional Archimax copula $C_{\psi,A}$ with sample sizes $n \in \{100, 500, 1000\}$. The Pickands dependence function $A$ follows the Marshall-Olkin family with parameters $\alpha_1 = \frac{3}{7}$ and $\alpha_2 = \frac{6}{7}$, chosen such that $\tau_A = \tfrac{2}{5}$. Four Archimedean generators $\psi$ are considered -- Clayton, Frank, Gumbel and Joe -- each for four values of Kendall's $\tau$, $\tau_\psi \in \{\tfrac{1}{5}, \tfrac{2}{5}, \tfrac{3}{5}, \tfrac{4}{5}\}$. Reported values are averages over $1000$ Monte Carlo repetitions.}\label{fig:beta_cfg_vs_pick_mar_ol_0.4}
\end{figure}
\begin{figure}[!ht]
	\centering
	\includegraphics[width=.47\textwidth]{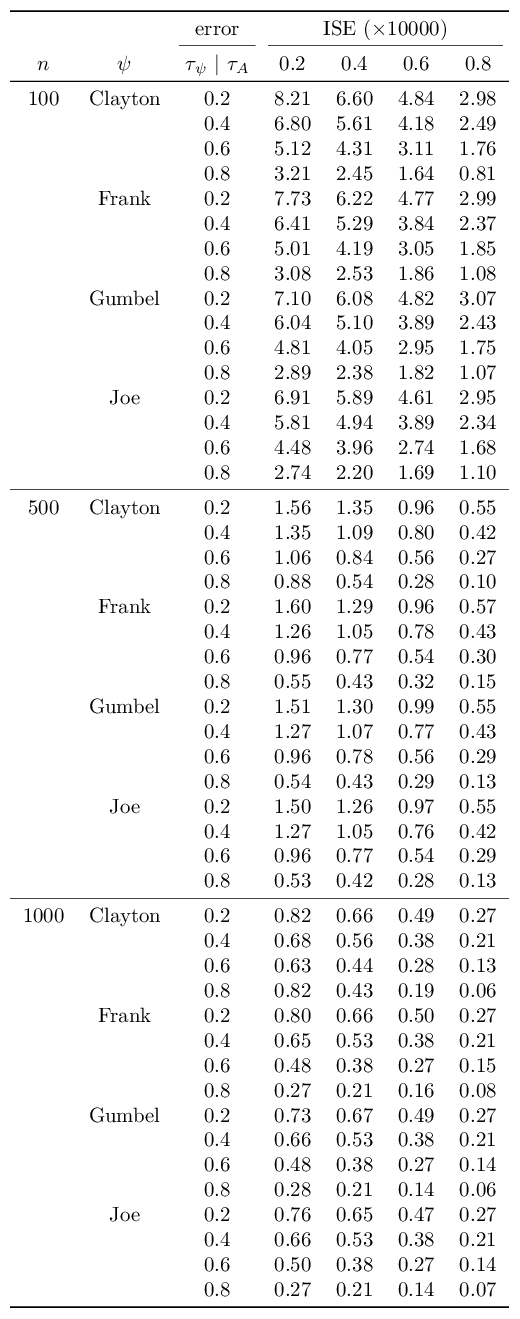}
	\caption{Average performance of the CFG-based estimator $(\beta_{\alpha_n,n}^\mathbf{CFG})_{1-\tau_{A_{\alpha_n,n}^\mathbf{CFG}}}$, measured via the integrated squared error ISE ($\times 10000$), based on samples drawn from a $2$-dimensional Archimax copula $C_{\psi,A}$ with sample sizes $n \in \{100, 500, 1000\}$. The Pickands dependence function $A$ follows the Gumbel family with parameters $\alpha \in [1,\infty)$, chosen such that $\tau_A \in \{\tfrac{1}{5}, \tfrac{2}{5}, \tfrac{3}{5}, \tfrac{4}{5}\}$. Four Archimedean generators $\psi$ are considered -- Clayton, Frank, Gumbel and Joe -- each for four values of Kendall's $\tau$, $\tau_\psi \in \{\tfrac{1}{5}, \tfrac{2}{5}, \tfrac{3}{5}, \tfrac{4}{5}\}$. Reported values are averages over $1000$ Monte Carlo repetitions.}\label{fig:beta_cfg_vs_pick_gumbel_all_vs_all}
\end{figure}
\begin{figure}[!ht]
	\centering
	\includegraphics[width=.47\textwidth]{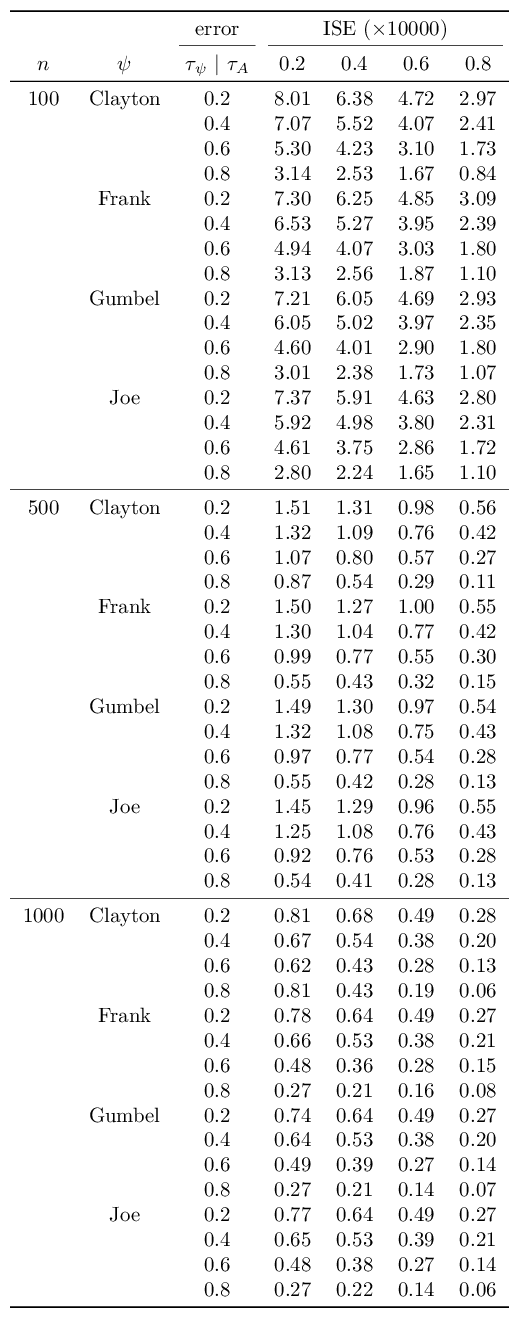}
	\caption{Average performance of the CFG-based estimator $(\beta_{\alpha_n,n}^\mathbf{CFG})_{1-\tau_{A_{\alpha_n,n}^\mathbf{CFG}}}$, measured via the integrated squared error ISE ($\times 10000$), based on samples drawn from a $2$-dimensional Archimax copula $C_{\psi,A}$ with sample sizes $n \in \{100, 500, 1000\}$. The Pickands dependence function $A$ follows the Galambos family with parameters $\alpha \in (0,\infty)$, chosen such that $\tau_A \in \{\tfrac{1}{5}, \tfrac{2}{5}, \tfrac{3}{5}, \tfrac{4}{5}\}$. Four Archimedean generators $\psi$ are considered -- Clayton, Frank, Gumbel and Joe -- each for four values of Kendall's $\tau$, $\tau_\psi \in \{\tfrac{1}{5}, \tfrac{2}{5}, \tfrac{3}{5}, \tfrac{4}{5}\}$. Reported values are averages over $1000$ Monte Carlo repetitions.}\label{fig:beta_cfg_vs_pick_galambos_all_vs_all}
\end{figure}
\begin{figure}[!ht]
	\centering
	\includegraphics[width=.47\textwidth]{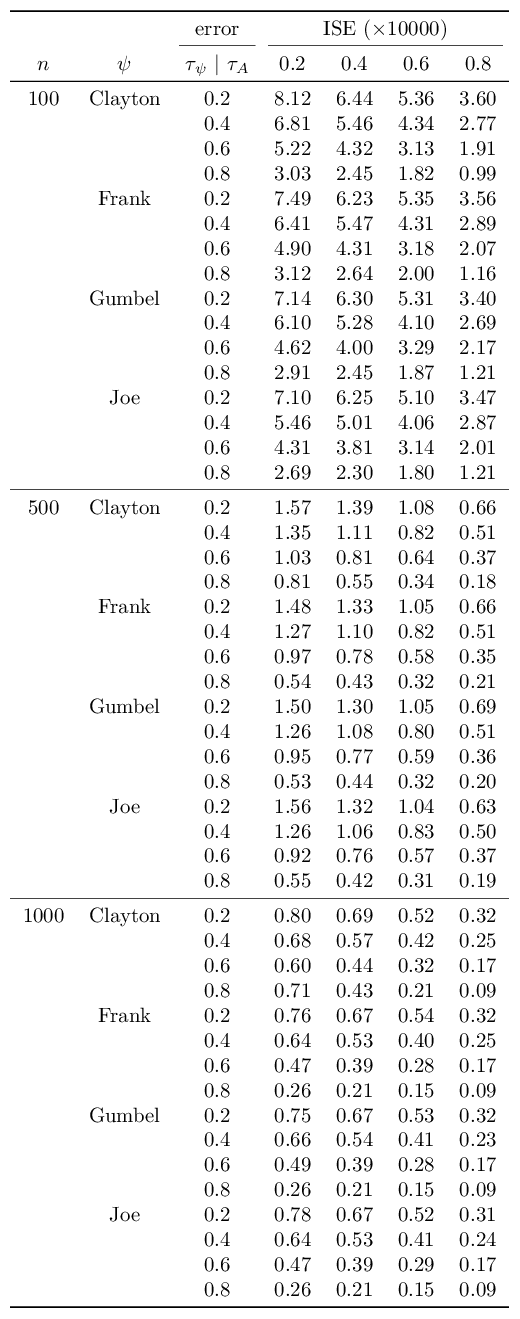}
	\caption{Average performance of the CFG-based estimator $(\beta_{\alpha_n,n}^\mathbf{CFG})_{1-\tau_{A_{\alpha_n,n}^\mathbf{CFG}}}$, measured via the integrated squared error ISE ($\times 10000$), based on samples drawn from a $2$-dimensional Archimax copula $C_{\psi,A}$ with sample sizes $n \in \{100, 500, 1000\}$. The Pickands dependence function $A$ follows the Marshall-Olkin family with parameters $\alpha_1, \alpha_2 \in (0,1)$, chosen such that $\tau_A \in \{\tfrac{1}{5}, \tfrac{2}{5}, \tfrac{3}{5}, \tfrac{4}{5}\}$. Four Archimedean generators $\psi$ are considered -- Clayton, Frank, Gumbel and Joe -- each for four values of Kendall's $\tau$, $\tau_\psi \in \{\tfrac{1}{5}, \tfrac{2}{5}, \tfrac{3}{5}, \tfrac{4}{5}\}$. Reported values are averages over $1000$ Monte Carlo repetitions.}\label{fig:beta_cfg_vs_pick_mar_ol_all_vs_all}
\end{figure}
\begin{figure}[!ht]
	\centering
	\includegraphics[width=1\textwidth]{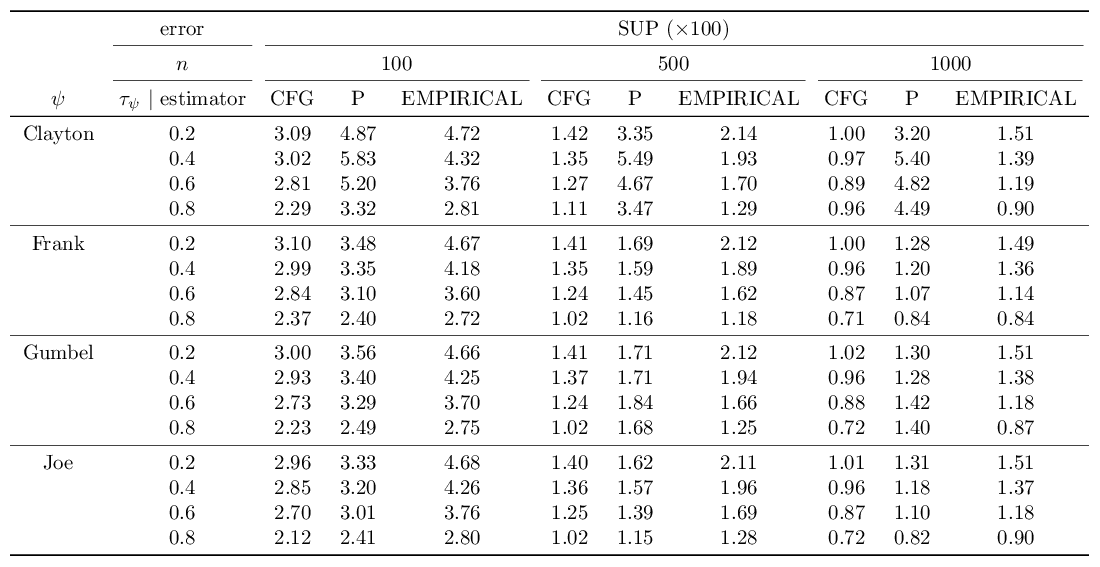}
	\caption{Average performance of the CFG-based estimator $C_{\alpha_n,n}^\mathbf{CFG}$, the Pickands-based estimator $C_{\alpha_n,n}^\mathbf{P}$ and the (bilinear extension of the) empirical copula estimator $\hat{C}_n$, measured via the uniform metric ($\times 100$), based on samples drawn from a $2$-dimensional Archimax copula $C_{\psi,A}$ with sample sizes $n \in \{100, 500, 1000\}$. The Pickands dependence function $A$ follows the Gumbel family with parameter $\alpha = \frac{5}{3}$, chosen such that $\tau_A = \tfrac{2}{5}$. Four Archimedean generators $\psi$ are considered -- Clayton, Frank, Gumbel and Joe -- each for four values of Kendall's $\tau$, $\tau_\psi \in \{\tfrac{1}{5}, \tfrac{2}{5}, \tfrac{3}{5}, \tfrac{4}{5}\}$. Reported values are averages over $1000$ Monte Carlo repetitions.}\label{fig:cop_cfg_vs_pick_gumbel_0.4}
\end{figure}
\begin{figure}[!ht]
	\centering
	\includegraphics[width=1\textwidth]{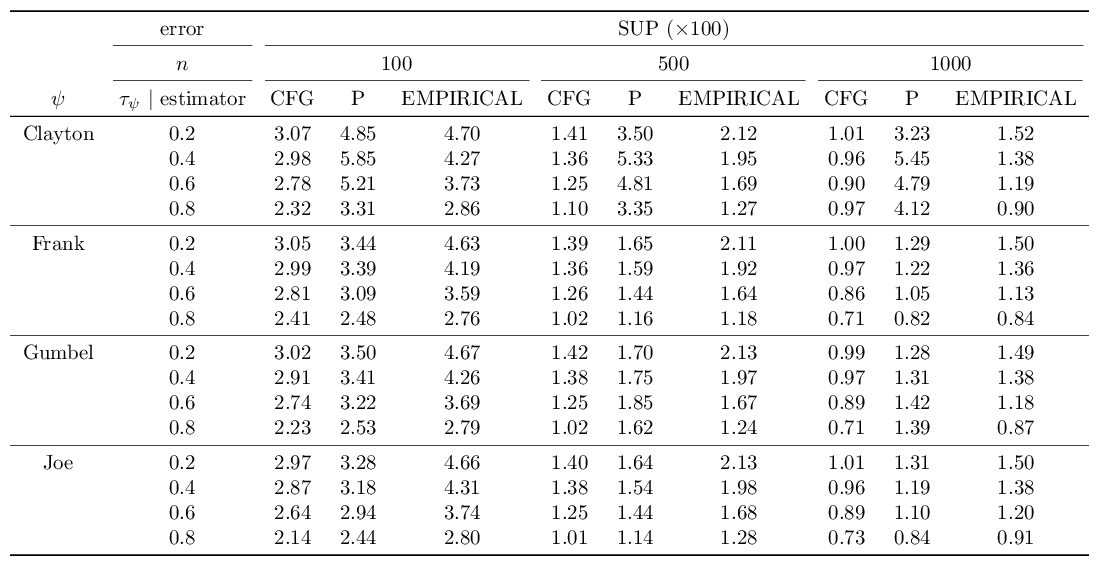}
	\caption{Average performance of the CFG-based estimator $C_{\alpha_n,n}^\mathbf{CFG}$, the Pickands-based estimator $C_{\alpha_n,n}^\mathbf{P}$ and the (bilinear extension of the) empirical copula estimator $\hat{C}_n$, measured via the uniform metric ($\times 100$), based on samples drawn from a $2$-dimensional Archimax copula $C_{\psi,A}$ with sample sizes $n \in \{100, 500, 1000\}$. The Pickands dependence function $A$ follows the Galambos family with parameter $\alpha = 0.9460773$, chosen such that $\tau_A = \tfrac{2}{5}$. Four Archimedean generators $\psi$ are considered -- Clayton, Frank, Gumbel and Joe -- each for four values of Kendall's $\tau$, $\tau_\psi \in \{\tfrac{1}{5}, \tfrac{2}{5}, \tfrac{3}{5}, \tfrac{4}{5}\}$. Reported values are averages over $1000$ Monte Carlo repetitions.}\label{fig:cop_cfg_vs_pick_galambos_0.4}
\end{figure}
\begin{figure}[!ht]
	\centering
	\includegraphics[width=1\textwidth]{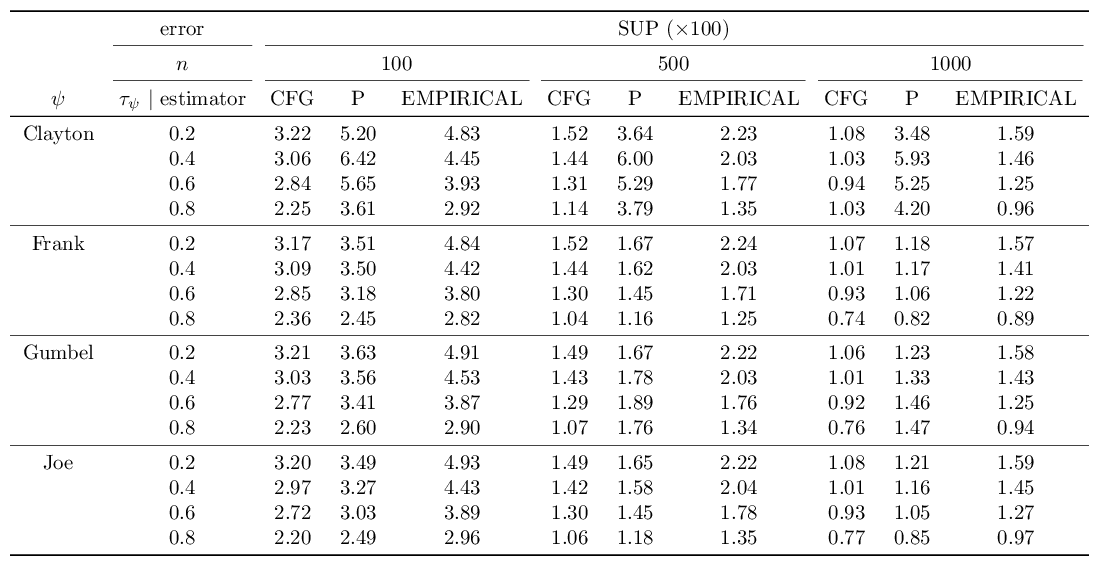}
	\caption{Average performance of the CFG-based estimator $C_{\alpha_n,n}^\mathbf{CFG}$, the Pickands-based estimator $C_{\alpha_n,n}^\mathbf{P}$ and the (bilinear extension of the) empirical copula estimator $\hat{C}_n$, measured via the uniform metric ($\times 100$), based on samples drawn from a $2$-dimensional Archimax copula $C_{\psi,A}$ with sample sizes $n \in \{100, 500, 1000\}$. The Pickands dependence function $A$ follows the Marshall-Olkin family with parameters $\alpha_1 = \frac{3}{7}$ and $\alpha_2 = \frac{6}{7}$, chosen such that $\tau_A = \tfrac{2}{5}$. Four Archimedean generators $\psi$ are considered -- Clayton, Frank, Gumbel and Joe -- each for four values of Kendall's $\tau$, $\tau_\psi \in \{\tfrac{1}{5}, \tfrac{2}{5}, \tfrac{3}{5}, \tfrac{4}{5}\}$. Reported values are averages over $1000$ Monte Carlo repetitions.}\label{fig:cop_cfg_vs_pick_mar_ol_0.4}
\end{figure}
\begin{figure}[!ht]
	\centering
	\includegraphics[width=.7\textwidth]{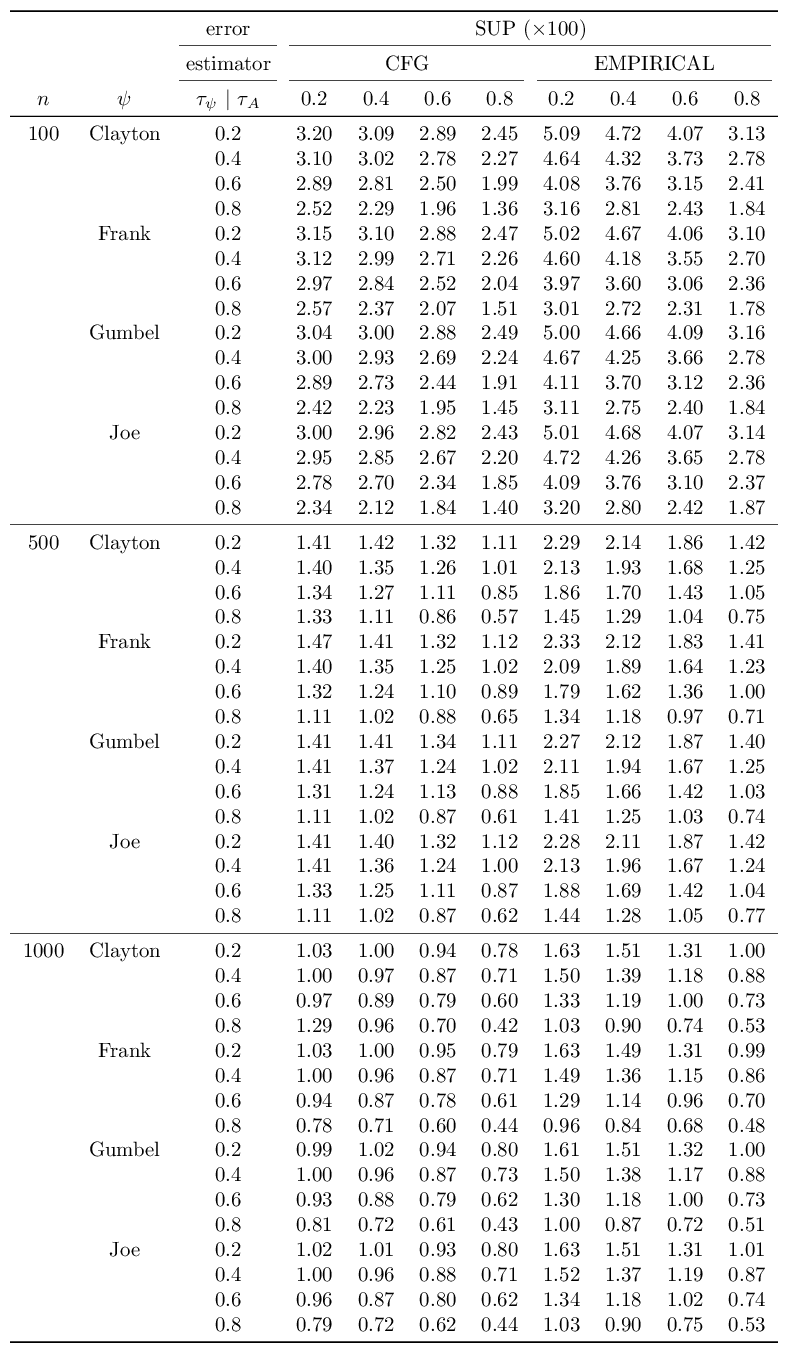}
	\caption{Average performance of the CFG-based estimator $C_{\alpha_n,n}^\mathbf{CFG}$ and the (bilinear extension of the) empirical copula estimator $\hat{C}_n$, measured via the uniform metric ($\times 100$), based on samples drawn from a $2$-dimensional Archimax copula $C_{\psi,A}$ with sample sizes $n \in \{100, 500, 1000\}$. The Pickands dependence function $A$ follows the Gumbel family with parameters $\alpha \in [1,\infty)$, chosen such that $\tau_A \in \{\tfrac{1}{5},\tfrac{2}{5},\tfrac{3}{5},\tfrac{4}{5}\}$. Four Archimedean generators $\psi$ are considered -- Clayton, Frank, Gumbel and Joe -- each for four values of Kendall's $\tau$, $\tau_\psi \in \{\tfrac{1}{5}, \tfrac{2}{5}, \tfrac{3}{5}, \tfrac{4}{5}\}$. Reported values are averages over $1000$ Monte Carlo repetitions.}\label{fig:cop_cfg_vs_pick_gumbel_all_vs_all}
\end{figure}
\begin{figure}[!ht]
	\centering
	\includegraphics[width=.7\textwidth]{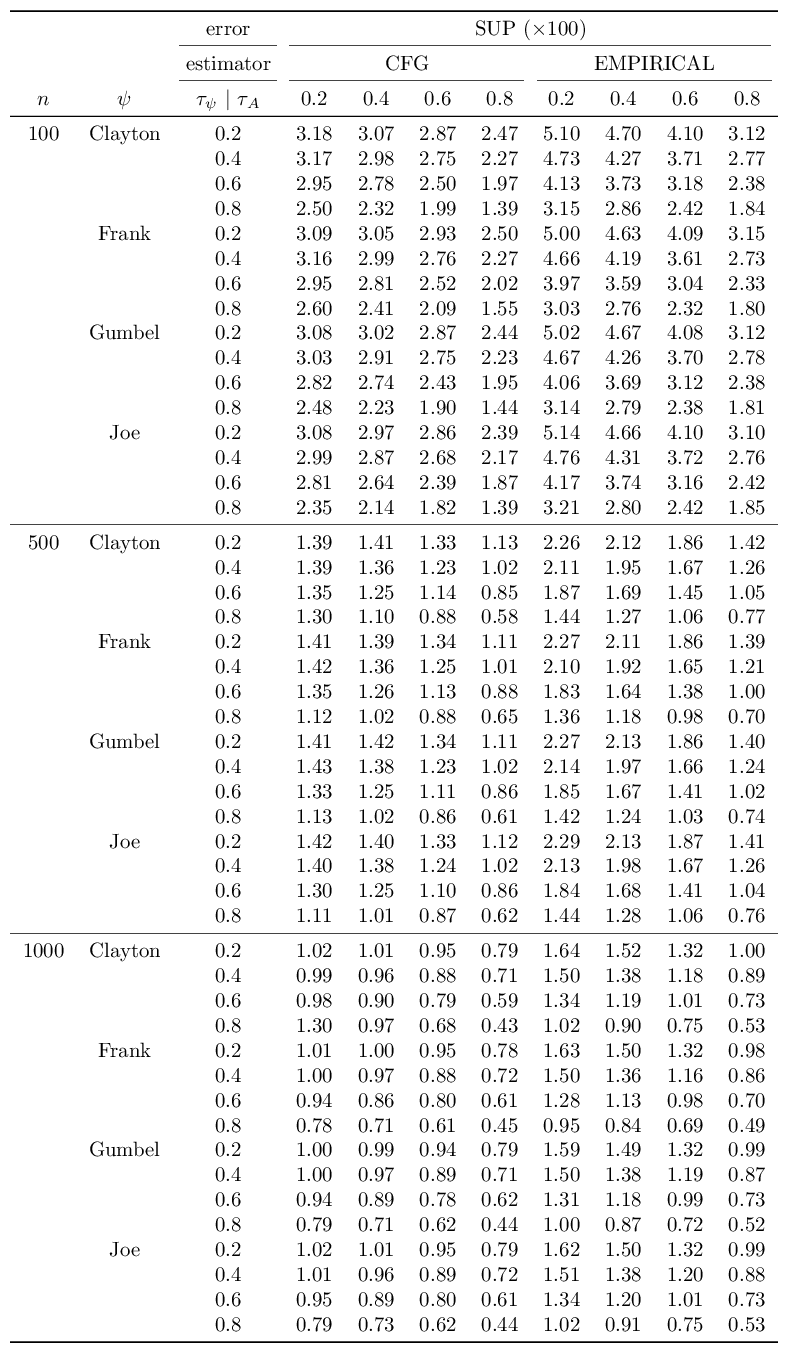}
	\caption{Average performance of the CFG-based estimator $C_{\alpha_n,n}^\mathbf{CFG}$ and the (bilinear extension of the) empirical copula estimator $\hat{C}_n$, measured via the uniform metric ($\times 100$), based on samples drawn from a $2$-dimensional Archimax copula $C_{\psi,A}$ with sample sizes $n \in \{100, 500, 1000\}$. The Pickands dependence function $A$ follows the Galambos family with parameters $\alpha \in (0,\infty)$, chosen such that $\tau_A \in \{\tfrac{1}{5},\tfrac{2}{5},\tfrac{3}{5},\tfrac{4}{5}\}$. Four Archimedean generators $\psi$ are considered -- Clayton, Frank, Gumbel and Joe -- each for four values of Kendall's $\tau$, $\tau_\psi \in \{\tfrac{1}{5}, \tfrac{2}{5}, \tfrac{3}{5}, \tfrac{4}{5}\}$. Reported values are averages over $1000$ Monte Carlo repetitions.}\label{fig:cop_cfg_vs_pick_galambos_all_vs_all}
\end{figure}
\begin{figure}[!ht]
	\centering
	\includegraphics[width=.7\textwidth]{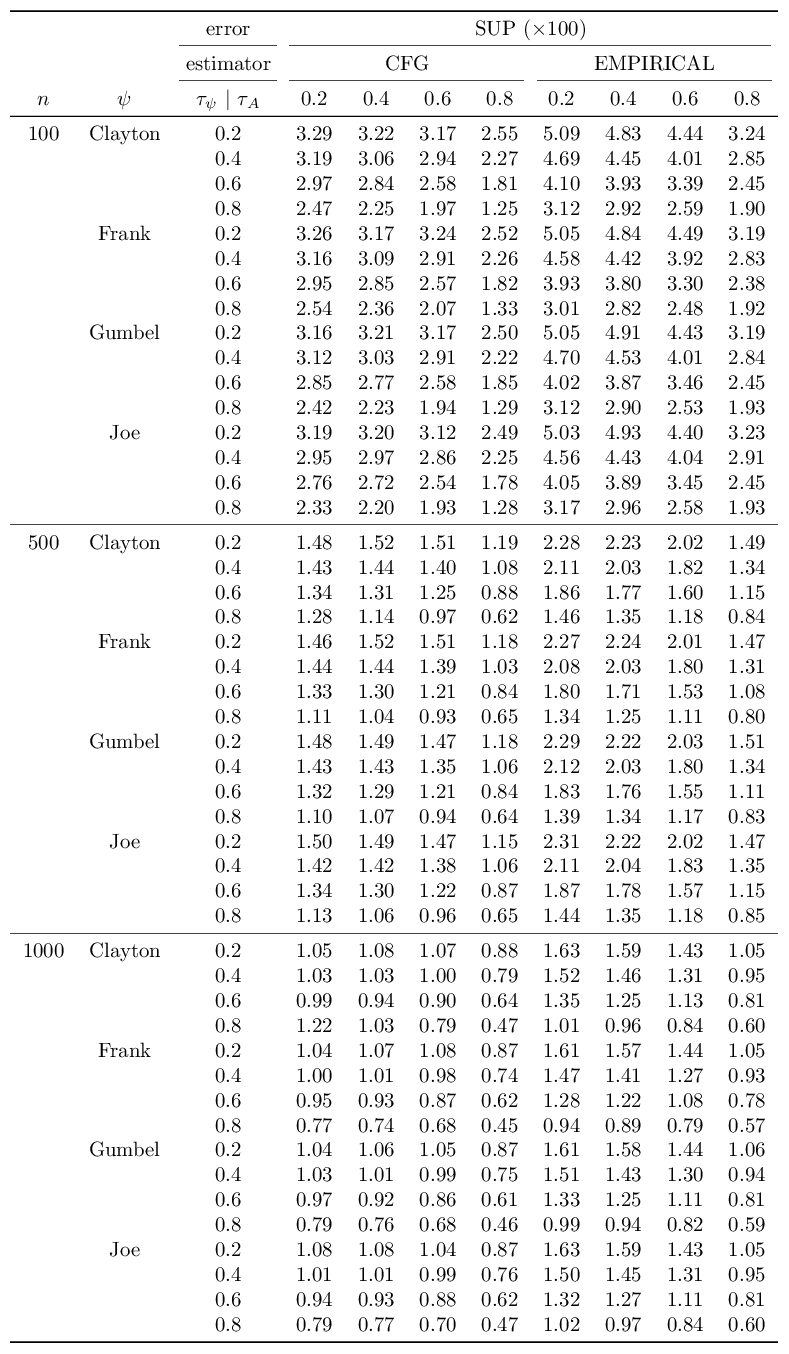}
	\caption{Average performance of the CFG-based estimator $C_{\alpha_n,n}^\mathbf{CFG}$ and the (bilinear extension of the) empirical copula estimator $\hat{C}_n$, measured via the uniform metric ($\times 100$), based on samples drawn from a $2$-dimensional Archimax copula $C_{\psi,A}$ with sample sizes $n \in \{100, 500, 1000\}$. The Pickands dependence function $A$ follows the Marshall-Olkin family with parameters $\alpha_1,\alpha_2 \in (0,1)$, chosen such that $\tau_A \in \{\tfrac{1}{5},\tfrac{2}{5},\tfrac{3}{5},\tfrac{4}{5}\}$. Four Archimedean generators $\psi$ are considered -- Clayton, Frank, Gumbel and Joe -- each for four values of Kendall's $\tau$, $\tau_\psi \in \{\tfrac{1}{5}, \tfrac{2}{5}, \tfrac{3}{5}, \tfrac{4}{5}\}$. Reported values are averages over $1000$ Monte Carlo repetitions.}\label{fig:cop_cfg_vs_pick_mar_ol_all_vs_all}
\end{figure}

\end{document}